\documentclass[12pt, a4paper, parskip=half, abstracton]{scrartcl}
 
 \pdfoutput=1 

\usepackage{array}
\usepackage{xcolor}
\usepackage{amscd,amssymb,amsfonts,amsmath,latexsym,amsthm}
\usepackage{hyperref}
\usepackage[all,cmtip]{xy}
\usepackage{mathrsfs}
\usepackage{graphicx}
\usepackage{bm} % for bold symbols in math mode \boldsymbol{}
\usepackage[nameinlink, capitalise]{cleveref}
\usepackage[normalem]{ulem}

\usepackage{slashed}
\usepackage[utf8]{inputenc}
\usepackage{microtype} % das verhindert live-texing bei mir. Koennen wir am Ende wieder einkommentieren.
\usepackage[english]{babel}
\usepackage{mathtools}
\usepackage{esvect}

\usepackage{tikz-cd}
\usetikzlibrary{decorations.markings,arrows.meta,bending}
\tikzset{->-/.style={decoration={
  markings,
  mark=at position #1 with {\arrow{>}}},postaction={decorate}}}
\tikzset{->-/.default=0.5}

\usepackage{pgfplots}

\pgfplotsset{compat=1.10}
\usepgfplotslibrary{fillbetween}
\usetikzlibrary{patterns}

\usepackage{etoolbox}
\def\hB{\hfill$\blacksquare$}% \newline\noindent}

\usepackage[nottoc]{tocbibind} % Bibliography im toc

\usepackage{defs_pp1}

\title{Coronas and Callias type operators in coarse geometry}
\author{
Ulrich Bunke\thanks{Fakult{\"a}t f{\"u}r Mathematik,
Universit{\"a}t Regensburg
\newline
ulrich.bunke@mathematik.uni-regensburg.de} \:and 
Matthias Ludewig \thanks{Institut f{\"u}r Mathematik und Informatik,
Universit{\"a}t Greifswald\newline
matthias.ludewig@uni-greifswald.de}
}

\numberwithin{equation}{section}
\newtheorem{theorem}{Theorem}[section] 
\newtheorem{prop}[theorem]{Proposition}
\newtheorem{lem}[theorem]{Lemma}

\newtheorem{ddd}[theorem]{Definition}
\newtheorem{kor}[theorem]{Corollary}
\newtheorem{ass}[theorem]{Assumption}

\newtheorem{construction}[theorem]{Construction}

\theoremstyle{remark}
\theoremstyle{definition}

\newtheorem{ex}[theorem]{Example}
\newtheorem{rem}[theorem]{Remark}

\newcommand{\Comod}{\mathbf{Comod}}

\newcommand{\ee}{\mathrm{e}}
\newcommand{\gee}{\mathrm{e}^{\mathrm{gr}}}

\newcommand{\EE}{\mathrm{E}}
\newcommand{\gEE}{\mathrm{E}^{\mathrm{gr}}}
\newcommand{\sepa}{\mathrm{sep}}

\newcommand{\an}{\mathrm{an}}

\newcommand{\nCalg}{C^{*}\mathbf{Alg}^{\mathrm{nu}}}
\newcommand{\gCalg}{(C^{*}\mathbf{Alg}^{\mathrm{nu}})^{\mathrm{gr}}}

\newcommand{\UBC}{\mathbf{UBC}}

\newcommand{\Yo}{\mathrm{Yo}}

\newcommand{\ho}{\mathrm{ho}}

\newcommand{\Res}{\mathrm{Res}}

\newcommand{\TB}{\mathbf{TB}}

\newcommand{\BC}{\mathbf{BC}}

\newcommand{\Cofib}{\mathrm{Cofib}}

\newcommand{\Fib}{{\mathrm{Fib}}}

\newcommand{\incl}{\mathrm{incl}}

\newcommand{\cP}{\mathcal{P}}

\newcommand{\CAlg}{{\mathbf{CAlg}}}

\newcommand{\cI}{{\mathcal{I}}}

\newcommand{\cZ}{{\mathcal{Z}}}

\newcommand{\cW}{{\mathcal{W}}}

\newcommand{\sfin}{\mathrm{sfin}}

\newcommand{\triv}{\mathrm{triv}}

\newcommand{\spec}{{\mathtt{spec}}}
\newcommand{\cO}{{\mathcal{O}}}
\newcommand{\cU}{{\mathcal{U}}}
\newcommand{\cY}{{\mathcal{Y}}}
\newcommand{\Var}{\mathrm{Var}}

\newcommand{\lf}{\mathrm{lf}}

\newcommand{\lc}{{lc}}

\renewcommand{\Dirac}{\slashed{D}}

\newcommand{\Spc}{\mathbf{Spc}}

\newcommand{\simp}{\mathrm{simp}}

\newcommand{\op}{\mathrm{op}}

\newcommand{\KK}{\mathrm{KK}}

\newcommand{\cone}{\mathrm{cone}}

\newcommand{\nCcat}{C^{*}\mathbf{Cat}^{\mathrm{nu}}}
\newcommand{\gCcat}{(C^{*}\mathbf{Cat}^{\mathrm{nu}})^{\mathrm{gr}}}

\renewcommand{\Spc}{\mathbf{Spc}}

\newcommand{\exa}{\mathrm{ex}}

\newcommand{\K}{\mathrm{K}}

\newcommand{\odd}{\mathrm{odd}}

 \newcommand{\gr}{\mathrm{gr}}
 
 \newcommand{\gK}{K^{\gr}}
	 \renewcommand{\ind}{\mathrm{ind}}
\renewcommand{\cX}{\mathscr{X}}
 
 \DeclareMathOperator{\stimes}{\hat{\otimes}}
 \newcommand{\TildeDirac}{\tilde{\smash{\slashed{D}}\vphantom{D}}}

\begin{document}  \maketitle 
\begin{abstract}
We interpret the coarse symbol and index class of a Callias type Dirac operator 
$\Dirac+\Psi$ on a manifold $M$ as a pairing between the coarse   symbol and index classes 
associated to $\Dirac$ and K-theory classes of the coarse corona of $M$ or $M$ itself determined by $\Psi$.
Local positivity of $\Dirac$ and local invertibility of $\Psi$ are incorporated in terms of support conditions  on the $K$-theoretic level.
  \end{abstract}

  \tableofcontents
  
  \section{Introduction}
  
The purpose of this paper  is to embed
 the index theory of  Callias type  Dirac operators  on   complete Riemannian manifolds  into the coarse homotopy theory  developed in \cite{buen, ass}.  Our main  innovation  is \cref{wegokwpergerfwefrwefwf} which expresses  the   index and symbol of Callias type operators
in terms of the coarse  corona and symbol pairings introduced  \cref{gojkwpergerwferfewrfwrefw} and
  \cref{grjweopgergweffefwef}.
  
 One of the achievements of the present paper  is to provide a $K$-theoretic interpretation of  Callias type Dirac operators beyond the Fredholm index case. 
 As an application of the theory we
discuss a version of the partitioned manifold index theorem \cite{zbMATH04109704}  and its multipartitioned generalization \cite{Siegel:2012aa, MR3704253, MR3825192}.  
New aspects are that we allow multipartitions along non-compact   submanifolds, work with the refined symbol classes in local $K$-homology,  and that we  take the support conditions into account.

 In the present paper we work in the setup of coarse homotopy theory developed in \cite{buen, ass}
 which is based on the categories $\BC$ and $\UBC$ of    bornological coarse  and uniform bornological coarse  spaces and the $KU$-module-valued coarse $K$-homology functor 
 \[
 K\cX:\BC\to \Mod(KU)
 \]
 from \cite{coarsek}.
 In order to keep this paper selfcontained, we review the relevant definitions and results of this theory.
The wish to present the setup  in a coherent way and to give at least complete definitions
accounts for most of  the length of this paper.

%\ml{[Ich habe mal Unter-Ueberschriften eingefuegt um die recht lange Intro etwas zu strukturieren.]}

\paragraph{Coarse index theory with supports.}
 As observed by J. Roe \cite{roe_coarse_cohomology} the basic global $K$-theoretic invariant of a generalized Dirac operator $\Dirac$ on a complete Riemannian mani\-fold $M$
 is its coarse index class $\ind\cX(\Dirac)$ in the coarse $K$-homology $K\cX(M)$ (see \cref{grwgjrogwregwre} and \cref{gjkregopregwergfwrefrfwrf}). This index class detects the invertibility of $\Dirac$
 on the natural $L^{2}$-space, but also retains some coarsely local information, e.g., on which parts of $M$ the zero order term in the Weizenboeck formula for $\Dirac$ is uniformly positive \cite{Roe:2012fk}.  

Classically, the local information on $\Dirac$ is  homotopy theoretically encoded by the analytic  symbol class
 $\sigma^{\an}(\Dirac)$ in  the locally finite analytic $K$-homology $ K^{\an}(M)\simeq \KK(C_{0}(M),\C)$ of $M$ (see \cref{kogpwegreffrefrfw} and \cref{jiorgwpegwerfrefrfw}, but note that in the present paper we prefer to work with $E$-theory).  
 The analytic symbol class  can, e.g.,  be paired with compactly supported $K$-theory classes of $M$. 
 If one realizes {$KK$}-theory classes by Kasparov modules \cite{kasparovinvent}, then 
  the results of these pairings can be interpreted
 in terms of indices of Fredholm operators constructed from $\Dirac$, see \cref{kgopwrerfwrfrfw}. But (somewhat surprisingly) one can also recover the coarse index $\ind\cX(\Dirac)$  
 from the symbol class $\sigma^{\an}(\Dirac)$, see \cref{jiogpgweergerfwef}.

In order to capture the symbol of $\Dirac$ within the coarse $K$-homology calculus,
  in  \cite{Bunke:2018aa, Bunke:2017aa} we proposed to 
consider a refined symbol class $\sigma(\Dirac)$ in the local $K$-homology $K^{\cX}(M)$ (see \cref{kogpwergrefrefwrfw} and \cref{lphegeetrhetrgtgteg}). There is a natural transformation $a:K^{\cX}\to K\cX$  (see \eqref{fqweqwfqewdq}) of local homology theories called the index map
 such that \begin{equation}\label{hertopetrgvertg}a_{M}(\sigma (\Dirac))=\ind\cX(\Dirac)\ .
\end{equation}  The classical symbol class in analytic $K$-homology can be recovered in terms of the  natural Paschke duality  transformation of local homology theories
 $p:K^{\cX}\to K^{\an}$ (see \eqref{r3rgrgsfegeg})  by $$p_{M}(\sigma(\Dirac))=\sigma^{\an}(\Dirac)\ .$$

The advantage of considering the  symbol $\sigma(\Dirac)$ instead of $\sigma^{\an}(\Dirac)$ is that one can incorporate  local positivity of $\Dirac$
geometrically in terms of additional support conditions 
  \cite{Bunke:2018aa,Bunke:2017aa}.   
    In coarse geometry, the right notion of a subset is a big family $\cY$,  i.e., a filtered family  of subsets which is closed under forming coarse thickenings. 
    Given such a big family $\cY $ on $M$ one can define the local $K$-homology $K^{\cX}_{\cY}(M)$ (see \cref{jgkgopekgregfwrefrwf})
 with support on $\cY$ and the coarse $K$-homology $K\cX(\cY):=\colim_{Y\in \cY} K\cX(Y)$ of $\cY$.
 The index map {then} refines to a map 
 \[
 a_{M,\cY}:K^{\cX}_{\cY}(M)\to K\cX(\cY),
 \]
 see \eqref{fwerqfdwedqewdqewdq}.

 Assume now that there exists  a member   $Y$  of $\cY$   such that $\Dirac^{2}$ is {positive} on $M\setminus Y$ in the sense that the  
  quadratic form determined by $\smash{\Dirac^{2}}$ on smooth sections which are compactly supported on $M\setminus \bar Y$  is   positive.    
  In  \cref{lphegeetrhetrgtgteg}, 
  we  define a refined symbol class $\sigma_{\cY}(\Dirac) $ in $K^{\cX}_{\cY}(M)$ such that 
  \[
  {
  a_{M,\cY}(\sigma_{\cY}(\Dirac)) = \ind\cX(\Dirac,\mathrm{on}\,\cY) }
  \]
   in $K\cX(\cY)$. 
 
 For example,  if $\Dirac$ is positive on all of $M$, then $\Dirac$ is invertible and $\ind\cX(\Dirac)=0$. 
  {Taking $\cY={\{\emptyset\}}$}, the class
  $\smash{\sigma_{{\{\emptyset\}}}}(\Dirac)$ in $\smash{K^{\cX}_{{\{\emptyset\}}}(M)}$  (called the $\rho$-invariant) captures the  reason for  the equality $ \ind\cX(\Dirac)=0$ in a homotopy theoretic manner.

  The   local {$K$-}homology $K^{\cX}$ (with supports) and the index map $a$  are instances of general constructions
  which can be performed with   arbitrary coarse homology theories, see   \cref{owkgpergwrefwerfw} and \cref{rtkohprethergetetg}. In the present paper we  just applied them to the coarse $K$-homology theory $K\cX$. These constructions are functorial with respect to natural transformations of coarse homology theories, an observation which waits to be exploited in the future. Furthermore, the index map is a natural transformation of local homology theories (say on the spectrum level) and therefore automatically compatible with Mayer-Vietoris boundary maps.
    In contrast, the  construction of the classical analytic index map \[
    a_{M}^{\an}:K_{*}^{\an}(M)\to K\cX_{*}(M)
    \]
     (see, e.g., \cite[Sec. 5 \& 6]{hr}) is based on complicated  analytic arguments using specific models of $K_{*}^{\an}(M)$ and a group-level Paschke duality  isomorphism
  which is not obviously liftable to the spectrum level in a functorial manner.

 \paragraph{Callias type operators.}
  
 By \cref{kopththretg} a Callias type operator is a generalized Dirac operator of the form $\Dirac+\Psi$ obtained from
 a generalized Dirac operator $\Dirac$ by adding a potential $\Psi$  such that $\Dirac\Psi+\Psi\Dirac$ is zero order (see \cref{8wrtlhkerhterhertrtgtrgetg}).
 The relevant index theoretic information on $\Dirac+\Psi$ is concentrated on the subset of $M$ on which the potential is not  positive in the sense that $\Psi^{2}$ dominates $\Dirac\Psi+\Psi\Dirac$ (see \cref{tlhpetrgggegrt}).  
 
 The classical condition is that the potential is positive outside of a compact subset of $M$ which turns the Callias type operator into a Fredholm operator. Partial positivity as considered in the present paper {was} already studied in      \cite{Guo_2022}.

 We consider a big family    $\cY $ such that the potential is positive    away from $\cY$.
 Then $\Dirac+\Psi$ is {positive} away from $\cY$.
 Since this positivity is caused by the  zero order term of $\Dirac+\Psi$  in \cref{kogpwergwerffwref} we can define
 a refined symbol class with support $\sigma(\Dirac+\Psi,\mathrm{on}\,\cY)$  in $K^{\cX}(\cY):=\colim_{Y\in \cY}K^{\cX}(Y)$.
 
  If $\Dirac$ was already {positive} away from a big family $\cZ$ 
 and $\Psi$ is asymptotically constant away from $\cZ$, then $\Dirac+\Psi$ is {positive} away from $\cY\cap \cZ$ with index $\ind\cX(\Dirac+\Psi,\mathrm{on}\,\cY\cap \cZ)$ in $K\cX(\cY\cap \cZ)$, and 
 we  define a refined symbol class
 $\sigma_{\cZ}(\Dirac+\Psi,\mathrm{on}\,\cY)$ in $K^{\cX}_{\cZ}(\cY)$.
  
  We want to  interpret the symbol class $\sigma_{\cZ}(\Dirac+\Psi,\mathrm{on}\,\cY)$  and the coarse  index class  
  \[
  \ind\cX(\Dirac+\Psi,\mathrm{on}\,\cY\cap \cZ)=a_{M,{\cZ}}(\sigma_{\cZ}(\Dirac+\Psi,\mathrm{on}\,\cY))
  \]
    in terms of 
pairings  between the   symbol  class $\sigma_{\cZ}(\Dirac)$   or the index class $\ind\cX(\Dirac,\mathrm{on}\,\cZ)$  with suitable $K$-theory classes  determined by  the potential $\Psi$.

{The} big family $\cY$   gives rise to the  uniform  $\cY$-corona $\partial_{u}^{\cY}M$, a compact topological space with the potential to capture the behaviour of $\Psi$ at $\infty$ {(see \cref{zkophergrtgtrgetrg})}. 
As part of the coarse $K$-homology calculus we  introduce the coarse corona pairing \cref{gojkwpergerwferfewrfwrefw}
\[
-\cap^{\cX}-:{K(\partial^{\cY} M)}\times  K\cX(\cZ)\to  \Sigma K\cX(\cY\cap \cZ)
\]

and 
the coarse symbol pairing \cref{grjweopgergweffefwef}
\[
-\cap^{\cX\sigma}-:K(\cY)\times K_{\cZ}^{\cX}(M)\to K^{\cX}_{\cZ}(\cY)\ ,
\]
involving the topological $K$-theory of  the coarse corona $\partial^{\cY}X$  (\cref{rkopthtrhegertg}) and  of $\cY$.
 Both pairings  are components of extra-natural transformations, see \cref{RemarkExtranatural}.

 Under the assumptions made above, the  potential $\Psi$ gives naturally rise to a class $[\Psi]$ in $K(\cY)$. Our  main result is \cref{wegokwpergerfwefrwefwf}, {which states that under suitable conditions on $\Dirac$ and $\Psi$, one has}
$$\sigma_{\cZ}(\Dirac+\Psi, \mathrm{on}\,\cY)=[\Psi]\cap^{\cX\sigma} \sigma_{\cZ}(\Dirac)$$ in $K^{\cX}_{\cZ}(\cY)$.
It provides the interpretation of the symbol class of the Callias type operator in terms of 
the coarse $K$-homology calculus.

We now consider the boundary map
\[
\partial:K(\partial_{u}^{\cY}M)
\to \Sigma K(\cY)
\]
 from  \eqref{fwerferfervwerv}.
We   assume that there exists a class $p$ in $K(\partial_{u}^{\cY}M)$ with $\partial p =[\Psi]$.
Then by the compatibility of $\partial$ with the index map shown in  \cref{erokgpwergrwefwerfwerfwer}
we get \begin{equation}
\label{gerwegwergwrgwergre}
\ind\cX(\Dirac+\Psi,\mathrm{on}\,\cY\cap \cZ)= c^{*} p \cap^{\cX} \ind\cX(\Dirac, \mathrm{on}\,\cZ)\end{equation}
in $K\cX(\cY\cap \cZ)$, where $c$ is the comparison map \eqref{ComparisonMap}.
As a consequence of \eqref{gerwegwergwrgwergre},  the assumption $[\Psi]=\partial p$  implies the non-obvious fact that the coarse index  
$\ind\cX(\Dirac+\Psi,\mathrm{on}\,\cY\cap \cZ)$ only depends on $\Dirac$ via its coarse index $\ind\cX(\Dirac, \mathrm{on}\,\cZ)$, instead of the finer information contained it its symbol class $\sigma_{\cZ}(\Dirac)$.

\begin{ex} \label{kgopwrerfwrfrfw}We assume that 
  $M$ is connected and that $\cY=\cB$ is the family of bounded subsets. Then we have
$K\cX(\cB)\simeq KU$ 
and under this identification $\ind\cX(\Dirac+\Psi,\mathrm{on}\,\cB)$ is the Fredholm index
of $\Dirac+\Psi$.  The class $[\Psi]$ in $K_{\cB}(M)$ is a compactly supported $K$-theory class of $M$ and the Fredholm index
of $\Dirac+\Psi$ is the result of the pairing of $\sigma^{\an}(\Dirac)$ with the class $[\Psi]$.

 In this case
$\partial^{\cB}M$ is the Higson corona $\partial_{h}M$ of $M$.  If $[\Psi]=\partial p$ for some class $p$ in $K(\partial_{h}M)$, then  the interpretation of the Fredholm  index of  $\Dirac+\Psi$ in terms of a pairing between the $K$-theory $K(\partial_{h}M)$ of 
the Higson corona and symbol $\sigma^{\an}(\Dirac)$  is known from 
\cite{MR1348799}.  
 \hB\end{ex}

 \paragraph{Structure of the paper.}

In the first  sections we will recall the main definitions and constructions
 of the coarse homotopy calculus, $K$-theory of $C^{*}$-algebras and categories, and coarse index theory of Dirac operators. 
 Prerequisites for reading this paper are basic homotopy theoretic language (e.g., {elements of the theory of} stable $\infty$-categories), basic $C^{*}$-algebra theory (continuous function calculus, maximal tensor product, Gelfand duality) and some global analysis (Riemannian manifold and generalized Dirac operators).

%The presentation of the new results and their applications starts with  \cref{ewrokgpwergwerffwef}. 

 In \cref{rfiohqeiorfqfewdqddq} we recall from \cite{buen, ass} the definitions of the categories $\BC$ of bornological coarse spaces and $\UBC$ of uniform bornological coarse spaces and the cone functor $\cO^{\infty}:\UBC\to \BC$. 
%We also discuss topological bornological spaces $\TB$.  
 
 In \cref{gojkperfrefrwfwrefw2} we recall  from \cite{buen} the notion of a coarse homology theory in general and some basic calculations.

In   \cref{kogpewgwergergwergwerfw} we recall from \cite{ass} and \cite{buen} the notions of a local and locally finite homology theory. 
We  explain the construction of the local homology theory associated to a coarse homology theory and the index map. 
We further introduce the local homology theory with supports associated to a coarse homology theory. 
  
  In \cref{gkpoegsergreg} we give a descriptive presentation  $E$- and $K$-theory via universal properties 
  for  (graded) $C^{*}$-algebras and categories in the style of \cite{Bunke:2023aa, budu}.
 
 \cref{gojkperfrefrwfwrefw} is devoted to the construction of the coarse $K$-homology functor $K\cX:\BC\to \Mod(KU)$ in terms of the  $C^{*}$-categories of controlled Hilbert spaces  taken from \cite{buen} and \cite{coarsek}.

 In \cref{ojgpwegwregrfrwef} and \cref{lhpkgerhretgrge} we introduce various $C^{*}$-algebras of functions associated to
 a big family by putting support conditions on the functions themselves or their   variation.
 We introduce the corona  associated to a pair of big families  and show a crucial technical commutator estimate in \cref{kohpkertphokgpertge}.

 In \cref{gkopegrergweferfwfr} we define the analytic locally finite $K$-homology functor $K^{\an}(-):=\EE(C_{0}(-),\C)$ 
 and verify that it is a locally finite local homology theory. 
 We then recall from  \cite{quro, bel-paschke}
 the Paschke transformation $p_{M}:K^{\cX}(M)\to K^{\an}(M)$ and  provide conditions on $M$ ensuring that  $p_{M}$ it is an equivalence. We also explain the factorization of the index map over a map $K^{\an}(M)\to K\cX(M)$
 which allows to recover the coarse index of $\Dirac$ from its  analytic symbol in \cref{jiogpgweergerfwef}.
 
  In  Sections \ref{regijowergefrfwrfrf} and \ref{regijowergefrfwrfrf2}, we construct the coarse corona pairing  and the coarse symbol pairing and state and verify their formal functorial properties.
 
 In \cref{rijeogegergwergf} we recall from \cite{Bunke:2017aa} the details of the construction of coarse index $\ind\cX(\Dirac,\mathrm{on}\,\cY)$ in   $K\cX(\cY)$   of a Dirac operator  which is  {positive}  away from a big family  $\cY$. We state the suspension theorem (identifying the index operators on products $\R\otimes M$) and the coarse  relative index theorem expressing the locality of the coarse index. 
  
  In \cref{gkopwerfefrefwfref} we first recall from \cite{Bunke:2018aa} how to capture symbols of  Dirac operators  which are {positive} away from $\cY$ as classes $\sigma_{\cY}(\Dirac)$  in $K^{\cX}_{\cY}(M)$.
  
  The main result of the present paper is   \cref{wegokwpergerfwefrwefwf} stated in  \cref{ewrokgpwergwerffwef} which identifies the symbol of the Callias type operator $\Dirac+\Psi$ with support on $\cY$ as the coarse symbol paring $[\Psi]\cap^{ \cX\sigma} \sigma(\Dirac)$. 
 
 In \cref{kgpwergrewfwwrfwref} we discuss the coarse version of the Dirac-goes-to-Dirac principle. 

 This principle
 is then  used in \cref{koregperegfrefrfrefw} in the discussion of the relation between our pairings  and  Mayer-Vietoris boundary maps leading to a version of the partitioned manifold theorem. 

In \cref{ewgkorpegrefwrf} we  derive  a version of the  multipartitioned manifold index theorem  by an interated application of the partitioned manifold theorem. 

Finally, in \cref{tlkhgwteggfregwgwergwregwregw}
we argue that the slant products  considered in \cite{Engel_2022} can be derived by specializing our pairings, and  that their basic properties  
can be derived  from general facts about these pairings.

 Most of the material discussed in this paper has an immediate generalization to the equivariant case for discrete group actions. 
 We decided not to formulate everything in this generality in order to keep the complexity of the paper reasonable.  
 
   {\em Acknowledgement: The authors were supported by the SFB 1085 (Higher Invariants) funded by the Deutsche Forschungsgemeinschaft (DFG).}

  \section{Coarse cohomology theories}
   
  \subsection{Uniform bornological coarse spaces and cones}\label{rfiohqeiorfqfewdqddq}
  
  In this section we recall the notions of  bornological coarse spaces, uniform bornological coarse spaces, topological bornological spaces and the cone construction  \cite{buen, equicoarse}. We just mention that coarse geometry has been invented by J. Roe \cite{roe_exotic_cohomology, roe_lectures_coarse_geometry} and  refer to  the references above  for further links to the literature on coarse geometry and attributions. 
 
  Let $X$ be a set. 
  \begin{ddd}\mbox{}
     \begin{enumerate}
  \item A bornology on $X$ is a subset  $\cB$  of the power set of $X$ which covers $X$ and is closed under forming
  subsets and finite unions.
  \item A coarse structure on $X$ is a subset $\cC$   of the power set of $X\times X$ which is closed under forming subsets, finite unions, compositions $(U,U')\mapsto U\circ U$ (in the sense of relations), the flip of $X\times X$, and which  contains the  diagonal $\diag(X)$ of $X$.
  \item A uniform structure on $X$  is a subset $\cU$   of the power set of $X\times X$ of subsets  $U$ with $\diag(X)\subseteq U$    which is closed under forming supersets, finite intersections, compositions, the flip, and which has the property  that for any $U$ in $\cU$ there exists a $V$ in $\cU$ with $V\circ V\subseteq U$.
   \end{enumerate}\end{ddd}
For a subset $Z$ of $X$ and a subset $U$ (called an entourage) of $X\times X$ we let 
\begin{equation}\label{fewrihjikj}
U[Z]:=\{x\in X\mid \exists z\in Z : (x,z)\in U\}
\end{equation}   
denote the $U$-thickening of $Z$.

Let $X$ be a set with a coarse structure $\cC$. 
The correct  notion of a subset in coarse geometry
is the notion of a big family. % is  the correct notion of a subset in coarse geometry.

\begin{ddd}\label{jieogewrgwref}
A big family is a {nonempty} %\fml{Sonst w\"aren sowohl $\emptyset$ als auch $\{\emptyset\}$ big families. Aber wir wollen glaube ich nur zweiteres. Ansonsten w\"are zum Beispiel die lokalisierte Roe-Algebra leer (statt null).} 
subset $\cY$ of the power set of $X$ which is closed under taking subsets, finite unions and $U$-thickenings for all $U$ in $\cC$.
\end{ddd}

We consider a big family $\cY$  as a  partially ordered set with respect to the inclusion relation. 
{Because $\cY$ is closed under taking finite unions, this partially ordered set is filtered}.
 
\begin{rem}
In the literature on bornological coarse spaces with the basic reference \cite{buen}, a  big family on $X$  is a family of subsets $(Y_{i})_{i\in I}$  indexed by a filtered poset {$I$} such that $i\le j$ implies $Y_{i}\subseteq Y_{j}$ and such that for every $i$ in $I$ and $U$ in $\cC$ there exists $i'$ in $I$ with $U[Y_{i}]\subseteq Y_{i'}$.  
In this convention the notion introduced in \cref{jieogewrgwref} is the special case of a selfindexing big family
which is in addition complete (this is the closedness under taking subsets).

Every big family $(Y_{i})_{i\in I}$ in the sense of  \cite{buen}  generates a big family $\cY:=\{Y\subseteq X\mid \exists i\in I : Y\subseteq Y_{i}\}$ as defined in \cref{jieogewrgwref}.
The canonical map $I\to \cY$, $i\mapsto Y_{i}$ of partially ordered sets  is cofinal.

One can define an equivalence relation on the collection of big families in the sense of \cite{buen}
such that $(Y_{i})_{i\in I}\sim (Y'_{i'})_{i'\in I'}$ if and only if the associated big families in the sense of \cref{jieogewrgwref} satisfy $\cY=\cY'$.

Since the constructions  with big families occuring in the present paper  only depend on
 the equivalence classes
   it suffices to work with complete selfindexing big families
as introduced in \cref{jieogewrgwref}.\hB
\end{rem}

 % \begin{ddd} A big family on $X$  is a family of subsets $(Y_{i})_{i\in I}$  indexed by a filtered poset \ml{$I$} such that $i\le j$ implies $Y_{i}\subseteq Y_{j}$ and such that for every $i$ in $I$ and $U$ in $\cC$ there exists $i'$ in $I$ with $U[Y_{i}]\subseteq Y_{i'}$. \end{ddd} 
 %A big family is  the correct notion of a subset in coarse geometry.
 
 \begin{ex}\label{BigFamilyConstr}
 If $Y$ is a subset of $X$, then 
\begin{equation}\label{ExampleGeneratedBigFamily}
  \{Y\}:=\{Z\subseteq X \mid \exists U\in \cC : Z\subseteq U[Y]\} .
\end{equation}
is called the big family generated by $Y$.
If $\cY$ and $\cZ$ are big families in $X$, then we define the big families
\begin{equation}\label{ExampleCapBigFamilies}
\cY\cap\cZ := \{Y\cap  Z \mid Y\in \cY ~ \& ~ Z\in \cZ \}
\end{equation}
and
   \begin{equation}\label{ExampleCupBigFamilies}\cY\cup \cZ:=\{W\subseteq X\mid \exists Y\in \cY\& Z\in \cZ:W\subseteq Y\cup Z\}
\end{equation} 
in $X$.

If $\cY$ is a big family in $X$ and $\cY'$ is a big family in $X'$, then we define the big family
\begin{equation}\label{ExampleTimesBigFamilies}
  \cY \times \cY' := \{Z \subseteq X\times X' \mid \exists Y \in \cY, ~ Y' \in \cY :Z\subseteq Y\times  Y'\}
\end{equation} 
in $X \times X'$.
In order to    simplify the notation we will set $X\times \cY':=\{X\}\times \cY'$. 

If $\cY$  is a big family in $X$ and $Z \subseteq X$ is a subset equipped with the induced coarse structure, then we define the induced big family \begin{equation}\label{koprhregertgerg}
   Z \cap \cY := \{ Z \cap Y \mid Y \in \cY\}
 \end{equation} 	
   in $Z$. \hB
\end{ex}

A big family on a topological space $X$  is a family of subsets $\cY$
which is closed under taking subsets and finite unions and has property that for every $Y$ in $\cY$ there exists $Y'$ in $\cY$ which contains an open neighbourhood of $\bar Y$.

We consider the following compatibility relations between bornologies, coarse structures and uniform structures.
\begin{ddd}\mbox{}
\begin{enumerate}
\item A uniform structure $\cU$  is compatible with a coarse structure $\cC$ if $\cC\cap \cU\not=\emptyset$.
\item A bornology is $\cB$ compatible with a coarse structure $\cC$ if it is a big family.
\item A bornology $\cB$ is compatible with a topology on $X$  if it is a big family (in the topological sense).
\end{enumerate}
\end{ddd}

Note that a compatible bornology (in the coarse or topological sense) is simply a big family which covers the whole space.

We consider the following conditions on maps $f:X\to Y$  between sets equipped with (some) of the above structures.

\begin{ddd}\mbox{}
\begin{enumerate}
\item  $f$ is proper if $f^{-1}(\cB_{Y})\subseteq \cB_{X}$.
\item $f$ is bornological, if $f(\cB_{X})\subseteq \cB_{Y}$.
\item   $f$ is controlled if $f(\cC_{X})\subseteq \cC_{Y}$.
\item $f$ is uniform if $f^{-1}(\cU_{Y})\subseteq \cU_{X}$.
\end{enumerate}
\end{ddd}

\begin{ex}
Let $X$ und $X'$ be sets with coarse structures and $f:X'\to X$ be a controlled map. 
If $\cY$ is a big family on $X$, then generalizing \eqref{koprhregertgerg} we define  the  induced  big family
\[
f^{-1}(\cY):= \{ Z'\subseteq X' \mid \exists Y\in \cY : Z'\subseteq f^{-1}(Y)\} 
\] 
 {in} $X'$. 
For a subset $Y$ of $X$ we have $\{f^{-1}(Y)\}\subseteq f^{-1}(\{Y\})$, but this inclusion is in general not an equality.
\hB
\end{ex}

  We will consider the following categories:
  \begin{ddd}\mbox{}
  \begin{enumerate}
  \item $\BC$ (bornological coarse spaces): Sets with coarse structures and compatible bornologies  and controlled and proper maps.
  \item $\UBC$ (uniform bornological coarse spaces): Sets with coarse structures, compatible bornologies  and uniform structures  and controlled, uniform and proper maps. 
  \item $\TB$ (topological bornological spaces): Topological spaces with compatible bornological structures and proper continuous maps.
  \end{enumerate}
  \end{ddd}
Recall that a uniform structure $\cU$ on a set $X$ induces a topology on $X$ generated by the subsets $U[x]$
for all $x$ in $X$ and $U$ in $\cU$. 
We have  obvious forgetful functors \[\iota:\UBC\to \BC\ , \quad \tau:\UBC\to \TB\] which we will often omit from the notation.
  
\begin{ex}
\label{epokpbbwerg}
A (generalized, i.e., we allow $d(x,y)=\infty$) metric space $(X,d)$ presents a uniform bornological coarse space with the following structures:
\begin{enumerate}
\item $\cB$ consists of the subsets which can be covered by finitely many metric balls.
\item $\cC$ is consists of all subsets which are contained a metric entourage \begin{equation}\label{rwgrefrefew}U_{r}:=\{(x,y)\in X\times X\mid d(x,y)\le r  \}
\end{equation}  for some $r$ in $[0,\infty)$.
\item $\cU$ consists of subsets containing a metric entourage $U_{r}$ for some $r\in (0,\infty)$ (note that $0$ is excluded here).
\end{enumerate}
A Lipschitz  continuous map between metric spaces is uniform and controlled. \hB 
\end{ex}

 \begin{ex}
 Let $\Top^{\lc}$ denote the category of locally compact topological spaces and proper continuous  maps.
 We have a fully faithful functor $\Top^{\lc}\to \TB$ which equips a locally compact space with the bornology of relatively compact subsets. 
 
 We have a functor $\Top_{\mathrm{mc}}\to \UBC$ from  metrizable compact   topological spaces to uniform bornological coarse spaces which equips a metrizable compact topological space with the maximal coarse and bornological structures and
 the canonical uniform structure (induced by any choice of metric on $X$). 
  \hB
  \end{ex}

The category $\BC$ has a symmetric monoidal structure $\otimes$   such that
$ X\otimes Y$ has the following description:
\begin{enumerate}
\item The underlying set  of $X\otimes Y$ is $X\times Y$.
\item An entourage of $X\times Y$ is coarse for $X\otimes Y$ if it is contained a product of entourages of $X$ and $Y$.
\item A subset of $X\times Y$ is bounded if it is contained in a product of bounded subsets of $X$ and $Y$.
\end{enumerate}
Similarly the category   $\UBC$ has a symmetric monoidal structure $\otimes$  such that the underlying bornological coarse space of  $X\otimes Y$ is as above and the uniform structure consists of entourages
which contain products of uniform entourages of $X$ and $Y$.
Finally, also $\TB$ has a symmetric monoidal structure given by the cartesian structure on the underlying topological spaces and the product bornology as above.

\begin{ddd} \label{kophehehtre}
The cone functor 
\begin{equation}\label{jkfklejlwddqed}
\cO^{\infty}:\UBC\to \BC
\end{equation}
is defined as follows:
\label{jkigowergerwfwfwrefw}
\mbox{}
\begin{enumerate}
\item objects: Let $X$ be in $\UBC$:
\begin{enumerate}
\item The underlying set of $\cO^{\infty}(X)$ is $\R\times X$.
\item The underlying bornology of $\cO^{\infty}(X)$ is the one of $\R\otimes X$ (where $\R$ is considered as a metric space).
\item\label{kogpwerokgweprfwerfrefwre} The coarse structure of $\cO^{\infty}(X)$ consists of those entourages $U$ of {the} coarse product $\R\otimes X$ which have that property that for every uniform entourage $V$ of $X$   there exists  $r$ in $\R$  such that
$((x,t),(x',t'))\in U$ with $\min(t,t')\ge r$ implies  $(x,x')\in V$.
\end{enumerate}
\item morphisms:
  A map $f:X\to Y$ in $\UBC$ induces a map
$\cO^{\infty}(f):\cO^{\infty}(X)\to \cO^{\infty}(Y)$ given by $(t,x)\mapsto (t,f(x))$.
\end{enumerate}
\end{ddd}

The cone boundary   
\begin{equation}\label{fewrfrwfrweffwrf}\partial^{\cone}:\cO^{\infty}(X)\to  \iota(\R\otimes X)
\end{equation}  
is given by the identity of underlying sets.  
Observe that in general the product  uniform structure on $\R\otimes X$ is not compatible with the coarse structure  of $\cO^{\infty}(X)$.

 The following notions will be relevant in the discussion of Mayer-Vietoris sequences for coarse or local homology theories.
Let $X$ be a bornological coarse space and $Y,Z$ be two subsets. 

\begin{ddd} The pair $(Y,Z)$ is said to be coarsely excisive if for every coarse entourage $U$ if $X$ there exists a coarse entourage $V$ of $X$ such that
$$U[Y]\cap U[Z]\subseteq V[Y\cap Z]\ .$$ Equivalenty, the pair $(Y,Z)$ is coarsely  excisive iff   $\{Y\}\cap \{Z\}$ {is equivalent to}  $\{Y\cap Z\}$.
\end{ddd}
Assume now that $X$ is a uniform space with uniform structure $\cU$. 

\begin{ddd} We say that $(Y,Z)$ is uniformly excisive   \cite[Def.\ 3.3]{ass} if there exists
  $U$ in $\cU$ and a function $\kappa: \{V\subseteq X\times X\mid  V\subseteq U\}\to \cP(X\times X)$ (the power set of $X\times X$) such that 
  \begin{enumerate}
  \item	
$V\subseteq V'\subseteq U$ implies $\kappa(V)\subseteq \kappa(V')$, \item for every $V$ in $\cU$
there exists $W$ in $\cU$ with $W\subseteq U$ and  $\kappa(W)\subseteq V$, \item %and
for every $W$ in $\cU$ with $W\subseteq U$ we have  
\[W[Y]\cap W[Z]\subseteq \kappa(W)[Y\cap Z]\ .\]
  \end{enumerate}  \end{ddd}

%\begin{rem}
%If we set $\cP(X\times X)^{U}:=\{W\subseteq X\times X\mid W\subseteq U\}$, then the condition on $\kappa:\cP(X\times X)^{U}\to \cP(X\times X)$ can be phrased as follows. 
%\end{rem}
%

If $X$ is a uniform bornological coarse space, then 
by \cite[Lem.\ 9.26]{equicoarse} the cone functor sends pairs $(Y,Z)$ of subsets which a coarsely and uniformly 
excisive to coarsely excisive pairs.

\begin{ex}\label{kogrpergwefwerfwrefw}
A typical example of coarsely (or uniformly) excisive decomposition is the decomposition $((-\infty,0]\times X,[0,\infty)\times X)$ of $\R\otimes X$ for a coarse (or uniform) space $X$.

If $X$ is presented by a path metric space, then any closed decomposition  $(Y,Z)$ is coarsely and uniformly excisive.

In contrast, the picture below shows a subset of the plane together with a non-coarsely excisive pair of subsets.
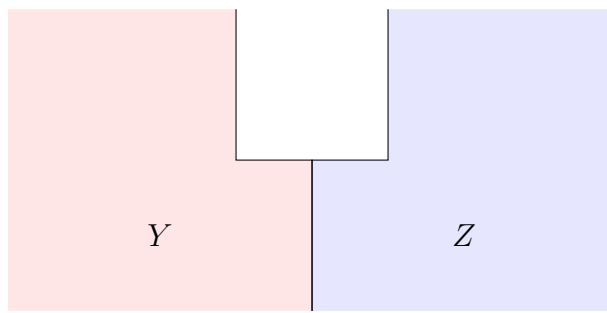
\begin{center}
\begin{tikzpicture}
\draw[name path=A] (2,0) --(2,2)--(3,2)--(3,4);
\draw[white, name path =B] (6,0) -- (6,4);
\draw[name path=C] (1,4) --(1,2)--(2,2)--(2, 0);
\draw[white, name path =D] (-2,0) -- (-2,4);
\tikzfillbetween[of=A and B]{blue, opacity=0.1};
\tikzfillbetween[of=C and D]{red, opacity=0.1};
\node at (0,1) {$Y$};
\node at (4,1) {$Z$};
\end{tikzpicture}
\end{center}
Note that this subset is not  a path metric space with the restricted metric.
\hB
\end{ex}

 \subsection{Coarse homology theories}\label{gojkperfrefrwfwrefw2}

In this section we recall   from  \cite{buen, equicoarse} the notion of a  coarse homology theory with values in a cocomplete stable $\infty$-category $\bD$ (e.g., the category of spectra $\Sp$ or the category of $KU$-modules $\Mod(KU)$).

Before we can state the axioms we need to recall some further notions from coarse geometry and introduce some notation.

Let $X,Y$ be   coarse spaces and $f,g:X\to Y$ be two controlled maps. 

\begin{ddd} We say that $f$ and $g$ are close  if the map
$$f\sqcup g: \{0,1\}\otimes X\to Y\ , \quad (0,x)\mapsto f(x)\ , \quad (1,x)\mapsto g(x)$$
is controlled. \end{ddd} Here $\{0,1\}$ has the maximal coarse   structure.

 \begin{ddd}A bornological coarse space $X$ is called flasque if it admits a 
 selfmap $f:X\to X$ (called a witness of flasqueness) with the following properties:
 \begin{enumerate}
 \item  $f$ is close to $\id_{X}$
 \item $f$ is non-expanding in the sense that    for every coarse enourage $U$ of $X$ also $\bigcup_{n\in \nat} (f^{n}\times f^{n})(U)$ is a coarse entourage
 \item \label{kojgpwerwrfwerfwf}$f$ is shifting  in the sense that  for every  bounded subset $B$ there exists $n$ in $\nat$ with $B\cap f^{n}(X)=\emptyset$.
  \end{enumerate}
\end{ddd}
 
\begin{ex}\label{rgjkokpwerwerfwref}
If $X$ is a bornological coarse space and $[0,\infty)$ has the bornological coarse structure induced by the metric, then
$[0,\infty)\otimes X$ is flasque with witness of flasqueness $f:X\to X$ given by  $f(t,x):=(t+1,x)$. %If $E$ is a coarse homology theory, then
  %$E([0,\infty)\otimes X)\simeq 0$.
   \hB
\end{ex}

 We consider a functor
 $$E:\BC\to \bD\ .$$
 If $\cY$ is a big family on $X$ in $\BC$ (see \cref{rfiohqeiorfqfewdqddq}), then we define 
 \begin{equation}
 \label{fqwedeqwed}
 E(\cY):= \colim_{Y\in \cY}E(Y)\ ,
\end{equation}
where we equip the members of $\cY$ with the bornological coarse structures induced from $X$.

  %For a subset $Z$ of $X$ we define the big family $Z\cap \cY:=(Z\cap Y_{i})_{i\in I}$ on $Z$.

 \begin{ddd}
 \label{mkeoggwger9} 
 The functor $E:\BC\to \bD$ is called a coarse homology theory if it has 
 the  following properties:
\begin{enumerate}
\item \label{wjoigwegfwerfw}Coarse invariance: 
 The projection $\{0,1\}\otimes X\to X$ induces an equivalence $E(\{0,1\}\otimes X)\to E(X)$ for every $X$ in $\BC$.
\item Excision: For every  $X$ in $\BC$ each and each pair $(\cY,Z)$ of a big family $\cY$ and a subset $Z$ of $X$ such that $X=Y\cup Z$ for some $Y$ in $\cY$ (i.e., a complementary pair), we have a push-out square 
$$
\xymatrix{E(\cY\cap Z)\ar[r]\ar[d] & E(Z) \ar[d] \\ E(\cY) \ar[r] & E(X) }\ . 
$$
 \item Vanishing on flasques: $E(X)\simeq 0$ if $X$ is flasque.
  \item $U$-continuity: For every  $X$ in $\BC$ the canonical morphism $$\colim_{U\in \cC} E(X_{U})\to E(X)$$ is an equivalence, where $X_{U}$ is the bornological coarse space derived from $X$ by replacing the coarse structure by the smaller coarse structure generated by $U$.
\end{enumerate}
\end{ddd}

  \begin{rem}
  If $Y$ is a subset of $X$, then we have a map $E(Y)\to E(\{Y\})$. 
  If $E$ is coarsely invariant, then this map is an equivalence.
  \hB
\end{rem}

\begin{rem}
Coarse invariance {of $E$} is equivalent to the condition that for every two maps $f,g:X\to Y$  in $\BC$ which are close to each other we have $E(f)\simeq E(g)$. \hB
\end{rem}

\begin{rem}
For a more symmetric version of excision let $\cZ$ be a big family which is complementary to $\cY$ in the sense that there exists a member $Z$ of $\cZ$ such that $(\cY,Z)$ is a complementary pair.  If $E$ is a coarse homology theory, then we also have a push-out square \begin{equation}\label{gwerojopewferferwf}\xymatrix{E(\cY\cap \cZ)\ar[r]\ar[d] & E(\cZ) \ar[d] \\ E(\cY) \ar[r] & E(X) }\ . 
\end{equation}

If $(Y,Z)$ is a coarsely excisive decomposition of $X$, then by \cite[Lem. 3.41]{buen} have a push-out square
 \[
  \label{gwerojopewferferwf2}
 \begin{tikzcd}
 	E(Y\cap Z)\ar[r]\ar[d] & E(Z) \ar[d] \\ E(Y) \ar[r] & E(X) \ .
 \end{tikzcd}
\]
\hB
\end{rem}

\begin{rem}\label{werkogpergwerfwf}
In \cite{buen} we have constructed a universal
coarse homology theory  \begin{equation}\label{oiujvodbsdfb}
\Yo^{s}:\BC\to \Sp\cX
\end{equation}
to a certain presentable stable infinity category $\Sp\cX$ such that for every cocomplete stable $\infty$-category $\bD$ 
the restriction 
\[
(\Yo^{s})^{*}:\Fun^{\colim}(\Sp\cX,\bD)\to \Fun(\BC,\bD)
\]
is  an equivalence onto the full subcategory of $ \Fun(\BC,\bD)$ of coarse homology theories.
In particular, every coarse homology theory $E:\BC\to \bD$ has an essentially unique colimit-preserving factorization
$$\xymatrix{\BC\ar[dr]_{\Yo^{s}}\ar[rr]^{E}&&\bD\\&\Sp\cX\ar@{..>}[ur]_{E}&}$$ indicated by the dotted
arrow and denoted by the same symbol. 
Since the  $\Sp\cX$ plays a role analogous to the category of spectra in classical homotopy theory it  is  called the category of coarse spectra.
\hB\end{rem}

\begin{ex}
A basic consequence of the axioms is a canonical  equivalence  
\begin{equation}\label{fqwedwedewdqwdwd}\Sigma E(-)\simeq E(\R\otimes -):\BC\to \bD\ .
\end{equation}
To this end we consider the  coarsely excisive decomposition of $\R\otimes X$ from \cref{kogrpergwefwerfwrefw} and get by specializing \eqref{gwerojopewferferwf2} a push-out square
 $$\xymatrix{E(X)\ar[r]\ar[d] &E([0,\infty)\otimes X) \ar[d] \\ E((-\infty,0]\otimes X)\ar[r] &E(\R\otimes X )} \ .$$
 We now use that $[0,\infty)\otimes X$ and$(-\infty,0] \otimes X$  are flasque  by \cref{rgjkokpwerwerfwref} in order to see that the lower left and upper right corners of this square vanish. This provides an equivalence of the lower right corner of the square with the suspension of the upper left corner. \hB
\end{ex}

A coarse homology theory can have the following additional properties:
\begin{enumerate}
\item Continuity: For every  $X$ in $\BC$ the canonical map 
\[
\colim_{Z\in \mathrm{LF}(X)} E(Z)\to  E(X)
\] 
is an equivalence, where $ \mathrm{LF}(X)$ is the poset of locally finite subsets $Z$ of $X$.
{Here} $Z$ is locally finite if the induced bornology of $Z$ consists of the  finite subsets.
\item Additivity:  For any family $(X_{i})_{i\in I}$ in $\BC$ the natural map $E(\bigsqcup^{\mathrm{free}}_{i\in I}X_{i})\to \prod_{i\in I} E(X_{i})$ (defined using excision and assuming that the product in $\bD$ exists) is an equivalence, where
the free union $\bigsqcup^{\mathrm{free}}_{i\in I}X_{i}$ is the bornological coarse space with underlying set $\bigsqcup_{i\in I}X_{i}$ with  the coarse structure   generated by entourages $\bigcup_{i\in I}U_{i}$ for families
$(U_{i})_{i\in I}$ in $\prod_{i\in I} \cC_{X_{i}}$, and whose bounded subsets are subsets $B$  such that
$B\cap X_{i}$ is  bounded for all $i$, and non-empty for at most finitely many $i$ in $ I$.
\item Strongness: $E(X)\simeq 0$ for weakly flasque bornological coarse spaces, where $X$ is called weakly flasque if it admits a non-expanding and shifting (see \cref{werkogpergwerfwf}.\ref{kojgpwerwrfwerfwf}) selfmap $f:X\to X$   such that
$\Yo^{s}(f)\simeq \Yo^{s}(\id_{X})$, see \eqref{oiujvodbsdfb}.
 
\end{enumerate}

\subsection{Local and locally finite homology theories}
\label{kogpewgwergergwergwerfw}

In this section we recall the notions of a local and a locally finite homology theory from \cite{ass, buen}.
We further recall the construction  via the cone functor  of the local homology theory associated to a coarse homology theory  and 
the implementation of additional support conditions.

We start with the notion of local finiteness which applies to functors defined on  spaces equipped with a bornology.
We consider a functor
$E:\UBC\to \bD$ or $E:\TB\to \bD$ and assume that $\bD$ is stable, complete and cocomplete. \begin{ddd}
 We say that $E$ is locally finite if the canonical map
 \begin{equation}\label{sdfvsdfvsfdvv}E(X)\to \lim_{B\in \cB} \Cofib(E(X\setminus B)\to E(X))
\end{equation} is an equivalence for every $X$ with bornology $\cB$.\end{ddd}
Equivalently, $E$ is locally finite if $$\lim_{B\in \cB} E(X\setminus B)\simeq 0$$ for every  $X$. The inclusions of locally finite functors into all functors
 are the right-adjoints of left  Bousfield localizations
 \begin{align} &(-)^{\lf}:\Fun(\UBC,\bD)\leftrightarrows \Fun^{\lf}(\UBC,\bD):\incl\ ,\nonumber\\ & (-)^{\lf}:\Fun(\TB,\bD)\leftrightarrows \Fun^{\lf}(\TB,\bD):\incl \label{ijgowrgrfwferfref}\end{align} whose units $E\to E^{\lf}$ have the components \eqref{sdfvsdfvsfdvv}, see \cite[Sec. 7.1.2]{buen}.

\begin{rem}
In  {this} remark we recall the  characterization of homology theories on the category $\Top$ of topological spaces.
We let $\ell:\Top\to \Spc$ be the Dwyer-Kan localization at the weak homotopy equivalences.   This is one of the presentations of the category $\Spc$ of spaces (anima).  Let $E:\Top \to \bD$ be a functor to a cocomplete stable $\infty$-category. 
The functor $E$ is a  homology theory on $\Top$  if it  sends weak  {homotopy} equivalences to equivalences and has the property  that the induced factorization $\Spc\to \bD$  
$$\xymatrix{\Top\ar[rr]^{E}\ar[dr]^{\ell}&&\bD\\&\Spc\ar@{..>}[ur]&}$$preserves colimits.

By the universal property of $\Spc$ the evaluation at $*$ is an equivalence $$\Fun^{\colim}(\Spc,\bD)\stackrel{\ev_{*}}{\simeq} \bD$$ so that the category of  $\bD$-valued homology theories is equivalent to $\bD$ itself.
For $D$ in $\bD$ we write $D(-)$ for the corresponding homology theory $X\mapsto \ell(X)\otimes D$
(written using the tensor structure of $\bD$ over $\Spc$).

A homology theory on $\Top$ in particular has the following properties:
\begin{enumerate}
\item homotopy invariance: The map $E([0,1]\times X)\to E(X)$ is an equivalence for every $X$ in $\Top$.
\item 
{excisiveness}: For every  pair of big families $(\cY,\cZ)$ of $X$ (in the topological sense) covering $X$  (i.e.,  there exists members $Y$ and $Z$ with $Y\cup Z=X$) we have a push-out square
$$\xymatrix{E(\cY\cap \cZ)\ar[r]\ar[d] &E(\cY) \ar[d] \\ E(\cZ) \ar[r] &E(X) }\ . $$
\end{enumerate}
Note that these two  properties do not suffice to characterize homology theories.
 The formulation of excision with big families avoids to talk about the usual additional point-set theoretical assumptions on decompositions of $X$ into two subsets. But if, e.g., $(Y,Z)$ is a closed decomposition of a metric space $X$ such that  the inclusions $Y\to U_{r}[Y]$
and $Z\to U_{r}[Z]$ are homotopy equivalences for all sufficiently small $r$, then we have a push-out 
$$\xymatrix{E(Y\cap Z)\ar[r]\ar[d] &E(Y) \ar[d] \\ E(Z) \ar[r] &E(X) }$$ 
provided $E$ is homotopy invariant and excisive.
  \hB
\end{rem}

Let $E:\TB\to \bD$ be a functor to a cocomplete and complete stable $\infty$-category.
 
\begin{ddd}[{\cite[7.27]{buen}}] 
\label{DefinitionLocallyFiniteHomologyTheory}
The functor $E$ is a  locally finite homology theory  if it is
\begin{enumerate}
\item homotopy invariant
\item excisive
\item locally finite.
\end{enumerate}\end{ddd}

\begin{rem}
 By \cite[7.1.4]{buen} the composition $$\Fun(\Top,\bD) \stackrel{ }{\to}  \Fun(\TB,\bD) \stackrel{(-)^{\lf}}{\to} \Fun^{\lf}(\TB,\bD)$$ of the restriction along the forgetful functor $\TB\to \Top$ and  the left-adjoint in \eqref{ijgowrgrfwferfref}
sends homology theories $E  $ to locally finite homology theories $E^{\lf} $. 
The original homology theory  can be recovered from $E^{\lf}$ as the homology theory represented by the spectrum $E^{\lf}(*)$.

On the other hand, if $F:\TB\to \bD$ is a  locally finite homology theory, then the spectrum $F(*)$ represents a cohomology theory $F(*)(-)$ and we have a natural transformation of functors $F(*)^{\lf}(-) \to F(-) : \TB \to \bD$, which attempts to recover $F$ from its value  on the point. 
By  \cite[Thm. 7.43]{buen}
 the map $F(*)^{\lf}(X)\to F(X)$ is an equivalence for a large class of topological bornological  spaces, e.g, for locally finite simplicial complexes with the topology and bornology induced by the spherical path metric \cite[Ex. 7.42]{buen}. \mbox{}
 \hB
  \end{rem}

\begin{rem}
We say that  $E:\TB\to \bD$ is open (or closed) excisive  {if} 
$E$ sends open (or closed) decompositions to push-out squares.
A locally finite and  open  (or closed) excisive functor $E:\TB\to \bD$ has an additional contravariant  functoriality for   topological bornological  inclusions $Y\to X$ of open (closed) subsets. Indeed,  by open (or closed) excision for every closed (or open) $B$ in $\cB\cap Y$ we have
\[
\Cofib(E(Y\setminus B)\to E(Y)) \stackrel{\simeq}{\longrightarrow} \Cofib(E(X\setminus B)\to E(X))\ .
\]
In  the limit in \eqref{sdfvsdfvsfdvv} we can restrict to closed (open) bounded subsets contained in $Y$.
By local finiteness get we a map
\begin{equation*}
\begin{aligned}
E(X) &\simeq  \lim_{B\in \cB} \Cofib(E(X\setminus B)\to E(X)) \to  \lim_{B\in \cB\cap Y}\Cofib(E(X\setminus B)\to E(X))\\&\simeq   \lim_{B\in \cB\cap Y} \Cofib(E(Y\setminus B)\to E(Y)) \simeq E(Y)\ .
\end{aligned}
\tag*{$\blacksquare$}
\end{equation*}
\end{rem}

We now consider functors $E:\UBC\to \bD$ to a stable cocomplete $\infty$-category $\bD$. Homotopy invariance  is defined as above for $\TB$. We will consider the following additional properties:
 \begin{enumerate}
 \item {excisiveness}: For every coarsely and uniformly excisive decomposition $(Y,Z)$ of $X$
 we have a push-out square $$\xymatrix{E(Y\cap Z)\ar[r]\ar[d] &E(Y) \ar[d] \\ E(Z) \ar[r] &E(X) }\ . $$
 \item vanishing on flasques: $E([0,\infty)\otimes X)\simeq 0$ for all $X$ in $\UBC$.
 \item $u$-continuity: $\colim_{U\in \cC\cap \cU} E(X_{U})\stackrel{\simeq}{\to} E(X)$, where
 $X_{U}$ is obtained from $X$ by replacing the coarse structure by the smaller coarse structure generated by $U$.
 \end{enumerate}

\begin{ddd}
\label{DefinitionLocalHomologyTheory}
The functor $E $ is called a local homology theory if it has the following properties:
\begin{enumerate}
\item homotopy invariance
\item excisiveness
\item vanishing on flasques
\item $u$-continuity.
\end{enumerate}
\end{ddd}

We have then following constructions of local homology theories:
\begin{enumerate}
\item If $E:\UBC\to \bD$ is a local homology theory which is coarsening invariant \cite[Def. 11.8]{ass} and $\bD$ is complete, then $E^{\lf}$ is again a (coarsening invariant, see \cite[Lem. 11.10]{ass}) local homology theory which is in addition locally finite.
The assumption on coarsening invariance is needed in order to ensure that $E^{\lf}$ is $u$-continuous. Without this assumption we would have the problem of  commuting  the limit defining $E^{\lf}$ with the colimit appearing in the $u$-continuity condition.
\item If $E:\TB\to \bD$ is a locally finite homology theory, then $E  \tau:\UBC\to \TB$ is a local homology theory \cite[Lem. 3.16]{ass}.
\item If $E:\BC\to \bD$ is a coarse homology theory, then $E  \iota:\UBC\to \bD$ is a local homology theory \cite[Lem 3.13]{ass}.
\item If  $ E:\BC\to \bD$ is a coarse homology theory which is in addition strong, then
$E\cO^{\infty}:=E\circ \cO^{\infty}:\UBC\to \bD$ is a local homology theory \cite[Lem. 9.6]{ass}.
It is coarsening invariant by \cite[Ex. 11.9]{ass}.
\end{enumerate}

 \begin{ex}\label{rtlphhrerhgtrge}
 The cone boundary \eqref{fewrfrwfrweffwrf} together with \eqref{fqwedwedewdqwdwd} induces a natural transformation of local homology theories
 \begin{equation}\label{bdgpbkdpgbdpgbdfg}a:\Sigma^{-1}E\cO^{\infty}\to  E\iota:\UBC\to \bD\ .
\end{equation}
which we call the index map. \hB
 \end{ex}

 Note that a local homology theory   $E:\UBC\to \bD$ with complete target is  not necessarily locally finite. 
By \cite[Sec. 11]{ass} we have comparison transformations 
\begin{equation}\label{gerfwerfreferfvwef}E(*)^{\lf}(-)\to  E^{\lf} \leftarrow E:\UBC\to \bD\ ,
\end{equation} 
where $E(*)^{\lf}(-)$ denotes the locally finite homology theory represented by the spectrum $E(*)$.

 \begin{ex}\label{wkorgpwertgwrewrefwerf} Let $E:\BC\to \bD$ be a  strong and additive coarse homology theory with complete target. If  $X$ in $\UBC$ is presented by a locally finite finite-dimensional simplicial complex with the spherical path metric, then  by \cite[Prop. 11.23]{ass} transformations from \eqref{gerfwerfreferfvwef} induce equivalences
 $$\Sigma E(*)^{\lf}(X)\stackrel{\simeq}{\to} (E\cO^{\infty})^{\lf}(X)\stackrel{\simeq}{\leftarrow} E\cO^{\infty}(X) \ .$$
Here we used the equivalence   $E\cO^{\infty}(*)\simeq \Sigma E(*)$. 
  \hB
 \end{ex}

 \begin{ex}
 \label{kofpwerfeggrwrg}
Let $X$ in $\UBC$ be a uniform bornological coarse space whose bornology is the bornology of relatively compact subsets of the underlying topological space.  Note that $X$ is then necessarily locally compact.
 
 We assume that $E:\BC\to \bD$  is a strong coarse homology theory with complete target $\bD$. 
By the coarsening invariance of $E\cO^{\infty}$ \cite[Prop. 9.33]{equicoarse}, \cite[Ex. 11.9]{ass} we know that
$E\cO^{\infty}(X)$, and therefore also $(E\cO^{\infty})^{\lf}(X)$ do not depend on the choice of $\cC$. 
Since the restrictions  the uniform structures as above to a compact  subset of $X$ is determined by the topology of $X$ alone we conclude that $(E\cO^{\infty})^{\lf}(X)$ does not depend on the uniform structure as well. 
So $(E\cO^{\infty})^{\lf}(X)$ is an invariant of the locally compact topological space $X$. 
By \cref{wkorgpwertgwrewrefwerf}, if $X$ is a locally finite finite-dimensional simplicial complex and $E$ is in addition additive, then
$(E\cO^{\infty})^{\lf}(X)\simeq \Sigma E(*)^{\lf}(X)$ where the right-hand side clearly  only depends on the locally compact space $X$. 

Assume now that $X$ in $\UBC$ has {the} additional structure of a finite-dimensional, locally finite  simplicial complex whose simplices are uniformly equicontinuous  and coarsely bounded.
Let $X_{\simp}$ be $X$ equipped with the uniform and coarse structure induced by the spherical path metric
of the simplices. In both cases we assume that the bornologies consist  of relatively compact subsets.
Then the identity of the underlying set  is a morphism $X_{\simp}\to X$ in $\UBC$.
We  have a commutative diagram
$$\xymatrix{E\cO^{\infty}(X_{\simp})\ar[r]^{\simeq}\ar[d] &(E \cO^{\infty})^{\lf}(X_{\simp})\ar[d]^-{\simeq} &\ar[l]_-{\simeq}\Sigma E(*)^{\lf}(X_{\simp})\ar[d]^{\simeq}\\E\cO^{\infty}(X)\ar[r] & (E\cO^{\infty})^{\lf}(X)&\Sigma E(*)^{\lf}(X)\ar[l]_-
{\simeq}}$$
which presents $\Sigma E(*)^{\lf}(X)$ as a retract of $E\cO^{\infty} (X)$.
\hB \end{ex}

\begin{rem}
The functor   $$ \Sigma E(*)^{\lf}(-): \UBC \to \bD$$ has the additional contravariant functoriality for inclusions of open  subsets.  We do not know whether
 $E\cO^{\infty} (-)$ has such a contravariant functoriality. \hB
 \end{rem}

\begin{ex}
\label{gkgoerpgwergrefwrefw} 
The consideration from \cref{kofpwerfeggrwrg} {applies} in particular to a complete  finite-dimensional Riemannian manifold $M$ presenting an object of $\UBC$. 
 Assume that $E$ is strong and additive with complete target $\bD$.
Then
$\Sigma E(*)^{\lf}(M)$ is  a retract of $E\cO^{\infty}(M)$.
We get a factorization of the index map over a locally finite version $a_{M}^{\lf}$
$$\xymatrix{\Sigma^{-1}E\cO^{\infty}(M)\ar[r]\ar@/^1cm/[rr]^{\mbox{}}\ar[dr]^{a_{M}}&\ar@{..>}[d]^{a^{\lf}_{M_{\simp}}}E(*)^{\lf}(M) & \ar[l]_-{\simeq}\Sigma^{-1}E\cO^{\infty}(M_{\simp})\ar[dl]_{a_{M_{\simp}}}\\&E(M)&}$$
of the index map.
\hB \end{ex}

We now take advantage of the geometric construction of local homology theories via the cone functor in order to add support conditions.
Let $X$ be in $\UBC$ and $\cY$ and $\cZ$ be two big families  in $X$. 
  
    \begin{ddd}
    \label{fewdqwedqewdqewd}
We define the big  families
  \[
  \cO^{\infty}(\cY) := \R \times \cY, \qquad \cO^{-}(\cY):= \{{\R^-}\} \times \cY
  \]
   and 
%   $$ \cO^{\infty}_{\cZ}(X):=(\R^{-}\times Z_{j} \cup [-n,\infty)\times X)_{(j,n)\in J\times \nat}\ , \quad 
%    \cO^{\infty}_{\cZ}(\cY):= \cO^{\infty}_{\cZ}(X)\cap \cO^{\infty}(\cY)$$
 {
  \[
  \cO^{\infty}_{\cZ}(X):= \{\R^-\} \times \cZ \cup \{\R^+\} \times X , \qquad {\cO^\infty_{\cZ}(\cY) :=} \cO^{\infty}_{\cZ}(X)\cap \cO^{\infty}(\cY)
  \]
  }
  in $\cO^{\infty}(X)$. 
  \end{ddd}
 
{ 
Recall   \cref{BigFamilyConstr} for the notations in the above definition.
}
  
  Note that all members of $\cO^{-}(\cY)$ are flasque.
  
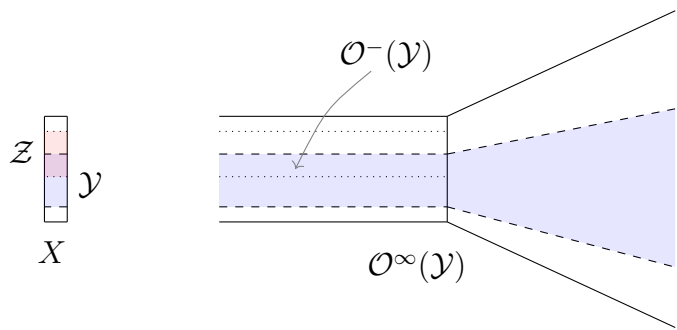
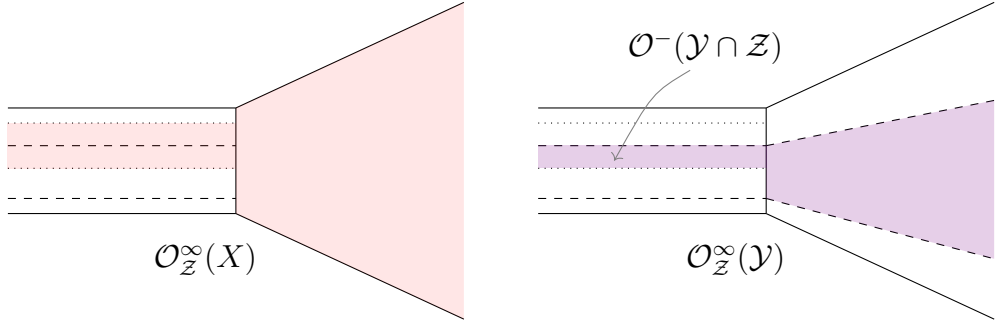
\begin{center}
\begin{tikzpicture}
\draw (-5,0.7) -- (-5,-0.7);
\draw (-5.3,0.7) -- (-5.3,-0.7);
\draw (-5,0.7) --(-5.3,0.7);
\draw[dashed, name path=A] (-5,0.2) --(-5.3,0.2);
\draw[dashed, name path=B] (-5,-0.5) --(-5.3,-0.5);
\draw[dotted, name path=C] (-5,0.5) --(-5.3,0.5);
\draw[dotted, name path=D] (-5,-0.1) --(-5.3,-0.1);
\draw (-5,-0.7) --(-5.3,-0.7);
\tikzfillbetween[of=A and B]{blue, opacity=0.1};
\tikzfillbetween[of=C and D]{red, opacity=0.1};
\node at (-5.6,0.23) {$\cZ$};
\node at (-4.7,-0.2) {$\cY$};
\node at (-5.2,-1.1) {$X$};
\draw (-3,0.7) --(0,0.7);
\draw[dashed, name path=A] (-3,0.2) --(0,0.2);
\draw[dashed, name path=B] (-3,-0.5) --(0,-0.5);
\draw[dotted] (-3,0.5) --(0,0.5);
\draw[dotted] (-3,-0.1) --(0,-0.1);
\draw (-3,-0.7) --(0,-0.7);
\draw (0, 0.7) -- (0, -0.7);
\draw (0, 0.7) -- (3, 2.1);
\draw (0, -0.7) -- (3, -2.1);
\draw[dashed, name path=C] (0, 0.2) -- (3, 0.8);
\draw[dashed, name path=D] (0, -0.5) -- (3, -1.3);
\tikzfillbetween[of=A and B]{blue, opacity=0.1};
\tikzfillbetween[of=C and D]{blue, opacity=0.1};
\node at (-0.4,-1.3) {$\cO^\infty(\cY)$};
\node at (-0.8,1.5) {$\cO^-(\cY)$};
\draw[->, color=gray] (-1,1.3) .. controls (-1.6,0.8) .. (-2,0);
\end{tikzpicture}
\\
\begin{tikzpicture}
\draw (-3,0.7) --(0,0.7);
\draw[dotted, name path=A] (-3,0.5) --(0,0.5);
\draw[dotted, name path=B] (-3,-0.1) --(0,-0.1);
\draw[dashed] (-3,0.2) --(0,0.2);
\draw[dashed] (-3,-0.5) --(0,-0.5);
\draw (-3,-0.7) --(0,-0.7);
\draw (0, 0.7) -- (0, -0.7);
\draw[name path=C] (0, 0.7) -- (3, 2.1);
\draw[name path=D] (0, -0.7) -- (3, -2.1);
\tikzfillbetween[of=A and B]{red, opacity=0.1};
\tikzfillbetween[of=C and D]{red, opacity=0.1};
\node at (-0.4,-1.3) {$\cO^\infty_{\cZ}(X)$};
\end{tikzpicture}
~~~~~ 
\begin{tikzpicture}
\draw (-3,0.7) --(0,0.7);
\draw[dotted] (-3,0.5) --(0,0.5);
\draw[dashed, name path=A] (-3,0.2) --(0,0.2);
\draw[dotted, name path=B] (-3,-0.1) --(0,-0.1);
\draw[dashed] (-3,-0.5) --(0,-0.5);
\draw (-3,-0.7) --(0,-0.7);
\draw (0, 0.7) -- (0, -0.7);
\draw (0, 0.7) -- (3, 2.1);
\draw (0, -0.7) -- (3, -2.1);
\draw[dashed, name path=C] (0, 0.2) -- (3, 0.8);
\draw[dashed, name path=D] (0, -0.5) -- (3, -1.3);
\tikzfillbetween[of=A and B]{red, opacity=0.1};
\tikzfillbetween[of=C and D]{red, opacity=0.1};
\tikzfillbetween[of=A and B]{blue, opacity=0.1};
\tikzfillbetween[of=C and D]{blue, opacity=0.1};
\node at (-0.4,-1.3) {$\cO^\infty_{\cZ}(\cY)$};
\node at (-0.8,1.5) {$\cO^-(\cY\cap \cZ)$};
\draw[->, color=gray] (-1,1.2) .. controls (-1.5,0.9) .. (-2,0);
\end{tikzpicture}
\end{center}

 We let $\UBC^{(2)}$ be the category of pairs $(X,\cZ)$ of $X$ in $\UBC$ and a big family $\cZ$. A morphism
 $f:(X,\cZ)\to (X',\cZ')$ is a morphism $f:X\to X'$ in $\UBC $ such that $f(\cZ)\subseteq \cZ'$. 
 We define the category $\BC^{(2)}$ of pairs $(X,\cZ)$ {of} bornological coarse spaces with a big family similarly.
 We have a forgetful functor $\iota:\UBC^{(2)}\to \BC^{(2)}$ and the functor
 \[
 \cO^{\infty}_{-}: \UBC^{(2)}\to \BC^{(2)}\ , \quad (X,\cZ)\mapsto (\cO^{\infty}(X),\cO^{\infty}_{\cZ}(X))\ .
 \]
 
 If $E:\BC\to \bD$ is a functor to a cocomplete stable $\infty$-category, then we get the functor
 \begin{equation}\label{vfdsovjsfdovsdfvsfdvsfv}E\cO^{\infty}_{-}:\UBC^{(2)}\to \bD\ , \quad  (X,\cZ)\mapsto E(\cO^{\infty}_{\cZ}(X))\ .
\end{equation}
 It comes with a natural transformation $E\cO^{\infty}_{-}\to E\cO^{\infty} $ induced by the maps
 $\cO_{\cZ}^{\infty}(X)\to \cO^{\infty}(X)$ for every $(X,\cZ)$ in $\UBC^{(2)}$, where we consider $E\cO^{\infty} $ as a functor on $\UBC^{(2)}$ using the forgetful functor $(X,\cZ)\mapsto X$.
 
 We now assume that $E:\BC\to \bD$ is a strong coarse homology theory.
 {As discussed above, the induced functor $E\cO^\infty : \UBC \to \bD$ is a local homology theory.}

 \begin{ddd}\label{owkgpergwrefwerfw}
 We call the functor $E\cO^{\infty}_{-}:\UBC^{(2)}\to \bD$ from \eqref{vfdsovjsfdovsdfvsfdvsfv} the associated local homology theory with support.
 \end{ddd}
 
 \begin{rem}
 Let $\cZ$  be a big family in $X$.
 For a subset $Y$ or a big family $\cY$ on $X$, in order to simplify the notation,  we will write
 $E\cO^{\infty}_{\cZ} (Y):=E\cO^{\infty}_{Y\cap \cZ} (Y)$ or $E\cO^{\infty}_{\cZ} (\cY):=\colim_{Y\in \cY} E\cO^{\infty}_{\cZ\cap Y} (Y)$.
This   notation is consistent since
$E\cO^{\infty}_{\cZ} (\cY)\simeq E(\cO^{\infty}_{\cZ}(\cY))$, where $\cO^{\infty}_{\cZ}(\cY)$ is defined in  \cref{fewdqwedqewdqewd}.
%
% If $\cY=(Y_{i})_{i\in I}$ and $\cZ=(Z_{j})_{j\in J}$ are two big families on $X$ in $\UBC$, then we have two potentially different interpretations of $E\cO^{\infty}_{\cZ}(\cY)$, namely as
% $E(\cO^{\infty}_{\cZ}(\cY))$ (using the big family introduced in \cref{fewdqwedqewdqewd},  and as $\colim_{i\in I} E\cO^{\infty}_{\cZ\cap Y_{i}} (Y_{i})$ using the functor from \cref{owkgpergwrefwerfw}.
%These evaluations are canonically equivalent as they only differ by the order to taking the colimits over the partially ordered set $I\times J$. 
%
%  
% using the functor from \cref{owkgpergwrefwerfw}.
%These evaluations are canonically e
%
  \hB
 \end{rem}

 \begin{lem}
We have a functorial fibre sequence
\begin{equation}\label{gerwfwerfrewfwr}E\iota(\cZ) \to E\cO^{{\infty}}_{{{\{\emptyset\}}}}(X) \to E\cO^{\infty}_{\cZ}(X)\ .
\end{equation} 
%\fuli{$E\cO^{\infty}_{(\emptyset)}(...)$ wäre aber auch richtig gewesen. Die big family hat das eine Mitglied. \ml{Ja, aber wir sind ja jetzt zur Konvention uebergegangen, dass Big Families Teilmengen und keine Familien sind. Darum muss es doch nun $\{\emptyset\}$ heissen.}}
  \end{lem}
\begin{proof}
This is the Mayer-Vietoris sequence for the decomposition
$(\{\R^{-}\}\times \cZ,\{\R^{+}\}\times X)$ of $\cO_{\cZ}^{\infty}(X)$ together with the flasqueness of the members of
$\{\R^{-}\}\times \cZ$.
\end{proof}

 The following lemma says that $E\cO^{\infty}_{-}(-)$ has  the analogues   for pairs  of the properties of a local homology theory.
 \begin{lem}\label{keopgegwerferfw} The functor $E\cO^{\infty}_{-}(-) :\UBC^{(2)}\to \bD$ has the  following properties:
 \begin{enumerate}
 \item homotopy invariant:  For $(X,\cZ)$ in $\UBC$ the projection
 $([0,1]\otimes X,[0,1]\times \cZ)\to (X,\cZ)$ induces an equivalence $E\cO^{\infty}_{[0,1]\times \cZ}([0,1]\otimes X)\to E\cO^{\infty}_{\cZ}(X)$. 
  \item\label{werkjgokwegergwerg9}  {excisiveness}: If $(X,\cZ)$ is in $\UBC^{(2)}$ and $(A,B)$ is a 
  coarsely and uniformly  excisive decomposition of $X$, then the square
  $$\xymatrix{E\cO^{\infty}_{\cZ}(A\cap B)\ar[r]\ar[d] &E\cO^{\infty}_{\cZ}(B) \ar[d] \\E\cO^{\infty}_{\cZ}(A) \ar[r] &E\cO^{\infty}_{\cZ}(X) }  $$ is a push-out square.
 \item 
 vanishing on flasques: For every $(X,\cZ)$ in $\UBC^{(2)}$ we have $E\cO^{\infty}_{\R\times\cZ}(\R\otimes X)\simeq 0$.
 \item $u$-{continuity}: For every $(X,\cZ)$ in $\UBC^{(2)}$ we have $\colim_{U\in \cC} E\cO^{\infty}_{\cZ}(X_{U})\simeq  E\cO^{\infty}_{\cZ}(X)\ .$
 \end{enumerate}
\end{lem}
\begin{proof}
One option is to redo the argument of  \cite[Lem. 9.6]{ass} with the additional support condition.
Alternatively one could employ the fibre sequence \eqref{gerwfwerfrewfwr} and the fact that $E\iota$ and $E\cO$ are known to be local homology theories (see \cite[Lem 3.13]{ass} for $E\iota$ and \cite[Lem. 9.6]{ass} for $E\cO$).
 \end{proof}

 We now assume that $E:\BC\to \bD$ is a  strong  coarse homology theory.  We  incorporate the support conditions into
 the index map \eqref{bdgpbkdpgbdpgbdfg}   as follows. We define the functor $$E\iota':\UBC^{(2)}\to \bD\ , \quad (X,\cZ)\mapsto E\iota (\cZ)\ .$$ 
 We added the superscript ${}'$ in order to make clear that this functor takes the big family instead of the space.
 
\begin{ddd}
\label{rtkohprethergetetg}  
The index map $$a:\Sigma^{-1}E\cO_{-}^{\infty}\to E\iota':\UBC^{(2)}\to \bD$$  with support conditions is the natural  transformation of functors $       \UBC^{(2)}\to \bD$ with components  
\begin{equation}
\label{bdgpbkdpgbdpgbdfg11}
 a_{X,\cZ}:\Sigma^{-1}E\cO^{\infty}_{\cZ}(X)\to E\iota(\cZ)  \end{equation}
  given by
\[
\Sigma^{-1}E\cO^{\infty}_{\cZ}(X)\stackrel{\partial^{\cone}}{\longrightarrow} \Sigma^{-1}E(\{\R^{-}\}\otimes \cZ\cup\{\R^{+}\}\otimes {X})\stackrel{\partial^{MV}}{\longrightarrow}
 E(\cZ)\ .
 \]
 \end{ddd}
 
 Here we consider  $\{\R^{-}\}\otimes \cZ\cup\{\R^{+}\}\otimes {X}$ as a  big family of ${\R}\otimes X$,   $\partial^{MV}$ is the Mayer-Vietoris boundary associated to the decomposition into $(\{\R^{-}\}\otimes \cZ,\{\R^{+}\}\otimes {X})$, and we use that 
 {
 the inclusion $X \to \R \times X$, $x \mapsto (0, x)$ induces an equivalence $E(\cZ) \simeq E(\{0\} \times \cZ)$.
}
 %the inclusions  $\{0\}\times Z \to [-n,m]\times Z$  of $\{0\}\times Z$ into  the member  of $\{\R^{-}\}\otimes \cZ\cap \{\R^{+}\}\otimes X$ is a 
%  coarse equivalence for every $n,m$ in $\nat$ and {$Z$ in $\cZ$}.

 \begin{rem}
 \label{IndexMapAsBoundary}
  The index map is the boundary map of the fibre sequence \eqref{gerwfwerfrewfwr}.
 In particular, we have a fibre sequence
 \begin{equation}
 \label{FibreSequenceIndexMap}
 \Sigma^{-1} E\cO^{\infty}_{{\{\emptyset\}}}(X) \to \Sigma^{-1} E\cO^{\infty}_{\cZ}(X)	\stackrel{a_{X, \cZ}}{\to} E(\cZ).
 \end{equation}
 \hB
\end{rem}

\subsection{$C^{*}$-algebras and categories and their $K$-theory}\label{gkpoegsergreg}

In this section we recall some basic constructions from non-commutative homotopy theory. 
In particular we recall the  $E$- and $K$-theory for graded $C^{*}$-algebra.  
Instead of providing explicit cycle-by-relation constructions  we will give a complete description 
in terms of universal properties. %Our philosophy is to deal with $K$-theory in an analysis-free 
%and formula-free manner. 
All needed classes will be derived from explicit  homomorphisms
between suitable $C^{*}$-algebras (see \cref{jiogowergwefrfwrf} and  \cref{okfpqfqewddedq}), and the   boundary operators in long exact sequences  
 come from general homotopy-theoretic principles (see \cref{egkopergwrefw}).

Let $G$ be a countable discrete group. By  $G\nCalg$ we denote the category of {not necessarily unital} $C^{*}$-algebras with a $G$-action
by automorphisms. The left action of $G$ on itself turns the Hilbert space $L^{2}(G)$ into a $G$-Hilbert space.
We let $K_{G}$ denote the algebra of compact operators on $L^{2}(G)\otimes \ell^{2}$ with the induced $G$-action. In the present paper we will need the cases of a trivial group and {the two-element group} $G=C_{2}$.
The category $G\nCalg$ has a symmetric monoidal structure $\otimes$ given by the maximal tensor product.  

%In the following we characterize the equivariant $E$-theory functor
%$$\ee^{G}:G\nCalg\to \EE^{G}$$
%essentially uniquely by its universal property.
We will consider the following properties of a functor $F:G\nCalg\to \bD$ to some cocomplete stable $\infty$-category $\bD$:
\begin{enumerate}
\item homotopy {invariance}: For every $A$ in $\nCalg$ the inclusion $A\to A\otimes C([0,1])$ (induced by the inclusion $\C\to C([0,1])$ as constant functions) induces an equivalence $F(A)\to F(A\otimes C([0,1]))$.
\item $K_{G}$-{stablility}: $F$ sends $K_{G}$-equivalences to an equivalences. Thereby   
$f:A\to B$ is a $K_{G}$-equivalence if $f\otimes \id_{K_{G}}:A\otimes K_{G}\to B\otimes K_{G}$ is a homotopy equivalence, or equivalently, is sent to an equivalence by every homotopy invariant functor. 
\item {exactness}: $F(0)\simeq 0$ and $F$ sends every exact sequence $0\to A\to B\to C\to 0$ in 
$G\nCalg$ to a fibre sequence $F(A)\to F(B)\to F(C)$ in $\bD$.
\item s-finitary: For every $A$ in $G\nCalg$ the natural morphism induces an equivalence $\colim_{A'\subseteq_{\sepa}A} F(A')\simeq F(A)$ , where $A'\subseteq_{\sepa}A$ is the poset of $G$-invariant separable subalgebras of $A$.
   \item sum-preserving: For any family $(A_{i})_{i\in I}$ in $G\nCalg$ the canonical map induced by the inclusions $A_{j}\to \bigoplus_{i\in I}A_{i} $ for all $j$ in $I$ induces an equivalence
   $\bigoplus_{i\in I} F(A_{i})\simeq F(\bigoplus_{i\in I}A_{i})$.
\end{enumerate}

\begin{rem}\label{egkopergwrefw}
Note that exactness is a property the functor can have or not. If $F$ is exact, then the boundary map
$\partial:F(C)\to \Sigma F(A)$ is implicitly encoded into the functor. 

Consider for example the $K$-theory functor $K:\nCalg\to \Mod(KU)$ described further below. This functor
turns out to be exact and therefore sends an exact sequence of $C^{*}$-algebras to a fibre sequence of $KU$-modules. Upon taking homotopy groups we get  a long exact sequence
$$K_{*}(A)\to K_{*}(B) \to K_{*}(X)\stackrel{\partial}{\to} K_{*-1}(A)$$ of abelian groups.
Classically one considers the graded group-valued functor  $K_{*}$ only. In order to even state the exactness of the latter one must provide the boundary as an additional datum. In the cycle-by-relation picture
it is given by an explicit formula involving unitaries and projections, and this approach requires various verifications of well-definedness, exactness, and naturality.  
Here different authors use different formulas which have to be compared.  Eventually, all these constructions,
at least up to sign, provide the boundary map  coming from homotopy theory. We refer to the discussion 
of this question in \cite{Bunke:2023ab}. 
\hB
\end{rem}

Following \cite{Bunke:2023aa} (for $G=e$) and \cite{budu} (for general $G$)  we adopt the following definition:

\begin{ddd}[{\cite{budu}}]\label{wgregwrefwrfw}
The equivariant $E$-theory functor
\[
\ee^{G}:G\nCalg\to \EE^{G}\]
 is  an initial     homotopy invariant, $K_{G}$-stable, exact,  sum-preserving  and $s$-finitary functor to a cocomplete stable $\infty$-category.
 %  \cite{Bunke:2023aa, budu}. 
 \end{ddd}
 
In other words, the equivariant $E$-theory functor is uniquely characterized by the universal property that $\EE^{G}$ is cocomplete and stable, and that 
$$(\ee^{G})^{*}: \Fun^{\colim}(\EE^{G}
,\bD) \to \Fun^{h,K_{G},\exa,\sfin,\oplus}(G\nCalg,\bD)$$
is an equivalence for any cocomplete stable $\infty$-category $\bD$, where  the target   is the full subcategory of $\Fun(G\nCalg,\bD)$ of functors having the list of properties indicated by the superscripts.
In the case of $G=\{e\}$ we will omit the superscript $G$.
%We ere $\EE(A,B):=\map_{\EE}(\ee(A),\ee(B))$ for the universal
% homotopy invariant, $K$-stable, exact,  sum-preserving  and $s$-finitary functor $\ee:\nCalg\to \EE$ to a presentable stable $\infty$-category 

Since for any $A$ in $G\nCalg$ precomposition
with $A\otimes -:G\nCalg\to G\nCalg$ preserves   homotopy invariant, $K_{G}$-stable, exact,  sum-preserving  and $s$-finitary functors the category  $\EE^{G}$ has furthermore a uniquely determined symmetric monoidal structure such that $\ee^{G}$ extends to a symmetric monoidal functor.
It turns out that $\EE^{G}$ is a presentably symmetric monoidal $\aleph_{1}$-presentable stable $\infty$-category \cite{budu}.
 
  The tensor unit of $\EE$ is $\ee(\C)$. One can check  \cite{Bunke:2023aa} that  the commutative ring spectrum  $$KU:=\map_{\EE}(\ee(\C),\ee(\C))$$ in $\CAlg(\Sp)$ is equivalent to the usual complex $K$-theory spectrum. 
  The symmetric monoidal functor $\Res_{G}:\nCalg\to G\nCalg$ equipping $C^{*}
$-algebras with the trivial $G$-action descends to a symmetric monoidal functor  $\Res_{G}:\EE\to \EE^{G}$.   Consequently,  $\EE^{G}$ is enriched over  the category $\Mod(KU)$ of $KU$-modules. 
  For $A,B$ in $G\nCalg$ we will write
  $$\EE^{G}(A,B):=\map_{\EE^{G}}(\ee^{G}(A),\ee^{G}(B))$$
  for the bivariant $\EE^{G}$-theory  $KU$-module of $A,B$.  Considering the mapping spectra in $\EE^{G}$ as $KU$-modules encodes Bott-periodicity in a natural way.

\begin{rem}
\label{kojgwpergwrefwefrwef}As shown in  \cite{budu}, for separable $A,B$ in $G\nCalg$  there is a canonical isomorphism  of groups  
between $\pi_{0}\EE^{G}(A,B)$ and the classical equivariant $E$-theory groups
as defined, e.g., in \cite{Guentner_2000} using homotopy classes of   asymptotic morphisms. Furthermore, the restriction of $\ee^{G}$ to   separable $G$-$C^{*}$-algebras yields a
functor $G\nCalg_{\sepa}\to (\EE^{G})^{\aleph_{1}} $, where  $ (\EE^{G})^{\aleph_{1}}$ denotes   the full subcategory of $\aleph_{1}$-compact objects in $\EE^{G}$. Upon going to the homotopy category this functor is equivalent to the equivariant $E$-theory functor 
  constructed  in \cite{Guentner_2000}.  
This justifies to consider the
functor characterized in \cref{wgregwrefwrfw} as the correct $\infty$-categorical enhancement of classical $E$-theory. \hB
\end{rem}

We next 
 describe the extension of $E$-theory to graded $C^{*}$-algebras. Our approach   is an $\infty$-categorical version of \cite{MR1694805}.  The category   $\gCalg$ of graded $C^{*}$-algebras is the same as the category $C_{2}\nCalg$ of $C_{2}$-$C^{*}$-algebras, but we equip it with  the graded version $\stimes$ of the maximal tensor product which involves Koszul sign rules. 
 One again checks that for any $A$ in $ C_{2}\nCalg$  precomposition by the functor $A\stimes -:C_{2}\nCalg\to C_{2}\nCalg$ preserves    homotopy invariant, $K_{{C_2}}$-stable, exact,  sum-preserving  and $s$-finitary functors. Consequently, 
the category $\EE^{C_{2}}$ has an essentially unique  symmetric monoidal structure $\stimes$ such that $\ee^{C_{2}}$ has an essentially unique symmetric monoidal refinement with respect to the graded tensor product $\stimes$ on $C_{2}\nCalg = \gCalg$. 

We now consider the graded $C^{*}$-algebra $\cS:=C_{0}(\R)$ with the {grading involution} sending $f(x)$ to $f(-x)$.
If $A$ is a  graded $C^{*}$-algebra and $\Phi$ is a selfadjoint odd unbounded  multiplier of $A$ such that $(i+\Psi)^{-1}\in A$, then using function calculus we can define a morphism of graded $C^{*}$-algebras 
\begin{equation}
\label{fqewfewsdfvv}
\Psi_*:\cS\to A\ , \quad  
f\mapsto f(\Psi)\ .
\end{equation}
We can recover $\Psi$ by $\Psi=\Psi_{*}(X)$, where $X$ is the multiplier of $\cS$ generating the identity morphism {(given by $(Xf)(x) := xf(x)$)}, and where we implicitly used the extension of $\Psi_{*}$ to some unbounded multipliers.
The algebra $\cS$ has a structure of a coalgebra in $\gCalg$ with counit $\epsilon:\cS\to \C$ given by $f\mapsto f(0)$ (induced by the multiplier $0$ on $\C$) and the coproduct 
\begin{equation}
\label{Coproduct}
	\Delta:\cS\to \cS\stimes \cS,
\end{equation}
induced by the multiplier
$X\stimes 1+1\stimes  X$ (see, e.g., \cite{zbMATH02068495} for a nice description).
Since $\ee^{C_{2}}$ is symmetric monoidal,    $\ee^{C_{2}}(\cS)$ becomes a  coalgebra in $\EE^{C_{2}}$. 

%We use the equivariant $E$-theory functor
%$$\ee^{C_{2}}:C_{2}\nCalg\to \EE^{C_{2}}$$ which similarly as the non-equivariant version can  be characterized as the initial
%homotopy invariant, $K_{C_{2}}$-stable, exact, sum-preserving and s-finitary functor to a presentable stable $\infty$-category \cite{budu}. 

 \begin{ddd}We denote by
  \[
  \gEE:= \Comod_{\EE^{C_{2}}}(\ee^{C_{2}}(\cS))
  \]
  the presentably symmetric monoidal   stable $\infty$-category  
  of comodules over the coalgebra $\ee^{C_{2}}(\cS)$ in $\EE^{C_{2}}$.
  We further define the lax symmetric monoidal functor
  $$\gee:= \ee^{C_{2}}(\cS)\stimes \ee^{C_{2}}(-):\gCalg\to \gEE\ .$$ \end{ddd}  Thus by  definition, the functor $\gee$
  sends  a graded $C^{*}$-algebra $A$ to the cofree comodule $\ee^{C_{2}}(\cS)\stimes \ee^{C_{2}}(A)$ over  {$\ee^{C_2}(\cS)$}.     
  The functor $\gee$ is homotopy invariant, $K_{{C_2}}$-stable, exact, sum-preserving and  s-finitary. 
  {We do not know whether this functor satisfies a universal property.}
   
     \begin{rem} In order to see that our definition reproduces the classical $E$-theory groups  for   graded $C^{*}$-algebras $A,B$    we observe (using the universal property of the cofree comodule functor) that 
  \begin{equation}\label{gweroighjweoirfwerfw}  \gEE(A,B)\simeq \EE^{C_{2}}(\cS\stimes A,B)\ .
\end{equation} 
 
   If $A$ and $B$ are separable,  then it  follows from   \cite{MR1694805} and the fact \cite{budu} that the group $\pi_{0}\EE^{C_{2}}(\cS\stimes A,B)$  is the classical $C_{2}$-equivariant $E$-theory group  that 
  $\pi_{0}\EE^{\gr}(A,B)$   is classical graded $E$-theory group of $A$ and $B$. \hB\end{rem}

      Let  $\triv:\nCalg\to \gCalg$ be the functor (the same as $\Res_{C_{2}}$) which equips a $C^{*}$-algebra with the trivial 
grading.  Then we have a commutative diagram 
 \begin{equation}\label{werfwefwvsdsf1}\xymatrix{\nCalg\ar[d]^{\triv}\ar[r]^{\ee}&\EE\ar[d]^-{\ee^{C_{2}}(\cS) \stimes \triv(-)}\\ \gCalg\ar[r]^-{\gee}&\gEE} \ .
\end{equation}  
We argue that the right vertical arrow   is fully faithful.  In other words,  
  the $E$-theory for graded $C^{*}$-algebras extends the $E$-theory for ungraded ones. To this end, for $A,B$ in $\EE$, we consider the   diagram 
\[
\begin{tikzcd}
&[0cm] &[-1cm]
\map_{\EE^{\gr}}(\ee^{C_{2}}(\cS)\stimes \triv(A),\ee^{C_{2}}(\cS)\stimes \triv(B))
\ar[d, "\epsilon_{*}"', "\simeq{,} \eqref{gweroighjweoirfwerfw}"]
\\
\map_{\EE}(A,B)
\ar[urr, dashed, "\simeq", bend left=9]
\ar[r, "\Res_{C_{2}}"']
\ar[d, equal]
&
	\map_{\EE^{C_{2}}}(\Res_{C_{2}}(A),\Res_{C_{2}}(B))
	\ar[r, "\epsilon^{*}"]
	\ar[d, "\simeq"]
	&
		\map_{\EE^{C_{2}}}(\ee^{C_{2}}(\cS)\stimes \triv(A),\triv(B))		
		\ar[d, "\simeq"]
 \\
\map_{\EE}(A,B)
\ar[rr, dotted, "\simeq", bend right=20]  \ar[r, "\pi^{*}"]
	&
		\map_{\EE}((\ee^{C_{2}}(\C)\rtimes C_{2})\otimes A,B)\ar[r, "(\epsilon\rtimes C_{2})^{*}"]
		&
			\map_{\EE}((\ee^{C_{2}}(\cS)\rtimes C_{2})\otimes A,B)
 \ ,
\end{tikzcd}
 \]
 where the two unnamed vertical equivalences are given by the dual Green-Julg adjunction 
 (see e.g. \cite[Prop. 3.61.2]{budu}).
One can further check using the explicit description of the unit of the dual Green-Julg  adjunction   that the map maked by $\pi^{*}$
is induced by the homomorphism $\pi:\C\rtimes C_{2}\to \C$ induced by the trivial representation of $C_{2}$. 
The dotted arrow is an equivalence since by \cite[Lem. 7.3]{Bunke:2017aa} 
the composition $$ \ee^{C_{2}}(\cS)\rtimes C_{2}\stackrel{\epsilon\rtimes C_{2}}{\to}
\ee^{C_{2}}(\C)\rtimes C_{2}\stackrel{\pi}{\to} \ee(\C)$$
is an equivalence.
We can conclude that the dashed arrow
is an equivalence as desired.

  We use the graded $E$-theory in order to define the $K$-theory functor  for graded $C^{*}$-algebras. \begin{ddd}\label{garefawefeawfawf} We define the lax symmetric monoidal $K$-theory functor for graded $C^{*}$-algebras  as the composition $$\gK: \gCalg \xrightarrow{\gee}\gEE \xrightarrow{\map_{\gEE}(\gee(\C),-)} \Mod(KU)\ .$$ We further define the  lax symmetric monoidal  $K$-theory functor for $C^{*}$-algebras by
 $$K:\nCalg\xrightarrow{\triv} \gCalg\xrightarrow{\gK} \Mod(KU)\ .$$ 
 \end{ddd}
  Since the right vertical arrow in \eqref{werfwefwvsdsf1} is fully faithful  
  we have  $K(-)\simeq \EE(\C,-)$ which shows that our definition reproduces the classical definition of the $K$-theory functor for $C^{*}$-algebras.
    
    \begin{rem}\label{jiogowergwefrfwrf} 
    In this remark we recall the model-free description of the $K$-theory classes associated to a projection or a unitary in an ungraded algebra.
   
   Let  $p$ be  a self-adjoint projection in a $C^{*}$-algebra $Q$.
    It gives rise to a homomorphism $\pi:\C\to {Q}$ sending $1$ to $p$.
   The morphism $\ee(\pi):\ee(\C)\to \ee(Q)$  in $\pi_{0}\EE(\C,{Q})\cong K_{0}(Q)$ is the $K$-theory class $[p]$ associated to $p$.
   
   Let $I$ be a $C^{*}$-algebra and $I^{+}$ denote its unitalization.
   A normalized unitary $u$ in $I$ is a unitary $u$ in $I^{+}$ which is sent to $1$ by the canonical homomorphism $I^{+}\to \C$. If $u$ is a normalized unitary  in $I$, then   
by function calculus it gives rise to a homomorphism
   \[
   \mu:C(S^{1},\{1\})\to I\ , \quad f\mapsto f(u)\ ,
   \]
    where $C(S^{1},\{1\}):=\ker(C(S^{1})\xrightarrow{f\mapsto f(1)}\C)$. The composition
   $\Omega \ee(\C)\simeq \ee({C}(S^{1},\{1\}))\xrightarrow{{e(\mu)}}\ee(I)$ represents the class
   $[u]$ in $\pi_{0}\Sigma \EE(\C,I)\cong K_{-1}(I)$ associated to $u$.

Assume that 
  \begin{equation}\label{gwergwpeofpwerf}0\to I\to A\to Q\to 0
\end{equation}  
is an exact sequence of $C^{*}$-algebras such that $A\to Q$  is a unit-preserving
  morphism between unital algebras. Let $p$ be a projection in $Q$. Then we can find a lift
  $\tilde p$ in $A$ which is selfadjoint and has spectrum in the unit interval $[0,1]$. We 
  get a normalized unitary $u:=e^{2\pi i\tilde p}$ in $I$ where we identify $I^{+}$ with the subalgebra $I+1_{A}\C$ of $A$.
  Using the  relation of the triangulated structure of $\ho\EE$ with Puppe sequences one can check that
  \begin{equation}\label{gwegrwerfwerfwef}
  \partial [p]=[u]\ ,
\end{equation}
where $\partial$ is the boundary operator
$K_{0}(Q)\to K_{-1}(I)$ associated to the exact sequence \eqref{gwergwpeofpwerf}.
    \hB    
\end{rem}

\begin{rem}
\label{okfpqfqewddedq}
Recall that  
\[
\gK(A)\simeq \gEE(\C,A)\simeq \EE^{C_{2}}(\cS,A)\ .
\]
  In particular, every homomorphism $\cS\to A$ of graded $C^{*}$-algebras (e.g., the one in \eqref{fqewfewsdfvv}) represents a class in $\gK_{0}(A)$. 
If this homomorphism is {$\Psi_{*}:\cS\to A$, $\Psi_{*} (f)=f(\Psi)$, for an unbounded multiplier $\Psi$ of $A$},  then we write $[\Psi]$ for the corresponding class in $\gK_{0}(A)$.
  
  {We will use the following variation on this construction:}
  Let $a$ be in $(0,\infty)$. 
  Then pull-back along the map 
  \begin{equation}\label{bsdfpokbspdvsdfvsfdv}
 {\kappa}: (-a,a)\to \R\ , \quad t\mapsto \frac{t}{\sqrt{a-|t|^{2}}}\end{equation}
  induces an isomorphism ${\kappa^*} : \cS\to C_{0}((-a,a))$ of graded $C^{*}$-algebras. 
It is a homotpy inverse of the extension by zero map $\epsilon:C_{0}((-a,a))\to \cS$.
 
 Assume that ${j} :B\to A$ is the inclusion of a graded subalgebra and $\Psi$ is an unbounded multiplier of $A$ as above such that
 $\epsilon(f)(\Psi)\in B$ for all $f$ in $C_{0}((-a,a))$. 
 Then  we get a class $[\Psi_{B}]\in \gK_{0}(B)$ represented by the homomorphism 
 %$C_{0}((-a,a))\to B$, $f\mapsto \epsilon(f)(\Psi)$, .
 \[
   \Psi_B(f) = (\kappa^*f)(\Psi):\cS\to  B\ .
 \]
It satisfies ${j} _{*}[\Psi_{B}]=[\Psi]$.
   \hB \end{rem}

    \begin{rem}\label{wgklopergwrefwef} In this remark we provide an explicit formula for the cup product of $K$-theory classes.
Let  $A, B$ be graded $C^{*}$-algebras, 
  $[a]$ be a class in $\gK_{0}(A)$ represented by $a:\cS\to A$, and $[b]$ be a class in $\gK_{0}(B)$ represented by 
  $b:\cS\to B$.  Then the  product 
  $[a]\cup [b]$ in $ \gK_{0}(A\stimes B)$ is   by definition the image of $([a],[b])$ under the structure map 
  \[
  \gK(A)\times \gK(B)\to \gK(A\stimes B)
  \]
 of the symmetric monoidal structure of $\gK$.
  Explicitly, $[a]\cup [b]$ is represented by the composition
  \[
  \cS\xrightarrow{\Delta}  \cS\stimes \cS\xrightarrow{a\stimes b}A\stimes B\ ,
  \]
{where $\Delta$ is the coproduct \eqref{Coproduct} of $\cS$.
}
 \hB
\end{rem}

\begin{ex}\label{ekohperhetrgertg}
If $V$ is a finite-dimensional graded Hilbert space with an action of $\Cl^{n}$ 
such that the generators act by anti-selfajoint operators, then 
$\End_{\Cl^{n}}(V)$ is a graded $C^{*}$-algebra.
The evaluation gives an isomorphism of $\Cl^{n}$-modules
\begin{equation}\label{qewfdqwedqq}  \Cl^{n}\stimes \Hom_{\Cl^{n}}(\Cl^{n},V)\stackrel{\cong}{\to}V  \ ,\end{equation} 
where $\Cl^{n}$ acts by right-multiplication in itself. It induces
  an isomorphism of graded $C^{*}$-algebras
\begin{equation}\label{qewfdqwedq}  \Cl^{n}\stimes \End_{\C} (\Hom_{\Cl^{n}}(\Cl^{n},V))\cong \End_{\Cl^{n}}(V) \ ,
\end{equation} 
where the first factor $\Cl^{n}$ in \eqref{qewfdqwedq} acts on  the first factor $\Cl^{n}$ in \eqref{qewfdqwedqq} by left multiplication. 
Since $\ee^{\gr}( \End_{\C} (\Hom_{\Cl^{n}}(\Cl^{n},V)))\simeq \ee^{\gr}(\C)$
by Morita invariance we therefore have an equivalence
$\ee^{\gr}(\End_{\Cl^{n}}(V))\simeq \ee^{\gr}(\Cl^{n})$.

Let $X:\R^{n}\to \Cl^{n}$ be the canonical linear map which we consider as an unbounded multiplier
on $C_{0}(\R^{n}, \Cl^{n})$. Then   
$(iX)_{*}:\cS\to C_{0}(\R^{n}, \Cl^{n} )$ represents an invertible class (Kasparov's Bott element) $\beta$  in
$\EE^{\gr}(\C,C_{0}(\R^{n}, \Cl^{n} ))$  and hence an equivalence
$$
\ee^{\gr}(\C)\stackrel{\beta}{\simeq}\ee^{\gr}(C_{0}(\R^{n}, \Cl^{n} ))\simeq \Omega^{n}\ee^{\gr}( \Cl^{n} )\simeq 
\Omega^{n}\ee^{\gr}(\End_{\Cl^{n}}(V))\ .$$
Upon tensoring with $\ee^{\gr}(A)$ and applying $\K^{\gr}$ its $n$-fold desuspension  $\Sigma^{n}\ee^{\gr}(\C)\stackrel{\simeq}{\to}\ee^{\gr}(\End_{\Cl^{n}}(V))$ induces an equivalence 
 \begin{equation}
 \label{gertgetrge}
\Sigma^{n} K^{\gr}(A) \simeq   K^{\gr}(A\stimes \End_{\Cl^{n}}(V)) 
\end{equation} 
for any graded $C^{*}$-algebra $A$.
In particular, 
\begin{equation}\label{gkwoperkgopwergfwf}
 K_{*}^{\gr}(A\stimes \End_{\Cl^{n}}(V))\cong K^{\gr}_{*-n}(A)\ . 
 \end{equation}
\hB
\end{ex}

\begin{ex}
We consider the exact sequence \eqref{gwergwpeofpwerf} and a 
 normalized unitary $u=e^{2\pi i \tilde p}$ in $I$ for some selfadjoint $\tilde p$ in $A$ with spectrum in $[0,1]$
 whose image in $Q$ is a projection.
 Taking $V=\Cl^{1}$ as a right $\Cl^{1}$-module we have $\End_{\Cl^{1}}(\Cl^{1})\cong \Cl^{1}$ and
\eqref{gkwoperkgopwergfwf} gives an isomorphism 
\begin{equation}\label{broptkbopdfgbdfgb}K_{-1}(I)\cong K_{-1}^{\gr}(I)\cong K_{0}^{\gr}(I\stimes \Cl^{1})\ .
\end{equation} 
 One can check using the explicit description of Kasparov's Bott element $\beta$ and \eqref{bsdfpokbspdvsdfvsfdv}
that the image of $[u]$ in $K_{0}^{\gr}(I\stimes \Cl^{1})$ under \eqref{broptkbopdfgbdfgb} is represented by
the map
\begin{equation}\label{jgiowjeroijfoweferfwef}
\cS\to I\stimes \Cl^{1}\ , \quad   f\mapsto f\left(\frac{2\tilde p-1}{\sqrt{1-(2\tilde  p-1)^{2}}}\stimes i\sigma\right)
\ .\end{equation} \hB
\end{ex}

We now extend the  $E$-theory and $K$-theory functors from $C^{*}$-algebras to $C^{*}$-categories. We start with the characterization of $C^{*}$-categories following  \cite{crosscat}.
A $\C$-linear  $*$-category $\bC$ is a category enriched in complex vector spaces    
with an involution $*:\bC^{\op}\to \bC$ which fixes objects and acts complex anti-linearly
on morphism vector spaces. 
A morphism between   {$\C$-linear $*$-categories} is a functor  which is compatible with the enrichment in $\C$-vector spaces and the involutions.

A $C^{*}$-algebra can be considered as a $\C$-linear $*$-category with a single object. We consider a  morphism $f$ in $\bC$. We define  the maximal norm  of $f$ by  $\|f\|_{\max}:=\sup_{\rho} \|\rho(f) \|$,
where $\rho$ runs over all functors $\rho:\bC\to A$ of $\C$-linear $*$-categories  from $\bC$ to $C^{*}$-algebras $A$. We say that $\bC$ is a pre-$C^{*}$-category if $\|f\|_{\max}<\infty$ for every morphism $f$ in $\bC$. Finally, $\bC$ is a $C^{*}$-category if it is a pre-$C^{*}$-category such that
its morphism spaces are complete in the norm $\|-\|_{\max}$.
A morphism between $C^{*}$-categories is just a morphism of $\C$-linear $*$-categories. 

\begin{rem}\label{joihorthrthgerg}
A morphism $\bC\to \bD$ of $C^{*}$-categories is an ideal inclusion if it induces a bijection on objects and
 inclusions  on the level of morphisms so that the image satisfies the obvious generalization of
the conditions for an ideal to categories. For an ideal inclusion we can form the quotient $\bD/\bC$
which has the same objects as $\bD$ and whose morphism spaces are the quotients of the morphism spaces of $\bD$
by the images of the corresponding morphism spaces of $\bC$. 
\hB
\end{rem}

A $G$-$C^{*}$-category is a $C^{*}$-category with an action of $G$ by automorpisms.
We let $G\nCcat$ denote the category of $G$-$C^{*}$-categories and equivariant morphisms.
 By  \cite{joachimcat} we have an adjunction
\begin{equation}\label{rfwrefwerfwervvsfvsfdvsvsdfv}A^{f}:G\nCcat\leftrightarrows G\nCalg:\incl
\end{equation} where the inclusion views a $G$-$C^{*}$-algebra as a $G$-$C^{*}$-category with a single object. 

 A graded $C^{*}$-category is a $C^{*}$-category with a strict $C_{2}$-action fixing objects.
  We let $\gCcat$ be the full subcategory of $C_{2}\nCcat$ of graded $C^{*}$-categories. 
  The adjunction 
  \eqref{rfwrefwerfwervvsfvsfdvsvsdfv} restricts to an adjunction 
  $$A^{f,\gr}:\gCcat\leftrightarrows \gCalg:\incl\ .$$
   We use this functor in order to extend the $E$-theory to graded $C^{*}$-categories.
   
   \begin{ddd}
We define the  $E$-theory of graded $C^{*}$-categories as the composition
  $$\gee:\gCcat\xrightarrow{A^{f,\gr}} \gCalg \xrightarrow{\gee} \gEE\ .$$
We further define  the $K$-theory of graded $C^{*}$-categories by 
$$\gK: \gCcat \xrightarrow{\gee}\gEE \xrightarrow{\map_{\gEE}(\gee(\C),-)} \Mod(KU)\ .$$
\end{ddd}
The compositions
\begin{equation}\label{efqwfewfqefe}\ee:\nCcat\xrightarrow{\triv} \gCcat\xrightarrow{\ee} \gEE
\end{equation} and 
\begin{equation}\label{efweqfdqewdewdwed}
K:\nCcat\xrightarrow{\triv} \gCcat\xrightarrow{\gK} \Mod(KU)
\end{equation}

are  the usual $E$- and $K$-theory functors for $C^{*}$-categories.
%We have a canonical  commutative diagram
%  $$\xymatrix{\nCcat\ar[dd]^{\triv}\ar[dr]^{K}&\\&\Mod(KU)\\\gCcat\ar[ur]_{\gK}}$$
%  expressing the fact that the $K$-theory for graded $C^{*}$-categories extends the $K$-theory of ungraded $C^{*}$-categories, where $\triv$ equips a $C^{*}$-category with the trivial grading.
%

%We extend the equivariant $E$-theory functor to $G$-$C^{*}$-categories (using the same symbol) by setting $$\ee^{G}:G\nCcat\xrightarrow{A^{f}}G \nCalg\xrightarrow{\ee^{G}} \EE^{G}\ .$$
One can check  that $\gee:\gCcat\to \gEE$  again has a symmetric monoidal refinement if we equip $\gCcat$ with the maximal tensor product. We will use that the $E$-theory functor   for $C^{*}$-categories from \eqref{efqwfewfqefe}  is a finitary homological functor in the sense of \cite[Def. 3.24]{coarsek}. In particular, it preserves filtered colimits, sends unitary equivalences to equivalences, sends exact sequences to fibre sequences, and annihilates flasque $C^{*}$-categories \cite[Def. 11.3]{cank}. %\fml{Was sind denn flasque $C^*$-categories?} 
The arguments are the same as for the case of $KK$-theory in \cite[Sec. 6 \& 7]{KKG}.

\subsection{$X$-controlled Hilbert spaces and coarse $K$-homology}\label{gojkperfrefrwfwrefw}

For the present paper the main example of a coarse homology theory is coarse $K$-theory $$K\cX:\BC\to \Mod(KU)$$  
with values in the stable $\infty$-category of $KU$-modules which  we describe in the following.

Let $X$ be a set. 
An $X$-controlled Hilbert space is a pair $(H,\chi)$ of a Hilbert space and finitely additive projection valued measure  $\chi:\cP(X)\to B(H)$. 
We say that $(H,\chi)$ is determined on points if $\bigoplus_{x\in X} \chi(\{x\})H\cong H$.
If $X$ has a bornology $\cB$, then we say that $(H,\chi)$ is locally finite if $\chi(B)$ is finite-dimensional for all $B$ in $\cB$.

%bounded subsets of $X$ 
%Let $X$ be a bornological coarse space. We first recall the notion of a locally finite  $X$-controlled Hilbert space which is determined on points.
% This is a pair $(H,\chi)$ consisting of  a Hilbert space $H$ and  a  finitely additive projection valued measure  $\chi:\cP(X)\to B(H)$ such that $\chi(B)$ is finite-dimensional for all bounded subsets of $X$ (this is the local finiteness condition) and $\bigoplus_{x\in X} \chi(\{x\})H\cong H$ (this is the condition of being determined on points). 

Let $(H,\chi)$ and $(H',\chi')$ be two $X$-controlled Hilbert spaces. If $U$ is an entourage of $X$, then a bounded operator $A:H\to H'$ is $U$-controlled  if $\chi'(Z')A\chi(Z)=0$ for all subsets $Z,Z'$ of $X$ with $Z'\cap U[Z]=\emptyset$.
If $X$ has a coarse structure $\cC$, then a  controlled operator is an operator which is $U$-controlled from some $U$ in $\cC$.
 
  If $f$ in $\ell^{\infty}(X)$ is a bounded function on $X$, then we can consider it as a $\diag(X)$-controlled operator $\chi(f):H\to H$  acting by $f(x)$ on the summand $\chi(\{x\})H$ of $H$.

We first construct a functor
$$\bC:\BC\to \nCcat$$
which associates to every bornological coarse space the Roe category $\bC(X)$.
\begin{ddd}\mbox{}\begin{enumerate}
\item Objects of $\bC(X)$: The objects of $\bC(X)$ are  the locally finite $X$-controlled Hilbert spaces $(H,\chi)$ which are determined on points.
\item Morphisms of  $\bC(X)$: The morphisms $A:(H,\chi)\to (H',\chi')$  are bounded operators
which can be approximated in norm  by controlled operators.
\item Involution: The involution on $\bC(X)$ takes the adjoint operator.
 \item $\bC(f)$: For a morphism $X\to Y$ of bornological coarse spaces the  functor  $\bC(f):\bC(X)\to \bC(Y)$ sends
 $(H,\chi)$ to $(H,f_{*}\chi)$. It acts by the identity on morphisms.
\end{enumerate}
\end{ddd}
In the following definition we use the $K$-theory functor for $C^{*}$-categories from \eqref{efweqfdqewdewdwed}. 
 
 \begin{ddd}
\label{grwgjrogwregwre} 
We define the coarse $K$-homology functor as the composition  
$$K\cX:\BC\xrightarrow{\bC}\nCcat\xrightarrow{K}\Mod(KU)\ .$$
\end{ddd}
\begin{theorem}[\cite{buen},\cite{coarsek},\cite{Bunke:2019aa}]\label{gpkopwergwrfwferfwf}
$K\cX$ is a coarse homology theory which is in addition continuous, additive and strong and has a lax symmetric monoidal refinement.
\end{theorem}

\begin{rem}Coarse $K$-homology as a $\Z$-graded group-valued functor for proper metric spaces has been introduced again by Roe, see \cite{higson_roe} for a text-book account. Upon taking homotopy groups the above  construction extends the classical one to general bornological coarse spaces.  
Note that $KU\simeq K\cX(*)$.

The idea of using $C^{*}$-categories of $X$-controlled Hilbert spaces appeared already in \cite{zbMATH02196178}. 
But note that the details are different. The $C^{*}$-category $\bC(X)$ from \cref{grwgjrogwregwre}  is defined for arbitrary bornological coarse spaces and its objects are locally finite, and the control is implemented by a  finitely additive projection-valued measure. The 
$X$-controlled Hilbert spaces in \cite{zbMATH02196178} are defined only for proper metric spaces and usually are not locally finite, and the control
is  implemented by an action of the algebra  $C_{0}(X)$  of continuous functions  vanishing at $\infty$.
\hB\end{rem}

\begin{rem}\label{wegkowpergrwfrewfrefw}
  \cref{gpkopwergwrfwferfwf} is deduced from \cite[Thm. 7.3]{coarsek} and the fact that
$K:\nCcat\to \Mod(KU)$ is a finitary homological functor in the sense of  \cite[Def. 3.23]{coarsek}.
Note that the $E$-theory functor $\ee:\nCcat\to \EE$ itself is a  finitary homological functor with values in the stable $\infty$-category  $\EE$. Therefore the composition
\[
\ee\cX:\BC\xrightarrow{\bC}\nCcat
\xrightarrow{\ee} \EE
\]
is an $\EE$-valued coarse homology theory.  We will in particular use coarse invariance, 
excision, vanishing on flasques and the version of \eqref{fqwedwedewdqwdwd} for $\ee\cX$. \hB
 \end{rem}

  Let $X$ be a bornological coarse space, $Y$ be a subset,  and $\cY$ be a big family on $X$.
  
    \begin{ddd}\label{jkohpertgtege} We let $\bC(Y\subseteq X)$ be the wide subcategory of the Roe category $\bC(X)$  given by operators of the form $\chi'(Y)A\chi(Y)$ for $A$ in $\bC(X)$.
   We further define $\bC(\cY\subseteq X)$ as the ideal in $\bC(X)$ generated by the wide subcategories $ \bC(Y\subseteq X)$ for all 
 members $Y$ of $\cY$.
   \end{ddd}
   
\begin{rem}\label{korphrethertgertg}
Note that $\bC(\cY\subseteq X)$ has the same set of objects as $\bC(X)$
and the inclusion $\bC(\cY\subseteq X)\to \bC(X)$ is an ideal inclusion in the sense explained in  \cref{joihorthrthgerg}
allowing to form the quotient $\frac{ \bC(X)}{\bC(\cY\subseteq X)}$. In contrast, the canonical morphism
$\bC(\cY)\to \bC(X)$ is not an ideal and we can not form the corresponding quotient.
But this morphism factorizes over a morphism
$ \bC(\cY)\to \bC(\cY\subseteq X)$ which by \cref{wroekgpwergwrfgwrf} below induces an equivalence in $K$-theory. \hB
\end{rem}

 Let $\cY$ be a big family in  a bornological coarse space $X$. \begin{lem}\label{wroekgpwergwrfgwrf}
We have a canonical equivalence  
 \begin{equation}\label{wregjeoigwegwre}
 K\cX(\cY)\simeq K(\bC(\cY\subseteq X)) \ .
\end{equation}
\end{lem} 
\begin{proof}
 By definition  $$K\cX(\cY)\simeq \colim_{Y\in \cY} K\cX(Y)\simeq  \colim_{Y\in \cY} K( \bC(Y))\ .$$ %Since  the $K$-theory functor for $C^{*}$-categories commutes with filtered colimits   the canonical morphism is an equivalence 
%$$K\cX(\cY)\simeq  \colim_{i\in I} K( \bC(Y_{i}))\stackrel{\simeq}{\to} K( \colim_{i\in I}\bC(Y_{i}))\ .$$  
By \cite[Lem. 6.10]{coarsek}
for every $Y$ in $\cY$ we have a unitary  equivalence
$\bC(Y)\to \bC(Y\subseteq X)$. Since $K$ sends unitary equivalences  of $C^{*}$-categories to equivalences
we get
$$\colim_{Y\in \cY} K(\bC(Y)\simeq \colim_{Y\in \cY}K( \bC(Y \subseteq X))\ .$$
Since $ \colim_{Y\in \cY} \bC(Y\subseteq X)\cong \bC(\cY\subseteq X)$ by definition
and $K$ preserves filtered colimits we get
$$ \colim_{Y\in \cY}K( \bC(Y\subseteq X))\simeq K(\bC(\cY\subseteq X))\ .$$ Combining   these equivalence 
we get the equivalence  \eqref{wregjeoigwegwre}. \end{proof}

\begin{rem}\label{kowprgwertgwerfw} In the argument for \cref{wroekgpwergwrfgwrf}
 we can replace $K$ by $\ee$.  We then  get a canonical equivalence
\[
 \ee\cX(\cY)\simeq \ee(\bC(\cY\subseteq \cX))\ .
\tag*{$\blacksquare$} 
\]
\end{rem}

\section{Coarse coronas}

\subsection{Coarse coronas and a commutator estimate}\label{ojgpwegwregrfrwef}

Let $X$ be a set, $U$ be an entourage on $X$, and $W$ be a subset of $X$. For  a function $f:X\to \C$  we define the $U$-variation of $f$ on $W$ by 
\[
\Var_{U}(f,W):=\sup_{(x,y)\in U\cap (W\times W)} |f(x)-f(y)|\ .
\]

For a set $X$ we let $\ell^{\infty}(X)$ denote the $C^{*}$-algebra of all bounded functions $X\to \C$
with the supremum norm $\|f\|:=\sup_{x\in X} |f(x)|$.

Let $\cY$ be a filtered family of subsets in $X$.
 
 \begin{ddd}
 	
\label{gjsoepgergseffs} \mbox{}
\begin{enumerate}
\item	
The $C^{*}$-algebra $\ell^{\infty}(\cY)$ of functions  {vanishing} away from $\cY$ is defined as the sub-$C^*$-algebra of $\ell^{\infty}(X)$ of functions $f$ satisfying 
 $$\lim_{Y\in \cY} \|f_{M\setminus Y}\|=0\ .$$ 
 \item
 \label{wkotpgegfergweg} 
For a coarse space $X$ with coarse structure $\cC$ we define the algebra of bounded functions with vanishing variation away from $\cY$  as 
\[
\ell^{\infty}_{\cY}(X):=\{f\in \ell^{\infty}(X)\mid  { \forall U\in \cC : {\lim_{Y\in \cY}}\Var_{U}(f,X\setminus Y ) =0 }\}\ .
\]
\end{enumerate}
 \end{ddd}
 
From now one we assume that $X$ is a coarse space.
We  have an exact sequence of $C^{*}$-algebras
\begin{equation}
\label{gerwrgfrqlll}
0\to \ell^\infty(\cY)\to \ell^\infty_{\cY}(X)\to C(\partial^{\cY} X)\to 0 \end{equation} 
defining the unital quotient $C^{*}$-algebra  $C(\partial^{\cY} X)$.

\begin{ddd}\label{rkopthtrhegertg}
The coarse $\cY$-corona of $X$ is the compact topological space $\partial^{\cY} X$ defined as the Gelfand dual of $C(\partial^{\cY}X)$.
\end{ddd}

\begin{rem}
 We have the following cartoon picture of the {coarse} $\cY$-corona:
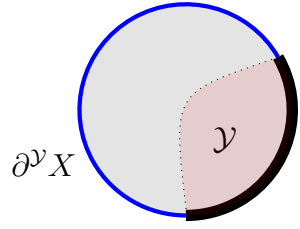
\begin{equation*}
\begin{tikzpicture}
\filldraw[color=gray!70,opacity=0.3] (0,0) circle (1.4cm);
\draw[ultra thick, blue] (1.2,0.7) arc (30:270:1.4);
\draw[line width=1.5mm, name path=B] (1.2,0.7) arc (30:-90:1.4);
\draw[dotted, name path=A] (1.2,0.7) .. controls (-0.2,0.2) and (-0.2,0.3) .. (0,-1.4);
\node at (0.5,-0.4) {$\cY$};
\node at (-1.88,-0.7) {$\partial^{\mathcal{Y}} X$};
%\node at (-0.5, 0.4) {$X$};
\tikzfillbetween[of=A and B]{red, opacity=0.1};
\end{tikzpicture}	
\end{equation*}
The grey interior is the space $X$, with the boundary depicting its $\cY$-completion, i.e., the Gelfand dual $\smash{\overline{X}^\cY}$ of the algebra $\ell^\infty_{\cY}(X)$.
The boundary $\smash{\overline{X}^\cY} \setminus X$ of this compactification consists of two parts: The boundary $\partial^{\cY} X$ (depicted blue in the picture), and {its} complement {in $\smash{\overline{X}^\cY} \setminus X$}, which comes from the {Gelfand dual of $\ell^\infty(X)$}. 
%\nml{Stone-\v{C}ech compactification of $X$.}\fml{Stone-Cech braucht stetige Funktionen. Hier nehmen wir $\ell^\infty$.}
\hB
\end{rem}

 \begin{ex}
 Let $W$ be a subset of $X$. 
 The coarse boundary of $W$ is defined as the  big family $\partial W$ of $X$ consisting of the subsets $Y$ of $X$ such that there exists a coarse entourage $U$ of $X$ with $Y\subseteq {U[W] \cap U[X\setminus W]}$. 
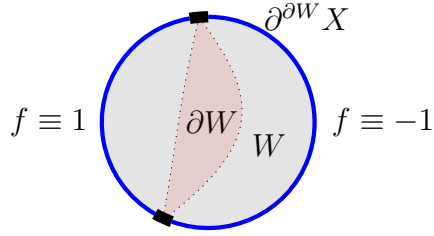
 \begin{equation*}
\begin{tikzpicture}
\filldraw[color=gray!70,opacity=0.3] (0,0) circle (1.4cm);
\draw[ultra thick, blue] (0,0)++(100:1.4) arc (100:240:1.4);
\draw[line width=1.5mm] (0,0)++(240:1.4) arc (240:250:1.4);
\draw[ultra thick, blue, name path=A] (0,0)++(250:1.4) arc (250:450:1.4);
\draw[line width=1.5mm] (0,0)++(450:1.4) arc (450:460:1.4);
\draw[dotted, name path=C] (0,0)++(95:1.4) .. controls (-0.4,0.2) and (-0.4,-0.2) .. (-0.6,-1.3);
\draw[dotted, name path=B] (0,0)++(95:1.4) .. controls (0.7,0.3) and (0.7,-0.3) .. (-0.6,-1.3);
\node at (0.05,0) {$\partial W$};
\node at (0.8,-0.3) {$W$};
\node at (1.3,1.4) {$\partial^{\partial W} X$};
\node at (2.3,0) {$f\equiv -1$};
\node at (-2.1,0) {$f\equiv 1$};
%\tikzfillbetween[of=A and B]{red, opacity=0.1};
\tikzfillbetween[of=B and C]{red, opacity=0.1};
%\tikzfillbetween[of=A and B]{gray, opacity=0.4};
\end{tikzpicture}	
\end{equation*}
Denote by $f$ the function such that $f|_W = 1$ and $f|_{X \setminus W} = -1$.
Then $f \in \ell^\infty_{\partial W}(X)$.
Its class $[f] $ in $C(\partial^{\partial W} X)$ defines a map 
$\partial^{\partial W} X \to \{-1,1\}$ which decomposes this corona into two disjoint components.
  \hB   
\end{ex}

\begin{rem}[A continuous description of the coarse corona]
\label{RemarkContinuousVersion} 
Assume that $X$ is a paracompact topological space and that the coarse structure on $X$ is compatible with the topology in the sense that there exists an open coarse entourage.  For example, the coarse and topological structures could be both induced by a metric.
We then define the $C^*$-algebra 
\[
C_{\cY}(X) := C(X) \cap \ell^\infty_{\cY}(X)
\]
 of continuous functions with bounded variation away from $\cY$, as well as the algebra 
\begin{equation}\label{wetopkgpwtgerfrwefwerfrwef9}
  C(\cY) := \colim_{Y \in \cY} \ker( C_b(X) \to C_b(X \setminus Y)) 
\end{equation}
of continuous functions that vanish away from $\cY$.

\begin{lem}\label{lpkrtherthergetr9}
We have  
$C(\cY)= \ell^\infty(\cY) \cap C(X)$ and the canonical  inclusion is an isomorphism
  $$\frac{C_{\cY}(X)}{C(\cY)} \stackrel{\cong}{\longrightarrow} \frac{\ell^\infty_{\cY}(X)}{\ell^{\infty}(\cY)}\ .$$
\end{lem}

\begin{proof}
It is clear from the definitions that $C(\cY)\subseteq \ell^\infty(\cY) \cap C_{\cY}(X)$. We 
show the converse inclusion. By assumption we can choose  an open coarse entourage  $U$     on $X$.
For every $Y$ in $ \cY$, we let $\rho_Y$ be a continuous function on $X$ taking values in $[0, 1]$, which is supported on $U[Y]$ and constant equal to one on $Y$. 
Here we use that by paracompactness of $X$, the open covering $(U[Y] ,X\setminus \bar Y)$ admits a subordinated partition of unity.

Then we have the equality
 $f = \lim_{Y \in \mathcal{Y}} \rho_Y  f$ and $\rho_Y   f\in \ker(C_{b}(X)\to C_{b}(X\setminus U[Y]))$. Since
 $\cY$ is a big family we have $U[Y]\in \cY$ for every $U$ in $\cC$ and we can conclude from \eqref{wetopkgpwtgerfrwefwerfrwef9} that
 $f\in C(\cY)$.

For the second assertion, we have to prove that for any function $f$ in $\ell^\infty_{\cY}(X)$ there exists $\tilde{f}$ in $C_{\cY}(X)$ such that $f - \tilde{f} \in \ell^\infty(\cY)$. 
We start with the open covering $(U[\{x\}])_{x\in X}$ of $X$. 
Since the topology of $X$ is paracompact, it admits  a locally finite open refinement $(O_j)_{j \in J}$ 
and a subordinate
 partition of unity $(\chi_j)_{j \in J}$.  Note that $O_{j}$ is $U$-bounded for every $j$ in $J$.

For each  $j \in J$ we pick a point $x_j$ in $O_j$.
Given a function $f$ in $\ell^\infty_{\cY}(X)$, we define the continuous bounded function
	\begin{equation}
	\label{DefinitionTildef}
	\tilde{f}(y) := \sum_{j \in J} f(x_j) \chi_j(y).
	\end{equation}
	%Since the cover is locally finite, this sum is actually finite for each $y$ and $\tilde{f}$ is a well-defined continuous function on $X$.
We claim that $f - \tilde{f} \in \ell^{\infty}(\cY)$.
	Let $\epsilon$ be in $(0, \infty)$.
	Then by the variation condition on $f$, there exists a member $Y$ of $\cY$ such that $\Var_{U}(f, X \setminus Y) \leq \epsilon$. % 
This implies {for} every $y$ in $X \setminus U[Y]$ that
	\begin{equation*}
	|f(y) - \tilde{f}(y)| \leq \sum_{{j \in J}} |f(y) - f(x_{j})|{\chi_j}(y) 
	\leq \epsilon \sum_{{j \in J}}{\chi_j}(y) = \epsilon\ ,
	\end{equation*}
where we used that the sum is actually taken only over those $j$ in $J$ such that ${y\in  U[x_{j}]\neq \emptyset}$ which implies that $x_{j},y\in X\setminus Y$.

Since $f\in \ell^\infty_{\cY}(X)$ and $f-\tilde f\in \ell^{\infty}(\cY)$ we see that $\tilde f\in \ell^\infty_{\cY}(X)$. By the first
assertion, $\tilde f\in C_{\cY}(X)$ as desired.
\end{proof}

Recall that the coarse  $\cY$-corona of $X$ is by \cref{rkopthtrhegertg} given by 
\[
 \partial^{\cY}X{:=}\spec\left(\frac{\ell^{\infty}_{\cY}(X)}{\ell^{\infty}(\cY)}\right) \ . 
 \]
 As a consequence of the second assertion of \cref{lpkrtherthergetr9}, if $X$ is a paracompact topological space such that the coarse structure contains an open entourage, then we get the alternative description
\begin{equation}
\label{HomeoCorona}
  \partial^{\cY} X \cong \spec\left(\frac{C_{\cY}(X)}{C(\cY)}\right)
\end{equation}
in terms of continuous functions on $X$.

In the case that $X$ is a proper metric space and that $\cY = \cB$ is the collection of bounded subsets of $X$, the
right-hand side of \eqref{HomeoCorona}  was considered in \cite[\S5.1]{roe_coarse_cohomology} and is known as the Higson corona $\partial_{h}X$.
\hB
\end{rem}

  \begin{ex}\label{okgpewgwefrewferfw}
  Let $\cY ,\cY',\cZ,\cZ'$ be big families on a coarse space $X$.
  Then we have a
  homomorphism
 \begin{equation}
 \label{WeDoNeedThis!}	
 \ell^{\infty}_{\cY}(\cZ)\otimes \ell^{\infty}_{\cY'}(\cZ')\to \ell^{\infty}_{\cY \cup \cY'}(\cZ\cap \cZ' )\ , \quad (f\otimes g)\mapsto  fg\ .
 \end{equation}
 For $\cZ=\cZ' = \{X\}$ this induces a well-defined homomorphism
  \[
  C(\partial^{\cY}X)\otimes C(\partial^{\cY'}X)\to C(\partial^{\cY {\cup} \cY'}X)\ ,
  \]
  which corresponds by Gelfand duality  to a map
  \[
  \partial^{\cY{\cup}  \cY'}X\to\partial^{\cY}X\times \partial^{\cY'}X \ .
\] 
%A similar reasoning applies to the topological coronas.  
\hB
  \end{ex}
  
  \begin{rem}
  \label{RemFunctorialityCorona}
  Let $\varphi: X \to X'$ be a morphism of coarse spaces.
  Then for $f$ in $\ell^\infty(X')$, a subset $Y'$ of $X'$ and a coarse entourage $U$ of $X$, we have 
  \[
  \Var_U(\varphi^{*}f, X\setminus \varphi^{-1}(Y')) \le \Var_{(\varphi \times \varphi)(U)}(f, X'\setminus Y')\ .
  \]
  Therefore, if $\cY$, $\cY'$ are big families on $X$, respectively $X'$ such $\varphi^{-1} (\cY^\prime) \subseteq \cY$, we get a $*$-homomorphism $\varphi^* : \ell^\infty_{\cY'}(X') \to \ell^\infty_{\cY}(X)$.
  Clearly, this homomorphism also sends $\ell^\infty(\cY')$ to $\ell^\infty(\cY)$.
  Therefore, by Gelfand, duality, $\varphi$ extends to a continuous map
  \[
     \varphi : \partial^{\cY} X \to \partial^{\cY'} X'.
  \]
  {In particular, if $X$ is a coarse space with two big families $\cY$, $\cY'$ such that $\cY' \subseteq \cY$, then we may apply the above observation to the identity map of $X$, which yields a continuous surjective map 
$\partial^{\cY} X \to \partial^{\cY'} X$.
}
  \hB
  \end{rem}

  Let $X$ be a coarse space and {$\cY$} be a big family on $X$. 
Let $(H,\chi)$, $(H',\chi')$ be $X$-controlled Hilbert spaces which are determined on points and 
  $A:H\to H'$ be a bounded operator.   
   The argument  for the following commutator estimate is taken from \cite{quro}.

  \begin{lem}
  \label{kohpkertphokgpertge} 
  If $f$  is in $\ell^{\infty}_{\cY}(X)$ and $A$ is $U$-controlled for some coarse entourage $U$, then 
 \[
 {\lim_{Y \in \cY} \|\chi'(X\setminus Y)(\chi'(f)A-A\chi(f))\chi(X\setminus Y)\|=0\ .}
 \]
  \end{lem}
  \begin{proof} 
Let $\epsilon$ in $(0,\infty)$ be given and set $\eta := \epsilon /4 \|A\|$. 
We then choose $Y$ in $\cY$ such that $\Var_{U}(f, X\setminus Y')\leq \eta$ for each $Y'$ in $\cY$ with $Y \subseteq Y'$. 
We define the partition $(S_{k})_{k\in \Z}$  of $X \setminus Y$ by\ \begin{equation*}
  S_k := \{ x \in X \setminus Y \mid (k-1) \eta \leq f(x) < k\eta\}\ .%, \qquad k \in \Z.
\end{equation*}
Since $f$ is bounded, only finitely many of these sets are non-empty.
If $k,l$ are in $\Z$, then  $x\in S_k$ and $y \in S_l$ implies $|f(x) - f(y)| \geq (|k-l|-1)\eta$.
Since  the $U$-variation of $f$ on   ${X \setminus} Y =\bigcup_{k\in \Z}S_{k}$ is bounded by $\eta$, the condition $|k-l|\geq 2$ implies that $S_k \cap U[S_l] = U[S_k] \cap S_l = \emptyset$.
Since $A$ is $U$-controlled we can conclude that $\chi'(S_k) A \chi(S_l) = 0$

We set
\begin{equation*}
 \tilde{f} := Y\cdot f + \eta \sum_{k \in \Z} k \cdot S_k,	
\end{equation*}
where for notational simplicity, we identify subsets of $X$ with the corresponding indicator function.
Then by construction, $\|\tilde{f} - f\ \| \leq \eta$ and hence
\begin{equation}
\label{FirstComparison}
  \|(\chi^\prime(f) A - A \chi(f)) - (\chi^\prime(\tilde{f}) A - A \chi(\tilde{f})) \| \leq 2\eta\|A\|	= \frac{\epsilon}{2}\ .
\end{equation}
Since $A$ is $U$-controlled, we have
\begin{equation}
\label{okwgpergrefwf}	
\begin{aligned}
&\chi'(X\setminus U[Y])(\chi'(\tilde{f})A-A\chi(\tilde{f}))\chi(X\setminus U[Y]) \\
&\qquad\qquad\qquad  = \eta\sum_{k \in\Z} k \cdot \chi'(X\setminus U[Y])(\chi^\prime(S_k) A - A \chi(S_k)) \chi(X \setminus U[Y]) \ .
\end{aligned} 
\end{equation}
Inserting the identities $\chi(X \setminus Y) = \sum_{k \in \Z} \chi(S_k)$ and $\chi'(X \setminus Y) = \sum_{k \in \Z} \chi'(S_k)$ and   using that $\chi'(S_k) A \chi(S_l) = 0$ whenever $|k-l| \geq 2$, we get
\begin{equation*}
  \sum_{k \in \Z} k\chi'(X \setminus Y)(\chi^\prime(S_k) A - A \chi(S_k))\chi(X \setminus Y) = \sum_{k \in \Z} (\chi^\prime(S_k) A \chi(S_{k-1}) - \chi^\prime(S_k) A \chi(S_{k+1}))\ .
\end{equation*}
The right-hand side is an operator  with norm bounded by $2\|A\|$.
Using $\chi(X\setminus U[Y])=\chi(X\setminus U[Y])\chi(X\setminus Y)$ and plugging the above equality into  \eqref{okwgpergrefwf}, we get
\begin{equation*}
\|\chi'(X\setminus U[Y])(\chi'(\tilde{f})A-A\chi(\tilde{f}))\chi(X\setminus U[Y])\| \leq 2\eta\|A\|= \frac{\epsilon}{2}.
\end{equation*}
Combining this with \eqref{FirstComparison}, we see that  
\begin{equation*}
\|\chi'(X\setminus Y')(\chi'(f)A-A\chi(f))\chi(X\setminus Y')\| \leq \varepsilon
\end{equation*} 
for all $Y'$ in $\cY$ with $U[Y] \subseteq Y'$.
   \end{proof}

Recall that   $X$ is a  coarse space with a big family ${\cY}$ and that $(H,\chi)$ and $(H',\chi')$ are  $X$-controlled Hilbert spaces which are determined on points. 
Let now $A:H\to H'$ be  a bounded operator which can be approximated in norm by controlled bounded operators, and  $f$ be in $\ell^{\infty}_{\cY}(X)$.

   \begin{kor}\label{vfjfvjpfvosdfvs}
    We have $\lim_{{Y \in \cY}}\| \chi'(X\setminus Y)(\chi'(f)A-A\chi(f))\chi(X\setminus Y)\|=0\ .$
    \end{kor}

Let $X$ be a bornological coarse space, $\cY$ be a big family in $X$ and recall the  \cref{jkohpertgtege}
of the ideal $\bC(\cY\subseteq X)$.

   \begin{kor}
   \label{gokpergergsgre}
For a morphism $A:(H,\chi)\to (H',\chi')$ in $\bC(X)$  and  a function $ f$ in $\ell^{\infty}_{\cY}(X)$ %a uniformly bounded function $f:X\to \C$ with $\lim_{i \in I}\Var_{U}(f,X\setminus Y_{i} ) =0 $
 we have
$\chi'(f)A-A\chi(f)\in \bC(\cY\subseteq X)$.  
\end{kor}

\subsection{Uniform coronas and the comparison map}\label{lhpkgerhretgrge}

In this section, we define the uniform version of the $\cY$-corona introduced in the previous section.
We consider a uniform space   $X$ with uniform structure $\cU$. 
{Recall that} a function $f:X\to \C$ is called uniformly continuous if %\fml{Keine Definitionsumgebung, da bekannter Begriff.}
\[
\lim_{U\in \cU^{\op}} \Var_{U}(f,X)=0\ .
\]  
  We let $C_{u}(X)$ denote the $C^{*}$-algebra of  uniformly continuous and  bounded functions on $X$ with the supremum norm. 
Note that $C_{u}(X)$ is a closed subalgebra of the $C^{*}$-algebra $C_{b}(X)$ of bounded continuous functions.

Let {$\cY$} be a filtered family of subsets of $X$. %\fml{Warum nochmal nicht eine big family?}

 \begin{ddd}
 \label{regkowpergerwfwfwrefw}
We define the $C^{*}$-algebra  
 $$C_{u}(\cY):=\colim_{{Y \in \cY}}  \ker(C_{u}(X) \to C_{u}(X\setminus Y))\ ,$$
 where the colimit is taken in $C^{*}$-algebras.
 \end{ddd}

  \begin{rem}  
  \label{RemarkOnCuY}
  $C_{u}(\cY)$ is a subalgebra of $C_u(X)$ in a canonical way and consists of uniformly continuous  functions which are asymptotically small away from   $\cY$. 
  We have an inclusion $C_{u}(\cY)\subseteq \ell^{\infty}(\cY)\cap C_{u}(X)$. 
  {This inclusion is an equality if $\cU$ is induced by a metric on $X$.}
{However,} for  general, non-metric, uniform  spaces  $X$ we do not expect   that this inclusion is an equality{:}
If $f$ is in $ \ell^{\infty}(\cY)\cap C_{u}(X)$, then it is not clear
how to approximate it by uniformly continuous functions supported on {a} suitable member  $Y$  of $\cY$.
\hB\end{rem}

 \begin{ex}
 \label{ExampleC0CuB}
 If $X$ is a proper metric space and $\cB$ is the bornology generated by bounded subsets {of $X$}, then  we have an equality $C_{0}(X)=C_{u}(\cB)=  \ell^{\infty}(\cB)\cap C_{u}(X)$. \hB  
 \end{ex}

Let $X$ be a set with compatible coarse and uniform structures $\cC$ and $\cU$.
Let $\cY$ be a big family on $X$.
Recall the \cref{gjsoepgergseffs}.\ref{wkotpgegfergweg} of $\ell^{\infty}_{\cY}(X)$.

\begin{ddd} 
\label{okpgwerfwerferwfw}
We define  the uniform corona algebra of $X$ by
   \[
   C_{u,\cY}(X):=\ell^{\infty}_{\cY}(X)\cap C_{u}(X)\ .
   \]
   \end{ddd}

We  have an exact sequence of $C^{*}$-algebras
\begin{equation}
\label{gerwrgfrq}
0\to C_{u}(\cY)\to C_{u,\cY}(X)\to C(\partial_{u}^{\cY} X)\to 0 \end{equation} 
defining the unital quotient $C^{*}$-algebra  $C(\partial_{u}^{\cY} X)$.

\begin{ddd}
\label{zkophergrtgtrgetrg} 
The  uniform  $\cY$-corona of $X$ is the compact topological space $\partial_{u}^{\cY} X$  defined as the Gelfand dual of  $C(\partial_{u}^{\cY}X)$. 
\end{ddd}

\begin{ex}
     If $X$ comes from a proper metric space (as in \cref{epokpbbwerg}) and $\cB$ is the big family of bounded subsets of $X$, then any continuous function with bounded variation away from $\cB$ is automatically uniformly continuous.
In view of \cref{ExampleC0CuB} and \cref{RemarkContinuousVersion} the uniform and the coarse corona coincide, 
\begin{equation*}
\partial_u^{\cB} X = \partial^{\cB} X, 	
\end{equation*}
and both agree with the Higson corona of $X$.
   %  \nml{We will use this notation $\partial_{h}X=\partial^{\cB}_{u} X$  also for general $X$ in $\UBC$.}\fml{Ich wuerde, $\partial X$ sollte die Higson-Corona sein. $\partial_u X$ sollte die uniforme Higson-Corona sein.}
     \hB
\end{ex}

\begin{rem} 
%\nml{We consider $X$ in $\UBC$.}\fml{Wir brauchen hier die gleichen Voraussetzungen wie oben.}
If $\cY, \cY'$ are big families such that $\cY\subseteq \cY'$%\nml{(in the sense that every member of the first is contained in some member of the second)}
, then
\[
C_{u,\cY}(X)\subseteq C_{u,\cY'}(X)\ .
\] 
The algebra $C_{u,\emptyset}(X)$ consists of functions which are constant on coarse components of $X$. 
Furthermore $C_{u,\{X\}}(X)=C_{u}(X)$ and $\smash{\partial^{\{X\}}_{u}}X$ contains $X$ as a subspace.
 \hB
\end{rem}
 
\begin{rem}
%\color{blue}
\label{RemarkUniformOnCone}
Consider the cone $\cO^\infty(X)$ of $X$ from \cref{jkigowergerwfwfwrefw}.
One of the main observations that will be exploited below is that pullback along the canonical projection $\pr : \cO^\infty(X) \to X$ yields a homomorphism
\[
  \pr^*  : C_u(X) \to \ell^\infty_{\cO^-(X)} (\cO^\infty(X)) \ ,
\]
which follows directly from the definition of the coarse structure of $\cO^-(X)$.
More generally, if $\cZ$ and $\cY$ are two big families on $X$, then the above homomorphism  restricts to a map
\[
  C_{u, \cZ}(\cY) := C_{u, \cZ}(X) \cap C_u(\cY) \to \ell^\infty_{\cO_{\cZ\cap \cY}^-(X)}(\cO^\infty(\cY)) \ ,
\]
using the big families in $\cO^\infty(X)$ defined in \eqref{fewdqwedqewdqewd}. 
\hB
\end{rem}

By \cref{RemarkOnCuY}, the canonical inclusion $C_{u, \cY}(X) \hookrightarrow \ell^\infty_{\cY}(X)$ sends $C_u(\cY)$ to $\ell^\infty(\cY)$.
Hence by Gelfand duality, it induces a comparison map
\begin{equation}
\label{ComparisonMap}
	c : \partial^{\cY} X \to \partial^{\cY}_u X \ .
\end{equation}

This map is a homeomorphism in many cases, as \cref{kpohrheh9} below shows.
We consider the following version of a bounded geometry condition on a   set $X$ with compatible uniform and coarse structures.

\begin{ass}\label{opthertherge}
We assume that $X$ admits a partition of unity $(\chi_{j})_{j\in J}$ such that:
\begin{enumerate}
\item \label{kohperhertgerg}The family $(\chi_{j})_{j\in J}$ is jointly uniformly continuous in the sense that for every $\epsilon$ in $(0,\infty)$ there exists a uniform entourage $V$ of $X$ such that for all $(x,y)$ in $ V$ and $j$ in $J$ we have $|\chi_{j}(x)-\chi_{j}(y)|<\epsilon$.
\item \label{okehpertgertge} The covering $(\supp(\chi_{j}))_{j\in J}$ is uniformly locally finite, i.e, 
$\sup_{x\in X}  |\{j\in J\mid  x\in \supp \chi_{j}\}|<\infty$
\item \label{okehpertgertg} The family $(\supp (\chi_{j}))_{j\in J}$ of subsets is coarsely bounded, i.e., there exists a coarse entourage $U$ such that $\supp (\chi_{j})\times \supp( \chi_{j})\subseteq U$ of $X$  for all $j$ in $J$.
%\item 
\hB
\end{enumerate}
\end{ass}

\begin{ex}  
If $(X,d)$ is a metric space of bounded geometry, then it satisfies \cref{opthertherge}.
Indeed,   we can find a suitable  family
$(x_{j})_{j\in J}$  of points in $X$ such that $(U_{2}[\{x_{j}\}])_{j\in J}$ is   uniformly locally finite  and
$(U_{1}[\{x_{j}\}])_{j\in J}$ is an open covering of $X$.
In order to define the required partition of unity we set
	\[
	 \tilde{\chi}_j(x) := \begin{cases} 1 - d(x, U_{1}[\{x_{j}\}])
	 & d(x, U_{1}[\{x_{j}\}])\le 1  \\
 	0 & \text{otherwise}
 \end{cases} \ ,
 \qquad
 \chi_j(x) := \frac{\tilde{\chi}_j(x)}{\sum_{k \in J} \tilde{\chi}_k(x)}
 \tag*{$\blacksquare$}
	\]
\end{ex}

\begin{prop}\label{kpohrheh9}
If $X$ is a set with compatible uniform and coarse structures
  satisfying \cref{opthertherge}, then \eqref{ComparisonMap} is a homeomorphism.
  
\end{prop}
\begin{proof}
Without loss of generality we can assume that %\nml{$\supp( \chi_{j})\not=\emptyset$} \ml
{$\chi_j \neq 0$} for all $j$ in $J$ and pick a point $x_{j}$ in $\supp (\chi_{j})$.

%\color{blue}
%	Set
	%\[
	% \tilde{\chi}_j(x) = \begin{cases} 1 - d(x, O_j)/r
	% & d(x, O_j) \leq r \\
 	%0 & \text{otherwise}
% \end{cases} \ ,
 %\qquad
 %\chi_j(x) = \frac{\tilde{\chi}_j(x)}{\sum_{k \in J} \tilde{\chi}_k(x)}
%	\]
	%Here $r$ is chosen as in the proposition. 
	%Each $\tilde{\chi}_j$ is a Lipschitz function with Lipschitz constant $1/r$, and since the sum of the $\tilde{\chi}_j$ is bigger or equal to one, the same is true for $\chi_j$.
	%\fml{Note to self: Check this properly!}
%For each $j$ in $J$, choose a point $x_j$ in $O_j$.
	
Let $f$ be in $\ell^\infty_{\cY}(X)$ and set 
\begin{equation*}
\tilde{f}(x) := \sum_{j \in J} f(x_j) \chi_j(x).
\end{equation*}
As  in \cref{RemarkContinuousVersion} and using \cref{opthertherge}.\ref{okehpertgertg} we see that $f$ is continuous, has vanishing variation away from $\cY$, and satisfies $f - \tilde{f} \in \ell^\infty(\cY)$. 
It remains to show that $\tilde{f}$ is uniformly continuous.
By  \cref{opthertherge}.\ref{okehpertgertge} the number $N:=\sup_{x\in X}  |\{j\in J\mid  x\in \supp \chi_{j}\}|$ is finite.
Let $\epsilon$ be in $(0,\infty)$ and choose by \cref{opthertherge}.\ref{kohperhertgerg} a uniform entourage $V $
such that $|\chi_{j}(x)-\chi_{j}(y)|\le \frac{\epsilon}{2N\|f\|_{\infty}}$ for all $(x,y)$ in $V$ and $j$ in $J$.
%In fact, we will show that it is even Lipschitz-continuous with Lipschitz constant $L = N \|f\|_\infty /r$. 
Then for all $(x,y)$ in $V$   we have
\[
\begin{aligned}
|\tilde{f}(x) - \tilde{f}(y)|
\leq \sum_{j \in J} |f(x_j)| |\chi_j(x) - \chi_j(y)|
\leq \|f\|_\infty \!\!\!\!\!\! \sum_{\substack{j \in J \\ x\in \supp( \chi_{j}) ~\text{or}~ y\in \supp (\chi_{j})  }} \!\!\!\!\!\! \frac{\epsilon}{2N\|f\|_{\infty}}    \leq \epsilon \ .
\end{aligned}
\]
 This shows that $\tilde f$
  uniformly continuous.
\end{proof}

Recall that $X$ is a set with compatible uniform and coarse structures and that $\cY$ is a big family in $X$.

\begin{ddd}
\label{fwerferfervwerv1}
We define the $K$-theory of $X$ with support in $\cY$ by 
\[
K(\cY):=K(C_{u}(\cY)) \ . 
\]
\end{ddd}

 \begin{rem}
 \label{InducedMapYsuppFnctns}
   Let $X$ and $X'$ be uniform spaces with filtered families $\cY$ and $\cY'$ of subsets of $X$, respectively $X'$. 
   If $f: X \to X'$ is a uniformly continuous map such that $f^{-1}(\cY') \subseteq \cY$, then the pullback homomorphism $f^* : C_u(X') \to C_u(X)$ sends $C_u(\cY')$ to $C_u(\cY)$, hence $f$ induces a map
   \[
     f^* : K(\cY') \to K(\cY).	
     \tag*{$\blacksquare$}
   \]
 \end{rem}

The exact sequence \eqref{gerwrgfrq} induces a boundary map
\begin{equation}
\label{fwerferfervwerv}
\partial:K(\partial_{u}^{\cY}X)\to \Sigma K(\cY)
\end{equation}
and the comparison map \eqref{ComparisonMap} induces a pullback map
\[
  	c^* : K(\partial_u^{\cY} X) \to K(\partial^{\cY} X)  \ .
\]
These maps involve the topological $K$-theory of the uniform and coarse $\cY$-corona of $X$, which is defined as the $K$-theory of the $C^{*}$-algebras of continuous functions on these spaces.

% The main usage of the $\cY_{\cZ}$-corona is to encode the topology of coarse decompositions of $X$.  

%  \begin{rem}
%  Let $X$ be in $\UBC$. 
%  Then the collection $\cB$ of bounded subsets is a big family covering $X$.
%   Let  $\cY$ be a further big family  covering $X$.
% Then $\cB\subseteq \cY$  and we have a maps of exact sequences 
%\[
%\xymatrix{
%0\ar[r]&C_{u}(\cB)\ar[d]\ar[r]&C_{u, \cB}(X)\ar[r]\ar@{=}[d]&C(\partial^{\cB}_uX)\ar[r]\ar@{->>}[d]&0\\
%0\ar[r]&C_{u}(\cY)\ar[r]&C_{u, \cB}(X)\ar[r]\ar@{^{(}->}[d]& \frac{C_{u, \cB}(X)}{C_{u}(\cY)}\ar[r]\ar@{^{(}->}[d]&0\\
%0\ar[r]&C_{u}(\cY)\ar@{=}[u]\ar[r]&C_{u,\cY}(X)\ar[r]&C(\partial_{u}^{\cY}X) \ar[r]&0
%}\ .
%\]
%By Gelfand duality we therefore get a projection followed by an inclusion of compact topological spaces
%\[
%\partial^{\cB}_uX  \hookleftarrow\spec\Big(\frac{C_{u, \cB}(X)}{C_{u}(\cY)}\Big) \twoheadleftarrow \partial_{u}^{\cY}X\ .
%\]
% %
%%A similar reasoning applies to $\partial_{\topp,\cY}X$.
% \hB
%  \end{rem}

%\fml{Hier habe ich einen falschen Remark geloescht. Ein entscheidender Punkt war eine Inklusion $C_u(\cY) \hookrightarrow C_{u, \cB}(X)$, die aber nicht existiert.}

\subsection{Analytic locally finite $K$-homology and  Paschke duality}
 \label{gkopegrergweferfwfr}

 Recall  \cref{grwgjrogwregwre}  of  the strong coarse $K$-homology theory   $K\cX:\BC\to \Mod(KU)$  (see also \cref{gpkopwergwrfwferfwf}) and the cone functor  $\cO^{\infty}:\UBC\to \BC$ introduced in \cref{kophehehtre}.
 As explaind in \cref{kogpewgwergergwergwerfw}, using the cone functor, we can associate  to any strong coarse homology theory  a  local homology theory.
 
   \begin{ddd}
   \label{kogpwergrefrefwrfw}  
   We define the  local $K$-homology functor for uniform bornological coarse spaces as the composition
  $$K^{\cX}=\Sigma^{-1}K\cX\cO^{\infty}:\UBC \xrightarrow{\cO^{\infty}} \BC\xrightarrow{\Sigma^{-1}K\cX} \Mod(KU)\ .$$
      \end{ddd}

 By specializing \eqref{bdgpbkdpgbdpgbdfg}
  we get  the index map, a natural transformation of functors
 \begin{equation}\label{fqweqwfqewdq}a:K^{\cX}\to  K\cX\circ \iota :\UBC  \to \Mod(KU)\ .
\end{equation}  
Note that if we  write the evaluation of $K\cX {\circ \iota}$ on objects of $\UBC$, then we will omit the symbol $\iota$.

 We now specialize the construction from  \eqref{vfdsovjsfdovsdfvsfdvsfv}.
\begin{ddd}\label{jgkgopekgregfwrefrwf}
We define the  local $K$-homology   with support by 
$$K^{\cX}_{-}:\UBC^{(2)}\to \Mod(KU)\ , \quad (X,\cZ)\mapsto K^{\cX}_{\cZ}(X):=\Sigma^{-1}K\cX(\cO^{\infty}_{\cZ}(X))\ .$$
\end{ddd}

The index map \eqref{bdgpbkdpgbdpgbdfg11} specializes to a natural transformation \begin{equation}\label{fwerqfdwedqewdqewdq}a_{X,\cZ}:K^{\cX}_{\cZ}(X)\to K\cX(\cZ)\ .
\end{equation}

\begin{rem}\label{kopgwergregfrfwf}
Note that $K^{\cX}_{\{X\}}(X)\simeq K^{\cX}(X)$  and  that we have a fibre sequence
\begin{equation}\label{vffdwerfgsefgdf}K^{\cX}_{{\{\emptyset\}}}(X)\to K^{\cX}(X)\stackrel{a_{X}}{\longrightarrow} K\cX
(X)\ ,
\end{equation} 
which is a special case of \eqref{FibreSequenceIndexMap}.
Thus the spectrum $\smash{K^{\cX}_{{\{\emptyset\}}}(X)}$ captures the secondary invariants of the reasons for the vanishing of the index of classes in $K^{\cX}(X)$. 

The fibre sequence \eqref{vffdwerfgsefgdf} is our analogue of the Higson-Roe surgery sequence
\cite[Def.\ 1.5]{zbMATH02196177} and $\smash{K^{\cX}_{{\{\emptyset\}}}}(X)$ is the analogue of the 
analytic structure set. It follows from \cref{keopgegwerferfw} that $\smash{K^{\cX}_{{\{\emptyset\}}}(-)}:\UBC\to \Mod(KU)$
is a local homology theory, so it is in particular excisive and homotopy invariant. The analogue of excisiveness for the analytic structure set of Higson-Roe has been
shown in  \cite{Siegel:2012aa}.

 Our order to compare our constructions with  \cite{zbMATH02196177}  precisely one would have to work with the equivariant generalization
 \begin{equation}\label{ojgpwergwrfwf} K^{G,\cX}_{{\{\emptyset\}}}(X)\to K^{G,\cX}(X)\stackrel{a^{G}_{X}}{\longrightarrow} K\cX^{G}
(X)
\end{equation} 
 for $G$-uniform bornological coarse spaces $X$. If $X$ is homotopy equivalent in $G\UBC$ to a  finite-dimensional
 locally finite simplicial complex with a free proper action of $G$, then it follows from the equivariant Paschke duality equivalence
 considered in \cite{bel-paschke} that  {the long exact sequence obtained from} \eqref{ojgpwergwrfwf} {by taking homotopy groups} is equivalent to the sequence defined in \cite[Def. 1.5]{zbMATH02196177}. 
 \hB
\end{rem}

 We consider a uniform bornological coarse space $X$
 with bornology $\cB$ and recall the \cref{regkowpergerwfwfwrefw} of $C_{u}(\cB)$. %{(here $\cB$ is viewed as a self-indexed big family)}.
 In this section it is useful to change notation and consider  the functor
 \[
 C_{0}:\UBC\to (\nCalg)^{\op}\ , \quad  X\mapsto C_0(X) := C_{u}(\cB)
 \]
 which on morphisms is given by the pull-back.
 
 \begin{rem}
 If $X$ is the uniform bornological coarse space associated to a proper metric space, then
 $C_{0}(X)$ is the usual algebra of continuous functions on $X$ vanishing at $\infty$. 
 This justifies the notation.
  \hB \end{rem}  
 % For proper   metric spaces  $X$ we can consider the analytic  version of 
  % locally finite  $K$-homology:
  \begin{ddd} \label{kogpwegreffrefrfw}
  We define the analytic locally finite $K$-homology as the composition
 $$K^{\an}:\UBC\xrightarrow{C_{0}}(\nCalg)^{\op}\xrightarrow{\EE(C_{0}(-),\C)}\Mod(KU)\ .$$ 
 \end{ddd}
 The following justifies the name, compare  \cref{DefinitionLocallyFiniteHomologyTheory}  and \cref{DefinitionLocalHomologyTheory}.
 
 \begin{prop}
 The functor $K^{\an}$ is a locally finite local homology theory.
 \end{prop}
 \begin{proof} For any $X$ in $\UBC$ the 
  map $C_{0}(X)\to C_{0}([0,1]\otimes X)$ is a homotopy equivalence of $C^{*}$-algebras. By homotopy invariance 
  of the $E$-theory functor $\ee:\nCalg\to \EE$ we see that $K^{\an}([0,1]\otimes X)\to K^{\an}(X)$ is an equivalence. This shows homotopy invariance of $\K^{\an}$.
  
  Assume that $(Y,Z)$ is a uniformly and coarsely excisive  decomposition of $X$ in $\UBC$. Then  we have a map of exact sequences of commutative $C^{*}$-algebras
  $$
\xymatrix{
0\ar[r]&A\ar[d]^{\cong}\ar[r]&C_{0}(X)\ar[r]\ar[d]&C_{0}(Y)\ar[r]\ar[d]&0\\
0\ar[r]&B\ar[r]&C_{0}(Z)\ar[r]&C_{0}(Z\cap Y)\ar[r]&0}\ ,
$$
where the maps in the right square are the restrictions along the inclusions of subsets and $A,B$ are defined as the kernels. We show that the left vertical map is an isomorphism as indicated. Assume that $f$ is in $A$ and sent to zero in $B$. Then $f_{|Z}=0$ and $f_{|Y}=0$ so that $f=0$ since $Y\cup Z=X$. Assume now that $f$ is in $B$. 
Then $f_{|Z\cap Y}=0$.  
 {As the $C^*$-algebra $C_{0}(Z)$ admits an approximate unit,} there exists a function $h$ in $C_{0}(Z)$ such that $\|fh-f\| \le \epsilon$. By definition of $C_{0}(Z)$ we can  in addition assume that there exists a bounded subset
$B$ of $Z$ such that $h_{|Y\setminus B}=0$.  
Let $\smash{\tilde f}$ be the extension by zero  of $fh$ to $X$.  
Then $\smash{\tilde f_{|X\setminus B}=0}$   and $\smash{\tilde f}$ is uniformly continuous since $(Y,Z)$ is uniformly excisive.  
Let $\hat f$ be the extension by zero of $f$ to $X$. Then $\|\smash{\hat f-\tilde f}\|\le \epsilon$. Since $\epsilon$ was arbitrary
we conclude that $\hat f\in A$  and that it is a preimage of $f$. 
   
We now use the exactness of the $E$-theory functor $\EE(-,\C)$  in order to get a map of fibre sequences
$$
\xymatrix{
 K^{\an}(Z\cap Y)\ar[r]\ar[d]&K^{\an}(Z)\ar[r]\ar[d]&\EE(B,\C)\ar[d]^{\simeq}\\
K^{\an}(Y)\ar[r]&K^{\an}(X)\ar[r]& \EE(A,\C)}\ .
$$
Since the right vertical map is an equivalence we conclude that the left square is a push-out.
This shows that $K^{\an}$ is excisive.

The algebra $C_{0}([0,\infty)\otimes X)$ is contractible by the homotopy
$C_{0}([0,\infty)\otimes X)\to C_{0}([0,1]\times [0,\infty)\otimes X)$ sending $f(t,x)$ to 
$(s,t,x)\mapsto f(st,x)$. This implies $K^{\an}([0,\infty)\otimes X)\simeq 0$ and therefore $K^{\an}$ vanishes on flasques.

Since $K^{\an}(X)$ does not depend on the coarse structure of $X$ at all the functor $K^{\an}$ is {trivially} $u$-continuous.

This completes the verification that $K^{\an}$ is a local homology theory. We now show that $K^{\an}$ is locally finite. 
For $B$ in $\cB$ we consider the exact sequence%\ml{Das ist glaube ich eine \"ubliche Notation. Wurde immer wieder von dem $A(B)$ verwirrt.}
$$0\to {C_B(X)} \to C_{0}(X)\to C_{0}(X\setminus B)\to 0\ ,$$
where ${C_B(X)}$ is defined as the kernel of the restriction map along $X\setminus B\to X$.
By the exactness of the $E$-theory functor $\EE(-,\C)$ we conclude that
\[\Cofib(K^{\an}(X\setminus B)\to K^{\an}(X))\simeq \EE({C_B(X)},\C)\ .\]
Since the $E$-theory functor preserves filtered colimits we get 
\begin{align*}
(K^{\an})^{\lf}(X)\simeq &\lim_{B\in \cB^{\op}} \Cofib(K^{\an}(X\setminus B)\to K^{\an}(X)) \\
&\simeq \lim_{B\in \cB^{\op}}  \EE({C_B(X)},\C)
\simeq
\EE(\colim_{B\in \cB} {C_B(X)},\C)\ .	
\end{align*}
Under this equivalence the canonical map $K^{\an}(X)\to  (K^{\an})^{\lf}(X)$ is  induced by the map of $C^{*}$-algebras \begin{equation}\label{vsdfvsfqvrsfvs}\colim_{B\in \cB} {C_B(X)}\to C_{0}(X)\ .
\end{equation}
In order to show that $K^{\an}$ is locally finite it therefore suffices to show that \eqref{vsdfvsfqvrsfvs}
   is an isomorphism of $C^{*}$-algebras. 
By \cref{regkowpergerwfwfwrefw}, the domain of this map is given by  
   \[
   \colim_{B\in \cB} \ker\Big(  \colim_{B'\in \cB} \ker\big(C_{u}(X)\to C_{u}(X\setminus B')\big)\to   \colim_{B'\in \cB} \ker\big(C_{u}(X\setminus B)\to C_{u}(X\setminus (B\cup B')\big)\Big) \ . 
   \]
   Since $\cB$ is filtered and $\ker$ is a finite limit we can interchange  $\colim_{B'\in \cB} $ with  $\ker$ and get  
   $$
   \colim_{(B', B) \in \cB\times \cB} \ker\Big( \ker\big(C_{u}(X)\to C_{u}(X\setminus B')\big) \to    \ker\big(C_{u}(X\setminus B)\to C_{u}(X\setminus (B\cup B')\big) \Big)\ . $$ 
   If $B'\subseteq B$, then the {map that the outer kernel is taken of is the zero map, so the argument of the colimit is just}  $\ker(C_{u}(X)\to
   C_{u}(X\setminus B')$. 
   Since the diagonal $\cB\to \cB\times \cB$ is cofinal, the domain of   \eqref{vsdfvsfqvrsfvs} is isomorphic to   $\colim_{B\in \cB} \ker(C_{u}(X)\to C_{u}(X\setminus B))$ which is precisely the definition of the codomain.
   \end{proof}

There is  a Paschke duality transformation 
\begin{equation}\label{r3rgrgsfegeg}p:K^{\cX}  \to K^{\an} 
\end{equation} 
whose components we will describe  following \cite{bel-paschke}. The  essential idea  for the construction  is taken from \cite{quro}.

  Recall that the bornology $\cB$ of $X$ is a big family. %We let $\cO^{\infty}(\cB)$ be the big family in
%$\cO^{\infty}(X)$ obtained by applying the cone functor to the members of $\cB$.
 
 In the following we use the maximal tensor product of $C^{*}$-categories in the special case that
 the first factor is a $C^{*}$-algebra. Then the set of objects of the tensor product is the set of objects of the right factor $ \bC(\cO^{\infty}(X)  )$, and the morphisms of the tensor product are generated by tensor products of algebra elements and morphisms. 
 Further note that the set of objects of the quotient on the right-hand side of \eqref{htrsoprhss}
 is also canonically identified with the set of objects of $\bC(\cO^{\infty}(X)  )$, see \cref{korphrethertgertg}.
 
 \begin{lem} 
 \label{htrsoprhss}
We have a morphism \begin{equation}\label{sfgrgwerggs} \mu:C_{0}(X)\otimes \bC(\cO^{\infty}(X)  )\to  \frac{\bC(\cO^{\infty}(\cB)\subseteq \cO^{\infty}(X))}{\bC(\cO^{-}(\cB)\subseteq \cO^{\infty}(X))}
\end{equation}
  given as follows:
 \begin{enumerate}
 \item objects: It  {is} the identity on objects. % the object $(H,\chi)$ to $(H,\chi)$.
 \item morphisms: It sends the morphism $f\otimes A:(H,\chi)\to (H',\chi')$ to  $[\chi'(\pr^{*}f)A]:(H,\chi)\to (H',\chi')$, %in $\frac{\bC(\cO^{\infty}(\cB)\subseteq \cO^{\infty}(X)))}{\bC(\cO^{-}(\cB)\subseteq \cO^{\infty}(X)))}$, 
 where $\pr:\R\times X\to X$ is the projection.
 \end{enumerate}
 \end{lem}
 
 \begin{proof}
 This is a special case of the more general \cref{iogjweoigwefewrfwerfwf} below, which reduces to \cref{htrsoprhss} upon setting $\cY = \cB$ and $\cZ = \{X\}$.
 \end{proof}

   The Paschke morphism \begin{equation}\label{frefwrfregwreg}p_{X}:K^{\cX}(X)\to K^{\an}(X)
\end{equation}  
is defined as the composition
\begin{align*} 
K^{\cX}(X)&\simeq \Sigma^{-1} K(\bC(\cO^{\infty}(X)))\simeq  \Sigma^{-1} \EE( \C,\bC(\cO^{\infty}(X)))\\ 
&\xrightarrow{C_{0}(X)\otimes-}\Sigma^{-1}  \EE( C_{0}(X),C_{0}(X)\otimes \bC(\cO^{\infty}(X))) \\
  &\xrightarrow{\,\,~~~\mu_{*}~~~\,} \Sigma^{-1} \EE\Big(C_{0}(X),\frac{\bC(\cO^{\infty}(\cB)\subseteq \cO^{\infty}(X))}{\bC(\cO^{-}(\cB)\subseteq \cO^{\infty}(X))}\Big) \\&
  \stackrel{!}{\simeq} \Sigma^{-1} \EE(C_{0}(X), \bC(\cO^{\infty}(\cB)\subseteq \cO^{\infty}(X)))\\& \stackrel{!!}{\to} \Sigma^{-1} \EE(C_{0}(X), \bC(\cO^{\infty}(*)))  \stackrel{!!!}{\simeq}   \EE(C_{0}(X), \C)\simeq K^{\an}(X)\ ,
  \end{align*}
  where for $!$ we use that the projection  
  \[
  \ee(\bC(\cO^{\infty}(\cB)\subseteq \cO^{\infty}(X)))\to \ee\Big(\frac{\bC(\cO^{\infty}(\cB)\subseteq \cO^{\infty}(X))}{\bC(\cO^{-}(\cB)\subseteq \cO^{\infty}(X))}\Big)
  \]
  induces an equivalence   since the functor $\ee$ is exact and we have $\ee(\bC(\cO^{-}(\cB)\subseteq \cO^{\infty}(X)))\simeq 0$.
  For the latter equivalence,
   we use (see \cref{wroekgpwergwrfgwrf} and \cref{kowprgwertgwerfw})  that 
  \[
  \ee(\bC(\cO^{-}(\cB)\subseteq \cO^{\infty}(X)))\simeq \colim_{n\in \nat,B\in \cB}\ee(\bC((-\infty,n]\times B) )
  \ ,
  \] 
  where $(-\infty,n]\times B$ has the bornological coarse  structure induced from $\cO^{\infty}(X)$. 
   These subspaces are flasque and therefore $\bC((-\infty,n]\times B)$ is flasque as a $C^{*}$-category and hence annihilated by $\ee$.
The morphism $!!$ is induced by the projections $B\to *$  for the members $B$ of $\cB$,
  and $!!!$ employs the equivalence 
  \[
  \ee(\bC(\cO^{\infty}(*)))
  \stackrel{\ee(\bC(\partial^{\cone}))}{\simeq}
  %\simeq
  \ee(\bC(\R))\stackrel{ \substack{\eqref{fqwedwedewdqwdwd},\\ \text{\cref{wegkowpergrwfrewfrefw}}}}{\simeq} 
   \Sigma \ee(\bC(*))\simeq \Sigma \ee( \C)\ .
   \]
We refer to \cite{bel-paschke} for more details and a better description of the Paschke morphism which exhibits $p_{X}$ as a component of a natural transformation.   In \cite{bel-paschke}  it is shown that the Paschke morphism $p_{X}$ is an equivalence
  under very general conditions on $X$.
  
  \begin{theorem}[\cite{bel-paschke}]\label{zhpoergtrghrrzhrzrjz}
 If $X$ is in $\UBC$ homotopy equivalent to   a countable, finite-dimensional,  locally
finite simplicial complex with the spherical path metric, then $p_{X}$ in \eqref{frefwrfregwreg} is an equivalence.
\end{theorem}

\begin{proof} 
In order to convince the reader that the statement is true  we sketch the proof and refer to  \cite{bel-paschke} for  more details.
We use (an import from \cite{bel-paschke}) the non-trivial fact  that $p_{X}$ is the component of a natural transformation $p:K^{\cX}\to K^{\an}$. 

By the homotopy invariance of $K^{\cX}$ and $K^{\an}$ we can then assume that $X$ itself is a locally finite finite-dimensional simplicial complex with the spherical path metric. 

We then argue by induction on the dimension. For the induction step we use that {gluing} the disjoint union of the $n$-simplices into the $n-1$-skeleton 
provides a coarsely and uniformly excisive decomposition, and that the disjoint union of  $n$-simplices 
is homotopy equivalent to a discrete space. 
In these two steps it is important that the simplices are equicontinuously parameterized, a fact guaranteed by the choice of the metric. 
We then use that both functors $K^{\cX}$ and $K^{\an}$ are excisive and homotopy invariant.

In order to start the induction we must show the assertion for discrete spaces. If $X$ is discrete, then
$K^{\cX}(X)=\prod_{x\in X} KU$ since $K\cX$ is additive, and also $K^{\an}(X)\cong \prod_{x\in X}KU$ since
$K^{\an}$ is locally finite. Under these identifications the map $p_{X}$ becomes the identity.
\end{proof}

\begin{rem}\label{trkohprthhetrh99}
We note that the Paschke morphism $p_{X}$ in \eqref{frefwrfregwreg} is similar but not the same
as the Paschke isomorphism as discussed, e.g., in \cite{higson_roe}. The maps have different domains.

The Paschke morphism from  \cite{higson_roe} is an equivalence for any locally compact metric space.
We do not know wether we can remove the additional assumption on $X$ in \cref{zhpoergtrghrrzhrzrjz}
for the Paschke morphism with domain $K^{\cX}(X)$. In fact, we do not expect that the functor $K^{\cX}$ is locally finite.
\hB\end{rem}

\subsection{The coarse corona pairing}\label{regijowergefrfwrfrf}

 In this and the next section we construct the coarse corona pairing $\cap^{\cX}$ and the coarse symbol pairing $\cap^{\cX\sigma}$   mentioned in the introduction. %\fml{Ich habe die Konstruktionen unten umgetauscht, nach dem Prinzip weniger Struktur --> Mehr Struktur.}
 The principal ideas going into their constructions
are  surely well known and variants in different setups have be considered previously, see, e.g., \cite{quro, zbMATH07364362}.

  We consider a   bornological coarse space $X$ with two big families $\cY$ and $\cZ$.  
  
In order to understand the following assertion note that the sets of objects of the $C^{*}$-categories in the domain and target of the functor in \eqref{iogjweoigwefewrfwerfwf11}
are canonically identified. 
   
 \begin{lem} 
 \label{ijgiowergrefreffw}    
 We have a morphism of $C^{*}$-categories
\begin{equation}
\label{iogjweoigwefewrfwerfwf11}
\nu : C(\partial^{\cY} X) \otimes \bC(\cZ\subseteq X)\to \frac{\bC(\cZ\subseteq X)}{\bC(\cY\cap \cZ\subseteq X)}\ .
\end{equation}
   given as follows:
 \begin{enumerate}
 \item objects:   It  acts as identity on objects.
 \item morphisms: It sends the morphism $[f]{\otimes} A:(H,\chi)\to (H',\chi')$ to $[\chi'(f)A]:(H,\chi)\to (H',\chi')$. % in $ \frac{\bC(X)}{\bC(\cY\subseteq X)}$.
 \end{enumerate}    
 \end{lem}
 
 \begin{proof}
  Let $f$ be in $\ell^{\infty}_{\cY}(X)$ and $A$ be in $\bC(\cZ\subseteq X)$.
 First note that if in addition  $f\in \ell^\infty(\cY)$, then $\chi^\prime(f)A\in \bC(\cY\cap \cZ \subseteq X)$.
 Hence the class $[\chi^\prime(f)A]$ only depends on the class  $[f]$ in  $\ell^{\infty}_{\cY}(X)/\ell^{\infty}(\cY) = C(\partial^{\cY} X)$,  {so} the map $\nu$ is well-defined on elementary tensors. 
 Furthermore, by  \cref{gokpergergsgre} we have $\chi'(f)A-A\chi(f)\in  \bC(\cY\subseteq X)$. 
 Since the individual terms belong to 
 $\bC(\cZ\subseteq X)$ we conclude that actually  $\chi'(f)A-A\chi(f)$ belongs to $\bC(\cY\cap \cZ\subseteq X)$. 
  
The arguments above show that $\nu$ is a well-defined functor (see the proof of \cref{iogjweoigwefewrfwerfwf} for an analogous argument with more details) defined on the algebraic tensor product  in the domain. 
Since its target is a $C^{*}$-category, $\nu$  uniquely  extends to a functor defined on the maximal tensor product, by the universal property of the latter.
 \end{proof}

 Upon applying  $K$-theory to   $\nu$ from \cref{ijgiowergrefreffw}  and using the symmetric monoidal structure of $K$  and \cref{wroekgpwergwrfgwrf} we get a pairing
  \begin{equation}
  \label{fwerfw2ffrwfreferferfw}
  K({\partial^{\cY} X})\times K\cX(\cZ)\to K\cX(\cZ,\cY\cap \cZ)\ .
\end{equation}  

\begin{ddd} 
\label{gojkwpergerwferfewrfwrefw}
  We define  the coarse corona pairing 
  \begin{equation}
  \label{fwerfewrfewrfwerf}
  -\cap^{\cX}-:K({\partial^{\cY} X}) \times K\cX(\cZ) \to \Sigma K\cX(\cY\cap \cZ)
\end{equation}   
as the composition 
 \begin{equation}\label{kortphertherther}
 K({\partial^{\cY} X})\times K\cX(\cZ) \stackrel{\nu}{\to} K\cX(\cZ,\cY\cap \cZ) \stackrel{\partial}{\to} \Sigma K\cX(\cY\cap \cZ)\ .
\end{equation}
of \eqref{fwerfw2ffrwfreferferfw} with the boundary map for the pair $(\cZ, \cY \cap \cZ)$.
\end{ddd}

\begin{rem}
\label{RemarkExtranatural}
  We now explain the naturality properties of the coarse corona pairing.
  To begin with, we consider the category $\BC^{(3)}$, whose objects are triples $(X, \cY, \cZ)$ of a space $X$ together with two big families $\cY$, $\cZ$ on $X$, and whose morphisms $f : (X, \cY, \cZ) \to (X', \cY', \cZ')$ are morphisms in $\BC$ such that $f(\cY) \subseteq \cY'$ and $f(\cZ) \subseteq \cZ'$.
  We let $\smash{\BC^{(3)}_{\mathrm{tw}}}$ be the following variation on the twisted arrow category of $\BC^{(3)}$:
  Its objects consist of an object $(X, \cY, \cZ)$ of $\BC^{(3)}$, a pair $(\tilde{X}, \tilde{\cY})$ of a 
 coarse space $\tilde{X}$ and a big family $\tilde{\cY}$ on $X$, as well as a controlled map $f: X \to \tilde{X}$ such that $f^{-1}(\tilde{\cY}) \subseteq \cY$.
  Morphisms of $\smash{\BC^{(3)}_{\mathrm{tw}}}$ are diagrams
  \begin{equation}
  \label{TypicalMorphismTw}
  \begin{tikzcd}
  	(\tilde{X}, \tilde{\cY})
  		&
  			(X, \cY, \cZ) 
  				\ar[l, "f"']
  				\ar[d, "h"]
  			\\
  	(\tilde{X}', \tilde{\cY}')
  		\ar[u, "g"]
  		&
  			(X', \cY', \cZ') 
  				\ar[l, "f^{\prime}"']
  \end{tikzcd}
  \end{equation}
  such that the underlying maps commute, where $h$ is a morphism in $\BC^{(3)}$, and where $g$ is a controlled map such that $g^{-1}(\tilde{\cY}) \subseteq \tilde{\cY}'$.
  Here the morphism direction is downward.
  
  There are two functors
  \[
    L, R : \BC^{(3)}_{\mathrm{tw}} \longrightarrow \Mod(KU)
  \]
  whose action on objects is given by
  \[
  \begin{aligned}
    L((\tilde{X}, \tilde{\cY}) \stackrel{f}{\leftarrow} (X, \cY, \cZ)) &= K(\partial^{\tilde{\cY}} \tilde{X}) \times K\cX(\cZ) \\
    R((\tilde{X}, \tilde{\cY}) \stackrel{f}{\leftarrow} (X, \cY, \cZ)) &= K\cX(\cY \cap \cZ) \ .
  \end{aligned}
  \]
Naturality of the coarse corona pairing is the assertion that there is a 
  natural transformation $L \to R$ whose component at  $f : (X, \cY, \cZ) \to (\tilde X, \tilde\cY)$ is the composition  
  \[
 \cap^{\cX} \circ (f^* \times \id): K(\partial^{\tilde \cY}\tilde{X}) \times K\cX(\cZ)\to \Sigma K\cX(\cY\cap \cZ)\ .
  \]
  This natural transformation arises from applying the $K$-theory functor to a similar (1-categorical) natural transformation at the level of $C^*$-categories, involving the morphism $\nu$ from \cref{ijgiowergrefreffw}.
  
In particular, the natural transformation %$L \to R$ 
 assigns to the morphism \eqref{TypicalMorphismTw} in $\smash{\BC^{(3)}_{\mathrm{tw}}}$ the commutative diagram
  \begin{equation}
  \label{ExtranaturalityDiagram}
  	\begin{tikzcd}
  		K(\partial^{\tilde{\cY}}\tilde{X}) \times K\cX(\cZ)
  			\ar[r, "f^* \times \id"]
  			\ar[d, "g^* \times h_*"']
  			&
  				K(\partial^{\cY}X) \times K\cX(\cZ)
  					\ar[r, "\cap^{\cX}"]
  					&
  						 \Sigma K\cX(\cY \cap \cZ)
  							\ar[d, "h_*"]
  							\\
  		K(\partial^{\tilde{\cY}'}\tilde{X}') \times K\cX(\cZ')
  			\ar[r, "f'^* \times \id"]
  			&
  				K(\partial^{\cY'}X') \times K\cX(\cZ')
  					\ar[r, "\cap^{\cX}"]
  					&
  						 \Sigma K\cX(\cY' \cap \cZ') \ .
  	\end{tikzcd}
  \end{equation}
  We summarize this by saying that $\cap^{\cX}$ is an extra-natural transformation.
  \hB
\end{rem}

\begin{rem}
\label{qewoifjqeiowdedwqdqdqer43r3r343}  
particular case of \cref{RemarkExtranatural} is the following:
If $f:(X,\cY, \cZ)\to (X',\cY', \cZ')$ is a morphism in ${\BC^{(3)}}$ satisfying the additional condition that $f^{-1}(\cY') \subseteq \cY$, then we have the commutative diagram
\[
\label{ExtranaturalityDiagramSpecial}
\begin{tikzcd}
K(\partial^{\cY}X)\times K\cX(\cZ)
	\ar[r, "\cap^{\cX}"]
	&
	\Sigma K\cX(\cY \cap \cZ)
		\ar[dd, "f_{*}"]
		\\  
K(\partial^{\cY'}X')\times K\cX(\cZ)
	\ar[u, "f^{*}\times \id"']
	\ar[d, "\id\times f_{*}"] 
	&
	\\ 
	K(\partial^{\cY'}X')\times K\cX(\cZ')
	\ar[r, "\cap^{\cX}"]
	& 
		\Sigma K\cX(\cY' \cap \cZ') 
\ . 
\end{tikzcd}
\]
This is the special case of \eqref{ExtranaturalityDiagram} for $g = f' = \id$ and $h = f$.
\hB
\end{rem}

Let $X$ be a bornological coarse space and {let} $\cW,\cW'$ be two big families on $X$ such that $X=W\cup W'$ for some members $W$ of $\cW$ and $W'$ of $\cW'$.
We then have the following compatibility of the coarse corona pairing with the Mayer-Vietoris boundary map corresponding to this decomposition:

\begin{lem}
 \label{LemCoarsePairingMV}
The following diagram commutes:
	\begin{equation}
	\label{MVcomatibility}
	\begin{tikzcd}
K(\partial^{\cY} X) \times \Sigma^{-1}K\cX(\cZ)
	\ar[r, "\cap^\cX"]
	\ar[d, "\id \times \partial^{\mathrm{MV}}"']
	&
		K\cX(\cY \cap \cZ)
			\ar[d, "\partial^{\mathrm{MV}}"]
			\\
K(\partial^{\cY}X) \times K\cX(\cZ \cap \cW \cap \cW^\prime)
	\ar[r, "\cap^\cX"]
	&
		\Sigma K\cX(\cY\cap\cZ \cap \cW \cap \cW^\prime)
	\end{tikzcd}
\end{equation}
\end{lem}

\begin{proof}
	Consider the commutative diagram
	\[
	\begin{tikzcd}
K(\partial^\cY X) \times K\cX(\cW \cap \cW^\prime\cap \cZ)
	\ar[dr, "\cap^{\cX}"]
	\ar[rrr]
	\ar[ddd, shift right=10]
	&[-3.5cm]
		&
			&[-3.5cm]
	  			K(\partial^\cY X) \times K\cX(\cW^\prime\cap \cZ)
	  				\ar[dl, "\cap^{\cX}"']
					\ar[ddd, shift left=10]
\\
	&
	\Sigma K\cX(\cW \cap \cW^\prime\cap \cZ \cap \cY) 
	  	\ar[r]
	  	\ar[d]
	  	&
	  		\Sigma K\cX(\cW \cap \cZ \cap \cY) 
	  		\ar[d]
	  		&
	  		\\
	&
	\Sigma K\cX(\cW^\prime \cap \cZ \cap \cY) \ar[r]
	  	&
	  		\Sigma K\cX(\cZ \cap \cY)
	  		\\
K(\partial^\cY X) \times K\cX(\cW^\prime \cap \cZ)
	\ar[ur, "\cap^{\cX}"']
	\ar[rrr]
	&
		&
			&
				K(\partial^\cY X) \times K\cX(\cZ) \ ,
					\ar[ul, "\cap^{\cX}"]
	\end{tikzcd}
	\] 
where both the interior and the exterior square is a pushout square, each giving rise to a Mayer-Vietoris boundary map. 
The four trapezoids commute by \cref{qewoifjqeiowdedwqdqdqer43r3r343}, {which} in each case {is} applied to a map of triples that is the identity on $X$.
We therefore obtain a map of pushout diagrams, which yields the desired commutative diagram \eqref{MVcomatibility}.
\end{proof}

\subsection{The coarse symbol pairing}
\label{regijowergefrfwrfrf2}

We consider a  uniform bornological coarse space $X$ with   big families $\cY$  and $\cZ$  and 
 let $  \cO^{-}(\cY) $ and  $\cO^{\infty}_{\cZ}(X)$  be the big families in $\cO^{\infty}(X)$ defined in    \eqref{fewdqwedqewdqewd}.

\begin{lem}
\label{iogjweoigwefewrfwerfwf} 
We have a morphism of $C^{*}$-categories
\begin{equation}
\label{ewfewfdqedqedq}
\mu : {C_{u}(\cY)}\otimes \bC(  \cO_{\cZ}^{\infty}(X)\subseteq\cO^{\infty}(X))\to \frac{\bC(\cO_{\cZ}^{\infty}(\cY)\subseteq \cO^{\infty}(X))}{\bC(\cO^{-}(\cZ\cap \cY)\subseteq \cO^{\infty}(X))}
\end{equation} 
 given as follows:
 \begin{enumerate}
 \item objects: It  acts as identity on objects.
 \item morphisms: It sends the morphism $f\otimes A:(H,\chi)\to (H',\chi')$ to $[\chi'(\pr^{*}f)A]:(H,\chi)\to (H',\chi')$, % in $\frac{\bC(\cO^{\infty}(\cY)\subseteq \cO^{\infty}(X))}{\bC(\cO^{-}(\cY)\subseteq \cO^{\infty}(X))}$
  where $\pr:\R\times X\to X$ is the projection.
 \end{enumerate}
%The morphism $\mu$ from \eqref{sfgrgwerggs} extends to a  morphism
% $$\mu_{u,\cY}:C_{u}(\cY)\otimes \bC(\cO^{\infty}(X))\to \frac{\bC(\cO^{\infty}(\cY))}{\bC(\cO^{-}(\cY))}$$
% given by the same formulas 
\end{lem}

\begin{proof}
Let $f$ be in ${C_{u}(\cY)}$ and  $A : (H,\chi)\to (H',\chi')$ be a morphism in $\bC(\cO_{\cZ}^{\infty}(X)\subseteq\cO^{\infty}(X))$.
We first show that $\chi^\prime(\pr^*f) A$ is a morphism of $\bC(\cO_{\cZ}^{\infty}(\cY)\subseteq \cO^{\infty}(X))$.
By assumption, $f$ can be approximated by functions which are supported on  members of $\cY$. 
This implies that $\pr^{*}f \in \ell^{\infty}(\cO^\infty(\cY))$.
 Similarly, $A$ can be approximated by morphisms supported on members of $\cO_{\cZ}^{\infty}(X)$.
 Consequently, the product $\chi'(\pr^{*}f)A$  can be approximated by morphisms supported on members of $\cO^\infty(\cY) \cap \cO^\infty_{\cZ}(X) = \cO^{\infty}_\cZ(\cY)$. 
 We conclude that
  $\chi^\prime(\pr^*f)A\in \bC(\cO_{{\cZ}}^{\infty}({\cY})\subseteq \cO^{\infty}(X))$ as desired.
 
We next show that $\mu$ is compatible with the composition.
 Given another morphism $B : (H^\prime, \chi^\prime) \to (H^{\prime\prime}, \chi^{\prime\prime})$ in $\bC(\cO^{\infty}_{{\cZ}}({\cY})\subseteq \cO^{\infty}(X))$ and a function $g$ in ${C_{u}(\cY)}$, we have
 \[
  \mu(g \otimes B) \mu(f \otimes A) - [\mu(gf \otimes BA)] = [\chi^{\prime\prime}(\pr^*g)(B \chi^\prime(\pr^{*}f) - \chi^{\prime\prime}(\pr^{*}f) B)A]
 \]
 We therefore have to show that $B \chi^\prime(\pr^{*}f) - \chi^{\prime\prime}(\pr^{*}f) B$ belongs to the ideal $\bC(\cO^{-}(\cZ \cap \cY)\subseteq \cO^{\infty}(X))$.
 
 To this end, we claim that the uniform continuity of $f$ implies that $\pr^*f\in \smash{\ell^\infty_{\cO^-(X)}(\cO^\infty(X))}$.
  Indeed, for each $\varepsilon$ in $(0,\infty)$, there exists a uniform entourage $U$ of $X$ such that $|f(x) - f(y)| \leq \varepsilon$ whenever $(x, y) \in U$. 
Let $V$ be a coarse entourage of $\cO^\infty(X)$. 
Then by \cref{jkigowergerwfwfwrefw}.\ref{kogpwerokgweprfwerfrefwre}, there exists $r$ in $\R$ such that $((x, t), (x^\prime, t^\prime))\in V$  and  $\min(t, t^\prime) \geq r$ imply $(x, x^\prime) \in U$.
Therefore  $\Var_{V}(\pr^{*}f, \cO^\infty(X) \setminus W) \leq \varepsilon$ for any member $W:=   (-\infty, s]\times X$ of $\cO^-(X)$ with $s \geq r$. 
This  shows that $\pr^{*}f\in \ell^{\infty}_{\cO^{-}(X)}(X)$, {as claimed}. 

In view of the claim \cref{gokpergergsgre} implies  $\chi''(\pr^{*}f)B-B\chi'(\pr^{*}f)\in \bC(\cO^{-}({X})\subseteq \cO^{\infty}(X))$. Furthermore, by the assumptions on $f$ and $B$ the individual terms  belong to $\bC(\cO^\infty(\cY) \cap \cO^\infty_{\cZ}(X)\subseteq \cO^{\infty}(X))$.
In total, $B \chi^\prime(\pr^{*}f) - \chi^{\prime\prime}(\pr^{*}f) B$ belongs to the ideal associated to the big family
\[
{\cO^-(X)} \cap \cO^\infty(\cY) \cap \cO^\infty_{\cZ}(X) = \cO^-(\cY \cap \cZ)\
\]
as desired.  
 
We thus get a functor $\mu$ defined on the algebraic tensor product, which then uniquely extends to the maximal tensor product.
 \end{proof}

We apply the $K$-theory functor to the $*$-homomorphism $\mu$ from \cref{iogjweoigwefewrfwerfwf}.
With a view on Definitions \ref{fwerferfervwerv1} and \ref{kogpwergrefrefwrfw} and \cref{wroekgpwergwrfgwrf}, the fact that $K$ is symmetric monoidal yields a pairing%\fml{Sieht so aus als koennten wir die support condition beim ersten Term weglassen.}
%Upon applying  $K$-theory to  $\mu$ from \cref{iogjweoigwefewrfwerfwf} and using the symmetric monoidal structure of $K$  and \cref{wroekgpwergwrfgwrf} we get a pairing
\begin{equation}
\label{wfawfadsfdsf}
\mu: K(\cY)\times  K^{\cX}_{\cZ}(X)\to \Sigma^{-1}K\cX(\cO^{\infty}_{\cZ}(\cY),  \cO^{-}(\cY\cap \cZ)) \ .
 \end{equation}
Note that the members of $\cO^{-}(\cY\cap \cZ)$ are flasque and consequently, that
\[
  K^{\cX}_{\cZ}(\cY)\simeq \Sigma^{-1}K\cX(\cO^\infty_{\cZ}(\cY))\to   \Sigma^{-1} K\cX(\cO^{\infty}_{\cZ}(\cY),  \cO^{-}(\cY\cap \cZ))\]
 is an equivalence.
 
\begin{ddd} \label{grjweopgergweffefwef}
We define the coarse symbol pairing  
\[
-\cap^{\cX\sigma}-: K(\cY)\times K^{\cX}_{\cZ}(X)\to   K_{\cZ}^{\cX}(\cY)
\]
as the composition
 \begin{equation}  
 \label{gier90gpergfwersfvfv}
 K(\cY)\times K_{\cZ}^{\cX}(X)  \stackrel{\mu}{\to}  {\Sigma^{-1}}K\cX(\cO^{\infty}_{\cZ}(\cY),  \cO^{-}(\cY\cap \cZ)) \stackrel{\simeq}{\leftarrow}  K_{\cZ}^{\cX}(\cY)\ . 
\end{equation} 
\end{ddd}
% 
% The map $(1)$ is given by the symmetric monoidal structure of $K$.
% The  equivalence marked by $(2)$  is induced by the canonical projection and uses that $K$ is exact and that 
% $\bC(\cO^{-}(\cY\cap \cZ)\subseteq \cO^{\infty}(X))$ is flasque and therefore has vanishing $K$-theory. Finally, for the equivalence marked by $(3)$ we use \cref{wroekgpwergwrfgwrf}.
%$$q_{u,\cY}:K^{\cX}(X)\to  \Sigma^{-1} \EE(C_{u}(\cY), \bC(\cO^{\infty}(\cY)))$$ as the composition  
%\begin{align}  K^{\cX}(X)&\simeq \Sigma^{-1} \EE( \C,\bC(\cO^{\infty}(X)))\xrightarrow{C_{u}(\cY)\otimes-}\Sigma^{-1}  \EE( C_{u}(\cY),C_{u}( \cY)\otimes \bC(\cO^{\infty}(X))) \label{gier90gpergfwersfvfv}\\ &
%  \xrightarrow{\mu_{u, \cY,*}} \Sigma^{-1} \EE(C_{u}(\cY),\frac{\bC(\cO^{\infty}(\cY))}{\bC(\cO^{-}(\cY))})
%  \stackrel{!}{\simeq} \Sigma^{-1} \EE(C_{u}(\cY), \bC(\cO^{\infty}(\cY)))\nonumber   \end{align}

  The name of this pairing will be justified in \cref{gkopwerfefrefwfref} where we show in \cref{wegokwpergerfwefrwefwf} that
  it sends  symbols of Dirac type operators to  symbols  (with support) of suitable Callias type operators.

%  \begin{rem}\label{gkopwergrefeferfwef}
%  Note that the obvious idea to extend the coarse index pairing to a pairing
%  $$-\cap^{\cX\sigma}-:K_{\cY}(X)\times K_{\cZ}^{\cX}(X)\to K_{\cZ}^{\cX}(\cY)$$ 
%  (see \cref{rkopgwregwerfwfwrefw})
%  does not work since $\cO^{-}(\cY)\cap \cO_{\cZ}^{\infty}(X)$ is not flasque  in general so that we do not have the analogue of the equivalence $(2)$ in \eqref{gier90gpergfwersfvfv}. 
%\hB
%  \end{rem}

\begin{rem}
\label{qewoifjqeiowdedwqdqdqe}
The categories $\UBC^{(3)}$ and $\UBC^{(3)}_{\mathrm{tw}}$ may be defined analogously to the categories $\BC^{(3)}$ and $\BC^{(3)}_{\mathrm{tw}}$ of \cref{RemarkExtranatural}, using spaces with compatible coarse and uniform structures instead of coarse spaces.
One may then say that the coarse symbol pairings $\cap^{\cX\sigma}$ are the components of an extra-natural transformation.
 
A particular case is the following analog of \cref{qewoifjqeiowdedwqdqdqer43r3r343}: For a morphism  $f:(X,\cY,\cZ)\to (X',\cY',\cZ')$ in $\UBC^{(3)}$ satisfying the additional condition that $f^{-1}(\cY') \subseteq \cY$, we have a commutative diagram 
\[
\begin{tikzcd}
K(\cY)\times K^\cX_\cZ(X)
	\ar[r, "\cap^{\cX{\sigma}}"]
	&
	\Sigma K^\cX_{\cZ}(\cY)
		\ar[dd, "f_{*}"]
		\\  
K(\partial^{\cY'}X')\times K^\cX_\cZ(X)
	\ar[u, "f^{*}\times \id"']
	\ar[d, "\id\times f_{*}"] 
	&
	\\ 
	K(\cY')\times K^\cX_{\cZ'}(X')
	\ar[r, "\cap^{\cX{\sigma}}"]
	& 
		\Sigma K^\cX_{\cZ'}(\cY') 
\ . 
\end{tikzcd}
\]
\hB
\end{rem}

       \begin{rem}
       \label{GeneralizationSymbolPairing}
       If $\cY'$ is a further  big family, then \eqref{ewfewfdqedqedq} can be generalized to a paring
      \begin{equation*}
      \label{fwefwefeqhhh}
      C_{u}(\cY)\otimes \bC(\cO^{\infty}_{\cZ}(\cY')\ \subseteq \cO^{\infty}(X))\to \frac{\bC(\cO_{\cZ}^{\infty}(\cY\cap \cY')\subseteq \cO^{\infty}(X))}{\bC(\cO^{-}(\cY\cap \cY'\cap \cZ)\subseteq \cO^{\infty}(X))}
\end{equation*} 
which yields
 \begin{equation}
 \label{fwefwefeq}
 -\cap^{\cX\sigma}-: K(\cY)\times K^{\cX}_{\cZ}(\cY')\to   K^{\cX}_{\cZ}(\cY\cap \cY')\ .\end{equation}  \hB
       \end{rem}

Let $X$ be in $\UBC$ and let $\cY$, $\cZ$ be  big families on $X$.
The following proposition describes the compatibility of the pairings introduced above with the index map \eqref{fqweqwfqewdq}. 
Note that we drop the forgetful functor $\iota:\UBC\to \BC$.

\begin{prop}
\label{erokgpwergrwefwerfwerfwer}
We have a commutative diagram
\begin{equation*}
\begin{tikzcd}
\Sigma K(\cY)\times K_{\cZ}^{\cX}(X)
	\ar[r, "\cap^{\cX\sigma}"]
	& 
		\Sigma K_{\cZ}^{\cX}(\cY)
		\ar[dd, "a_{\cY,\cZ}"]
		\\ 
K(\partial^{\cY}_u X)\times K^{\cX}_{\cZ}(X)
	\ar[u, "\partial\times \id"] 
	\ar[d, "c^* \times a_{X,\cZ}"]
	&
	\\ 
K(\partial^{\cY}X)\times K\cX(\cZ)
	\ar[r, "\cap^{\cX}"]
	&
	\Sigma K\cX(\cY\cap \cZ)
	\ .	
\end{tikzcd}
\end{equation*}
 \end{prop} 
 
 \begin{proof}
 	Consider the diagram of in $C^*$-categories with exact rows
 	\[
\begin{tikzcd}
0
	\ar[d]
	& 
	0
		\ar[d]
			\\
C_{u}(\cY)\otimes \bC(\cO^{\infty}_{\cZ}(X) \subseteq \cO^{\infty}(X))
	\ar[d]
	\ar[r, "\mu"]
	&
		\displaystyle\frac{\bC(\cO^{\infty}_{\cZ}(\cY)\subseteq \cO^{\infty}(X))}{\bC(\cO^{-}(\cY\cap \cZ)\subseteq \cO^{\infty}(X))}
		\ar[d]
			\\
 C_{u,\cY}(X)\otimes \bC(\cO^{\infty}_{\cZ}(X) \subseteq \cO^{\infty}(X)) 
 	\ar[d]
 	\ar[r]
 	&
 		\displaystyle\frac{\bC(\cO^{\infty}_{\cZ}(X) \subseteq \cO^{\infty}(X))}{\bC(\cO^{-}(\cY\cap \cZ)\subseteq \cO^{\infty}(X))}
 			\ar[d]
 				\\
 C(\partial_{u}^{\cY}X)\otimes \bC(\cO^{\infty}_{\cZ}(X) \subseteq \cO^{\infty}(X))
 	\ar[d]
	 \ar[r]
	 &
	 \displaystyle\frac{\bC(\cO^{\infty}_{\cZ}(X) \subseteq\cO^{\infty}(X))}{\bC(\cO^{\infty}_{\cZ}(\cY)\subseteq \cO^{\infty}(X))}
	 	\ar[d]
	 	&
	 		\\
	 		0
	 			&
	 			~~ 0 \ ,
\end{tikzcd}
\]
where the  horizontal maps are all given by the description $f \otimes A \mapsto [\chi(\pr^*f)A]$, thus making the diagram commutative.
We now apply the $K$-theory functor to obtain a map of fiber sequences of $KU$-modules. 
In particular, inserting \cref{grwgjrogwregwre}, we get the commutative square
\[
\begin{tikzcd}	
  \Sigma K(\cY) \times K\cX(\cO^\infty_{\cZ}(X)) 
  	\ar[r, "\mu"]
  	&
  		\Sigma K\cX(\cO^\infty_{\cZ}(\cY), \cO^-(\cY \cap \cZ)) 
  		\\
  K(\partial_u^{\cY} X) \times K\cX(\cO^\infty_{\cZ}(X))
  	\ar[r] 
  	\ar[u, "\partial \times \id"]
  	&
  		K\cX(\cO^\infty_{\cZ}(X), \cO^\infty_{\cZ}(\cY)) \ ,
  			\ar[u, "\partial"']
  \end{tikzcd}	
\]
involving the boundary maps for these fiber sequences.
Desuspending this diagram and using \cref{kogpwergrefrefwrfw}, we may expand the resulting diagram to involve the coarse symbol map
\begin{equation}
\label{FirstDiagram}
\begin{tikzcd}
  \Sigma K(\cY) \times K_{\cZ}^{\cX}(X)
  	\ar[r, "\mu"]
  	\ar[rr, bend left=15, "\cap^{\cX\sigma}"]
  	&
  		K\cX(\cO^\infty_{\cZ}(\cY), \cO^-(\cY \cap \cZ))
  		&
  			\Sigma K_{\cZ}^\cX(\cY)
  			\ar[l, "\simeq"']
  		\\
  K(\partial_u^{\cY} X) \times K_{\cZ}^\cX(X)
  	\ar[r]
  	\ar[u, "\partial \times \id"]
  	&
  		\Sigma^{-1} K\cX(\cO^\infty_{\cZ}(X), \cO^\infty_{\cZ}(\cY))
  		\ar[u, "\partial"']
  		\ar[r, "\partial"]
  		&
  			K\cX(\cO^\infty_{\cZ}(\cY))  \ .
  			\ar[u, equal] 
\end{tikzcd}
\end{equation}
Here the rightmost square commutes by compatibility of the boundary maps for the map of fiber sequences
\[
\begin{tikzcd}[column sep=0.5cm]
	K\cX(\cO^\infty_{\cZ}(\cY))
		\ar[r]
		\ar[d]
		&
			K\cX(\cO^\infty_{\cZ}(X))
			\ar[r]
			\ar[d]
			&
				K\cX(\cO^\infty_{\cZ}(X), \cO^\infty_{\cZ}(\cY))
				\ar[d, equal]
				\\
	K\cX(\cO^\infty_{\cZ}(\cY), \cO^-(\cY \cap \cZ))
		\ar[r]
		&
			K\cX(\cO^\infty_{\cZ}(X), \cO^-(\cY \cap \cZ))
			\ar[r]
			&
				K\cX(\cO^\infty_{\cZ}(X), \cO^\infty_{\cZ}(\cY)) \ .
\end{tikzcd}	
\]
The projection map $\pr: \cO^\infty(X) \to X$ is controlled and satisfies $\pr^{-1}(\cY) = \cO^\infty(\cY)$, hence extends to a map $\pr : \partial^{\cO^\infty(\cY)}\cO^\infty(X) \to \partial^{\cY} X$ between the corresponding coarse coronas (see \cref{RemFunctorialityCorona}).
This allows to write the second row of \eqref{FirstDiagram} as a coarse corona pairing over the space $\cO^\infty(X)$, for the big families $\cO^\infty_{\cZ}(X)$ and $\cO^\infty(\cY)$:
\begin{equation}
\label{SecondDiagram}
\begin{tikzcd}
K(\partial_u^{\cY} X) \times K_{\cZ}^\cX(X)
  	\ar[r]
  	\ar[d, "c^* \times \id"']
  	&[-1.8cm]
  		\Sigma^{-1} K\cX(\cO^\infty_{\cZ}(X), \cO^\infty_{\cZ}(\cY))
  		\ar[r, "\partial"]
  		&
  			K\cX(\cO_{\cZ}^\infty(\cY))
  			\\
K(\partial^{\cY} X) \times K_{\cZ}^\cX(X)
  	\ar[d, "\pr^* \times \id"']
  			\\
  K(\partial^{\cO^\infty(\cY)} \cO^\infty(X)) \times \Sigma^{-1}K\cX(\cO^\infty_{\cZ}(X))
  	\ar[ruu, bend right=10, "\nu"']
  	\ar[rruu, bend right=20, "\cap^{\cX}"']
\end{tikzcd}
\end{equation}
Here the top left triangle commutes by definition of the top left horizontal arrow, while the other triangle is the definition of the coarse corona pairing.
The bottom arrow of the diagram \eqref{SecondDiagram} may be extended to the following  diagram:
\begin{equation}
\label{ThirdDiagram}
	\begin{tikzcd}
K(\partial^{\cO^\infty(\cY)} \cO^\infty(X)) \times \Sigma^{-1}K\cX(\cO^\infty_{\cZ}(X))
	\ar[r, "\cap^\cX"]
	\ar[d, "\id \times \partial^{\mathrm{MV}}"']
	&
		K\cX(\cO_{\cZ}^\infty(\cY)) = \Sigma K_{\cZ}^{\cX}(\cY)
			\ar[d, "\partial^{\mathrm{MV}}"]
			\ar[ddd, shift left=10, bend left=60, "a_{\cY, \cZ}"]
			\\
K(\partial^{\cO^\infty(\cY)} \cO^\infty(X)) \times K\cX(\{0\} \otimes \cZ)
	\ar[r, "\cap^\cX"]
	&
		\Sigma K\cX(\{0\} \otimes (\cY\cap\cZ))
		\\
K(\partial^{\cO^\infty(\cY)} \cO^\infty(X)) \times K\cX(\cZ)
	\ar[d, "\texttt{i}^* \times \id"']
	\ar[u, "\id \times {\texttt{\texttt{i}}}_*", "\simeq"']
		\\
K(\partial^{\cY} X) \times K\cX(\cZ)
	\ar[r, "\cap^\cX"]
	&
		\Sigma K\cX(\cY\cap\cZ)
		\ar[uu, "\texttt{i}_*"', "\simeq"]
	\end{tikzcd}
\end{equation}
The upper square is commutative by compatibility of the coarse corona pairing with the Mayer-Vietoris boundary by applying    \cref{LemCoarsePairingMV} with  
$ \cO^\infty(X)$, $\{\R^-\} \times X$, $ \{\R^+\} \times X$,  $ \cO^\infty_{\cZ}(\cY)$, and $ \cO_{\cZ}^\infty(X)$ in place of 
$X $, $\cW  $, $\cW^\prime$, $\cY$ and $\cZ$. 
Commutativity of the bottom square is extra-naturality of the coarse corona pairing for the inclusion map ${\texttt{i}}: X \to \cO^\infty(X)$, see \cref{qewoifjqeiowdedwqdqdqer43r3r343}.
The triangle on the right is essentially the definition of the index map, see \cref{IndexMapAsBoundary}. 
Finally, since $\pr \circ \texttt{i} = \id$, it is clear that the diagram 
\begin{equation}
\label{FourthDiagram}
\begin{tikzcd}
K(\partial^{\cY} X) \times K^{\cX}_{\cZ}(X) 
		\ar[d, "\pr^* \times \id"']
		\ar[dddd, "\id \times a_{X, \cZ}"', bend right=60, shift right=20]
	\\		
K(\partial^{\cO^\infty(\cY)} \cO^\infty(X)) \times \Sigma^{-1}K\cX(\cO^\infty_{\cZ}(X))
	\ar[d, "\id \times \partial^{\mathrm{MV}}"']
	\\
K(\partial^{\cO^\infty(\cY)} \cO^\infty(X)) \times K\cX(\{0\} \otimes \cZ)
	\\	
K(\partial^{\cO^\infty(\cY)} \cO^\infty(X)) \times K\cX(\cZ)
	\ar[u, "\id \times \texttt{i}_*", "\simeq"']
	\ar[d, "\texttt{i}^* \times \id"']
	\\
K(\partial^{\cY} X) \times K\cX(\cZ)
\end{tikzcd}	
\end{equation}
commutes.
The assertion of the proposition follows from pasting the diagrams \eqref{FirstDiagram}--\eqref{FourthDiagram}.
 \end{proof}

\begin{rem}
In the case $X$ is a coarsifying proper metric space and  $\cY=\cB$ the commutative
diagram  \eqref{erokgpwergrwefwerfwerfwer} is essentially equivalent to the statement of \cite[Thm. 36]{zbMATH05025124}. \mbox{\hspace{1cm}}   
\hB
\end{rem}

 \begin{rem}
In this remark we explain why in the construction of the coarse index pairing we can not simply replace the local $K$-homology $K^{\cX}$ by its   locally finite version $K^{\cX,\lf}$ (see \cref{kogpewgwergergwergwerfw}) or the locally finite $K$-homology $K^{\an}$ introduced in \cref{kogpwegreffrefrfw}.
We fix $X$ in $\UBC$ with a big family $\cY  $. 
For $B$ in $\cB$ we have a restriction map
\[
K(\cY)\to K (\cY\cap (X\setminus B))\ .
\]
By naturality of the coarse symbol pairing we can produce a map of diagrams indexed by $B$ in $\cB^{\op}$
$$K (\cY)\times \Cofib(K^{\cX}(X\setminus B)\to K^{\cX}(X))\to     \Cofib(K^{\cX}( \cY\cap  (X \setminus B) )\to K^{\cX}(\cY))\ .$$ 
But note that taking the limit over $\cB^{\op}$ does \textbf{not} produce a map 
\begin{equation}\label{sfdvsdfvsdfvsdfv} K(\cY)\times  K^{\cX,\lf}(X)\to K^{\cX,\lf}(\cY)\ . 
\end{equation} 
The problem is that in the construction above we must first take the colimit over $\cY$ and and then the limit over $\cB^{\op}$, while for the target in \eqref{sfdvsdfvsdfvsdfv}  we would first  take the limit and then the colimit. 

Similarly it is not clear whether in general  the coarse index pairing has a factorization over the Paschke transformation $p_{X}:K^{\cX}(X)\to K^{\an}(X)$ from \eqref{frefwrfregwreg}, i.e., whether we can construct a  pairing
$$ K(
\cY)\times K^{\an}(X)\stackrel{-\cap^{\an \sigma}-}{\to} K^{\an}(\cY)$$
rendering the obvious comparison square commutative.
But note that if $\cY$ consists of bounded subsets, then such a pairing exists.
\hB
\end{rem}

 Let $X$ be  in $\UBC$ and let $\cY$, $\cY' $ and $\cZ$  be big families on $X$.  
 Recall the generalization of the coarse symbol pairing discussed in \cref{GeneralizationSymbolPairing}. 
 
 \begin{lem} 
 \label{koppwegfrgwgregwe}
 The following square commutes 
 \[
\begin{tikzcd}[column sep=2cm]
K(\cY) \times K(\cY') \times K^{\cX} _{\cZ}(X)  
	\ar[d, "\id\times \cap^{\cX\sigma}"']  
	\ar[r, "{\text{\eqref{WeDoNeedThis!}}}\times \id"]
	& 
		K( \cY\cap \cY') \times K\cX(X)
			\ar[d, "\cap^{\cX\sigma}"]
			\\ 
K(\cY)\times {K}^{\cX}_{\cZ}( \cY')   
	\ar[r, "\cap^{\cX\sigma}"] 
	& 
		{K}_{\cZ}^{\cX}(\cY\cap \cY')
\end{tikzcd}
\]
\end{lem}
\begin{proof}
This follows from unfolding definitions, using \eqref{fwefwefeq} and   the associativity of the symmetric  monoidal structure of $K$.
\end{proof}

\section{Applications to Dirac operators}

\subsection{The coarse index  of generalized Dirac operators}\label{rijeogegergwergf}

Let $M$ be a complete Riemannian manifold. 
Its underlying metric space represents a uniform bornological coarse space which we will also denote by $M$. 
We consider a generalized Dirac operator 
  $$\Dirac:C^{\infty}(M,E)\to C^{\infty}(M,E)$$ of degree $n$. Thus    $E$ is a hermitean vector bundle  of graded modules over the Clifford algebra $\Cl^{n}$ and $\Dirac$ is  an odd first order  formally selfadjoint differential operator commuting with the $\Cl^{n}$-action and such that
  \begin{equation}\label{wrgwerrwfrfw}\Dirac^{2}=\nabla^{*}\nabla
+R_{0}
\end{equation} for some hermitean connection $\nabla$ on $E$ and  $R_{0}$ in $C^{\infty}(M,\End(E))$.
We consider $\Dirac$ and its powers as symmetric operators on the Hilbert space $L^{2}(M,E)$ with
domain $C_{c}^{\infty}(M,E)$. 
By \cite{Chernoff_1973} the unbounded operator  $\Dirac$ and its powers are essentially selfadjoint and we will use the same symbols also to denote their unique selfadjoint extensions.

The Dirac operator gives rise to a coarse index class $\ind\cX(\Dirac)$ in $K\cX_{-n}(M )$ whose construction we will describe below in detail.  Under additional local positivity assumptions this index class is actually supported   on a suitable big family on $M$.

\begin{rem}
 That local positivity implies a restriction of the support of the coarse index has been observed in \cite{Roe:2012fk}. 
In the generality of the present paper (and even equivariantly) for  Callias type operators, this has been studied in \cite{Guo_2022}. \hB
\end{rem}

The norm  in $L^{2}(M,E)$ of a section $\phi$   will be denoted by $\|\phi\|$.
Let $Y$  be a   subset of $M$.  

\begin{ddd}
\label{kopherthergrge}
We say that  $\Dirac$ is  positive  away from   $Y$  if there exists {a number} $a$ in $(0,\infty)$ such that
$\|\Dirac \phi\|^{2}\ge a^{2} \|\phi\|^{2}$  for all $\phi$ in $C_{c}^{\infty}(M\setminus \bar Y,E)$.

We will say that $\Dirac$ is positive away from a big family $\cY$ if it is positive away from  {some}  member of $\cY$.
\end{ddd}

\begin{ex} We consider the summand $R_{0}$ from \eqref{wrgwerrwfrfw}. Let $\sigma(R_{0}(m))$ denote the spectrum of the selfadjoint endomorphism $R_{0}(m)$ of the fibre  of $E$ at $m$ in $M$.
If $\inf_{m\in M\setminus Y} \min(\sigma(R_{0}(m)))>0$, then $\Dirac$ is {positive} away from  $Y$. \hB
\end{ex}

We assume that $\Dirac$ is {positive} away from a big family $ \cY$.  One should expect that $\Dirac$ is close to being invertible outside $\cY$ and that the relevant index theoretic information is supported on this big family. Indeed, following insights of   \cite{Roe:2012fk}, one constructs   a  refined index class  $\ind \cX(\Dirac,\mathrm{on}\,\cY)$ in $K\cX_{-n}(\cY)$  such that 
 $$j_{*} \ind \cX(\Dirac,\mathrm{on}\,\cY)=\ind\cX(\Dirac)\ ,$$ where $j_{*}: K\cX(\cY)\to K\cX(M)$  is the  canonical morphism. As we need the details of the construction  in the subsequent sections 
 we recall  the precise {definition} of this index class from \cite{Bunke:2017aa} {in \cref{gjkregopregwergfwrefrfwrf}}.

 %  \begin{rem}
%  For separable $A$ it is known that the set of homotopy classes 
%  $[\cS,\hat K\stimes A]$ is actually in bijection with $\gK(A)$.
%  Here $\hat K$ is a graded version of the compact operators given by the compact operators on
%  $\ell^{2}\oplus (\ell^{2})^{\op}$ with the induced grading. \hB
%  \end{rem}
%

  Let $X$ be a bornological coarse space. A graded $X$-controlled Hilbert space determined on points of degree $n$ is an $X$-controlled Hilbert space $(H,\chi)$ determined on points {(see \cref{gojkperfrefrwfwrefw})} such that $H$ is in addition  a graded module over $\Cl^{n}$, {with} the generators of $\Cl^{n}$ acting as anti-selfadjoint operators. 
  A bounded operator $A$ in $B(H)$ is called locally compact if
  $\chi(B)A$ and $A\chi(B)$ are compact for all bounded subsets $B$ of $X$. Note that if $(H,\chi)$ is locally finite,  then local compactness is  automatic. \begin{ddd} The Roe algebra $C^{*}(H,\chi)$ associated to $(H,\chi)$  is the $C^{*}$-subalgebra of $B(H)$ generated by $\Cl^{n}$-equivariant controlled and locally compact operators.\end{ddd}
  
  \begin{ex}
  If $(H,\chi)$ is an object of the Roe category $\bC(X)$ of $X$, then it is of degree $0$ and
  $\End_{\bC(X)}((H,\chi))=C^{*}(H,\chi)$. \hB
  \end{ex}

  We now come back to the Dirac operator $\Dirac$ on $E\to M$. 
  We consider the graded Hilbert space $L^{2}(M,E)$ with its induced $\Cl^{n}$-action, where the $L^{2}$-scalar product is defined using the hermitean structure of $E$ and the Riemannian volume measure. 
  In order to turn it into a graded $M$-controlled Hilbert space determined on points of degree $n$, we choose a   partition $(B_{i})_{i\in I}$ of $M$ by  {regular} Borel subsets with  {uniformly bounded diameter} and base points $b_{i}$ in $B_{i}$ for every $i$ in $I$. 
We then define the projection valued measure $$\chi:=\sum_{i\in I}\chi_{B_{i}}  \delta_{b_{i}}\ ,$$
where $  \chi_{B_{i}}$ is the multiplication  on $L^{2}(M,E)$ by the characteristic function of $B_{i}$. Thus for a subset $Z$ of $M$ we have 
   $\chi(Z)=\sum_{i\in I, b_{i}\in Z} \chi_{B_{i}}  $.     
   
   By construction, the graded $M$-controlled Hilbert space $(L^{2}(M,E),\chi)$
    is determined on points, but   it is  not locally finite if $M$ has positive dimension. 
    We let 
    \[
    C^{*}(M,E):={C^*}(L^{2}(M,E),\chi)
    \] 
    denote the associated Roe algebra. 
    Note that it is independent of the choice of the partition $(B_{i})_{i\in I}$. 
   We let furthermore  $C^{*}(\cY\subseteq M,E)$ be the subalgebra of $C^{*}(M,E)$ generated by operators of the form $\chi(Y) A\chi(Y)$ for $A$ in $C^{*}(M,E)$ and $Y$ a member of $\cY$.
   Note that $C^{*}(\cY\subseteq M,E)$ is an ideal and again does not depend on the choice of the partition.
   
   The following lemma is \cite[Lem.\ 2.3]{Roe:2012fk}.  
   
   \begin{lem}\label{werkogperfwfwrf}
 If   $\Dirac$  is {positive} away from  $\cY$, then there exists $a$ in $(0,\infty)$ such that for every $f$ in $C_{0}((-a,a))$ we have $f(\Dirac)\in C^{*}(\cY\subseteq M,E)$.
 \end{lem}
 
 \begin{proof}
 We repeat the proof  \cite[Lem. 2.3]{Roe:2012fk} since this lemma is absolutely crucial for our purpose and {in the reference, the statement is only proven under the stronger assumptions that $\Dirac$ is associated to a Clifford bundle and the term $R_{0}$ in \eqref{wrgwerrwfrfw} is positive  on $M\setminus \bar Y$.
 }
 
  Let $Y$ be a member of $\cY$ such that $\Dirac$ is positive away from $Y$.      
  For $s$ in $[0,\infty)$ we let $U_{s}$  be the metric entourage of width $s$ as in \cref{epokpbbwerg}.

 For {an} even {function} $f$ in $\cS$  let $\hat f$ denote the Fourier transform (in general a distribution) such that 
 \begin{equation}
 \label{fojqw0fewfqdedqe}
 f(x)=\frac{1}{\pi} \int_{-\infty}^{\infty} \hat f(t) \cos(tx) dt\ ,
\end{equation} 
where the integral has to be interpreted appropriately.  We assume  $\supp(\hat f)\subseteq [-r,r]$ for some $r$ in $(0,\infty)$ and show that
 \begin{equation}
 \label{qewfwefwdqde}
 \| f(\Dirac)\chi(M\setminus U_{b+r}[Y])\|\le  \sup_{|s|\ge |a|} |f(s)|\ ,
\end{equation}
{where $b$ in $(0,\infty)$ is a bound on the diameter of the sets $B_i$ used in the construction of the projection valued measure on $L^{2}(M,E)$.}
(This is the analogue of \cite[Lem.\ 2.5]{Roe:2012fk}).
   
 The symmetric operator $\Dirac^{2}$ defines a  {hermitean} sesquilinear form $(\phi, \psi) \mapsto \langle \phi, \Dirac^2 \psi\rangle$   on $C^{\infty}_{c}(M\setminus \bar Y,E)$ which by assumption  is bounded below by $a^{2}$  for some $a$ in $(0,\infty)$.
 Let $A$ denote the square root of the Friedrichs extension of $\Dirac^{2}$ on $L^{2}(M\setminus \bar Y,E)$.
 By finite propagation speed  of the solutions of the  wave equation,   for any $h$ in $ L^{2}(M\setminus  \bar Y,E)$ with $\supp(h)\in M\setminus U_{r}[Y]$ we have
 $\cos(t\Dirac)h=\cos(tA)h$ for $t\in [-r,r]$.  Replacing $x$ by $\Dirac$ or $A$ in \eqref{fojqw0fewfqdedqe} and since $\supp(\hat f)\subseteq [-r,r]$ we get 
 $ f(\Dirac)h= f( A)h$.
  Since $\|f(A)\|\le \sup_{|s|\ge |a|} |f(s)|$ we conclude  \eqref{qewfwefwdqde}.

 One next derives the analogue of \cite[Lem.\ 2.6]{Roe:2012fk}
 by the same argument as in the reference: For {arbitrary (not necessarily even)} $f$ in $\cS$ with $\supp(f)\in [-r,r]$ we have
 \begin{equation}
 \label{vfdvrvvsdfvsfdv}
 \| f(\Dirac)\chi(M\setminus U_{2r+b}[Y])\|\le 2 \sup_{|s|\ge |a|} |f(s)|\ .
\end{equation} 
 To this end, we observe for odd $f$ that
 $$\| f(\Dirac)\chi(M\setminus U_{2r+b}[Y])\|^{2}\le \| |f|^{2}(\Dirac)\chi(M\setminus U_{2r+b}[Y])\|$$
 and that $\supp (\widehat{|f|^{2}})\subseteq [-2r,2r]$.
 {Then} $|f|^{2}$ is even so that we can apply the estimate already proven.
 A general function $f$ can be written as a sum of an even and an odd function which accounts for the factor $2$ in the right-hand side of \eqref{vfdvrvvsdfvsfdv}.
 
{Finally, we repeat the argument of \cite{Roe:2012fk} that \eqref{vfdvrvvsdfvsfdv} implies \cref{werkogperfwfwrf}.}
%In order to deduce the assertion of \cref{werkogperfwfwrf}, we finally repeat the argument from the reference that   \cite[Lem.\ 2.6]{Roe:2012fk} implies \cite[Lem.\ 2.3]{Roe:2012fk}.
Assume that $f$ is in $C_{0}((-a,a))$.
{As functions with compactly supported Fourier transform are dense in $C_0(\R)$,} 
given $\epsilon$ in $(0,\infty)$, there exists $g$ in $\cS$ with $\hat g$ compactly supported and $\|f-g\|_{\infty}\le \epsilon$, where we consider $f$ as an element of $\cS$   {through} extension by zero.
Then $\|f(\Dirac)-g(\Dirac)\|\le \epsilon$, $g$ has bounded propagation, and {by \eqref{vfdvrvvsdfvsfdv}}, $\| g(\Dirac)\chi(M\setminus U_{r+b}[Y])\|\le \epsilon$ for $r$  {so large} that $\supp(\hat g)\in [-r,r]$. 
 We can conclude that $f(\Dirac)$ is in distance $2\epsilon$ from an element of $C^{*}(\cY\subseteq M,E)$. 
 Since $\epsilon$ is arbitrary we conclude that
 $f(\Dirac)\in C^{*}(\cY\subseteq M,E)$ as asserted.
    \end{proof}

It follows from \cref{werkogperfwfwrf} that if $\Dirac$ is positive away from $Y$, then  there exists $a$ in $(0,\infty)$ such that
  the homomorphism 
  \[
  \cS\to C^{*}(M,E)\ , \quad f\mapsto f(\Dirac)
  \]
   restricts to a homomorphism
   \[
   (\Dirac_{\cY})_{*}:C_{0}((-a,a))\to C^{*}(\cY\subseteq M,E)\ .
   \]
    As explained in \cref{okfpqfqewddedq}, it represents a class 
   \begin{equation}\label{bdfgwrtbgwggsdtg}
    [\Dirac_{\cY}] \quad  \mbox{in} \quad \gK_{0}(C^{*}(\cY\subseteq M,E))
     \end{equation}      such that $j_{*}[\Dirac_{\cY}]=[\Dirac]$.
    
    In order to identify this class with a class of $\K\cX_{-n}(M)$ we must incorporate the gradings and and degrees in the definition of Roe categories and compare the  graded $K$-theory of Roe algebra $C^{*}(\cY\subseteq M,E)$ with the graded $K$-theory of a corresponding graded Roe category.

 Generalizing the construction of the Roe category $\bC(X)$ in \cref{gojkperfrefrwfwrefw}, for every bornological coarse space $X$ we can consider the graded $C^{*}$-category  $\bC[n](X)$  of graded locally finite  $X$-controlled Hilbert spaces $(H,\chi)$ determined on points  of degree $n$ and  
  morphisms which are bounded    operators that can be approximated  in norm by controlled
   $\Cl^{n}$-equivariant operators. We get a functor
   $$\bC[n]:\BC\to \gCcat% \quad X\mapsto \bC[n](X)\ , \quad (f:X\to Y)\mapsto (\bC(f):(H,\chi)\mapsto (H,f_{*}\chi))
   \ .$$
   We  define the 
   coarse $K$-homology functor of degree $n$ as the composition
   $$K\cX[n]:\BC\xrightarrow{\bC[n]}  \gCcat \xrightarrow{\gK}\Mod(KU)\ .$$
   As observed in \cite{Bunke:2017aa}  (or using \cref{ekohperhetrgertg}) we have a canonical  equivalence of functors 
   \begin{equation}\label{rewgeopgkpwerfwerfwerf}K\cX[n]\simeq \Sigma^{n}K\cX:\BC\to \Mod(KU)
\end{equation} 
   which first of all shows that $K\cX[n]$ is indeed a coarse homology theory, and that it  is equivalent to $K\cX$ up to a shift.

  In order to compare $\bC[n](M)$ and $C^{*}(M,E)$ we form the intermediate graded $C^{*}$-category $\tilde \bC[n](M)$ of degree $n$ whose objects are triples $(H,\chi,U)$ with $(H,\chi)$ in $\bC[n](M)$
  and $U:H\to L^{2}(M,E)$ a $\Cl^{n}$-equivariant controlled isometric embedding. The 
  morphisms $A:  (H,\chi,U)\to   (H',\chi',U')$  in $\tilde \bC[n](X)$ are the morphisms  $A:(H,\chi)\to (H',\chi')$  in $\bC[n](X)$.

 We now assume that $M$ is not zero-dimensional. 
 Then {the graded $M$-controlled Hilbert space $(L^{2}(M,E),\chi)$ is  ample}  which means that any graded locally finite $M$-controlled Hilbert space admits a controlled isometric embedding into it.  
 Consequently,  the forgetful functor
 \begin{equation}
 \label{vrewjvoievjiosvdfvdsfvs}
 \cF:\tilde \bC[n](M)\to \bC[n](M)\ , \quad (H,\chi,U)\mapsto (H,\chi)
\end{equation} 
 is surjective on objects. Since it is fully faithful by definition, it is therefore an equivalence of graded $C^{*}$-categories. 
  
     We consider the Roe algebra $C^{*}(M,E)$ as a  graded $C^{*}$-category with a single object.
     We then have a functor  \begin{equation}
 \label{gergwerfwrfweferfw}
 {\cI}:\tilde \bC[n](M)\to C^{*}(M,E)\ , \quad (A:(H,\chi,U)\to (H',\chi',U'))\mapsto U'AU^{*}\ .
\end{equation} 
  One can show that $C^{*}(M,E)$   is generated by the image of ${\cI}$ {and} it has been observed in \cite{Bunke:2017aa}
that      
\[
\gK({\cI}):\gK(\tilde \bC[n](M)) \to \gK(C^{*}(M,E))
\]
 is an equivalence.
The zig-zag of functors
\begin{equation}
\label{eqfewdqcdvadvad}
\bC[n](M) \stackrel{\cF}{\leftarrow} \tilde \bC[n](M)\stackrel{{\cI}}{\rightarrow} C^{*}(M,E)
\end{equation}
thus induces upon applying $K^{\gr}$ an equivalence 
  $$\kappa_{M}:\Sigma^{n}K\cX(M)\stackrel{\eqref{rewgeopgkpwerfwerfwerf}}{\simeq} K\cX[n](M)  \stackrel{\simeq }{\leftarrow} \gK(\tilde \bC[n](M)) \stackrel{\simeq }{\to}\gK(C^{*}(M,E))\ .$$ 
  The zig-zag \eqref{eqfewdqcdvadvad} restricts to a zig-zag of ideals 
  \begin{equation*}
  \bC[n](\cY\subseteq M) \stackrel{\cF}{\leftarrow} \tilde \bC[n](\cY\subseteq M)\stackrel{{\cI}}{\rightarrow} C^{*}(\cY\subseteq M,E)
\end{equation*}
which in turn induces an equivalence 
\begin{equation}\label{qedqefq2t5}\kappa_{\cY}:\Sigma^{n}K\cX(\cY)\simeq \gK(C^{*}(\cY\subseteq M,E))
\end{equation} such that
 we have   commutative square
 \begin{equation}\label{fqrfqwdewdqdefdeqfeqwfq}\xymatrix{\Sigma^{n}K\cX(\cY)\ar[d]\ar[r]^-{\simeq }_-{\kappa_{\cY}} &\gK(C^{*}(\cY\subseteq M,E)) \ar[d] \\ \Sigma^{n}K\cX(M) \ar[r]^-{\simeq}_-{\kappa_{M}} &\gK(C^{*}(M,E)) }\ , 
\end{equation}  
  where the vertical maps are induced by the inclusions.
  
Recall the definition of the class $[\Dirac_{\cY}]$ in \eqref{bdfgwrtbgwggsdtg}.
  \begin{ddd}\label{gjkregopregwergfwrefrfwrf} If $\Dirac$ is {positive} away from $\cY$, then 
  the coarse index class with support $$\ind\cX(\Dirac, \mathrm{on}\,\cY)\quad  \mbox{in} \quad K\cX_{-n}(\cY)$$ is defined uniquely by the condition $\kappa_{\cY}(\ind\cX(\Dirac, \mathrm{on}\,\cY))\simeq [\Dirac_{\cY}]$. 
  \end{ddd}
 
   We now formulate the basic  properties of the coarse index class.
  For $i=0,1$ let $\Dirac_{i}$ be a generalized Dirac operators of degree $k_{i}$ on complete Riemannian manifolds $M_{i}$. % Then we can form $\Dirac_{0}\stimes \id+\id \stimes \Dirac_{1}$ on $M_{0}\times M_{1}$.
  If $\Dirac_{1}$ is positive away from {$\cY$}, then $\Dirac_{0}\stimes \id+\id \stimes \Dirac_{1}$ is positive away from $M_{0}\times \cY$.
 Using the explicit description of the cup product from \cref{wgklopergwrefwef}, one can check that the image under the symmetric monoidal structure $$K\cX(M_{0})\times K\cX(\cY)\to K\cX(M_{0}\otimes \cY)$$
  of $(\ind\cX(\Dirac_{0}),\ind\cX(\Dirac_{1},\mathrm{on}\,\cY))$ is the index $\ind\cX(\Dirac_{0}\stimes \id+\id \stimes \Dirac_{1},\mathrm{on}\,M_{0}\times \cY)$ of the product Dirac operator of degree $k_{0}+k_{1}$.
 
We take $M_{0}=\R$ and $\Dirac_{0}=\sigma \partial_{t}$ on the trivial bundle with fibre $\Cl^{1}$.  
Then $$\ind\cX(\sigma \partial_{t}) \in  \pi_{-1} K\cX(\R)\cong   \pi_{-2}K\cX(*)\cong KU_{-2}$$ is the inverse of the Bott element  $\beta$.
On the manifold $\R \times M$, there is a suspended Dirac operator $\Sigma D$, given by
  \begin{equation}
  \label{gwerpogjowkepferfwerfwrf} 
  \Sigma \Dirac := \id_{\pr^{*}E}\stimes \sigma  \partial_{t}+ \pr^{*} \Dirac\stimes \id_{\Cl^{1}} 
\end{equation} 
acting on sections of the bundle $\pr^{*}E\stimes \Cl^{1}$.
{
Here $\sigma$ is the odd generator of the tensor factor $\Cl^{1}$ (acting by left multiplication) with $\sigma^{*}=-\sigma$ and $\sigma^{2}=-1$.
A special case of the above statements on products is then the following.}

Recall that the equivalence $\delta^{MV}:{K}\cX({\R}\otimes \cY)\stackrel{\simeq}{\to} \Sigma K\cX(\cY)$ (see \eqref{fqwedwedewdqwdwd}) induced by Mayer-Vietoris boundary for the decomposition of ${\R}\otimes \cY$ into ${\R^{-}}\times \cY$ and ${\R^{+}}\times \cY$.
  
   \begin{prop}[Suspension theorem] 
   \label{SuspensionTheorem} 
We have the equality  
\[ 
   \beta \cdot \delta^{MV}\ind\cX(\Sigma \Dirac,\mathrm{on}\,{\R \times \cY})=\ind\cX(\Dirac,\mathrm{on}\,\cY)\ .
\]
\end{prop}   
  
{
Here on the left, we multiplied using the $KU$-module structure on the values of $ K\cX$.% (\R \otimes \cY)$.
} 
   
 For $i=0,1$ we consider generalized Dirac operators $\Dirac_{i}$ on complete Riemannian manifolds $M_{i}$ which are positive away from  big families families $\cY_{i}$.
  Assume there exist subsets $W_{i}$ in $M_{i}$  such that there exists  a Riemannian isometry  $e:M_{0}\setminus  W_{0}\to M_{1}\setminus W_{1}$ which is a morphism of bornological coarse spaces  and covered by an isomorphism of all bundle data. Note that $e$ is not necessarily an isomorphism of bornological coarse spaces, {as $e$ is not necessarily distance-preserving}.
  We let $\cW$ and $\cW'$ be the big families generated by $W$ and $W'$, respectively.  
 By excision we get a morphism
\begin{equation}
\label{vfpojopewgerwg}
e_{*}:K\cX(\cY_{0},\cY_{0}\cap \cW_{0}) \to K\cX(\cY_{1},\cY_{1}\cap \cW_{1})\ .
\end{equation} 
We let 
\[
\pi_{i}:K\cX(\cY_{i})\to K\cX(\cY_{i},\cY_{i}\cap \cW_{i}) 
\]
denote the projections. 
The following is \cite[Thm. 10.5]{Bunke:2017aa}.

\begin{prop}[The relative index theorem]
 \label{RelativeIndexTheorem}
%\nml{Under the equivalence \eqref{vfpojopewgerwg} }\fml{Der Punkt war doch gerade, dass $e_*$ nicht notwendigerweise eine aequivalenz ist, oder?}
We have the equality
\[
e_{* }\pi_{0}(\ind\cX(\Dirac_{0},\mathrm{on}\,\cY_{0}))=\pi_{1}(\ind\cX(\Dirac_{1} ,\mathrm{on}\,\cY_{1}))\ .
\]
\end{prop}

  \subsection{Symbols with support}\label{gkopwerfefrefwfref}

{As before, let  $M$ be a complete Riemannian manifold.}     
       Classically, the symbol  of a generalized  Dirac operator on $M$  is the analytic  locally finite $K$-homology  class $\sigma^{\an}(\Dirac)$ in $K^{\an}(M)$. 
   In the present section, we propose  to consider a refined symbol class $\sigma(\Dirac)$ in $K^{\cX}(M)$ such that $p_{M}(\sigma(\Dirac))=\sigma^{\an}(\Dirac)$, see \eqref{r3rgrgsfegeg} for $p$. 
   If $\Dirac$ is positive away from a big family $\cY$, then we introduce a further refined symbol class $\sigma_{\cY}(\Dirac )$ in $K^{\cX}_{\cY}(M)$ which in addition captures the reason for the fact that the $a_{M,\cY}(\sigma(\Dirac))=\ind(\Dirac,\mathrm{on}\,\cY)$ is supported on $\cY$.

 Since by definition $K_{\cY}^{\cX}(M)\simeq \Sigma^{-1} K\cX (\cO_{\cY}^{\infty}(M))$ (see  \cref{fewdqwedqewdqewd}) and $K\cX$ receives indices of Dirac type operators it is natural to represent $\sigma_{\cY}(\Dirac)$ as a coarse index of a generalized Dirac operator $\smash{\TildeDirac}$ naturally derived from $\Dirac$  on the cone over $M$.
     
   \begin{construction}
   \label{kogpweerfwerfrwefw}
   \em
   We assume that $\Dirac$ has degree $n$. Then we 
  consider the manifold $\tilde M:=\R\times M$  with the warped product metric
  \[
  \tilde g=dt^{2}+{h(t)^2}g
  \]
 for some choice of a smooth function $h:\R\to(0,\infty)$  with $h(t)=1$ for $t\le 0$ and $h(t)\ge 1$ for $t\ge 0$. 
 We let $\smash{\TildeDirac}$ be the degree $n+1$ Dirac operator on the bundle $\tilde E:=\pr_{M}^{*}E\stimes \Cl^{1}\to \tilde M$ explicitly given  by the formula  (see \cref{RemarkAdaption} for the motivation)
\begin{equation}
\label{FormulaAdaptedOperator}
\TildeDirac =  \id_{\pr^{*}E}\stimes \sigma  \left(\partial_t + \frac{{n}}{2} \frac{\dot{h}(t)}{h(t)} \right) + \frac{1}{h(t)} \pr^{*} \Dirac\stimes \id_{\Cl^{1}},
\end{equation}
where $\sigma$ is the odd generator of the tensor factor $\Cl^{1}$.
\end{construction}

\begin{rem}
\label{RemarkAdaption}
Abstractly, $\smash{\TildeDirac}$ is the result obtained by adapting the suspension $\Sigma \Dirac$ of $\Dirac$ on $\R\times M$ given in \eqref{gwerpogjowkepferfwerfwrf}, to the metric $\tilde g$, using the adaptation procedure given in \cite[Sec.\ 4]{Bunke:2018aa}.
The process of adaptation is uniquely fixed by the following requirements: (1) It is a local construction.
(2) When applied to the spin Dirac operator on the product it produces the spin  Dirac operator on the warped product.
(3) Finally it is compatible with forming twisted Dirac operators and adding zero order terms.
\hB 
 \end{rem}

We consider the big family  \begin{equation}\label{qfwefdqeq423}
 \tilde M_{\cY} := {\R} \times \cY\cup \{\R^{+}\}\times M\ .
\end{equation}  
on $\tilde M$, see \cref{BigFamilyConstr} for notation.
 \begin{lem}
 \label{LemmaPositivityOnCone}
 If $\Dirac$ is positive away from $\cY$, then $\smash{\TildeDirac}$ is positive away from {$\tilde M_{\cY}$.}
 %\begin{equation}
 %\label{qfwefdqeq423}
 %\tilde M_{\cY} := {\R} \times \cY\cup \{\R^{+}\}\times M\ .
%\end{equation}  
\end{lem}

\begin{proof}

	Let $Y$ be a member of $\cY$ such that $\Dirac$ is positive on $M \setminus Y$ and choose $a$ in $(0,\infty)$ as in \cref{kopherthergrge}.
	We claim that $\smash{\TildeDirac}$ is positive away from 
	\[
	(\R \times Y) \cup (\R^+ \times M) = (\R \times M) \setminus ({(-\infty,0)}\times ( M \setminus Y)) \ .
	\]
	Indeed, let $\tilde{\phi}$ be in $ C^\infty_c({(-\infty,0)}\times (M \setminus Y), \tilde{E})$.
	Then since $h \equiv 1$ on the support of $\tilde{\phi}$, we have
	\[
	  \TildeDirac \tilde{\phi} =  (\id_{\pr^*E} \stimes \sigma \partial_t + \Dirac \stimes \id_{\Cl^1}) \tilde{\phi} \ .
	\]
	As the summands in the bracket commute (in the graded sense), we therefore get
	\[
	\TildeDirac{}^2 \tilde{\phi} 
	= \id_{\pr^*E} \stimes (\sigma \partial_t)^2 \tilde{\phi} + \Dirac^2 \stimes \id_{\Cl^1} \tilde{\phi},
	\]
	so
	\begin{equation*}
	\begin{aligned}
	  \|\TildeDirac \tilde{\phi}\|^2
	  = \langle \tilde{\phi}, \TildeDirac{}^2 \tilde{\phi}\rangle
	  = - \langle \tilde{\phi}, \partial_{t}^{2}{\tilde{\phi}}\rangle + \langle \tilde{\phi}, \Dirac^2 \tilde{\phi} \rangle = \|\partial_{t}{\phi}\|^{2} + \|\Dirac\tilde{\phi}\|^2 \geq a^2 \|\tilde{\phi}\|^2.
	\end{aligned}
	\end{equation*}
	Here we used that $\|\Dirac \tilde{\phi}(t)\|_{\{t\} \times M}^2 \geq a^2 \|\tilde{\phi}(t)\|^2$ for each $t$. 
\end{proof}

In view of \cref{LemmaPositivityOnCone} the \cref{gjkregopregwergfwrefrfwrf}   provides the coarse index class   $\ind\cX(\smash{\TildeDirac},\mathrm{on}\, \tilde M_{\cY})$ in $K\cX_{-n-1}( \tilde M_{\cY})$.

For the moment we write $\tilde M_{h}$ and $\TildeDirac_{h}$ for the  manifold and Dirac operator associated to the warping function $h$.
 The set of these functions is partially ordered and the constant function $h\equiv 1$
is the minimal element. 
If $h',h$ are two such functions such that 
\[
h'\le h \ ,
\] then the identity map of $\R\times M$
is a map of bornological coarse spaces $q:\tilde M_{h}\to \tilde M_{h'}$.
We assume that $\Dirac$ is positive away from $\cY$.
 We add a subscript $h$ or $h'$ to the notation of the big family $\tilde M_{\cY}$
indicating from which space its coarse structure is induced.
\begin{prop}[{\cite[Prop. 4.12]{Bunke:2018aa}}]\label{bjgbpdgdfbdg}
If $\Dirac$ is positive away from $\cY$, then
we have the equality \begin{equation}\label{wrgerferf}q_{*} \ind\cX(\TildeDirac_{h},\mathrm{on}\,\tilde M_{h,\cY})=\ind\cX(\TildeDirac_{h'},\mathrm{on}\,\tilde M_{h',\cY})
\end{equation} 
in $K\cX_{-n-1}(\tilde M_{h',\cY})$.
\end{prop}

%\fml{Ich habe jetzt deinen Beweis verstanden, habe ihn aber versucht, in zwei groessere Schritte aufzuteilen, aehnlich zu meinem vorigen Versuch. Das Ergebnis unten. Der urspruengliche Beweis ist darunter auskommentiert.}

\begin{proof}
We emphasize that the all analysis of Dirac operators used in this proof is  hidden in the suspension theorem \cref{SuspensionTheorem} and the coarse relative index theorem, \cref{RelativeIndexTheorem}.
Besides these results, the argument below only uses formal properties of coarse homology theories and some simple coarse geometric insights.

 We consider the manifold
 $\hat  M:=\R\times \R\times M$ and let $(s,t)$ denote the coordinates on the factor $\R\times \R$. 
 We choose a real-valued smooth increasing function
 $\chi$ on $\R$ such that $\chi(s)=0$ for $s\le 0$ and $\chi(s)=1$ for $s\ge 1$. 
  We then consider the function on $\hat M$ given by 
 \[
 \tilde h(s,t,m):=(1-\chi(s)) h^{\prime}(t)+\chi(s) h(t)\ .
 \]
 \begin{center}
\begin{tikzpicture}
	\draw[color=white, name path=Y] (-7, -1.5) -- (0, -1.5);
	\draw[color=white, name path=X] (-7, 3) -- (0, 3);
	\draw[color=white, name path=W] (0, -1.5) -- (7, -1.5);
	\draw[color=white, name path=Z] (0, 3) -- (7, 3);
	\draw[thick, name path=A] (0, 0) -- (7, 0);
	\draw[thick] (0, 3) -- (0, -1.5);
	\draw[thick, name path=B] (0, 1.5) -- (7, 1.5);
%	\tikzfillbetween[of=A and B]{gray, opacity=0.1};
	\tikzfillbetween[of=X and Y]{gray, opacity=0.05};
	\tikzfillbetween[of=W and A]{gray, opacity=0.1};
	\tikzfillbetween[of=Z and B]{gray, opacity=0.1};
\node at (3.5, 0.75) {transition region};
\node at (3.5, 2.25) {$[1, \infty) \times ([0,\infty)\times M)_h$};
\node at (3.5, -0.75) {$(-\infty, 0] \times ([0,\infty)\times M)_{h'}$};
\node at (-4, 0.75) {$\R \times (-\infty, 0] \times M$};
\node at (-0.7, 1.5) {{\color{gray}$s=1$}};
\node at (-0.7, 0) {{\color{gray}$s=0$}};
\end{tikzpicture}
 \end{center}

We view $\hat{M}$ as a bornological coarse space with the warped product metric
  $ds^{2}+dt^{2}+\tilde h g$.  
    We then get following morphisms of bornological coarse spaces
 \[
 \R \otimes \tilde M_{h}\stackrel{\hat q}{\longrightarrow}  \hat M \stackrel{\hat q^{\prime}}{\longrightarrow} \R \otimes \tilde M_{h^{\prime}} \ ,
 \]
 all induced by the identity of the underlying sets.

The families
\[
\hat{M}_{\cY} := \R \times \R \times \cY , \quad
\hat{M}^{\{\leq a\}}:= \{(-\infty,a] \}\times \R\times M \ ,  \quad 
\hat{M}^{\{\geq a\}}:= \{[a,\infty)\} \times \R\times M 
 \]
are all big families on $\hat{M} $ and so are
  \[  
\hat{M}^{\{\leq a\}}_{\cY} := \hat{M}^{\{\leq a\}} \cap \hat{M}_{\cY} \ ,\quad  \hat{M}^{\{\geq a\}}_{\cY} := \hat{M}^{\{\geq a\}}\cap \smash{\hat{M}_{\cY}} \ . 
\]
 We define $\smash{\hat \Dirac}$  by a two-dimensional version of \cref{kogpweerfwerfrwefw}.
Notice that $\hat{q}$ and $\hat{q}'$ induce local isomorphisms of the Dirac operators on the subsets $[1,\infty)\times \R\times M$ and $(-\infty,0]\times \R\times M$, respectively.
 
We consider the the projections
\begin{align*}
\pi_{\le 1} : K\cX (\R \otimes \tilde{M}_{h, \cY}) &\longrightarrow K\cX (\R \otimes \tilde{M}_{h, \cY}, \{(-\infty,1]\} \otimes \tilde{M}_{h, \cY}) \ ,
\\
\pi_{\ge 0}' : K\cX (\R \otimes \tilde{M}_{h', \cY}) &\longrightarrow K\cX (\R \otimes \tilde{M}_{h', \cY}, \{[0,\infty)\} \otimes \tilde{M}_{h', \cY}) \ ,
\\
\tilde{\pi}_{\le 1} : K\cX (\hat{M}_{  \cY}) &\longrightarrow K\cX (\hat{M}_{  \cY}, \hat{M}^{\{\le 1\}}_{ \cY})\ ,	\\
\tilde{\pi}_{\ge 0} : K\cX (\hat{M}_{  \cY}) &\longrightarrow K\cX (\hat{M}_{  \cY}, \hat{M}^{\{\ge 0\}}_{  \cY})	
\end{align*}
from the absolute to the relative coarse cohomology, and let 
 \begin{align*}\hat q^{\mathrm{Rel}}:& K\cX(\R\otimes \tilde M_{h,\cY},\{(-\infty,1]\}\otimes \tilde M_{h,\cY})\to K\cX(\hat M_{ \cY},\hat M^{\{\le 1\}}_{ \cY})\\
\hat q^{\prime,\mathrm{\mathrm{Rel}}}:&K\cX(\hat M_{ \cY},\hat M^{\{\ge 0\}}_{ \cY})\to  K\cX(\R\otimes \tilde M_{h',\cY},\{[0,\infty)\}\otimes \tilde M_{h,\cY})
\end{align*}
the maps in relative coarse $K$-homology.
{Then the relative index theorem,  \cref{RelativeIndexTheorem},} provides   the identities
\begin{align}
\label{ApplRelativeIndex1}
 \hat{q}^{\mathrm{Rel}}_*\pi_{\leq 1,*}\ind\cX(\Sigma \tilde\Dirac_h, \mathrm{on}\,\R \times \tilde{M}_{h, \cY})
  &=
  \tilde{\pi}_{\leq 1,*}\ind\cX(\hat\Dirac , \mathrm{on}\,\hat{M}_{ \cY}) \ ,
  \\\label{ApplRelativeIndex2}
  \pi_{\geq 0,*}'\ind\cX(\Sigma \TildeDirac_{h'}, \mathrm{on}\,\R \times \tilde{M}_{h', \cY})
  &=
\hat{q}^{\prime,\mathrm{Rel}}_*\tilde\pi_{\geq 0,*}\ind\cX(\hat \Dirac , \mathrm{on}\,\hat{M}_{  \cY})  \ .
\end{align}
  
In the calculations below we use the following relation beween Mayer-Vietoris boundaries and boundaries in pair sequences. 
Let $X$ be a bornological coarse space with a big family $\cY$ and let  $(W,W')$ be a coarsely excisive decomposition of $X$.
Then we   have a commutative diagram
 \begin{equation}\label{gweprkopwefklwpoekplml}
\begin{tikzcd}
K\cX(\cY) \ar[ddr, "\partial^{MV}"]
\ar[r, "\delta^{MV}"]
\ar[d, "\pi"']
&
\Sigma K\cX(\cY \cap W \cap W') \ar[dd, "\iota", "\simeq"'] 
\\
K\cX(\cY, \cY \cap \{W'\}) \\
K\cX(\cY \cap \{W\}, \cY \cap \{W\} \cap \{W'\}) \ar[r, "\partial"]
\ar[u, "j", "\simeq"']
&
\Sigma K\cX(\cY \cap \{W\} \cap \{W'\}) \ ,
\end{tikzcd}
\end{equation}
where $\pi$ is the projection from the absolute to the relative coarse cohomology, and $j$ and $\iota$ are induced by the canonical inclusions.
The morphism $j$ is an equivalence by excision and $\iota$ is an equivalence by coarse invariance and the assumption that $(W, W')$ is coarsely excisive.

Pasting two instances of these diagrams, we get the commutative diagram
\begin{equation}\label{gweiojgoewrgfewrfw}
  \begin{tikzcd}
K\cX(\R \otimes \tilde{M}_{h, \cY}) 
\ar[r, "-\delta^{MV}_{1}"]
\ar[d, "\pi_{\leq 1}"']
\ar[ddr, "-\partial^{MV}_{1}"]
&
\Sigma K\cX(\tilde{M}_{h, \cY}) 
\ar[dd, "{\iota_1}", "\simeq"'] 
\\
K\cX(\R \otimes \tilde{M}_{h, \cY}, \{(-\infty, 1]\} \otimes \tilde{M}_{h, \cY})
\ar[ddd, bend right=80, "{\hat{q}^{\mathrm{Rel}}}"'] 
\\
\qquad \qquad \qquad K\cX(\{[1, \infty)\} \otimes \tilde{M}_{h, \cY}, \{\{1\}\} \otimes \tilde{M}_{h, \cY})
\ar[r, "\partial"]
\ar[u, "\simeq"', "j_1"]
\ar[d, "{\hat{q}^{\mathrm{rel}}}"']
&
\Sigma K\cX(\{\{1\}\} \otimes \tilde{M}_{h, \cY})
\ar[d, "{\hat{q}^{0}}"']
\\
K\cX(\hat{M}^{\{\geq 1\}}_{  \cY}, \hat{M}^{\{\geq 1\}}_{  \cY} \cap \hat{M}^{\{\leq 1\}}_{  \cY})
\ar[r, "\partial"]
\ar[d, "\simeq", "\tilde{j}_1"']
&
\Sigma K\cX(\hat{M}^{\{\geq 1\}}_{  \cY} \cap \hat{M}^{\{\leq 1\}}_{  \cY})
\ar[dd, equal]
\\
K\cX(\hat{M}_{ \cY}, \hat{M}_{  \cY}^{\{\leq 1\}}) 
\\
K\cX(\hat{M}_{  \cY})
\ar[r, "-\tilde{\partial}^{MV}"]
\ar[u, "\tilde{\pi}_{\leq 1}"]
& 
\Sigma K\cX(\hat{M}^{\{\leq 1\}}_{  \cY} \cap \hat{M}_{  \cY}^{\{\geq 1\}})
\end{tikzcd} 
\end{equation}

Here $\delta_1^{MV}$ is the Mayer-Vietoris map for the coarsely excisive decomposition of $\R \otimes \tilde{M}_{h, \cY}$ into $ (-\infty, 1] \times \tilde{M}_{h}$ and $[1, \infty) \times \tilde{M}_{h})$, while $\tilde{\partial}^{MV}$ is the Mayer-Vietoris boundary for the pair of big families $(\hat{M}^{\{\leq 1\}} , \hat{M}^{\{\geq 1\}} )$ in $\hat{M} $.
As above, the maps $\hat q^{{?}}$ are all induced by $\hat q$.
The additional minus signs come from the fact that the order of the partitions is swapped in comparison to \eqref{gweprkopwefklwpoekplml}.
Consequently we have
\begin{equation}
\begin{aligned}
\label{ResultFirstDiagram}
-\hat{q}_*^{0} \iota_{1,*} \delta_{1,*}^{MV}\Ind\cX( \Sigma\TildeDirac_{h}, \text{on}~\R \times \tilde{M}_{h, \cY})
&\stackrel{\eqref{gweiojgoewrgfewrfw}}{=}
\partial_{*}\tilde{j}_{1,*}^{-1} \hat{q}^{\mathrm{Rel}}_* {\pi}_{\leq 1,*}\Ind\cX( \Sigma\TildeDirac_{h}, \text{on}~\R \times \tilde{M}_{h, \cY})
\\
&\stackrel{ \eqref{ApplRelativeIndex1}}{=} 
	\partial_{*}\tilde{j}_{1,*}^{-1} \tilde{\pi}_{\leq 1,*}\Ind\cX( \hat{\Dirac} , \text{on}~\hat{M}_{ \cY})
	\\
&\stackrel{\eqref{gweiojgoewrgfewrfw}}{=}
	-\tilde{\partial}^{MV}_{*} \Ind\cX( \hat{\Dirac} , \text{on}~\hat{M}_{ \cY}) \ .
\end{aligned} \end{equation} 
Similarly, we get the commutative diagram
\begin{equation}\label{gweiojgoewrgfewrfw1}
 \begin{tikzcd}
K\cX(\hat{M}_{  \cY}) 
\ar[r, "\tilde{\partial}^{MV}"]
\ar[d, "\tilde{\pi}_{\geq 0}"']
&
\Sigma K \cX(\hat{M}_{  \cY}^{\{\leq 0\}} \cap \hat{M}_{  \cY}^{{\{\geq 0\}}})
\ar[dd, equal]
\\
K\cX(\hat{M}_{  \cY}, \hat{M}_{  \cY}^{\{\geq 0\}})
\ar[ddd, bend right=80, shift right=3, "{\hat{q}^{\prime,\mathrm{Rel}}}"'] 
\\
K\cX(\hat{M}_{ \cY}^{\{\leq 0\}}, \hat{M}_{ \cY}^{\{\leq 0\}} \cap \hat{M}_{ \cY}^{\{\geq 0\}})
\ar[r, "\partial"]
\ar[u, "\tilde{j}_0", "\simeq"']
\ar[d, "{\hat{q}^{\prime,\mathrm{rel}}}"']
&
\Sigma K\cX (\hat{M}_{ \cY}^{\{\leq 0\}} \cap \hat{M}_{ \cY}^{\{\geq 0\}})
\ar[d, "{\hat{q}^{\prime,0}}"]
\\
\qquad K\cX(\{(-\infty, 0]\} \otimes \tilde{M}_{h', \cY}, \{\{0\}\} \otimes \tilde{M}_{h', \cY})
\ar[r, "\partial"]
\ar[d, "j_0'"', "\simeq"]
&
\Sigma K\cX(\{\{0\}\} \otimes \tilde{M}_{h', \cY})
\\
K\cX(\R \otimes \tilde{M}_{h', \cY}, \{[0, \infty)\} \otimes \tilde{M}_{h', \cY})
\\
K\cX(\R \otimes \tilde{M}_{h', \cY})
\ar[r, "\delta_0^{MV}"]
\ar[ruu, bend right=5, "\partial_0^{MV}"', near end]
\ar[u, "\pi_{\geq 0}'"]
& 
\Sigma K\cX(\tilde{M}_{h', \cY}) \ ,
\ar[uu, "{\iota'_{0}}"']
\end{tikzcd} 
\end{equation}
where $\delta_0^{MV}$ is the Mayer-Vietoris boundary for the coarsely excisive partition of $\R \otimes \tilde{M}_{h'}$ into $ (-\infty, 0] \times \tilde{M}_{h'}$ and $[0, \infty) \times\tilde{M}_{h'}$.
Notice also that $ \smash{\hat{M}_{  \cY}^{\{\leq 1\}} = \hat{M}_{  \cY}^{\{\leq 0\}}}$ and $\smash{\hat{M}_{  \cY}^{\{\geq 1\}} = \hat{M}_{  \cY}^{\{\geq 0\}}}$
so that the bottom arrow of \eqref{gweiojgoewrgfewrfw} coincides with to top arrow of \eqref{gweiojgoewrgfewrfw1} up to sign.
%up to sign so that we can paste the two diagrams.
 
 %Using \eqref{ApplRelativeIndex2} in the second step, 
 We  therefore get
\begin{equation}
\begin{aligned}
\label{ResultSecondDiagram}
\hat{q}^{\prime,0}_*\tilde{\partial}^{MV}_{*}\Ind\cX(\hat \Dirac , \mathrm{on}\,\hat{M}_{\tilde{h}, \cY})
%&=
%\partial_{*}\hat{q}^{\prime,\mathrm{rel}}_*\tilde{j}_{0,*}^{-1} \tilde{\pi}_{\geq 0,*}\Ind\cX(\hat \Dirac_{\tilde{h}}, \mathrm{on}\,\hat{M}%_{\tilde{h}, \cY})
&\stackrel{\eqref{gweiojgoewrgfewrfw1}}{=}
\partial_{*}j_{0,*}^{\prime,-1}\hat{q}^{\prime,\mathrm{Rel}}_* \tilde{\pi}_{\geq 0,*}\Ind\cX(\hat \Dirac , \mathrm{on}\,\hat{M}_{\tilde{h}, \cY})
\\
	&\stackrel{ \eqref{ApplRelativeIndex2}}{=}
	\partial_{*}j_{0,*}^{\prime,-1}{\pi}'_{\geq 0,*}\Ind\cX(\Sigma \TildeDirac_{h'}, \mathrm{on}\,\R \times \tilde{M}_{h', \cY})
	\\
	& \stackrel{\eqref{gweiojgoewrgfewrfw1}}{=}
	\iota'_{0,*} \delta_{0,*}^{MV}\Ind\cX(\Sigma \TildeDirac_{h'}, \mathrm{on}\,\R \times \tilde{M}_{h', \cY}) \ .
\end{aligned}
\end{equation} 

Notice that the diagram
\[
\begin{tikzcd}
\tilde{M}_{h}
\ar[rr, "q"]
\ar[d, "{(t,m) \mapsto (1,t,m)}"']
	& &
	\tilde{M}_{h'}
	\ar[d, "{(t,m)\mapsto (0,t,m)}"]
	\\
\R \otimes \tilde{M}_{h}
\ar[r, "\hat{q}"]
&
\hat{M} 
\ar[r, "\hat{q}'"]
&
\R \otimes \tilde{M}_{h'}
\end{tikzcd}
\]
does not strictly commute, but the two compositions are close to each other. 
 We conclude that
\begin{equation}
	\label{CloseDiagram}
	\iota'_{0,*} q_* = \hat{q}^{\prime,0}_* \hat{q}^{0}_* \iota_{1,*}:K\cX(\tilde M_{h,\cY})  
\to K\cX(\{\{0\}\}\otimes \tilde M_{h',\cY})\ .	
\end{equation}

We now calculate
\begin{eqnarray*}
q_* \Ind(\TildeDirac_h, \mathrm{on}\,\tilde{M}_{h, \cY})
	&\stackrel{\text{\cref{SuspensionTheorem}}}{=}&
	\beta \cdot
 q_* \delta^{MV}_{1,*} (\Ind \cX(\Sigma \TildeDirac_h, \mathrm{on}\,\R \times \tilde{M}_{h, \cY}))
	\\
	&\stackrel{\eqref{CloseDiagram}}{=}&
	\beta \cdot
 \iota^{\prime,-1}_{0,*}  \hat{q}^{\prime,0}_*\hat{q}^{0}_*\iota_{1,*} \delta_{1,*}^{MV} \Ind\cX(\Sigma \TildeDirac_h, \mathrm{on}\,\R \times \tilde{M}_{h, \cY})
	\\
&\stackrel{\eqref{ResultFirstDiagram}}{=}&
	\beta \cdot
 \iota^{\prime,-1}_{0,*} \hat{q}^{\prime,0}_*\tilde{\partial}^{MV}_{*} \Ind \cX( \hat{\Dirac}_{ }, \mathrm{on}\,\hat{M}_{  \cY})
	\\
&\stackrel{\eqref{ResultSecondDiagram}}{=}&
	\beta \cdot
\delta_{0,*}^{MV} \Ind \cX(\Sigma \TildeDirac_{h'}, \mathrm{on}\,\R \times \tilde{M}_{h', \cY})
\\
&\stackrel{\text{\cref{SuspensionTheorem}}}{=}&
\Ind \cX(\TildeDirac_{h'}, \mathrm{on}\,\tilde{M}_{h', \cY}) \ ,
\end{eqnarray*}
as desired.
\end{proof}

We now assume that $\lim_{t\to \infty} h(t)=\infty$. 
Then the identity of underlying sets is a map
\begin{equation}\label{fwerferwffrfwfwrfw}\iota:\tilde M\to \cO^{\infty}(M)
\end{equation}
in $\BC$. It sends $\tilde M_{\cY}$ to $\cO^{\infty}_{\cY}(X)$.

We assume that $\Dirac$ is positive away from $\cY$.

\begin{ddd}
\label{lphegeetrhetrgtgteg} 
We define the symbol of $\Dirac$  with support in $\cY$ as the class  
\[
\sigma_{\cY}(\Dirac ):=\iota_{*}( \beta \cdot \ind\cX(\TildeDirac,\mathrm{on}\,\tilde M_{\cY} ))\quad \mbox{in} \quad
K_{\cY,-n}^{\cX}(M)\cong   K\cX_{-n+1}(\cO_{\cY}^{\infty}(M))\ .\]
\end{ddd}

For $\cY=\{M\}$ we write $\sigma(\Dirac):=\sigma_{\{M\}} (\Dirac)$ in $K^{\cX}_{-n}(M)$.
By \cref{bjgbpdgdfbdg} the class $\iota_{*}(  \ind\cX(\smash{\TildeDirac},\mathrm{on}\,\tilde M_{\cY}))$ is independent of the choice of the warping function $h$.
 Consequently the symbol class  $\sigma_{\cY} (\Dirac )$ is well-defined.

Recall the Paschke morphism \eqref{frefwrfregwreg}.

\begin{ddd}\label{jiorgwpegwerfrefrfw}
We define $\sigma^{\an}(\Dirac):=p_{M}(\sigma(\Dirac))$ in $K^{\an}_{-n}(M)$.
\end{ddd}

\begin{rem}\label{jiogpgweergerfwef}
In this remark we explain how $\ind\cX(\Dirac)$ can be recovered from $\sigma^{\an}(\Dirac)$.
Let $g$ be the Riemannian metric on $M$. 
Then we consider a Riemannian metric $g'$ with $g\le g'$, which is automatically complete, and let $M'$ be the uniform bornological coarse space presented by the new metric.
Then the identity map of underlying sets $i:M'\to M$ is a contraction and therefore a morphism in $\UBC$. Let $\Dirac'$ be the adaptation of $\Dirac$ to the new   metric $g'$ (see \cref{RemarkAdaption}).
Using   an analogous argument as in \cref{bjgbpdgdfbdg} based on $\hat M=\R\times M$ with warped product metric metric $dt^{2}+\chi(t) g'+(1-\chi(t))g$ one can show that $i_{*}\ind\cX(\Dirac')=\ind\cX(\Dirac)$.  
We get a commutative diagram
$$
\xymatrix{K^{\cX}(M)\ar[d]^{a_{M}}\ar[r]^{p_{M}}&\ar@{..>}[dl]^{a^{\lf}_{M'}}K^{\an}(M)&K^{\cX}(M')\ar[l]_{p_{M'}} \ar@/_1cm/[ll]^{i}\ar[d]^{a_{M'}}\\K\cX(M)&&K\cX(M')\ar[ll]_{i_{*}}}\ , \xymatrix{\sigma(\Dirac)\ar@{|->}[d]^{a_{M}}\ar@{|->}[r]^{p_{M}}&\sigma^{\an}(\Dirac)&\sigma(\Dirac')\ar@{|->}[l]_{p_{M'}} \ar@{|->}@/_1cm/[ll]^{i}\ar@{|->}[d]^{a_{M'}}\\ \ind\cX(\Dirac)&& \ind\cX(\Dirac')\ar@{|->}[ll]_{i_{*}}}\ .
$$
We now choose the metric $g'$ adapted to a locally finite triangulation of $M$ so that
$p_{M'}$ is an equivalence and get the dotted arrow (see \cref{gkgoerpgwergrefwrefw}).
We then have $\ind\cX(\Dirac)=a^{\lf}_{M'}(\sigma^{\an}(M))$. \hB
 \end{rem}

Let $a_{M,\cY}:K_{\cY}^{\cX}(M)\to K\cX(\cY)$ be the  index  map \eqref{fwerqfdwedqewdqewdq}.
\begin{lem}\label{jogergreffwfrf}
 We have the equality
$$a_{M,\cY}(\sigma_{\cY}(\Dirac )) =  \ind\cX(\Dirac,\mathrm{on}\,\cY)\quad \mbox{in} \quad K\cX_{-n}(\cY)\ .$$
\end{lem}

\begin{proof}
  The first step in the construction of $a_{M,\cY}$ is the application of the cone boundary $\partial^{\cone}:\cO^{\infty}(M)\to \R\otimes M$ from \eqref{fewrfrwfrweffwrf}. 
  By \cref{bjgbpdgdfbdg}, it sends the class $\iota_{*}(\beta\cdot \ind\cX(\smash{\TildeDirac},\mathrm{on}\,\tilde \cY))$ to the class $\beta\cdot \ind\cX(\Sigma \Dirac,\mathrm{on}\,\tilde M_{\cY})$ in $K\cX_{-n+1}((\tilde M_{\cY})_{\R\otimes M})$, where the subscript inducates that $\tilde M_{\cY}$ has the bornological coarse structures induced from $\R\otimes M$ and $\iota$ is as in \eqref{fwerferwffrfwfwrfw}.
  
The second step is the application of the Mayer-Vietoris boundary
for the decomposition $(\{ \R^{-}\}\times \cY,\{\R^{+}\}\times M)$. 
 Using that
 $\Sigma \Dirac$ is actually positive away from
 $\R\times \cY$ and  the suspension theorem,  \cref{SuspensionTheorem},  the image of this class under the Mayer-Vietoris boundary is $\ind\cX(\Dirac,\mathrm{on}\,\cY)$ in $K\cX_{-n}(\cY)$.
  \end{proof}

 \begin{rem}
 Assume that $\Dirac$ is positive on $M$. Then we can take $\cY=\{\emptyset\}$. We get a class 
 $ \sigma_{\{\emptyset\}}(\Dirac)$ in $K^{\cX}_{\{\emptyset\},-n}(M)$ which captures the reason for
 $\ind\cX(\Dirac)=0$,  usually called the $\rho$-invariant, see \cref{kopgwergregfrfwf}. \hB
 \end{rem}

  We consider a  section
 $\Psi$ in $C^{\infty}(M, \End_{\Cl^{n}}(E)^{\odd})$.
 
 \begin{ddd}\label{8wrtlhkerhterhertrtgtrgetg}
 We say that $\Psi$ is a potential for $\Dirac$ if $\Psi$ takes values in selfadjoint endomorphisms and {the graded commutator} $[\Dirac,\Psi]$ is an operator of order zero.
 \end{ddd}

Note that $[\Dirac,\Psi]$ is a zero-order operator if $\Psi$
commutes in the graded sense with the symbol of $\Dirac$.

   If $\Psi$ is a potential {for $\Dirac$}, then  we can form the new Dirac type operator $\Dirac+\Psi$ of degree $n$.
  
  \begin{ddd}
  \label{kopththretg}
  We call $\Dirac+\Psi$ the Callias type operator associated to $\Dirac$ and $\Psi$.
  \end{ddd}

%  
%In addition we consider a graded $\Cl^{k}$-modules $V$ and a smooth  function
%  $\Psi:M\to \End^{\odd}(V)^{sa}$. We can then form the Callias type operator
%  $$C(\Dirac,\Psi) = \Dirac + \Psi :C^{\infty}(M,E\otimes V)\to C^{\infty}(M,E\otimes V) $$  which is a generalized Dirac operator of degree $n+k$. 
  
  We will consider the condition that  $\Dirac+\Psi$ is {positive} outside of a subset $Y$ of $M$. But we want the stronger condition  that this positivity is caused by $\Psi$ alone. To this end we 
 calculate \begin{equation}\label{fewfqwefqwedqewdewd}(\Dirac+\Psi)^{2}=\Dirac^{2} + [\Dirac,\Psi]+\Psi^{2 }\ .  
   \end{equation}
  
By $\sigma(A)$ we denote the spectrum of a selfadjoint operator $A$. So $\sigma(\Psi^{2}(m))$ below is the spectrum of the finite-dimensional operator $\Psi^{2}(m)$ on the fibre $E_{m}$ of the bundle $E$ at $m$ and $\min\:\sigma(\Psi^{2}(m))$ is its minimal element. 
\begin{ddd}\label{tlhpetrgggegrt}
We say that $\Psi$  is positive  away from $Y$ if
 \begin{equation}
 \label{fqqewdwedqfqqfefdqed} 
\inf_{m\in M\setminus Y} \Big({\min} \:\sigma(  \Psi^{2}(m))  - \|\nabla \Psi(m)\|  \Big) >0\ .
 \end{equation}  
If $\cY$ is a big family then we say that $\Psi$ is positive away from $\cY$ if it is positive away from a member of $\cY$.%\fml{In die Definition hinein bewegt.}
\end{ddd}

  \begin{lem}\label{gkopwergwerfwerf}
    If $\Psi$ is positive  away from $Y$, then $\Dirac +\Psi$ is positive away from $Y$ {in the sense of \cref{kopherthergrge}}.
    \end{lem}
    
    \begin{proof}
    For $\phi \in C^\infty_c(M \setminus \bar{Y}, E)$, we have
    \[
    \begin{aligned}
      	\|(\Dirac + \Psi)\phi\|^2 &= \|\Dirac\phi\|^2 + \langle \phi, ([\Dirac, \Psi] + \Psi^2) \phi\rangle \\
      	&\geq \inf_{m \in M \setminus Y} \Big(\min \sigma(\Psi^2(m)) - \|[\Dirac, \Psi]\|\Big) \|\phi\|^2\ .
    \end{aligned}
    \]
    As $\|[\Dirac, \Psi]\| \leq \|\nabla\Psi\|$ pointwise, the result follows from the assumption of \eqref{fqqewdwedqfqqfefdqed}.
  \end{proof}

{As a consequence of \cref{gkopwergwerfwerf},} if $\Psi$ is positive away from $Y$,   then  by \cref{gjkregopregwergfwrefrfwrf} 
 we get a class 
 %by $[(\Dirac+\Psi)_{\cY}]$
 %in $\gK_{0}(C^{*}(\cY\subseteq M,E))$ which we interpret as a class
 $$\ind\cX(\Dirac+\Psi,\mathrm{on}\,\cY) \quad \mbox{in}\quad  K\cX_{-n}(\cY)\ .$$ % using the identification $\kappa_{\cY}
%$ from  \eqref{qedqefq2t5}.

%  
%     We assume that 
% \begin{equation}\label{fqqewdwedqfqqfefdqed} \inf_{m\in M\setminus Y} \inf \:\sigma(|([\Dirac,\Psi]+\Psi^{2})(m)|-\|R_{0}(m)\|)>0\ ,\end{equation}  
%  where for $A$ in $\End(E_{m})$  the symbol $\sigma(A)$ denotes the spectrum of $A$.  Then in particular $\Dirac+\Psi$ is positive outside of $Y$
%  and
%  as explained in \cref{rijeogegergwergf}, and we  have an index class   
%  $$\ind \cX(\Dirac+\Psi, \mathrm{on}\,\cY)\:\:\mbox{in}\:\:K\cX_{n}(\cY)\ .$$
   %In the following we want to interpret the symbol of $\Dirac+\Psi$ 
    %as a $K$-theoretic product of $\sigma(\Dirac)$  with some class associated to $\Psi$.
    
    We abbreviate $\tilde \Psi:=\pr^{*}\Psi\otimes \id_{\Cl^{1}}$ in $C^{\infty}(\R\times M,\tilde E)$ and let $\smash{\TildeDirac}$ as above {be} the adaptation of $\Sigma \Dirac$ to the warped product metric on $\R\times M$ {(see \cref{kogpweerfwerfrwefw})}.
    Then  $\smash{\TildeDirac}+\tilde\Psi$ is the adaptation of $\Sigma(\Dirac+\Psi)$. 
    Therefore by  definition, with $\iota$ as in \eqref{fwerferwffrfwfwrfw}
    \[
    \iota_{*}(\beta \cdot \ind\cX(\TildeDirac+\tilde\Psi ))=\sigma(\Dirac+\Psi)\quad  \mbox{in} \quad K_{-n}^{\cX}(M)\ .
    \] 
    
     \begin{lem}
     \label{fofkpwedweqdewd}
      $\tilde \Psi$ is positive away from $\R\times Y$.
  \end{lem}
  
  \begin{proof}
 We note that 
 \begin{equation}
 \label{NablaPsiTilde}
 	\|  \nabla\tilde \Psi (t,m)\|^{2}= {h^{-2}(t)}\| \nabla\Psi(m)\|^{2} \ .
 \end{equation}
 Since $h^{-1}(t)\le 1$ for all $t$ in $\R$ we have 
 \[
 \min(\sigma( \tilde \Psi(t,m)^{2})-{\|\nabla\tilde \Psi (t,m)\|}
 \geq \min(\sigma ( \Psi(m)^{2})-{\|\nabla \Psi (m)\|}
 \]
 {for each $m $ in $M \setminus Y$ and each $t$ in $\R$.
 Taking the infimum over $\R \times (M \setminus Y)$, the result follows from the positivity of $\Psi$ away from $Y$.
 }
\end{proof}

  Assume that  $\Psi$ is positive away from $Y$. \begin{ddd}\label{kogpwergwerffwref}
  We define the symbol with support of the Callias type  operator $\Dirac+\Psi$  by 
  \[
  \sigma(\Dirac+\Psi,\mathrm{on}\,\cY):=\iota_{*}(\beta \cdot \ind\cX(\TildeDirac+\tilde\Psi ,\mathrm{on}\,\cY))\quad  \mbox{in} \quad  K^{\cX}_{-n}( \cY)\ .
  \]
  \end{ddd}
  
  The symbol is well-defined by \cref{fofkpwedweqdewd}, \cref{gkopwergwerfwerf} and \cref{gjkregopregwergfwrefrfwrf}.
  %, by \cref{gjkregopregwergfwrefrfwrf} 
  %we get a refinement of the symbol $\sigma(\Dirac+\Psi)$ to a class $$\sigma(\Dirac+\Psi,\mathrm{on}\,\cY):=\iota_{*}(\beta \ind\cX(\TildeDirac+\tilde\Psi ,\mathrm{on}\,\cY)$$ in $K^{\cX}_{-n}( \cY)$.
   \cref{jogergreffwfrf} implies:
   \begin{kor}\label{wkpgwergwrefrfw} We have the equality
  $$a_{\cY}(\sigma(\Dirac+\Psi,\mathrm{on}\,\cY))=\ind\cX(\Dirac+\Psi,\mathrm{on}\,\cY)\quad \mbox{in} \quad K\cX_{-n}(\cY)\ .$$
  \end{kor}
  
    \begin{rem}\label{jifioprgergwregwergf}
  Assume that $\Dirac$ is {positive} away from  a big family $\cZ$. Without further assumptions on $\Psi$ the construction of the Callias type operator $\Dirac+\Psi$ destroys this positivity
  since we do not have control over the term $[\Dirac,\Psi]$ on $(X\setminus Z) \cap Y $. 
  So we can not conclude that $\Dirac+\Psi$ is {positive} away from $Y\cap Z$.
      \hB
  \end{rem}

In order to improve on the point made in \cref{jifioprgergwregwergf} we adopt the following definition.
Let $\Psi$ be a potential and $\cZ$ be a big family on $M$. We consider the big family  \begin{equation}
\label{wrogijoiwergregfrwefwref}\tilde \cY_{\cZ}:=\{\R^{-}\}\times (\cY\cap \cZ)\cup \{\R^{+}\} \times \cY 
 = \tilde{M}_{\cZ} \cap (\R \times \cY)
\end{equation}  
in $\tilde M$.

%Let $\cZ=(Z_{i})_{i\in I}$.

 \begin{ddd}
 \label{owkpgwgerferrgehghd}
We say $\Psi$ is asymptotically constant away from $\cZ$ if
\[
\lim_{Z\in \cZ} \|\nabla \Psi_{|M\setminus Z}\|_{\infty}=0\ .
\]
\end{ddd}

  \begin{lem}
  \label{ogjoiergwefrefwfw9} 
  We assume that $\Dirac$ is positive away from $\cZ$, and that  $\Psi$ is positive away from $\cY$ and asymptotically constant away from $\cZ$. 
  Then $\smash{\TildeDirac +{\tilde{\Psi}}}$ is positive away from $\tilde \cY_{\cZ}$.
  \end{lem}
  
\begin{proof}
By \cref{gkopwergwerfwerf} and \cref{fofkpwedweqdewd} we know that $\smash{\TildeDirac +{\tilde{\Psi}}}$ is  positive away from $\R\times \cY$.
 We can therefore find a member $Y$ of $\cY$ and number $b$ in $(0,\infty)$ such that
\[
\| (\TildeDirac+\tilde \Psi)\tilde \phi\|^{2}\ge  b\|\tilde \phi\|^{2}
\]
 for any compactly supported section $\tilde \phi$ of $\pr^{*}E$ with support on $\R\times (M\setminus \bar Y)$.

 By \eqref{fewfqwefqwedqewdewd}, 
for any compactly supported smooth section $\tilde{\phi}$ of $\pr^{*}E$ we have
\begin{equation}\label{jgwioergfrefrw}
  \|(\TildeDirac + \tilde{\Psi})\tilde{\phi}\|^2 \geq \|\TildeDirac\tilde{\phi}\|^2 - \|\tilde{\nabla} \Psi_{|\supp(\tilde \phi)}\|_{\infty}\|\tilde{\phi}\|^2 + \|\tilde{\Psi}\tilde{\phi}\|^2\ .
\end{equation}
 Assume that $Z$ is a member of $\cZ$ and $a$ is a number in $(0,\infty)$ such that the quadratic form  associated to 
$\smash{\Dirac^{2}}$ is  bounded below by $\smash{a^{2}}$ on $ M\setminus Z$. 
Then by (the proof of) \cref{LemmaPositivityOnCone}, the quadratic form associated to $\smash{\TildeDirac{}^{2}}$ is bounded below by $a^{2}$ on $\R^{-}\times (M\setminus Z)$.
 We now  enlarge the member $Z$ of $\cZ$ such that  in addition $ \|\nabla \Psi_{|M\setminus Z}\|_{\infty}\le a^{2}/2$.  
Then by \eqref{NablaPsiTilde} and the fact that $h$ is bounded below by $1$ we have $\|\smash{\tilde \nabla \tilde \Psi_{|\R\times (M\setminus Z)}}\|_{\infty}\le a^{2}/2$.
Consequently, if $\tilde{\phi}$ is supported on $(-\infty,0) \times M \setminus \bar Z$, then the right-hand side of \eqref{jgwioergfrefrw} is larger than $a^2/2 \cdot \|\tilde{\phi}\|^2$.

Since any compactly supported section $\tilde \phi$ of $\pr^{*}E$ supported on $\R\times (M\setminus \bar Y) \cup (-\infty,0)\times (M\setminus \bar Z)$ can be written as a sum {of two sections}, {one} compactly supported  on $\R\times (M\setminus \bar Y)$ and one on $(-\infty,0)\times ( M\setminus \bar Z)$,  we see that for such sections 
$$\|(\TildeDirac+\tilde \Psi)\tilde \phi\|^{2}\ge \min (b,\frac{a^{2}}{2}) \|\tilde \phi\|^{2}\ .$$

This shows that $\smash{\TildeDirac +{\tilde{\Psi}}}$ is positive away from $\tilde \cY_{\cZ}$, because $\R\times (M\setminus \bar Y) \cup \R^{-}\times (M\setminus \bar Z)$ contains the complement of the member $\R^{-}\times (U_{1}[Y]\cap U_{1}[Z])\cup \R^{+}\times U_{1}[Y]$ of $\tilde{\cY}_\cZ$, where $U_{1}$ is the metric entourage of width $1$
\end{proof}

\begin{ddd} 
\label{ktogpwegrefwf} 
Under the assumptions of \cref{ogjoiergwefrefwfw9}, we define
\[
\sigma_{\cZ}(\Dirac+\Psi,\mathrm{on}\, \cY):=\iota_{*}(\beta\cdot \ind\cX(\TildeDirac+\tilde \Psi, \mbox{on {$\tilde{\cY}_\cZ$}
} ))
\]
in $ K^{\cX}_{\cZ,-n}(\cY)$.
\end{ddd}

\begin{lem} 
Under the assumptions of \cref{ogjoiergwefrefwfw9}
we have 
\begin{equation}
\label{qefewdad}
a_{\cY,\cZ}(\sigma_{\cZ}(\Dirac+\Psi,\mathrm{on}\, \cY))=\ind\cX(\Dirac+\Psi,\mathrm{on}\,\cY\cap \cZ)  \end{equation} 
in $K\cX_{-n}(\cY\cap \cZ)$.
\end{lem} 

\begin{proof}
We have a diagram
$$\xymatrix{\tilde M\ar[rd]^{\iota}\ar[rr]^{i}&&\R\times M\\&\cO^{\infty}(M)\ar[ur]^{\partial^{\cone}}&} \ . $$ in $\BC$ all induced by the identity maps of the underlying sets.
In view of the definition of $a_{\cY,\cZ}$ in terms of the cone boundary and the suspension equivalence
we get 
$$\kappa_{*}a_{\cY,\cZ}(\sigma_{\cZ}(\Dirac+\Psi,\mathrm{on}\, \cY)) = \partial^{MV} i_{*}(\beta\cdot \ind\cX(\TildeDirac+\tilde \Psi, \mbox{on {$\tilde{\cY}_\cZ$})}
)\ ,$$ where $\partial^{MV}$ is the Mayer-Vietoris boundary for the decomposition
of $\tilde{\cY}_\cZ$ into $\{\R^{-}\}\times (\cY\cap \cZ)$ and $\{\R^{+}\}\times \cY$
and $\kappa:\cY\cap \cZ = \{0\}\times (\cY\cap \cZ)\to  \{\R^{-}\}\times (\cY\cap \cZ)\cap\{\R^{+}\}\times \cY$
consist of coarse equivalences and induces an equivalence in coarse $K$-homology.
By \cref{bjgbpdgdfbdg}
we have
$$i_{*} \ind\cX(\TildeDirac+\tilde \Psi, \mbox{on {$\tilde{\cY}_\cZ$})}=\ind\cX(\Sigma(\Dirac+\Psi),\mathrm{on}\,\tilde{\cY}_\cZ)\ .$$
The right-hand side actually comes from an index supported on the smaller family $\R\times (\cY\cap \cZ)$.
  The equality \eqref{qefewdad} now follows from the suspension theorem \cref{SuspensionTheorem}, as the equivalence $\Sigma^{-1} K\cX(\R \otimes (\cY \cap \cZ)) \simeq K\cX(\cY \cap \cZ)$ is induced by the Mayer-Vietoris boundary.
\end{proof}

\subsection{Callias type operators and pairings}\label{ewrokgpwergwerffwef}

As in the previous section we consider a generalized Dirac operator $\Dirac$ of degree $n$ on a bundle  $E\to M$ which is positive away from a big family $\cZ$ (see \cref{kopherthergrge}) and a potential $\Psi$ in $C^{\infty}(M,\End_{\Cl^{n}}(E)^{\odd})$ which is very positive away from a second big family $\cY$ (see \cref{ekopthertheth} {below}) and asymptotically constant away from  $\cZ$ (see \cref{owkpgwgerferrgehghd}).    
In this section, we  express the coarse index $\ind\cX(\Dirac+\Psi,\mathrm{on}\,\cY\cap \cZ)$ in ${K}\cX(\cY\cap \cZ)$ and  the symbol $\sigma_{\cZ}(\Dirac+\Psi,\mathrm{on}\,\cY)$ in  ${K}^{\cX}_{\cZ}(\cY)$ of the Callias type operator $\Dirac+\Psi$ in terms of   the coarse corona pairing $\cap^{\cX}$ and the coarse  symbol pairing $\cap^{\cX\sigma}$ introduced in \cref{gojkwpergerwferfewrfwrefw} and  \cref{grjweopgergweffefwef}, respectively.
    
     We assume that $n=k+l$ for integers $k,l$ and $E\cong E_{0}\stimes V$.  Here $V$ is a finite-dimensional graded Hilbert space with a right action of $\Cl^{l}$ such that the generators of $\Cl^{l}$ act by anti-selfadjoint operators. Furthermore $E_{0}\to M$  is a  graded hermitean bundle
   of $\Cl^{k}$-modules carrying a Dirac operator $\Dirac_{0}$ of degree $k$. We assume that     $\Dirac=\Dirac_{0}\stimes \id_{V}$ and that $\Psi=\id_{E_{0}}\stimes \Psi_{0}$ for some map $\Psi_{0}:M\to \End_{\Cl^{l}}(V)^{\odd}$ with  selfadjoint  values.
   In particular, the graded Roe algebra  of $\Cl^{k+l}$-invariant operators associated to the bundle $E\to M$ is given by 
   \[
   C^*(M, E) \cong \End(V) \stimes C^*(M, E_0) \ ,
   \]
where $\End(V)$ denotes the graded algebra of $\Cl^{l}$-equivariant operators on $V$ and $C^{*}(M,E_{0})$ is the  graded Roe algebra of $\Cl^{k}$-equivariant operators on $E_{0}\to M$.
   We consider a big family $\cY$ on $M$. 
   Recall \cref{tlhpetrgggegrt} of the notion of  local positivity of $\Psi_{0}$.
   
   \begin{ddd}
   \label{ekopthertheth}
   We say that $\Psi_{0}$ is very positive away from $\cY$ if it is positive away from $\cY$ and  $ e^{-\Psi^{2}}$ belongs to $\End(V) \stimes C_{u}(\cY)$.
   \end{ddd}
   
  If $\Psi_{0}$ is very positive away from $\cY$, then it  generates a homomorphism 
  \[\Psi_{0,*}:\cS\to   \End(V) \stimes C_{u}(\cY)
  \] 
  {which represents a class
    $[\Psi_{0}]$ in $K^{\gr}_{0}(   \End(V) \stimes C_{u}(\cY) )\cong K^{l}(\cY)$, see   \cref{okfpqfqewddedq}}.
  
  Recall the coarse symbol pairing \cref{grjweopgergweffefwef}.
 Let $\cZ$ be a second big family on $M$.

  \begin{theorem}\label{wegokwpergerfwefrwefwf} 
  Assume: %\fml{Eigentlich muss doch $\cZ \subseteq \cY$ gelten, damit das ueberhaupt erfuellt sein kann, oder?}
  \begin{enumerate}
  \item   \label{ogpwgwergwefrefw1} $\Dirac_{0}$ is {positive} away from $\cZ$.
  \item  \label{ogpwgwergwefrefw}
 $\Psi_{0}$  is asymptotically constant away from $\cZ$. 
  \item \label{fwefwdewdqded} $\Psi_{0}$ is very positive away from $\cY$. 
  \item  \label{fwefwdewdqded3333} $\|\nabla \Psi_{0}\|$ is uniformly bounded on $M$. 
\end{enumerate}
Then we
 have the equality
  \begin{equation}\label{erwvfvsdfvsfs}    \sigma_{\cZ}(\Dirac+\Psi,\mathrm{on}\,\cY)= [\Psi_{0}]\cap^{\cX\sigma} \sigma_{\cZ}(\Dirac_{0})    \end{equation} in $K^{\cX}_{\cZ,-n}(\cY)$.
  \end{theorem}
  
  \begin{rem}
  Note that the right-hand side of \eqref{erwvfvsdfvsfs} is defined without assuming \ref{ogpwgwergwefrefw}.    But this assumption is necessary in order to get the left-hand side well-defined.
Indeed, the conditions imply that \cref{ktogpwegrefwf} applies and provides the class
  \[
  \sigma_{\cZ}(\Dirac+\Psi,\mathrm{on}\,\cY)\quad \mbox{in}\quad  K^{\cX}_{\cZ,-n}(\cY)\ .
  \]
  Both sides of \eqref{erwvfvsdfvsfs} do not require the technical condition \ref{fwefwdewdqded3333} which is only needed in the proof of this equality.
  \hB
  \end{rem}

  \begin{proof}
{
We apply \cref{kogpweerfwerfrwefw} to obtain the operators $\smash{\TildeDirac}$ and  $\smash{\TildeDirac + \tilde{\Psi}}$ on $\tilde{M}$.
Since $\Dirac$ is positive away from $\cZ$, \cref{LemmaPositivityOnCone} implies that $\smash{\TildeDirac}$ is positive away from the big family $\tilde{M}_{\cZ} \subset \tilde{M}$ as defined in \eqref{qfwefdqeq423}.
Hence by \cref{werkogperfwfwrf}, there exists $a$ in $(0, \infty)$  
such that
  \[
  \TildeDirac_{0,*}:C_{0}((-a,a))\to C^{*}( \tilde M_{\cZ}\subseteq \tilde M, \tilde{E}_0)
  \]
is well-defined.
Similarly, as by \cref{gkopwergwerfwerf}, $D + \Psi$ is positive away from $\cY$, the operator $\smash{\TildeDirac + \tilde{\Psi}}$ on $\tilde{M}$ is positive away from the big family $\tilde{\cY}_{\cZ}$ defined in \eqref{wrogijoiwergregfrwefwref}, and we get a well-defined $*$-homomorphism
  \[
  (\TildeDirac+\tilde \Psi)_{*}:C_{0}((-a,a))\to C^{*}( \tilde \cY_{\cZ}\subseteq \tilde M, \tilde{E}) \ .
  \]
}
  We also have 
\[
\Psi_{0,*}:C_{0}((-a,a))\to \End(V)\stimes C_{u}(\cY)\ .
\]

%   In order to interpret the right-hand side of \eqref{erwvfvsdfvsfs}  we use the equivalence 
%  $$ K^{\gr}_{0}(   \End(V) \stimes C_{u}(\cY) ) \stackrel{\eqref{werfwefwvsdsf}}{\cong} K_{l} ( C_{u}(\cY)) $$ 
%   and that the class
%   $q_{u, \cY}(\sigma(\Dirac_{0}))$ in $\Sigma^{k+1}\gEE(C_{u} (\cY),\bC(\cO^{\infty}(Y))$
%   induces a homomorphism morphism $K_{l}(    C_{u}(\cY) )\to    K_{k+l}^{\cX}(\cY)$ which we also denote by $q_{u, \cY}(\sigma(\Dirac_{0}))$.
%  
%
%
%  
%  We assume that
%     $\frac{1}{i+\Psi_{0}}$ belongs to $  \End(V) \stimes C_{u}(\cY)  $.
%    This implies that $\Psi_{0}$ is an unbounded multiplier of $  \End(V) \stimes C_{u}(\cY)  $ and generates a homomorphism $$\Psi_{0,*}:\cS\to   \End(V) \stimes C_{u}(\cY) \ .$$ The latter represents a class
%    $[\Psi_{0}]$ in $K^{\gr}_{0}(   \End(V) \stimes C_{u}(\cY) ) $.

%  From $\Psi$ want to derive  a class $[\Psi]$ in $\gEE(\C,  \End(V) \stimes C_{u}(\cY) )\simeq \Sigma^{n} \gEE(\C,C_{u}(\cY))$ represented by
%  $$\cS \to  \End(V) \stimes C_{u}(\cY) \ , \quad f\mapsto f(\Psi)\ .$$
%  For this we must require that
%  $\frac{1}{i+\Psi}  \in \End(V) \stimes C_{u}(\cY)  $.
%  Because of the warping of the metric and the uniform continuity of $\Psi$ the variation of $\pr^{*}\Psi$ on $\tilde M$ vanishes uniformly for $t\to \infty$. Consequently, 
%  $\tilde \Psi:=\pr^{*}\Psi$ generates a homomorphism 
% $$\cS\to  \End(V) \stimes C_{u}(\cY,\{(-\infty,0]\times M\})\ , \quad f\mapsto f( \tilde \Psi)\ .$$

Recall from \cref{gkpoegsergreg} that the coproduct $\Delta : \cS \to  \cS \stimes \cS$ is the $*$-homomorphism determined uniquely by
\[
  f \mapsto f(X \stimes 1 + 1 \stimes X)\ ,
\]
where $X$ is the unbounded multiplier of $\cS$ given by $(Xf)(x) = xf(x)$.
The multiplier $X$ restricts to a bounded multiplier on each of the subalgebras $C_0((-a, a))$. 
In particular we get a restriction 
\[
 \Delta_a : C_0((-a, a)) \to C_0((-a, a)) \stimes C_0((-a, a)) 
\] of the coproduct, {making the} 
 diagram
\begin{equation}\label{vojpofvsdfvsdfvs}
\begin{tikzcd}
C((-a,a))
\ar[r, "\epsilon"] 
\ar[d, "\Delta_{a}"'] 
&
\cS 
\ar[d, "\Delta"]
\\  
C((-a,a))  \stimes C((-a,a))  \ar[r, "\epsilon \otimes \epsilon"] 
&
\cS\stimes \cS 
\end{tikzcd} 
\end{equation}
commute,
where $\epsilon: C_0((-a, a)) \to \cS$ is the extension-by-zero map.

In view of the above and Remarks \ref{okfpqfqewddedq}, \ref{wgklopergwrefwef}, the composition 
   \[C_{0}((-a,a))\xrightarrow{\Delta_{a}}C_{0}((-a,a))  \stimes C_{0}((-a,a))  \xrightarrow{\Psi_{0,*} \stimes \TildeDirac_{0,*}  } \End(V)  \stimes C_{u}(\cY )\stimes  C^{*}( {\tilde{M}_{\cZ}}\subseteq \tilde M,\tilde E_{0})
   \] 
   represents the product  
   \[  [ \Psi_{0}]\cup [\TildeDirac_{0}] \quad   \mbox{in} \quad \gK_{0}( \End(V)\stimes C_{u}(\cY)  \stimes  C^{*}({\tilde{M}_{\cZ}}\subseteq \tilde M,\tilde E_{0}))  \ .
\]

In analogy to  \cref{iogjweoigwefewrfwerfwf} (with the same arguments and using the analogue of \cref{gokpergergsgre} for $A$ in $C^{*}(\tilde M,\tilde E)$),
{there is a well-defined pairing}
\begin{equation}
\label{qefddascdcwa}
\hat \mu:   \End(V)\stimes C_{u}(\cY) \stimes  C^{*}(\tilde M_{\cZ}\subseteq \tilde M,\tilde E_{0})
\to
  \frac{  C^{*}( \tilde \cY_{\cZ}\subseteq \tilde M,\tilde E)}{C^{*}(\{ \R^{-}\}\times (\cY\cap \cZ)\subseteq \tilde  M,\tilde E)}
\end{equation}
  defined {as the matrix extension of} $\hat \mu(f\stimes A):= [\chi(\pr^{*}f\stimes \id) (\id\stimes A)]$. 
  
  \begin{lem}
  \label{CrucialLemma}
  The diagram 
  \begin{equation}
\label{svffvrevdfvsdfv}
  \begin{tikzcd}[column sep=0.8cm]
  C((-a,a))
  \ar[r, "(\TildeDirac+\tilde \Psi)_{*}"] 
  \ar[d, "\Delta_{a}"']
  &
   C^{*}( \tilde \cY_{\cZ}\subseteq \tilde M, \tilde E)
  \ar[dd,"p_{0}"]
  \\
  C_{0}((-a,a))\stimes C((-a,a))
  \ar[d, "\Psi_{0, *} \stimes \TildeDirac_{0, *}"']
  \\
  \End(V) \stimes C_{u}(\cY) \stimes  C^{*}(\tilde M_{\cZ}\subseteq \tilde M,\tilde E_{0})
  \ar[r, "\hat \mu"]
  & 
  \displaystyle \frac{  C^{*}( \tilde \cY_{\cZ}\subseteq\tilde  M,\tilde E)}{C^{*}(\{ \R^{-}\}\times (\cY\cap \cZ)\subseteq \tilde  M,\tilde E)}
  \end{tikzcd}
\end{equation} 
commutes, where the right vertical  map $p_{0}$ is the quotient projection.
\end{lem}

\begin{proof}
We have a sequence of horizontal ideal inclusions and projections
\[
\begin{tikzcd}[column sep=0.6cm]
C^{*}(\tilde \cY_{\cZ}\subseteq \tilde M,\tilde E)
\ar[d,"p_{0}"]
\ar[r]
&
C^{*}(\tilde M,\tilde E)
\ar[d]
\ar[r]
&
\cM(C^{*}(\tilde M,\tilde E))
\ar[d, "p"]
\\
\displaystyle\frac{C^{*}(\tilde \cY_{\cZ}\subseteq \tilde M,\tilde E)}{C^{*}(\{ \R^{-}\}\times (\cY\cap \cZ)\subseteq \tilde  M,\tilde E)}
\ar[r, "!"]
&
\displaystyle\frac{C^{*}( \tilde M,\tilde E)}{C^{*}(\{ \R^{-}\}\times M\subseteq \tilde  M,\tilde E)}
\ar[r]
&
\displaystyle\frac{\cM(C^{*}( \tilde M,\tilde E))}{C^{*}(\{ \R^{-}\}\times M\subseteq \tilde  M,\tilde E)}	
\end{tikzcd}
\]
where {$\cM(\cdot)$ denotes taking the multiplier algebra}.
In order to see that the marked arrow is injective we  observe that
$$C^{*}(\{ \R^{-}\}\times (\cY\cap \cZ)\subseteq \tilde  M,\tilde E)=C^{*}(\tilde \cY_{\cZ}\subseteq \tilde M,\tilde E)\cap C^{*}(\{ \R^{-}\}\times M\subseteq \tilde  M,\tilde E)\ .$$

The diagram \eqref{svffvrevdfvsdfv} therefore maps injectively to the diagram 
	  \begin{equation}
\label{svffvrevdfvsdfv2}
  \begin{tikzcd}[column sep=0.8cm]
  \cS
  \ar[r, "(\TildeDirac+\tilde \Psi)_{*}"] 
  \ar[d, "\Delta"']
  &
   \cM(C^{*}(\tilde M, \tilde E))
  \ar[dd, "p"]
  \\
  \cS \stimes \cS
  \ar[d, "\Psi_{0, *} \stimes \TildeDirac_{0, *}"']
  \\
  \End(V) \stimes C_{u}(\cY) \stimes  C^{*}(\tilde M,\tilde E_{0})
  \ar[r, "\hat \mu"]
  & 
  \displaystyle \frac{  \cM(C^{*}(\tilde{M},\tilde E))}{C^{*}(\{ \R^{-}\}\times M\subseteq \tilde  M,\tilde E)} \ , 
  \end{tikzcd}
\end{equation} 
where $\hat{\mu}$ is the extension of \eqref{qefddascdcwa} defined by the same formula. 
We must check the commutativity of \eqref{svffvrevdfvsdfv2}. 	 
By continuity and a density argument  it suffices to check this commutativity on the elements $e^{-tx^{2}}$ and $xe^{-tx^{2}}$ of $\cS$ for all $t$ in $(0,\infty)$.
The coproducts of these elements are given by
\begin{equation}\label{sdfvsferwdfvsv}\Delta(e^{-tx^{2}})=e^{-tx^{2}}\stimes e^{-tx^{2}}\ , \quad  \Delta(xe^{-tx^{2}})=xe^{-tx^{2}}\stimes e^{-tx^{2}}+e^{-tx^{2}}\stimes xe^{-tx^{2}}\ .
\end{equation}

 The first relation in \eqref{sdfvsferwdfvsv} implies the first equality in
\[
\hat \mu((\Psi_{0,*}\stimes \TildeDirac_{0,*})(\Delta(e^{-tx^{2}})))= \hat \mu (e^{-t\Psi_{0}^{2}} \stimes e^{-t\TildeDirac{}_{0}^{2}}) = p(e^{-t\tilde \Psi^{2}})p(e^{-t\TildeDirac{}^{2}} ) \ ,
\]
and  the second is a reformuation of the definition $\hat \mu$ in terms of $p$.

By the  usual finite propagation speed argument (use  \cref{werkogperfwfwrf} with $\cY=\{\tilde M\}$) we have $e^{-t\TildeDirac{}^{2}}\in C^{*}(\tilde M,\tilde E)$.
In contrast, the operator $\smash{e^{-t \tilde{\Psi}^2}}$ is not contained in $C^*(\tilde{M}, \tilde{E})$, but it is  fortunately a multiplier for this algebra. 
For this reason we passed to multiplier algebras above.

The maps $t \mapsto e^{-tx^2}$ and $t \mapsto x e^{-t x^2}$ both are smooth functions $(0, \infty) \to \cS$.
Consequently,  postcomposing  these functions with any $*$-homomorphism $\cS \to A$ yields smooth functions $(0, \infty) \to A$.
In particular, 
\[
t\mapsto \tilde{\Psi}_{*}(e^{-tx^{2}})=\smash{e^{-t \tilde{\Psi}{}^2}} \qquad \text{and} \qquad t\mapsto \TildeDirac_{*}(e^{-tx^{2}})=\smash{e^{-t \TildeDirac{}^2}}
\]
 are smooth functions with values in the right upper corner of \eqref{svffvrevdfvsdfv2}. We will denote their derivatives suggestively by
$-\smash{\tilde{\Psi}{}^{2}e^{-t \tilde{\Psi}{}^2}}$ and $-\smash{\TildeDirac{}^{2}e^{-t \TildeDirac{}^2}}$, respectively.

Since $\Psi_0$ is very positive away from $\cY$ by Assumption \ref{fwefwdewdqded}, we have $e^{-t \Psi_0^2} \in \End(V) \stimes C_u(\cY)$. 
Using only boundedness and uniform continuity, the argument given in the proof of \cref{iogjweoigwefewrfwerfwf} shows that 
\[
\smash{e^{-t \tilde{\Psi}_0^2}} = \pr^* e^{-t \Psi_0^2} \in \End(V) \stimes \smash{\ell^{\infty}_{\{\R^{-}\}\times M}{(\tilde M)}}\ .
\]
By \cref{gokpergergsgre}, for all $s,t$ in $(0,\infty)$ the commutator of $\smash{e^{-s\tilde \Psi^{2}}}$ and $\smash{e^{-t\TildeDirac{}^{2}}}$ is contained in $C^*(\{\R^-\} \times M \subseteq \tilde{M}, \tilde{E})$.  
Consequently, $\smash{p(e^{-t\tilde \Psi^{2}})}$ and $\smash{p(e^{-s\TildeDirac{}^{2}})}$ commute, and so do their derivatives.
For example, $\smash{p(e^{-s\TildeDirac{}^{2}})}$ commutes with $\smash{p(\tilde \Psi^{2}e^{-t\tilde \Psi^{2}})}$.

We must show the equality of semigroups 
\begin{equation}
\label{mgrmglkrtgerl}
  p(e^{-t\tilde \Psi^{2}}) p(e^{-t\TildeDirac{}^{2}} )= p(e^{-t(\TildeDirac+\tilde \Psi)^{2}})\ .
\end{equation}

To this end we consider the smooth function
\[
 (0,\infty)\times (0,\infty)\ni (s,t) \mapsto   p(e^{-s\tilde \Psi^{2}}) p(e^{-s\TildeDirac{}^{2}} ) p(e^{-t(\TildeDirac+\tilde \Psi)^{2}})\ 
\in \frac{\cM(C^*(\tilde{M}, \tilde{E}))}{C^*(\{\R^- \}\times M \subseteq \tilde{M}, \tilde{E})} .
\]
  We calculate
\begin{eqnarray}
\lefteqn{-\frac{\partial}{\partial s} p(e^{-s\tilde \Psi^{2}}) p(e^{-s\TildeDirac{}^{2}})p( e^{-t(\TildeDirac+\tilde \Psi)^{2}})}&&\label{eokgpgwergwefwref}\\&=&
p(  \tilde \Psi^{2} e^{-s\tilde \Psi^{2}} ) p(e^{-s\TildeDirac{}^{2}} ) p(e^{-t(\TildeDirac+\tilde \Psi)^{2}})+
p( e^{-s\tilde \Psi^{2}} )p( \TildeDirac{}^{2} e^{-s\TildeDirac{}^{2}} ) p( e^{-t(\TildeDirac+\tilde \Psi)^{2}})\nonumber\\
&=&  p(e^{-s\TildeDirac{}^{2}} ) p( \tilde \Psi^{2} e^{-s\tilde \Psi^{2}}  )p(e^{-t(\TildeDirac+\tilde \Psi)^{2}})+
p( e^{-s\tilde \Psi^{2}} )p(  \TildeDirac{}^{2} e^{-s\TildeDirac{}^{2}} ) p( e^{-t(\TildeDirac+\tilde \Psi)^{2}})
\nonumber\\
&=&  p(e^{-s\TildeDirac{}^{2}} ) p( \tilde \Psi^{2} e^{-s\tilde \Psi^{2}}  \cdot e^{-t(\TildeDirac+\tilde \Psi)^{2}})+
p( e^{-s\tilde \Psi^{2}} )p(  \TildeDirac{}^{2} e^{-s\TildeDirac{}^{2}} \cdot e^{-t(\TildeDirac+\tilde \Psi)^{2}})
\nonumber\\
&\stackrel{!}{=}&  p(e^{-s\TildeDirac{}^{2}} ) p( e^{-s\tilde \Psi^{2}}  \cdot  \tilde \Psi^{2}  e^{-t(\TildeDirac+\tilde \Psi)^{2}})+
p( e^{-s\tilde \Psi^{2}} )p(   e^{-s\TildeDirac{}^{2}} \cdot  \TildeDirac{}^{2} e^{-t(\TildeDirac+\tilde \Psi)^{2}})
\nonumber\\
&=&p(e^{-s\TildeDirac{}^{2}} ) p( e^{-s\tilde \Psi^{2}} ) p(\tilde \Psi^{2}e^{-t(\TildeDirac+\tilde \Psi)^{2}})+
p( e^{-s\tilde \Psi^{2}} )p(  e^{-s\TildeDirac{}^{2}}) p(\TildeDirac{}^{2} e^{-t(\TildeDirac+\tilde \Psi)^{2}})\nonumber\\
&=& p( e^{-s\tilde \Psi^{2}} ) p(e^{-s\TildeDirac{}^{2}} ) p((\tilde \Psi^{2}+\TildeDirac{}^{2} )e^{-t(\TildeDirac+\tilde \Psi)^{2}}) \nonumber\\&=& -\frac{\partial}{\partial t} p(e^{-s\tilde \Psi^{2}}) p(e^{-s\TildeDirac{}^{2}})p( e^{-t(\TildeDirac+\tilde \Psi)^{2}})\nonumber
\end{eqnarray}
{In the marked equality, we used the identities
\begin{equation}
\label{IdentitiesToJustify}
\begin{aligned}
\tilde \Psi^{2} e^{-s\tilde \Psi^{2}}  \cdot e^{-t(\TildeDirac+\tilde \Psi)^{2}} &= e^{-s\tilde \Psi^{2}}  \cdot  \tilde \Psi^{2}  e^{-t(\TildeDirac+\tilde \Psi)^{2}}, 
\\
\TildeDirac{}^{2} e^{-s\TildeDirac{}^{2}} \cdot e^{-t(\TildeDirac+\tilde \Psi)^{2}}
&= 
e^{-s\TildeDirac{}^{2}} \cdot  \TildeDirac{}^{2} e^{-t(\TildeDirac+\tilde \Psi)^{2}}
\end{aligned}
\end{equation}
which require justification.
Recall here that we use $ \tilde \Psi^{2} e^{-s\tilde \Psi^{2}}$ as a symbol for the operator ${\smash{\tilde{\Psi}_*(x^2 e^{-sx^2}) = -\partial_{s}e^{-s\tilde \Psi^{2}}}}$, and it is not a priori clear that the right hand sides of \eqref{IdentitiesToJustify} are even defined.}
Let $\dom(\tilde \Psi^{2})$ denote the domain of the unbounded selfadjoint operator determined by $\smash{\tilde \Psi^{2}}$  equipped with the 
graph norm.
Then $\smash{e^{-s\tilde \Psi^{2}}:\dom(\tilde \Psi^{2})\to H}$ is a norm-differentiable family of operators with (negative) derivative
\[
  \dom(\tilde \Psi^{2})\xrightarrow{\tilde \Psi^{2}} H \xrightarrow{e^{-s\tilde \Psi^{2}}} H\ .
  \]
   This composition happens to have a continuous extension to a bounded operator on $H$ which we denoted by  $\tilde \Psi^{2} e^{-s\tilde \Psi^{2}}$ above.
We apply \cref{ijgowergwerfwefwef} {below} with $H=L^{2}(\tilde M, \tilde E)$, $\smash{A=t^{1/2}\TildeDirac}$, $\smash{B=t^{1/2}\tilde \Psi}$, $D=C_{c}^{\infty}(\tilde M,\tilde E)$. 
Note that $\smash{\{\TildeDirac,\tilde \Psi\}}$ is bounded by  Assumption \ref{fwefwdewdqded3333}.
Since $\smash{e^{-t(\TildeDirac+\tilde \Psi)^{2}}:H\to \dom((\TildeDirac+\tilde \Psi)^{2})}$ is continuous, 
   \cref{ijgowergwerfwefwef}.\ref{kotrhperthertgertgetg}
 implies that  $\smash{e^{-t(\TildeDirac+\tilde \Psi)^{2}}}$  corestricts to a bounded map from $H$ to $\smash{\dom(\tilde \Psi^{2})}$. 
Hence we can write
$\smash{\tilde \Psi^{2} e^{-s\tilde \Psi^{2}}  \cdot e^{-t(\TildeDirac+\tilde \Psi)^{2}}}$ as the composition of bounded operators
$$ H \xrightarrow{e^{-t(\TildeDirac+\tilde \Psi)^{2}}} \dom(\tilde \Psi^{2})\xrightarrow{\tilde \Psi^{2}} H \xrightarrow{e^{-s\tilde \Psi^{2}}} H\ .$$
This finishes the justification for the first equality in \eqref{IdentitiesToJustify} . 
For the second equality, we argue similarly using the claim that $e^{-t(\TildeDirac+\tilde \Psi)^{2}}$  restricts to a bounded map from $H$ to $\dom(\TildeDirac{}^{2})$.

By \eqref{eokgpgwergwefwref}, for every $t$ in $(0,\infty)$ the function 
\[
 (0,1)\ni u\mapsto p(e^{-tu\tilde \Psi^{2}}) p(e^{-tu\TildeDirac{}^{2}})p( e^{-(1-u)t(\TildeDirac+\tilde \Psi)^{2}})
 \] 
is constant. 
This together with
the following  two {norm} limits % \nml{in ${\cM(C^{*}(\tilde M,\tilde E))}$}
  \begin{equation}
  \label{vsfdvsdfvsdvvs}  \lim_{u\downarrow 0}   e^{-tu\TildeDirac{}^{2}} e^{-tu\tilde \Psi^{2}} e^{-(1-u)t(\TildeDirac+\tilde \Psi)^{2}}=
e^{-t(\TildeDirac+\tilde \Psi)^{2}}\end{equation} 
and  \begin{equation}\label{vsfdvsdfvsdvvs1}\lim_{u\uparrow 1} e^{-tu\TildeDirac{}^{2}}e^{-tu\tilde  \Psi^{2}  } e^{-(1-u)t(\TildeDirac+\tilde \Psi)^{2}} =e^{-t\tilde \Psi^{2}}e^{-t\TildeDirac{}^{2}}
\end{equation}   
 implies \eqref{mgrmglkrtgerl}.

We now justify the limits \eqref{vsfdvsdfvsdvvs} and \eqref{vsfdvsdfvsdvvs1}.
 {For the first,} we use the decomposition
 $$ e^{-tu\TildeDirac{}^{2}} e^{-tu\tilde \Psi^{2}}e^{-(1-u)t(\TildeDirac+\tilde \Psi)^{2}} = e^{-tu\TildeDirac{}^{2}} e^{-tu\tilde \Psi^{2}} e^{-t\frac{1}{2}(\TildeDirac+\tilde \Psi)^{2}} \cdot e^{-t(\frac{1}{2}-u)(\TildeDirac+\tilde \Psi)^{2}}  \ .$$
 The second {factor} {on the right hand side} converges to $e^{-t\frac{1}{2}(\TildeDirac+\tilde \Psi)^{2}} $ and we must show that
$$\lim_{u\downarrow 0}  e^{-tu\TildeDirac{}^{2}}e^{-tu\tilde \Psi^{2}}  e^{-t\frac{1}{2}(\TildeDirac+\tilde \Psi)^{2}} =
e^{-t\frac{1}{2}(\TildeDirac+\tilde \Psi)^{2}} \ .$$
We use the decomposition
\begin{equation}
\label{fqwefewdwdwedq}  
H
\xrightarrow{e^{-t\frac{1}{2}(\TildeDirac+\tilde \Psi)^{2}}}   
\dom(\TildeDirac{}^{2})\cap \dom(\tilde \Psi^{2}) 
\xrightarrow{e^{-tu\tilde \Psi^{2}}}
\dom(\TildeDirac)  
\xrightarrow{e^{-tu\TildeDirac{}^{2}}} 
H
\ .
 \end{equation} 
The first map is a bounded operator by  \cref{ijgowergwerfwefwef}.\ref{kotrhperthertgertgetg} (applied to 
 $\smash{A=t^{1/2}\TildeDirac}$, $\smash{B=t^{1/2}\tilde \Psi}$).
The last map is a continuous family of bounded operators which tends to the inclusion $\smash{\dom(\TildeDirac)}\to H$ as $u\downarrow 0$. 
For the  middle map  in \eqref{fqwefewdwdwedq}  we apply \cref{kpwergerwgw9}  with  $\smash{A=\TildeDirac}$ and $B=\tilde \Psi$, $H=L^{2}(\tilde M,\tilde E)$ and $D=C_{c}^{\infty}(\tilde M,\tilde E)$. The assumptions
on the functions of $B$ are easily verified. 
 \cref{kpwergerwgw9}.\ref{erwjigowergwefwref2} states that the {middle} map is a continuous family of bounded operators which tends to the inclusion if $u\downarrow 0$. This finishes the verification of \eqref{vsfdvsdfvsdvvs}.

For \eqref{vsfdvsdfvsdvvs1}  we consider the adjoint and use the decomposition
\begin{equation}\label{fwefqwedqwedqwdq}  H\xrightarrow{e^{-tu \TildeDirac^{2}} } \dom(\TildeDirac)\xrightarrow{e^{-t u \tilde \Psi^{2} }} 
\dom(\TildeDirac)\cap \dom(\tilde \Psi)
\stackrel{\text{\cref{ijgowergwerfwefwef}}.\ref{kotrhperthertgertgetg}}{=}\dom(\TildeDirac+\tilde \Psi) \xrightarrow{e^{-(1-u)t(\TildeDirac+\tilde \Psi)^{2}}} H \ .\end{equation} 
  As $u\uparrow 1$ the first map converges in norm to 
  $e^{-t \TildeDirac^{2}}: H \to  \dom(\TildeDirac)$.  
  The second map is continuous by  \cref{kpwergerwgw9}.\ref{erwjigowergwefwref1} and converges to
  $\smash{e^{-t\tilde \Psi^{2}}:  \dom(\TildeDirac)\to \dom(\TildeDirac)\cap \dom(\tilde \Psi)}$ as $u\uparrow 1$.
  Finally the last map converges in norm to
the inclusion $\dom(\TildeDirac+\tilde \Psi)\to H$.
This gives  \eqref{vsfdvsdfvsdvvs1}.

For the odd part we have  
\[
\hat \mu((\tilde \Psi_{0,*}\stimes \TildeDirac_{0,*})(\Delta(xe^{-tx^{2}})))=p(\tilde \Psi e^{-t\tilde \Psi^{2}})p(e^{-t\TildeDirac{}^{2}} )+ p( e^{-t\tilde \Psi^{2}})p(\TildeDirac e^{-t\TildeDirac{}^{2}} )
\]
and therefore must show that
\begin{equation}\label{ihfioefvsfdv}p(\tilde \Psi e^{-t\tilde \Psi^{2}})p(e^{-t\TildeDirac{}^{2}} )+  p(\TildeDirac e^{-t\TildeDirac{}^{2}} )p( e^{-t\tilde \Psi^{2}}) =p((\TildeDirac+\tilde \Psi)e^{-t(\TildeDirac+\tilde \Psi)^{2}} )\ .
\end{equation} 
We show 
\begin{equation}\label{gwergewfwref}p(\tilde \Psi e^{-t\tilde \Psi^{2}})p(e^{-t\TildeDirac{}^{2}})=p(\tilde \Psi e^{-t(\TildeDirac +\tilde \Psi)^{2}})\ ,
\end{equation}  where the right-hand side is defined by  \cref{ijgowergwerfwefwef}.\ref{kotrhperthertgertgetg0}. 
{For $\epsilon$ in $(0,\infty)$, we observe using} the even case that
 \begin{eqnarray*}
p(e^{-\epsilon \tilde \Psi^{2}} \tilde \Psi e^{-t\tilde \Psi^{2}})p(e^{-t\TildeDirac^{2}})&=&
p(\tilde \Psi e^{-\epsilon \tilde \Psi^{2}} ) p(e^{-t\tilde \Psi^{2}})p(e^{-t\TildeDirac^{2}})\\&=&
p(\tilde \Psi e^{-\epsilon \tilde \Psi^{2}})  p(e^{-t(\TildeDirac +\tilde \Psi)^{2}})\\&=&
p( e^{-\epsilon \tilde \Psi^{2}}\tilde \Psi e^{-t(\TildeDirac +\tilde \Psi)^{2}})  \end{eqnarray*}
Since  $\tilde \Psi e^{-t(\TildeDirac +\tilde \Psi)^{2}}$ and $ \tilde \Psi e^{-t\tilde \Psi^{2}}$ corestrict to bounded maps from $H$ to $\dom(\tilde \Psi)$ (we use   \cref{ijgowergwerfwefwef}.\ref{kotrhperthertgertgetg0}
 for the first case)  we can take the limit $\epsilon\downarrow 0$ and get \eqref{gwergewfwref}. 
 {The idenity} 
\begin{equation}\label{gwergewfwref1}p(\TildeDirac e^{-t\TildeDirac^{2}})p(e^{-t\tilde \Psi^{2}})=p(\TildeDirac e^{-t(\TildeDirac +\tilde \Psi)^{2}})\ 
\end{equation}
{can be shown} by an analogous argument.
\end{proof}

\begin{kor}
\label{CorollaryCupHom}
We have the equality 
\[
\hat  \mu_*([\Psi_0] \cup [\smash{\TildeDirac_0}]) = p_{0,*}([\TildeDirac + \tilde{\Psi}])\quad \mbox{in}\quad K^{\gr}_{0}\left( \frac{  C^{*}( \tilde \cY_{\cZ}\subseteq \tilde M,\tilde E)}{C^{*}(\{ \R^{-}\}\times (\cY\cap \cZ)\subseteq \tilde  M,\tilde E)}\right)\ .\]
%the quotient of the class represented by the odd unbounded multiplier $\smash{\TildeDirac + \tilde{\Psi}}$.
\end{kor}
  
  Recall that  on the one hand $\sigma_{\cZ}(\Dirac+\Psi,\mathrm{on}\,\cY)$ is obtained   in \cref{gkopwerfefrefwfref} from the class represented by $(\smash{\TildeDirac}+\tilde \Psi)_{*}$ on $\tilde \cY_{\cZ}$ in $K^{\gr}(C^{*}(\tilde \cY_{\cZ}\subseteq \tilde M,E))$ via an identification of   the $K$-theory  of this Roe algebra with the $K$-theory of the $C^{*}$-category  going into the definition of the the coarse $K$-homology functor  $K(\bC(\tilde \cY_{\cZ}\subseteq \tilde M))\simeq K\cX(\tilde \cY_{\cZ})$.   On the other hand the coarse symbol pairing $-\cap^{\cX\sigma}-$ is defined in \cref{grjweopgergweffefwef} in terms of 
 the latter  $C^{*}$-categories directly. The remaining  work in this proof  of \cref{wegokwpergerfwefrwefwf} consists of  transporting the construction of  $-\cap^{\cX\sigma}-$
  through these identifications until it can directly be compared with the class represented by $\hat \mu(\Psi_{0,*}\stimes  \smash{\TildeDirac}_{0,*}) \Delta_{a}   $.

  The formula \eqref{ewfewfdqedqedq} for $\mu$  also defines  pairings
  
  \begin{equation*}
\label{qefddascdcwa33}
{
\hat \mu_{0}:   C_{u}(\cY) \stimes  C^{*}(\tilde M_{\cZ}\subseteq \tilde M,\tilde E_{0})
\to
  \frac{  C^{*}( \tilde \cY_{\cZ}\subseteq \tilde M,\tilde E_{0})}{C^{*}(\{ \R^{-}\}\times (\cY\cap \cZ)\subseteq \tilde  M,\tilde E_{0})}
}\ ,
\end{equation*}
  \[
  \tilde \mu':  C_{u}(\cY ) \otimes  \tilde \bC[n](\tilde M_{\cZ}\subseteq\tilde M)
  \to 
  \frac{ \tilde \bC[n]( \tilde \cY_{\cZ}\subseteq \tilde M) }{\tilde \bC[n](\{ \R^{-}\}\times (\cY\cap \cZ)\subseteq \tilde M) }
  \]
  and  
  \[ \mu':    C_{u}(\cY) \otimes   \bC[n](\tilde M_{\cZ}\subseteq \tilde M)
  \to 
  \frac{   \bC[n](\tilde  \cY_{\cZ}\subseteq \tilde  M) }{  \bC[n](\{\R^{-}\}\times (\cY\cap \cZ)\subseteq \tilde  M) }\ ,
  \]
  where the notation $\bC[n](-)$ and $\tilde \bC[n](-)$ {was} introduced in \cref{rijeogegergwergf}.

Recall the functor $\cI$ from \eqref{gergwerfwrfweferfw}.

   \begin{lem}
 The following square commutes:
 \[
 \begin{tikzcd}
 C_{u}(\cY) \otimes  \tilde \bC[n](\tilde M_{\cZ}\subseteq \tilde M)
 \ar[r, "\tilde \mu'"]
 \ar[d, "1 \otimes \cI"']
 & \displaystyle
 \frac{ \tilde \bC[n]( \tilde \cY_{\cZ}\subseteq \tilde  M) }{\tilde \bC[n](\{ \R^{-}\}\times (\cY\cap \cZ)\subseteq  \tilde M) }\ar[d, "\cI"]  
 \\   
  C_{u}(\cY) \stimes  C^{*}(\tilde M_{\cZ}\subseteq \tilde M,\tilde E_{0})
  \ar[r, "\hat \mu_{0}"] 
  &  \displaystyle
  \frac{  C^{*}( \tilde \cY_{\cZ}\subseteq  \tilde  M,\tilde E_0)}{C^{*}(\{ \R^{-}\}\times (\cY\cap \cZ)\subseteq  \tilde  M,\tilde E_0)} \ .
  \end{tikzcd} 
  \]
\end{lem}

  \begin{proof}
  Let $A:(H,\chi,U)\to (H',\chi',U')$ be a morphism in $ \tilde \bC[n](\tilde M_{\cZ}\subseteq \tilde M)$ and $f$ be in ${C_{u}(\cY)}$. 
    Let $\tilde \chi$ denote the projection-valued measure on $L^{2}(\tilde M,\tilde E_{0})$.
  Then we have 
  \[
  \begin{aligned}
  (1 \otimes \cI)\circ \tilde{\mu}'(f \otimes A) &=  \cI([\chi'(\pr^*f) A])
  = [U'\chi'(\pr^*f) A U^*] 
  \\
  \hat \mu_{0} \circ (1 \otimes \cI) (f \otimes A) &= \hat  \mu_{0} (f \otimes U'AU^*)
  = [\tilde{\chi}(\pr^*f) U'A U^*] 
  \end{aligned}
  \]
%Then $\ml{\cI}(\chi'(\pr^{*}f) A)=U'\chi'(\pr^{*}f) AU^{*}$ goes into $\tilde \mu'$  and $\tilde \chi(\pr^{*}f )U'AU^{*}$ appears in the definition of $\mu$.
The difference {of the two representatives} is given by 
\[U'\chi'(\pr^{*}f) AU^{*}-\tilde \chi( \pr^{*}f) {U'}AU^{*}=(U'\chi'(\pr^{*}f)-\tilde \chi(\pr^{*}f)U')AU^{*}\ .
\]
Since $f$ is {uniformly continuous} we have $\pr^{*}f\in \ell^{\infty}_{\{\R^{-}\}\times M}(\tilde M)$ {(see \cref{RemarkUniformOnCone}).}
Since $U'$ is controlled for 
the {coarse} structure on $\smash{\tilde M}$, \cref{kohpkertphokgpertge} implies that the above difference is in $C^{*}(\{\R^{-}\}\times (\cY\cap \cZ) \subseteq\tilde  M,\tilde E)$.
%we have $$\lim_{n\to \infty} \|\tilde \chi([n,\infty)\times M)(U'\chi'(\pr^{*}f)U^{\prime,*}-\tilde \chi(\pr^{*}f))   \tilde \chi([n,\infty)\times M)\|=0\ .$$ This implies  that $U'\chi'(\pr^{*}f) AU^{*}- \tilde\chi(\pr^{*}f) U'AU^{*}\in C^{*}(\{\R^{-}\}\times (\cY\cap \cZ) \subseteq\tilde  M,\tilde E)$.  
\end{proof}

 The morphism $\cF$ from  \eqref{vrewjvoievjiosvdfvdsfvs} and the map $\iota$ from \eqref{fwerferwffrfwfwrfw} induce the respective vertical arrows in the following diagram
\[
\begin{tikzcd} 
 C_{u}(\cY) \stimes  \tilde \bC[n](\tilde M_{\cZ}\subseteq \tilde M)
 \ar[r, "{\tilde{\mu}'}"]
 \ar[d, "\id \stimes \cF"] 
 & \displaystyle
 \frac{ \tilde \bC[n]( \tilde \cY_{\cZ}\subseteq \tilde  M) }{\tilde \bC[n](\{ \R^{-}\}\times (\cY\cap \cZ)\subseteq\tilde M) }
 \ar[d, "\cF"]  
 \\   
 C_{u}(\cY) \stimes   \bC[n](\tilde M_{\cZ}\subseteq \tilde M)
 \ar[r, "{\mu'}"] 
 \ar[d,"\id\stimes \iota"]
 & \displaystyle
 \frac{  \bC[n]( \tilde \cY_{\cZ}\subseteq \tilde M)}{\bC[n](\{ \R^{-}\}\times (\cY\cap \cZ)\subseteq \tilde  M)}
 \ar[d,"\iota"] 
 \\
  C_{u}(\cY) \stimes   \bC[n](\cO_{\cZ}^{\infty}(\cY)\subseteq \cO^{\infty}(M))
  \ar[r, "\mu"]  
  & \displaystyle
  \frac{  \bC[n](\cO_{\cZ}^{\infty}(\cY)\subseteq \cO^{\infty}(M))}{\bC[n](\cO^{-}(\cY\cap \cZ)\subseteq \cO^{\infty}(M))}
  \end{tikzcd}
  \]
  which obviously commutes.
  We tensor with $\End(V)$, apply $\Sigma^{-l}\gK$ and immediately  use the equivalence $\Sigma^{-l}\gK(\End(V)\stimes -)\simeq \gK(-)$ in order to save space. 
  We get a commutative diagram
   \[ 
   \small
  \begin{tikzcd}[column sep=-2cm] 
  \gK(C_{u}(\cY)){\times} \gK(\bC[n](\cO^{\infty}_{\cZ}(M)))
  \ar[dr,"\hat \boxtimes"]
  & 
  \\
  &
  \gK(C_{u}(\cY)\stimes \bC[n]( \cO^{\infty}_{\cZ}(M))
  \ar[dr, "{\mu}", near end]
  \\
    \gK(C_{u}(\cY)){\times} \gK(\tilde \bC[n](\tilde M_{\cZ}\subseteq \tilde M))
  \ar[dr] 
  \ar[dd, "\simeq", "\id \times \cI"']
  \ar[uu, "\id \times (\iota\circ \cF)"] 
  &
  &
  \gK( \frac{\bC[n](\cO_{\cZ}^{\infty}(\cY)\subseteq \cO^{\infty}({M}))}{\bC[n](\cO^{-}(\cY\cap \cZ)\subseteq \cO^{\infty}({M}))})   
  \\
  &
   \gK(C_{u}(\cY)\stimes  \tilde \bC[n](\tilde M_{\cZ}\subseteq \tilde M))
   \ar[dd, "\id \stimes \cI"', "\simeq"]
   \ar[dr, "{\tilde \mu'}", near end]
   \ar[uu, "\id\stimes \iota\circ \cF"]
   &
   \\
  \gK(C_{u}(\cY)){\times} \gK( C^{*}(\tilde M_{\cZ}\subseteq \tilde M,\tilde E_{0}))
  \ar[dr]
   &
   &
   \gK(\frac{ \tilde \bC[n]( \tilde \cY_{\cZ}\subseteq \tilde M) }{\tilde \bC[n](\{ \R^{-}\}\times (\cY\cap \cZ)\subseteq \tilde  M) })
   \ar[dd, "\simeq"', "\cI"]
   \ar[uu, "\iota\circ \cF"']
  \\ 
  & 
  {\Sigma^{-l}} \gK(\End(V)\stimes C_{u}(\cY)\stimes   C^{*}(\tilde M_{\cZ}\subseteq \tilde M,\tilde E_{0}))
  \ar[dr, "{\hat \mu}", near end]
  &
  \\
  &
  &   
  {\Sigma^{-l}}\gK(\frac{  C^{*}(\tilde \cY_{\cZ}\subseteq \tilde M,\tilde E)}{C^{*}(\{ \R^{-}\}\times (\cY\cap \cZ)\subseteq  \tilde  M,\tilde E)})
  \ .
  \end{tikzcd}
  \]
  We consider the class $([\Psi_{0}],[\smash{\TildeDirac_{0}}])$ in the lower left corner.
  Its image under the up--right composition is by definition $\mu_{*}([\Psi_0]\hat \boxtimes \beta^{-1}  \sigma_{\cZ}(\Dirac_0))$ while its image under the right-up composition is $\iota_{*} \cF_{*}\hat \mu_{*}([\Psi_{0}]\cup [\smash{\TildeDirac_{0}}])$. 
  This gives the first equality in the following display, while \cref{CorollaryCupHom}
   implies the second:
  \[
  \beta^{-1}\mu([\Psi_0] \hat \boxtimes  \sigma_{\cZ}(\Dirac_0))=\iota_{*} \cF_{*}\cI^{-1}_{*}\hat \mu_{*} ([\Psi_{0}]\cup [\TildeDirac_{0}])
  = \iota_*  \cF_* \cI^{-1}_{*} ([\TildeDirac + \tilde{\Psi}]) \ .
  \]
 In view of the \cref{grjweopgergweffefwef} of the coarse  symbol pairing  and   \cref{gjkregopregwergfwrefrfwrf}  of the coarse index,  \cref{lphegeetrhetrgtgteg} of the symbol,   and  \eqref{eqfewdqcdvadvad} we can conclude  that
  $$ [\Psi_{0}]\cap^{\cX\sigma} \sigma_{\cZ}(\Dirac_{0})  = \beta \iota_{*}\ind\cX(\TildeDirac+\tilde \Psi,\mathrm{on}\,\tilde \cY_{\cZ})=\sigma_{\cZ}(\Dirac+\Psi, \mathrm{on}\, \cY)$$
  as asserted in \eqref{erwvfvsdfvsfs}.  \end{proof}

\subsection{Technical Results}

%\fml{Habe dies zu einer Subsection gemacht. Eine einzige subsubsection im paper fand ich doch unschoen.}

{In this section, we provide some technical results that were used in the proof of the crucial \cref{CrucialLemma}.}

Let $H$ be a Hilbert space and $A,B$ be two unbounded selfadjoint
operators on $H$. 

%\fml{Habe das Assumptions Environment verwendet.}

\begin{ass}
\label{AssumptionsAB}
We assume that there is a linear subspace $D$  of  $\dom(A)\cap \dom(B)$ {preserved by both $A$ and $B$} such that  $A^{k}_{|D},B^{k}_{|D}$ and $(A_{|D} +B_{|D})^{k}$  are  essentially selfadjoint for $k=1,2$.  
We let $A+B$ denote the unique selfadjoint extension of  $A_{|D}+B_{|D}$.
\end{ass}
  
If $C$ is a selfadjoint operator  on $H$ with $D\subseteq \dom(C)$ such that $C_{|D}$ is essentially selfadjoint,  then
$\dom(C)$ is the closure of $D$ with respect to the graph norm $\|x\|_{C}:=\|Cx\|+\|x\|$. 
For $k,l$ in $\nat$  we equip $\dom(A^{k})\cap \dom(B^{l})$ with the norm 
\[
\|x\|_{A^{k},B^{l}}:=\|x\|_{A^{k}}+\|x\|_{B^{l}} \ .
\]

\begin{lem}  \label{ijgowergwerfwefwef}
{Let $A$ and $B$ satisfy \cref{AssumptionsAB} and assume in addition}
% \nml{We assume that $A$ and $B$ preserve $D$, and}
 that  
  $$\{A,B\}:=AB+BA:D\to H$$
  extends to a bounded operator {on $H$}.
  \begin{enumerate} 
  \item \label{kotrhperthertgertgetg0} We have an isomorphism 
  $\dom(A+B) =\dom(A)\cap \dom(B)$.
  \item \label{kotrhperthertgertgetg} We have an isomorphism 
  $\dom((A+B)^{2}) =\dom(A^{2})\cap \dom(B^{2})$.
  \item\label{kotrhperthertgertgetg1} The operator $A:D\to D$ extends to  a bounded operator
  $$A:\dom(A^{2})\cap \dom(B^{2}) \to  \dom(A)\cap \dom(B)\ .$$
  \end{enumerate}
  \end{lem}
\begin{proof}
For two norms $\|-\|, \|-\|'$ on the same domain we write $ \|x\|\lesssim \|x\|'$
for the statement that there exists a constant $C$ in $(0,\infty)$ such that $  \|x\|\le C\|x\|'$ for all $x$ in the domain.

By the triangle inequality  we have \begin{equation}\label{heporigpertge}\|- \|_{A+B} \lesssim \|-\|_{A,B} 
\end{equation}on $D$. 
{Using that} the restrictions of 
$A$, $B$ and $A+B$ are essentially selfadjoint on $D$  we can further conclude that $$\dom(A)\cap \dom(B)\subseteq \dom(A+B)\ .$$
This is a non-trivial consequence of our assumptions and does not follow directly
from \eqref{heporigpertge} which would only imply that the closure of $D$ in the norm $\|-\|_{A,B}$
is contained in $\dom(A+B)$. 
Namely, if $\phi$ is in $\dom(A)\cap \dom(B)$, then by the symmetry of $A$ and $B$
the linear functionals $D\ni \psi \mapsto \langle A\psi,\phi\rangle$ and 
 $D\ni \psi \mapsto \langle B\psi,\phi\rangle$ are bounded. Then also 
 $D\ni \psi \mapsto \langle (A+B)\psi,\phi\rangle$ is bounded which implies $\phi\in \dom(A+B)^{*}$
 and hence  $\phi\in \dom(A+B)$ since $(A+B)_{|D}$ is essentially selfadjoint.

{
For the reverse inclusion
\[
 \dom(A+B)\subseteq \dom(A)\cap \dom(B)\ ,
 \]
it suffices to show
\begin{equation}
\label{fqwedqdqwed}
\|-\|_{A,B}  \lesssim \|- \|_{A+B} \ , \end{equation} 
on $D$, as $\dom(A+B)$ is the closure of $D$ with respect to $\|-\|_{A+B}$. }
We calculate for $x$ in $D$
\begin{eqnarray*}
\|(A+B)x\|^{2}&=&\langle (A+B)x,(A+B)x \rangle\\&=&\langle  x,(A+B)^{2}x \rangle\\
&=&\langle  x,(A^{2}+B^{2}+ \{A,B\})x \rangle\\&\ge&
\|Ax\|^{2}+\|Bx\|^{2} -\|\{A,B\}\| \|x\|^{2}
\end{eqnarray*}
which implies \eqref{fqwedqdqwed}.
This finishes the argument for Assertion \ref{kotrhperthertgertgetg0}.

The argument for  Assertion \ref{kotrhperthertgertgetg} is a more complicated version of the argument above.
It is clear from $(A+B)^{2}=A^{2}+B^{2}+\{A,B\}$ and the boundedness of $\{A,B\}$  that
\begin{equation}\label{gwergwrfwrefw}\|x\|_{(A+B)^{2}}\lesssim \|x\|_{A^{2},B^{2}}
\end{equation} on $D$.
 {Using that} the restrictions of $A^{2},B^{2}$ and $(A+B)^{2}$ to $D$  are essentially selfadjoint,
we can further conclude
$$\dom(A^{2})\cap \dom(B^{2})\subseteq  \dom((A+B)^{2})\ .$$ 
%This is a non-trivial consequence of our assumptions and does not follow directly
%from \eqref{gwergwrfwrefw} which would only imply that the closure of $D$ in the norm $\|-\|_{A^{2},B^{2}}$
%is contained in $\dom((A+B)^{2})$.
%and
%$$\|x\|_{(A+B)^{2}}\lesssim \|x\|_{A^{2},B^{2}}$$ on $D$.
%It is clear from $(A+B)^{2}=A^{2}+B^{2}+\{A,B\}$ and the boundedness of $\{A,B\}$ that
% $$\|x\|_{(A+B)^{2}}\lesssim \|x\|_{A^{2},B^{2}}$$ on $D$.
{To get the converse inclusion,} it suffices to show the reverse estimate \begin{equation}\label{gwerpojopfwefwerfwerf}  \|x \|_{A^{2},B^{2}}\lesssim \|x\|_{(A+B)^{2}}  
\end{equation} 
on $D$. % as this would imply the reverse inclusion
%$$ \dom((A+B)^{2})\subseteq \dom(A^{2})\cap \dom(B^{2})\ ,$$
%since $\dom((A+B)^{2})$ is the closure of $D$ with respect to $\|-\|_{(A+B)^{2}}$. 
 For $x$ in $D$
we have
\[
\begin{aligned}
  \|(A+B)^{2}x\|^{2}
 &= \langle x, (A+B)^4 x\rangle
 \\
 &= \langle x, (A^2 + B^2 + \{A, B\})^2 x\rangle
 \\
 &= \langle x, (A^4 + B^4 + A^2B^2 + B^2 A^2 + R) x\rangle
 \\
 &= \|A^2 x\|^2 + \|B^2x\|^2 + \langle x, (A^2B^2 + B^2 A^2 ) x\rangle + \langle x, R x\rangle \ ,
\end{aligned}
\]
where
\[
R := A^2 \{A, B\} +  \{A, B\} A^2 + B^2 \{A, B\} +  \{A, B\} B^2 +\{A, B\}^2 \ .
\]
%The first term of $\langle x, Rx\rangle$ may be estimated as
%\[
%\begin{aligned}
%|\langle x, A^2 \{A, B\} x\rangle| 
%&= |\langle  A^2 x, \{A, B\} x\rangle| 
%\\
%&\leq \|\{A, B\}\|\cdot \|A^2 x\|\|x\|
%\\
%&\leq  \frac{\|\{A, B\}\|^2}{4\varepsilon} \|x\|^2 + \varepsilon \|A^2x\|^2 \ ,
%\end{aligned}
%\]
%for any $\varepsilon >0$, 
%and the next three may by estimated similarly.
%In total, we get
%\[
%|\langle x, Rx\rangle| \leq 2\varepsilon \|A^2 x\|^2 + 2\varepsilon \|B^2 x\|^2 + \left(1 + \frac{1}{\varepsilon}\right)\|\{A, B\}\|^2 \|x\|^2 \ .
%\] 
%
Because of
\[
(AB)^*AB + (BA)^*BA  = A^2B^2 + B^2 A^2 + \{A, B\}^2 - A\{A, B\} B - B\{A, B\} A \ ,
\]
we get
\[
\langle x, (A^2B^2 + B^2 A^2 ) x\rangle \geq -\langle x, (\{A, B\}^2 - A\{A, B\} B - B\{A, B\} A) x\rangle \ .
\]
Therefore 
 $$ \|(A+B)^{2}x\|^{2}\ge \|A^2 x\|^2 + \|B^2x\|^2 - \langle x, (\{A, B\}^2 - A\{A, B\} B - B\{A, B\} A) x\rangle +  \langle x, R x\rangle\ .$$
 Using the Cauchy-Schwarz inequality, the boundedness of $\{A,B\}$ and estimates of the form
 $|uv|\le \epsilon u^{2} +\frac{4}{\epsilon} v^{2}$ for every $\epsilon$ in $(0,\infty)$ several times
 we obtain   the desired estimate \eqref{gwerpojopfwefwerfwerf}.
% 
% In view of the symmetry under exchangeing $A$ and $B$ it suffices to show that
% $e^{-(A+B)^{2}}$ on $H$ corestricts to a bounded operator $$e^{-(A+B)^{2}}:H\to  \dom(B^{k})$$
%A vector $x$ in $H$ belongs to $\dom(B^{k})$ if and only if
%$ \| B^{k}e^{-\epsilon B^{2}}x\| $ is uniformly bounded for small $\epsilon$.   
%It therefore suffices to show that   the family of operators
% that $B^{k} e^{-\epsilon B^{2}} \cdot e^{-(A+B)^{2} }$ is uniformly bounded in $\epsilon$. 
%The equality $(A+B)^{2}=A^{2}+B^{2}+\{A,B\}$ on $D$ and the assumption
%that $\|\{B,A\}\|\le C$  implies 
%  the  inequality
%$$B^{2}\le  (A+B)^{2}+C \ ,$$
%of quadratic forms on $D$ which in turn implies the inequality of bounded selfadjoint operators
%$$e^{-2 (A+B)^{2}}  \le e^{-2B^{2}} e^{2C} $$ on $H$
%and therefore  
%$$B^{k}e^{-\epsilon   B^{2}}\cdot e^{-2 (A+B)^{2}} \cdot B^{k}e^{-\epsilon B^{2}} \le  e^{C}B^{k} e^{-\epsilon  B^{2}}\cdot e^{-   2B^{2}}  \cdot B^{k}e^{-\epsilon B^{2}}\ .$$
%We conclude, using the $C^{*}$-identity in the first step, that
%$$\|B^{k}e^{-\epsilon B^{2}}    \cdot e^{-  (A+B)^{2}}\|^{2}=\|B^{k}e^{-\epsilon B^{2}}\cdot e^{-2 (A+B)^{2}} \cdot B^{k}e^{-\epsilon B^{2}} \|\le 
%e^{2C} \| B^{k}e^{-\epsilon B^{2}}  \cdot   e^{-2B^{2}}  \cdot  B^{k}e^{-\epsilon B^{2}} \|\le 2k e^{2C-4k} \ .$$
%The right-hand side does not depend on $\epsilon$ which gives the desired estimate.
This finishes the verification of Assertion \ref{kotrhperthertgertgetg}.

For  Assertion \ref{kotrhperthertgertgetg1} we use the non-trivial consequence of the proof of
Assertion \ref{kotrhperthertgertgetg} that
$\dom(A^{2})\cap \dom(B^{2})$ is the closure of $D$ with respect to the norm $\|-\|_{A^{2},B^{2}}$.
It is clear that $A$ restricts to a bounded operator  $A:\dom(A^{2})\cap \dom(B^{2}) \to  \dom(A)$.
It suffices to show that   \begin{equation}\label{fdbojosdvsdfvsfdv}\|BAx\| \lesssim  \|x\|_{A^{2},B^{2}}\end{equation}
on $D$.
We calculate
 \begin{eqnarray*}
\langle BA x,BA x\rangle &=&-\langle BA x,AB x\rangle + \langle BA x,\{A,B\} x\rangle\\&=&
-\langle ABA x,B x\rangle + \langle BA x,\{A,B\} x\rangle\\&=&
\langle BA^{2} x,B x\rangle -\langle \{A,B\}A x,B x\rangle + \langle BA x,\{A,B\} x\rangle\\&=&
\langle A^{2} x,B^{2} x\rangle -\langle \{A,B\}A x,B x\rangle + \langle BA x,\{A,B\} x\rangle \ .
\end{eqnarray*}
We get the estimate
$$\|BAx\|^{2}\le\|A^{2}x\|\|B^{2}x\| +\|\{A,B\}\| \|Ax\|\|Bx\|+\|BAx\|\|\{A,B\}\|\|x\|$$
which implies \eqref{fdbojosdvsdfvsfdv}.
\end{proof}

Note that the assumptions for \cref{ijgowergwerfwefwef} are invariant under exchanging  $A$ and $B$. 
In the following lemma we break this symmetry.

  \begin{lem}\label{kpwergerwgw9} 
  {We keep \cref{AssumptionsAB} and}  assume in addition that   
  \[
  [0,\infty)\ni \epsilon\mapsto e^{-\epsilon B^{2}}x
  \]
   is a differentiable $\dom(A)$-valued function for all $x$ in $D$.
\begin{enumerate}  
\item \label{erwjigowergwefwref} 
The family of  commutators
 $([A,e^{-\epsilon B^{2}}]:D\to H)_{\epsilon\in (0,\infty)}$ extends to a norm-continuous family $$[0,\infty)\ni \epsilon \mapsto [A,e^{-\epsilon B^{2}}]\in B(H)$$  vanishing at $\epsilon=0$.   
 \item  \label{erwjigowergwefwref1} The family of bounded operators 
  $(e^{-\epsilon B^{2}})_{\epsilon\in (0,\infty)}$ on $H$  restricts to a norm continuous family of bounded operators
  $$(0,\infty)\ni \epsilon\mapsto e^{-\epsilon B^{2}}:\dom(A)\to \dom(A)\cap \dom(B)\ .$$
%We have  $\lim_{\epsilon\downarrow 0}[A,e^{-\epsilon B^{2}}]=0$ as bounded operators on $H$.
  \item \label{erwjigowergwefwref2}
 The family of bounded operators 
  $(e^{-\epsilon B^{2}})_{\epsilon\in [0,\infty)}$ on $H$ restricts to a continuous family of  bounded  operators $$[0,\infty)\ni \epsilon\mapsto e^{-\epsilon B^{2}} :\dom(A^{2})\cap \dom(B^{2})\to \dom(A)\ .$$
 % \item  \label{orkhpwegrgfwreffwer}The family of bounded operators 
  %$(e^{-\epsilon B^{2}})_{\epsilon\in [0,\infty)}$ on $H$ restricts to a continuous family of  bounded  operators $$[0,\infty)\ni \epsilon\mapsto e^{-\epsilon B^{2}} :\dom(A)\cap \dom(B^{2})\to \dom(|A|^{1/2})\ .$$
   \end{enumerate}
     \end{lem}
     
  \begin{proof}
 On $D$ we have the relation
  $[A,B^{2}]=(\{A,B\}B-B\{A,B\})$ and therefore 
  $$e^{-r\epsilon B^{2}} [A,B^{2}] e^{-(1-r)\epsilon B^{2}} = e^{-r\epsilon B^{2}}(\{A,B\}B-B\{A,B\})e^{-(1-r)\epsilon B^{2}}\ .$$  
 Decomposing the right-hand side as 
 $$ e^{-r\epsilon B^{2}}\cdot \{A,B\}\cdot Be^{-(1-r)\epsilon B^{2}} -  e^{-r\epsilon B^{2}}   B\cdot \{A,B\} \cdot e^{-(1-r)\epsilon B^{2}}$$ we see that  it extends to  a norm-continuous family
  \[
  (0,\infty)\times (0,1)\ni (\epsilon,r)\mapsto \epsilon\: e^{-r\epsilon B^{2}} [A,B^{2}] e^{-(1-r)\epsilon B^{2}} \in B(H)\ .
  \]      
  We claim the integral of this family over $r$ in $[0,1]$ exists and defines a norm-continuous family  
  $$ [0,\infty) \ni \epsilon \mapsto  \int_{0}^{1}   \epsilon \  e^{-r\epsilon B^{2}} [A,B^{2}] e^{-(1-r)\epsilon B^{2}} dr \in B(H)$$
   vanishing at $\epsilon=0$.
    In order to estimate the norm of the intregrand
  we  split the interval of integration into the halfs $[0,1/2]$ and $[1/2,1]$. We further fix $c$ in $(0,\infty)$.
  We see that  there exists a constant $C$ in $(0,\infty)$ such that for all $r$ in $[0,1/2]$ and $\epsilon$ in $(0,c)$      $$ \epsilon \| e^{-r\epsilon  B^{2}} \{A,B\} B e^{-(1-r)\epsilon  B^{2}}\|\le  \epsilon^{1/2} \| e^{-r\epsilon  B^{2}}\|  \| \{A,B\}\| \| \epsilon^{1/2}  B e^{-(1-r)\epsilon  B^{2}}\|  \le   C\epsilon^{1/2}  $$ 
 and 
 $$ \epsilon \| e^{-r\epsilon  B^{2}} B\{A,B\}  e^{-(1-r)\epsilon  B^{2}}\|\le \epsilon^{1/2} r^{-1/2}  \| \epsilon^{1/2}r^{1/2}B e^{-r\epsilon  B^{2}}\|  \| \{A,B\}\| \| e^{-(1-r)\epsilon  B^{2}}\| \le C \epsilon^{1/2} r^{-1/2}\ .$$
   These estimates imply the claim for the integral over $[0,1/2]$. The other half of the integral is discussed similarly.

  For all $x,y$ in $D$ and $\epsilon$ in $(0,\infty)$ we have by our assumption that
  $$-\partial_{r} \langle y, [A,e^{-r \epsilon B^{2}}] x\rangle= \langle y,\epsilon\:   e^{-rB^{2}} [A,B^{2}] e^{-(1-r)\epsilon B^{2}}x\rangle\ .$$
Integrating and using that $D$ is dense in $H$ we get  the equality
  $$[A,e^{- \epsilon B^{2}}] =-\epsilon  \int_{0}^{1} e^{-r\epsilon B^{2}} [A,B^{2}] e^{-(1-r)\epsilon B^{2}} dr$$
 on $D$.  By continuous extension we get a norm-continuous family
  $$[0,\infty)\ni \epsilon \to   [A,e^{-\epsilon B^{2}}] \in B(H)$$  
  vanishing at $\epsilon=0$. This finishes the proof of Assertion \ref{erwjigowergwefwref}.

We now show Assertion \ref{erwjigowergwefwref1}.
We know that $$(0,\infty)\ni \epsilon \mapsto e^{-\epsilon B^{2}}:H\to \dom(B)$$ is a norm-continuous family of bounded operators. We must show that 
$$(0,\infty)\ni \epsilon \mapsto e^{-\epsilon B^{2}}:\dom(A)\to \dom(A)$$  is a norm-continuous family of bounded operators. To this end we must 
see that
$$(0,\infty)\ni \epsilon \mapsto A  e^{-\epsilon B^{2}}:\dom(A)\to H$$ is a norm-continuous family of bounded operators.
This follows from
\begin{equation}\label{qwfewffdq}A  e^{-\epsilon B^{2}}=[A,  e^{-\epsilon B^{2}}]+e^{-\epsilon B^{2}}A\ ,
\end{equation}  
%the assumption that $e^{-\epsilon B^{2}}$ sends $D$ to $\dom(A)$, 
and Assertion \ref{erwjigowergwefwref}.

We now show Assertion \ref{erwjigowergwefwref2}.
We know that $$[0,\infty)\ni \epsilon \mapsto e^{-\epsilon B^{2}}:\dom(B^{2})\to H$$ is a continuous family of bounded operators.
We must show that $$[0,\infty)\ni \epsilon \mapsto  A e^{-\epsilon B^{2}}:\dom(A^{2})\cap \dom(B^{2})\to H$$ is a continuous family of bounded operators.
This follows from \eqref{qwfewffdq} and the decomposition
$$\dom(A^{2})\cap \dom(B^{2})\stackrel{A}{\to} \dom(B) \stackrel{e^{-\epsilon B^{2}}}{\to} H\ ,$$
where the first map is continuous by \cref{ijgowergwerfwefwef}.\ref{kotrhperthertgertgetg1}
 \end{proof}

\subsection{The Dirac-goes-to-Dirac principle}\label{kgpwergrewfwwrfwref}

We consider a complete Riemannian manifold $M$ with a uniformly continuous and controlled smooth function $f:M\to \R$ such that $0$ is a regular value.  Then $N:=f^{-1}(\{0\})$ is an embedded smooth manifold which decomposes
$M$ into closed subspaces $M_{\pm}:=f^{-1}( \R^{\pm})$.
The pair of subsets $(M_{+},M_{-})$ is a coarsely and uniformly excisive decomposition of $M$.
If $\cY$, $\cZ$ are a big families on $M$, then by \cref{keopgegwerferfw}
this decomposition gives rise to a Mayer-Vietoris fibre sequence 
\begin{equation*}
K_{\cZ}^{\cX}(M_{-}\cap \cY)\oplus K_{\cZ}^{\cX}(M_{+}\cap \cY)\to K_{\cZ}^{\cX}(\cY)\xrightarrow{\delta^{MV}} \Sigma K_{\cZ}^{\cX}(N\cap \cY)\ . 
\end{equation*} 
 We assume that $\Dirac$ is a Dirac operator of degree $n$ which is {positive} away from a big family $\cZ$ (see \cref{kopherthergrge}), and that $\Psi$ is a potential which is very positive away from $\cY$ (see \cref{ekopthertheth}) and asymptotically constant away from $\cZ$ (see \cref{owkpgwgerferrgehghd}). 
 We then   consider the  Callias-type operator 
$\Dirac+\Psi$. We shall assume that $N$ has a   tubular neighbourhood on which all data has a  product structure.
This means that near $N$   the operator has the form
$ \Sigma(\Dirac + \Psi)_{|N}$  for some Callias type operator $(\Dirac+ \Psi)_{|N}$ of degree $n-1$ on $N$,  see \eqref{gwerpogjowkepferfwerfwrf} for notation and \cref{ojgopewrgwegfrefefw} {below} for more details.

\begin{rem}
\label{ojgopewrgwegfrefefw}
One can reconstruct $(\Dirac + \Psi)_{|N}$ as follows.
Assume that $\Dirac+\Psi$ acts on the bundle $E\to M$ of graded {right} $\Cl^{n}$-modules.
Let $\sigma$ in $\End_{\Cl^{n}}(E_{|N})$ be the Clifford multiplication by the normal vector pointing in direction $M_{+}$.
We let $e_{n}$ be the $n$th generator of $\Cl^{n}$ %\nml{acting from the right} 
and $z$ be the grading of $E$.
Then $iz\sigma e_{n}$ is an even selfadjoint involution. 
We let $E_{0}$ be the $1$-eigensubbundle of $E_{|N}$  for $iz\sigma e_{n}$. 
The grading $z$ induces by restriction a grading $z_{|N}$  of $E_{0}$.   
The right action of the subalgebra $\Cl^{n-1}$ generated by  remaining generators $e_{1},\dots,e_{n-1}$ of $\Cl^{n}$ preserves the subbundle and therefore induces a right action on $E_{0}$. The left Clifford multiplication  by tangent vectors along $ N$ also preserves $E_{0}$ and induces a Clifford bundle structure on $E_{0}\to N$. 
Since $\sigma$ and hence $iz\sigma e_{n}$ are parallel by the  product structure assumption, the connection of $E$ induces a connection $\nabla_{0}$ on $E_{0}$. We let $\Dirac_{|N}$ denote the Dirac operator defined by this Dirac bundle structure on $E_{0}$.
The restriction of $\Psi$ to $N$ commutes with $iz\sigma e_{n}$ (it anticommutes with $z$ and $\sigma$ and commutes with $e_n$) and therefore restricts to an endomorphism $\Psi_{|N}$ on $E_{0}$.  
In order to see that $\Dirac+\Psi$ is isomorphic to $\Sigma(\Dirac_{|N}+\Psi_{|N})$ near $N$
we define  an isomorphism
\begin{equation}
\label{IsoE0EN}
E_{0}\stimes {\Cl^{1}}\to E_{|N}\ , \quad v\otimes (a+be_{n})\mapsto  av+bve_{n} \ ,	
\end{equation}
where we consider $\Cl^{1}$ as generated by $e_{n}$ and $a,b$  are in  $\C$. 
 This map is compatible with the gradings, the right $\Cl^{n}$-actions, the left Clifford multiplication tangent vectors of $M$ and the restriction of the connection along $N$. 
 We define the Clifford multiplication by $\sigma$ on $E_{0}\stimes {\Cl^{1}}$ by $ \sigma (v\stimes c)=\pm  v\stimes e_{n}c $, where $v$ is homogenous of parity $\pm$.
 Since we assume a product structure of all data near $N$ it is then clear that
$\Dirac+\Psi$ is isomorphic to $ \Sigma(\Dirac\oplus \Psi)_{|N}$ in a neighbourhood of $N$. \hB
\end{rem}

The assumptions on $\Psi$ imply that   $\Psi_{|N}$ is very positive away from $N\cap \cY$ and is asymptotically constant away from $N\cap\cZ$. 
Below in \cref{tkopherthtere9}  we will need the condition that   $\Dirac_{|N}$ is positive away from $N\cap \cZ$.

\begin{rem}
 In general we can not deduce the local positivity of $\Dirac_{|N}$ from the local positivity of $\Dirac$.
For example, consider the spin Dirac operator on the Riemannian product $M:=\R\times S^{1}$, where  $S^{1}$ has the bounding  spin structure.
We further consider  the decomposition along $N:=\R\times S^{0}$. Then $\Dirac$ is positive everywhere, i.e, away from $\emptyset$, but $\Dirac_{|N}$ is not.

On the other hand, if the positivity of $\Dirac$ away from $\cZ$ is caused by the zero order term in the Weizenboeck formula \eqref{wrgwerrwfrfw}, then because of the assumption of a product structure, $\Dirac_{|N}$ is positive away from $N\cap \cZ$. 
See \cref{khoperthgtrgetg} below for a further example.
 \hB
 \end{rem}

  Let $U$ be {a} tubular neighbourhood of $N$ on which all data has a product structure.
  Assume that $\Dirac$ is positive away from the big family $\cZ$.
  
 \begin{ddd} 
 We say that the width of $U$ goes to $\infty$ away from $\cZ$ if for every $R$ in $(0,\infty)$ there exists  a member $Z$ of $\cZ$ and an isometric embedding  $(-R,R)\times (N\setminus Z)\to U\setminus Z$ compatible with the product structure of the bundle data.
   \end{ddd}
   
   
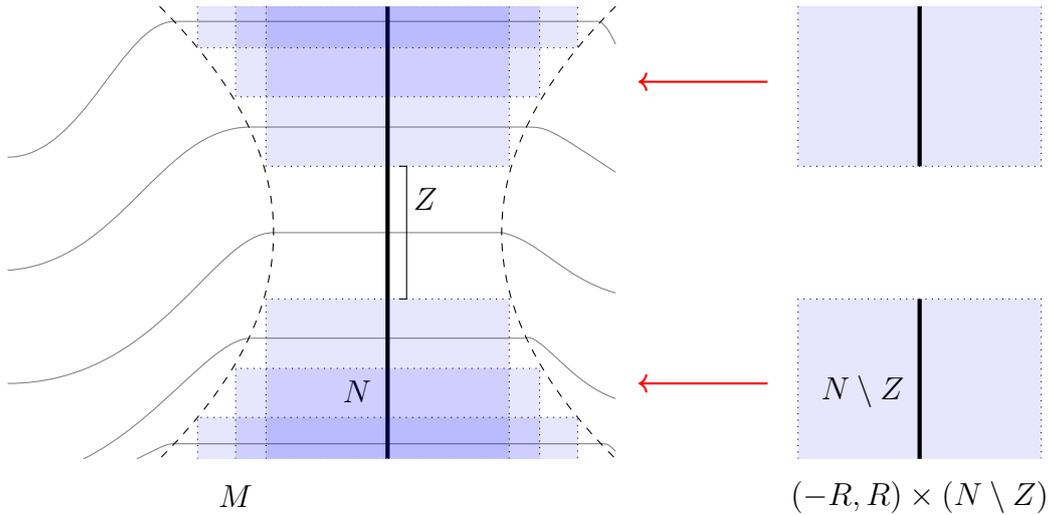
\begin{equation*}
\begin{tikzpicture}
\draw[gray] (-2.8, 2.8) --(2.8,2.8);
\draw[gray] (-2.8, 2.8) .. controls (-3.5, 2.8) and (-4,1) ..(-5,1);
\draw[gray] (2.8,2.8) .. controls (2.9,2.7) ..(3,2.5);
\draw[gray] (-1.8, 1.4) --(1.8,1.4);
\draw[gray] (-1.8, 1.4) .. controls (-3, 1.4) and (-3.5,-0.4) ..(-5,-0.5);
\draw[gray] (1.8,1.4) .. controls (2.1,1.4) ..(3,0.8);
\draw[gray] (-1.5, 0) --(1.5,0);
\draw[gray] (-1.5, 0) .. controls (-2.3, 0) and (-3,-2) ..(-5,-2);
\draw[gray] (1.5,0) .. controls (2,-0.1) and (2.3,-0.6) ..(3,-0.8);
\draw[gray] (-1.8, -1.4) --(1.8,-1.4);
\draw[gray] (-1.8, -1.4) .. controls (-2.3, -1.4) and (-3,-2.5) ..(-4,-3);
\draw[gray] (1.8,-1.4) .. controls (2, -1.3) and (2.4,-2) ..(3,-2.2);
\draw[gray] (-2.8, -2.8) --(2.8,-2.8);
\draw[gray] (-2.8, -2.8) .. controls (-2.9,-2.8) .. (-3.3,-3);
\draw[gray] (2.8, -2.8) .. controls (2.9,-2.8) .. (3,-2.9);
\draw[ultra thick] (0, 3) --(0,-3);
\node at (-0.4,-2.1) {$N$};
\draw[dashed] (-3, 3) .. controls (-1, 1) and (-1, -1) .. (-3, -3);
\draw[dashed] (3, 3) .. controls (1, 1) and (1, -1) .. (3, -3);
\draw[dotted, name path=A1] (-2.5, 3) --(-2.5,2.45) -- (2.5,2.45) -- (2.5, 3);
\draw[dotted, name path=AA1, white] (-2.5, 3)  -- (2.5, 3);
\tikzfillbetween[of=A1 and AA1]{blue, opacity=0.1};
\draw[dotted, name path=A2] (-2.5, -3) --(-2.5,-2.45) -- (2.5,-2.45) -- (2.5, -3);
\draw[dotted, name path=AA2, white] (-2.5, -3) -- (2.5, -3);
\tikzfillbetween[of=A2 and AA2]{blue, opacity=0.1};
\draw[dotted, name path=B1] (-2, 3) --(-2,1.8) -- (2,1.8) -- (2, 3);
\draw[dotted, name path=BB1, white] (-2, 3)  -- (2, 3);
\tikzfillbetween[of=B1 and BB1]{blue, opacity=0.1};
\draw[dotted, name path=B2] (-2, -3) --(-2,-1.8) -- (2,-1.8) -- (2, -3);
\draw[dotted, name path=BB2, white] (-2, -3)  -- (2, -3);
\tikzfillbetween[of=B2 and BB2]{blue, opacity=0.1};
\draw[dotted, name path=C1] (-1.6, 3) --(-1.6,0.88) -- (1.6,0.88) -- (1.6, 3);
\draw[dotted, name path=CC1, white] (-1.7, 3)  -- (1.7, 3);
\tikzfillbetween[of=C1 and CC1]{blue, opacity=0.1};
\draw[dotted, name path=C2] (-1.6, -3) --(-1.6,-0.88) -- (1.6,-0.88) -- (1.6, -3);
\draw (0.15, 0.88) -- (0.25, 0.88) -- (0.25, -0.88) -- (0.15, -0.88);
\draw[dotted, name path=CC2, white] (-1.7, -3)  -- (1.7, -3);
\tikzfillbetween[of=C2 and CC2]{blue, opacity=0.1};
\draw[dotted, name path=D1] (-1.6+7, 3) --(-1.6+7,0.88) -- (1.6+7,0.88) -- (1.6+7, 3);
\draw[dotted, name path=DD1, white] (-1.7+7, 3)  -- (1.7+7, 3);
\tikzfillbetween[of=D1 and DD1]{blue, opacity=0.1};
\draw[dotted, name path=D2] (-1.6+7, -3) --(-1.6+7,-0.88) -- (1.6+7,-0.88) -- (1.6+7, -3);
\draw[dotted, name path=DD2, white] (-1.7+7, -3)  -- (1.7+7, -3);
\tikzfillbetween[of=D2 and DD2]{blue, opacity=0.1};
\draw[ultra thick] (7, 3) --(7,0.88);
\draw[ultra thick] (7, -3) --(7,-0.88);
\node at (0+7, -3.5) {$(-R, R)\times (N \setminus Z)$};
\node at (-2, -3.5) {$M$};
\node at (-0.75+7,-2.1) {$N \setminus Z$};
\node at (0.5,0.45) {$Z$};
\draw[ thick, ->, red] (5, 2) --(3.3,2);
\draw[ thick, ->, red] (5, -2) --(3.3,-2);
\end{tikzpicture}	
\end{equation*}
%\fuli{Vielleicht kann man noch sagen wo $Z$ ist. Dicker Strich rechts in der Mitte}
   
   \begin{lem}\label{khoperthgtrgetg}
   If the width of $U$ goes to $\infty$ away from $\cZ$ and $\Dirac$ is positive away from $\cZ$, then $\Dirac_{|N}$ is  positive away from $N\cap \cZ$.
   \end{lem}
   
\begin{proof}
Since $\Dirac$ is positive away from $\cZ$, according to \cref{kopherthergrge} there exists $a$ in $(0,\infty)$ and a member $Z$ of $\cZ$ such that $\|\Dirac\phi\|^2 \geq a^2\|\phi\|^2$ for each $\phi$ in $C^\infty_c(M \setminus \bar Z, E)$.

Using the identification \eqref{IsoE0EN}, we may consider sections $\phi$ in $C^{\infty}(M,E)$ supported on the image of $(-R,R)\times (N\setminus \bar Z) $ in product form $\phi = f\stimes \psi$ with $f$ in $C_{c}^{\infty}((-R,R))\otimes   \Cl^{1}$ and $\psi$ in $C_{c}^{\infty}(N\setminus \bar Z, E_0)$.
For such a section $\phi$, we have%\fuli{bin mir nicht sicher ob der Zwschenschritt hilfreich ist. \ml{Ich finde schon.}}
\begin{equation}
\label{SplitUp}
	\begin{aligned}
\|\Dirac \phi\|^2_{L^2(M, E)}
	&= \langle \Dirac^2 \phi, \phi \rangle_{L^2(M, E)} 
	= \langle -f^{\prime\prime} \otimes \psi + f \otimes \Dirac_{|N}^2 \psi, f \otimes \psi\rangle_{L^2(M, E)} \\
	&= \|f'\|^2_{L^2((-R, R), \Cl^1)} \|\psi\|^2_{L^2(N, E_0)} + \|f\|^2_{L^2((-R, R), \Cl^1)}
	\|\Dirac_{|N}\psi\|^2_{L^2(N, E_0)}.
	\end{aligned}
\end{equation}
For every $R$ in $(0, \infty)$, we can find $f$ in $C_{c}^{\infty}((-R,R))\otimes \Cl^{1}$ with $\|f\|^2_{L^2} = 1$ and $\|f'\|^2_{L^2} \leq \frac{2}{R^2}$. %\fuli{bist Du dir sicher mit dem ZÃâ¬hler $2$. Ich hatte das offen gelassen, um keinen Fehler einzubauen. Ich denke, manchal werden die Dinge einfacher verstÃâ¬ndlich und lesbar, wenn man nicht alles explizit macht, was mÃÂ¶glich ist \ml{Ja, ich hab es nachgerechnet =)}}
Choose $R$ such that $\frac{2}{R^2} \leq \frac{a^2}{2}$ and enlarge the member $Z$ so that there exists an isometric embedding $(-R, R) \times (N\setminus Z) \to U \setminus Z$ and choose a function $f$ as above corresponding to this $R$.
Then by \eqref{SplitUp}, for all $\psi$ in $C^\infty_c(N \setminus \bar Z, E_0)$, we have
\[
\begin{aligned}
  \|\Dirac_{|N}\psi\|^2_{L^2(N, E_0)} 
  &= \|\Dirac \phi\|^2_{L^2(M, E)} - \|f^\prime\|^2_{L^2((-R, R), \Cl^1)} \|\psi\|_{L^2(N, E_0)}^2 \\
  &\geq a^2 \|\phi\|^2_{L^2(M, E)} - \frac{2}{R^2} \|\psi\|_{L^2(N, E_0)}^2
  \\
  &\geq a^2 \|f\|^2_{L^2((-R, R), \Cl^1)}\|\psi\|^2_{L^2(N, E_0)} - \frac{a^2}{2} \|\psi\|_{L^2(N, E_0)}^2\\
  &= \frac{a^2}{2} \|\psi\|_{L^2(N, E_0)}^2
  %\geq (a^2 \|f\|^2_{L^2((-R, R), \Cl^1)} - \|f'\|^2_{L^2((-R, R), \Cl^1)})\|\psi\|^2_{L^2(N, E_0)} \geq \frac{a^2}{2}\|\psi\|^2_{L^2(N, E_0)}
\end{aligned}
\ .
\]
Hence $\Dirac_{|N}$ is positive away from $N \cap \cZ$.
\end{proof}

 The following theorem is called the Dirac-goes-to-Dirac principle and has been observed in various variations, a first instance 
goes back to  \cite{Baum_1989}.
%Recall that we assume that $\Dirac$ is positive away from $\cZ$,  that $\Dirac_{|N}$ is positive away from $N\cap \cZ$,  that $\Psi$ is very positive away from $\cY$, and that we assume a tubular neighbourhood of $N$ where all data has a product structure.

\begin{theorem}\label{tkopherthtere9} We assume that $\Dirac$ is positive away from $\cZ$,  that $\Dirac_{|N}$ is positive away from $N\cap \cZ$,  and that $\Psi$ is very positive away from $\cY$.
If $N$ has a  tubular neighbourhood with product structure of uniform width, then 
\[
\delta^{MV}(\sigma_{\cZ}(\Dirac+\Psi,\mathrm{on}\,\cY))=\beta^{-1}\sigma_{N\cap\cZ}((\Dirac+\Psi)_{|N},\mathrm{on}\,N\cap \cY)
\]
 in  $K_{\cZ,-n-1}^{\cX}(N\cap\cY)$, 
\end{theorem}

\begin{proof}
  We  apply the relative index theorem (\cref{RelativeIndexTheorem}) to the following data:
  We consider 
  \begin{equation}\label{gojwekrpofewfw}
M_{0}:=\tilde M\  , \quad \cY_{0}:=\tilde \cY_{\cZ} \ , \quad \Dirac_{0}:=\TildeDirac+\tilde \Psi\ ,
\end{equation}
  see  \eqref{wrogijoiwergregfrwefwref}. 
  Furthermore,  we let 
  \[
  M_{1}:=(\R\otimes N)^\sim\ , \quad 
  \cY_{1}:= (\R \times (N \cap \cY))^\sim_{\R\times (N \cap \cZ)}
  %\{\R^{-}\}\times \R\times (N\cap \cY\cap \cZ)\cup \{\R^{+}\}\times \R\times N\cap \cY\ , 
  \quad \Dirac_{1}:=(\Sigma(\Dirac+\Psi)_{|N})^\sim\ ,
  \]
where $(\cdots)^{\sim}$ refers to applying the \cref{kogpweerfwerfrwefw}.
Let $U \cong I \times N$ be the tubular neighborhood of $N$ with uniform width.
By assumption, we have an isomorphism of all data over $\R^+ \times U \cong \R^+ \times I \times N$.
Let $W_0 \subset M_0$ and $W_1 \subset M_1$ be the complements of these sets and let $\cW_i := \{W_i\}$ be the big family generated by $W_i$. 
One checks that the identification  $e: M_{0}\setminus W_{0}\to M_{1}\setminus W_{0}$
is a morphism of bornological coarse spaces.
  
Then by \cref{RelativeIndexTheorem}, we have
\begin{equation}
\label{Prop 4.9 applied}
e_{*}\pi_{0}(\ind\cX(\Dirac_{0},\mathrm{on}\,\cY_{0}))=\pi_{1}(\ind\cX(\Dirac_{1} ,\mathrm{on}\,\cY_{1}))\ ,
\end{equation}
in $K\cX(\cY_{1},\cY_{1}\cap \cW_{1})$, where $\pi_{i}:K\cX(\cY_{i})\to K\cX(\cY_{i},\cY_{i}\cap \cW_{i})$ are the projections.

We  consider the Mayer-Vietoris  boundary for the coarsely excisive compositions of $M_{i}$ given by
$(  M_{0,-},  M_{0,+})$ with 
\[
 M_{0,\pm}=\R\times M_{\pm}\ ,
\] and $(  M_{1,-}, M_{1,+})$ with 
\[
 M_{1,\pm}=\R\times  \R^{\pm }\times N\ .
\]
 By naturality of the Mayer-Vietoris boundaries we get the commutative diagram
\begin{equation*} 
\begin{tikzcd}
K\cX(\cY_{0},\cY_{0}\cap \cW_{0})
\ar[r, "\delta_{0}^{MV}"]
\ar[d, "e_{*}%\simeq
"%, "\eqref{vfpojopewgerwg}"'
]
&
\Sigma K\cX(\tilde{N}\cap \cY_{0}, \tilde{N}\cap {\cY}_{0}\cap \cW_{0})
\ar[d, "\simeq"]
\\
K\cX(\cY_{1},\cY_{1}\cap \cW_{1})
\ar[r, "\delta_{1}^{MV}"]
&
\Sigma K\cX(\tilde{N}\cap \cY_{1}, \tilde{N}\cap \cY_{1}\cap \cW_{1}) \ ,
\end{tikzcd}
\end{equation*} 

so by
\eqref{Prop 4.9 applied}, we get
\begin{equation}
\label{Step111}
	\delta_0^{MV}(\pi_{0}(\ind\cX(\Dirac_{0},\mathrm{on}\,\cY_{0}
)
	)) = \delta_1^{MV}(\pi_{1}(\ind\cX(\Dirac_{1} ,\mathrm{on}\,\cY_{1})))
\end{equation}
in $K\cX_{-n-1}(\tilde{N}\cap \cY_{0}, \tilde{N}\cap \cY_{0}\cap \cW_{0}) \cong K\cX_{-n-1}(\tilde{N}\cap \cY_1, \tilde{N}\cap \cY_1\cap \cW_1)$.

\begin{lem} 
\label{tkopghwtgregrfrfw}
We have
\begin{equation}
\label{wlkegjerogergefwerffefw}
\delta_{1}^{MV} \ind\cX  (\Dirac_{1} ,\mathrm{on}\,  \cY_{1})=
\beta^{-1}\ind\cX( ((\Dirac+\Psi)_{|N})^\sim,\mathrm{on}\, \tilde N\cap \cY_{1})\ .
\end{equation}
in {$K\cX_{-n-1}(\tilde{N}\cap \cY_1)$}. 
\end{lem}

\begin{proof}
The identity of underlying sets induces a map of bornological coarse spaces
${(\R\otimes N)^\sim} \to \R\otimes \tilde N$. By \cref{bjgbpdgdfbdg} the induced map in coarse $K$-homology sends
$\ind\cX(\Dirac_{1},\text{on }\cY_{1})$ to $\ind\cX(\Sigma (( \Dirac +  \Psi)_{|N})^\sim, \text{on }\cY_{1})$. 
By naturality of the Mayer-Vietoris boundaries  we get the first equality in
\begin{align*}\delta_{1}^{MV}(\ind\cX(\Dirac_{1} ,\mathrm{on}\, \cY_{1}))&=\delta_{1}^{MV}(\ind\cX(\Sigma (( \Dirac+  \Psi)_{|N})^\sim,\mathrm{on}\,\cY_{1}))\\&=\beta^{-1}\ind\cX( ((\Dirac+\Psi)_{|N})^\sim,\mathrm{on}\,\tilde{N}\cap \cY_{1})\ ,
\end{align*}
{while the second one follows from the suspension theorem (\cref{SuspensionTheorem}).}
\end{proof}

Combining \eqref{Step111} with \eqref{wlkegjerogergefwerffefw} and using the definition of $ \Dirac_0$ and $\cY_{0}$ in \eqref{gojwekrpofewfw}, we get
\begin{equation*}
	  \pi_0(\delta_0^{MV}(\ind\cX(\TildeDirac+\tilde \Psi,\mathrm{on}\,\tilde{\cY}_{\cZ}))) = \beta^{-1} \pi_1(\ind\cX( ((\Dirac+\Psi)_{|N})^\sim,\mathrm{on}\,\tilde{N}\cap \cY_{1}))
\end{equation*} 
Here we implicitly employed the fact that $\pi_i$ intertwines the relative with the absolute Mayer-Vietoris boundary. %\fuli{diese Formulierung gefÃâ¬llt mir noch nicht.}
Applying the map of coarse spaces $\iota : \tilde{M} \to \cO^\infty(M)$ to this equation and using naturality of the Mayer-Vietoris boundary, \cref{ktogpwegrefwf} yields the following result.

\begin{kor}
We have
\begin{equation}
\label{wlkegjerogergefwerffefw1}
\pi_0 (\delta^{MV} ( \sigma_{ \cZ} (\Dirac+\Psi ,\mathrm{on}\, \cY)))
=
 \pi_0 (\beta^{-1} \sigma_{N\cap \cZ} ( (\Dirac+\Psi)_{|N},\mathrm{on}\,N\cap \cY))\ .
\end{equation}
in $K\cX_{-n -1}( \cO_{\cZ}^{\infty}(N\cap \cY),\cO^{\infty}_{\cZ}(N\cap \cY)\cap \cW_{0})$. 
\end{kor}

We now observe that every member of $\cO^{\infty}_{\cZ}(N \cap \cY)\cap \cW_{0}$ is contained in a member of $\cO^{-}(N \cap \cY\cap \cZ)$. Since these subsets are flasque  the map 
\[
K\cX(\cO^{\infty}_{\cZ}(N \cap \cY)\cap \cW_{0})\to K\cX(\cO_{\cZ}^{\infty}(N\cap \cY))
\]
 vanishes and
the projection $\pi_0$ is a split injection.
We can therefore omit $\pi_0$ in \eqref{wlkegjerogergefwerffefw1} and get the desired equality.
 \end{proof}
 
 \begin{rem}
 Even if we do not know that  $\Dirac_{|N}$ is positive away from $N\cap \cZ$
 we have a class 
 $\delta^{MV} \sigma_{\cZ}(\Dirac+\Psi,\mathrm{on}\, \cY)$ in $K_{\cZ}^{\cX}(N\cap \cY)$
 whose image under $K_{\cZ}^{\cX}(N\cap \cY)\to K^{\cX}(N\cap \cY)$ is
 $\beta^{-1}\sigma((\Dirac+\Psi)_{|N}, \mathrm{on}\, N\cap  \cY)$.  \hB
 \end{rem}

\subsection{Decompositions and Mayer-Vietoris}\label{koregperegfrefrfrefw}

We consider a bornological coarse space $M$ (we use the symbol $M$ since it will later be the bornological coarse space associated to a complete Riemannian manifold).  
Let $\tilde f:M\to \R$ be a controlled function. 
Then we define the big families
\[
\cM_{\pm}:= \tilde f^{-1}(\{\R^{\pm}\})
\]
 on $M$ and set $\cY:=\cM_{-}\cap \cM_{+}$.
We consider the function 
\begin{equation}\label{bjbkjsldkbfjlsdfbvsdfvsvs}f:=\frac{\tilde f}{\sqrt{1+\tilde f^{2}}}\quad \mbox{in} \quad  \ell^{\infty}_{\cY}(M) 
\end{equation} 
 and observe that   
$$(\frac{f+1}{2})^{2}- \frac{f+1}{2} \in \ell^{\infty}(\cY)\ ,$$  see \cref{gjsoepgergseffs} for notation. 
The class of $\frac{f+1}{2}$ in $C(\partial^{\cY}X) = \smash{\frac{\ell^{\infty}_{\cY}(M)}{ \ell^{\infty}(\cY)}}$  is a projection
 and represents a $K$-theory class
 \begin{equation}
 \label{pclass}
 p:=\left[\frac{f+1}{2}\right]
 \end{equation}
   in $K_0(C (\partial^{\cY} X))$, see \cref{jiogowergwefrfwrf}.

%\nml{On the one hand} 
We have the pairing \eqref{fwerfewrfewrfwerf}   $$p\cap^{\cX}-:K\cX(M)\to \Sigma K\cX(\cY) \ .$$ On the other hand,  the decomposition of $M$ into
$(\cM_{-},\cM_{+})$ induces in view of  \eqref{gwerojopewferferwf} a Mayer-Vietoris fibre sequence
$$K\cX(\cM_{-})\oplus K\cX(\cM_{+})\to K\cX(M)\xrightarrow{\partial^{MV}} \Sigma K\cX(\cY)\ .$$

\begin{lem}\label{gojsekorpgergesrfrefs}
We have an equivalence  $p\cap^{\cX}-\simeq \partial^{MV}:  K\cX(M)\to  \Sigma K\cX(\cY)$. 
 \end{lem}
 \begin{proof}
 The origin of the Mayer-Vietoris boundary map for coarse $K$-homology $K\cX$ is the excisiveness of the  $\Mod(KU)$-valued functor $K\cX$ on  $\BC$. It sends the square \begin{equation}\label{wergojopwefwerfr}\xymatrix{\cY\ar[r]\ar[d] &\cM_{+} \ar[d] \\ \cM_{-}\ar[r] &M } 
\end{equation} 
of big families or objects in $\BC$ to %\nml{a} 
{the} left square in  \begin{equation}\label{ogjermf432r}\xymatrix{K\cX(\cY)\ar[r]\ar[d] &K\cX(\cM_{+}) \ar[d]\ar@{..>}[r]&K\cX(\cM_{+},\cY)\ar[d]^{\simeq} \\ K\cX(\cM_{-})\ar[r] &K\cX(M) \ar@{..>}[r]&K\cX(M,\cM_{-})} 
\end{equation}in $\Mod(KU)$ which has the property of being cocartesian. It extends to a map of fibre sequences as indicated so that the right vertical map is an equivalence.
The Mayer-Vietoris boundary map is by definition
\begin{equation}\label{fqwfqweewdqed}\partial^{MV}:K\cX(M)\to K\cX(M,\cM_{-}) \stackrel{\simeq}{\leftarrow} K\cX(\cM_{+},\cY)\stackrel{\partial}{\to} \Sigma K\cX(\cY)\ ,\end{equation} 
where $\partial$ is the boundary map associated to the upper fibre sequence.
Since we want to compare this map with a construction of $p\cap^{\cX}- $ defined on the level of $C^{*}$-categories
we must also express the $\partial^{MV}$ in terms of $C^{*}$-categorical constructions.

 The square \eqref{wergojopwefwerfr}
 yields an excisive square of $C^{*}$-categories \cite[Lemma 6.10]{coarsek}
 \begin{equation}\label{fpoerjfkopwfewdqewd}\xymatrix{\bC(\cY\subseteq M)\ar[r]\ar[d] &\bC(\cM_{+} \subseteq M)\ar[d] \\ \bC(\cM_{-}\subseteq M)\ar[r] &\bC(M) } \ .
\end{equation} 
 Applying the $K$-theory functor for $C^{*}$-categories we get a square
 $$\xymatrix{K(\bC(\cY\subseteq M))\ar[r]\ar[d] &K(\bC(\cM_{+} \subseteq M))\ar[d] \\K( \bC(\cM_{-}\subseteq M))\ar[r] &K(\bC(M)) }$$ 
 which corresponds to the left push-out square in \eqref{ogjermf432r}. 
  % By definition of $K\cX$  it induces the Mayer-Vietoris boundary
 %$$\partial^{MV}:K\cX(M)\simeq  K(\bC(M))\xrightarrow{\partial^{C^{*}}} \Sigma K(\bC(\cY\subseteq M))\simeq \Sigma K\cX(\cY)\ .$$  
 Unfolding the definition of $\nu$ in \eqref{iogjweoigwefewrfwerfwf11} and \eqref{kortphertherther} the map 
 $p\cap^{\cX}-$ is given  by the following composition
 \begin{align*}
 K\cX(M)&\simeq K(\bC(M)) \xrightarrow{p\otimes \id}
  K\left(\frac{\ell^{\infty}_{\cY}(M) }{ \ell^{\infty}(\cY)}\otimes \bC(M)\right) \\&\stackrel{\nu}{\to} 
  K\left(\frac{\bC(M)}{\bC(\cY\subseteq M)}\right) \xrightarrow{\partial} \Sigma K(\bC(\cY\subseteq M))\simeq \Sigma K\cX(\cY)
 \end{align*}
We now argue with the commutative diagram
\[
\begin{tikzcd}[column sep=-0.5cm]
K(\bC(M)) 
	\ar[rr]
	\ar[d, "\nu \circ (p \otimes \id)"]
	&& 
		K\Big(\frac{\bC(M)}{\bC(\cM_{-}\subseteq M)}\Big)
		\\ 
K\Big(\frac{\bC(M)}{\bC(\cY\subseteq M)}\Big)
	\ar[dr, "\partial"']
	&&
	K\Big(\frac{\bC(\cM_{+}\subseteq M)}{\bC(\cY\subseteq M)}\Big)
	\ar[dl, "\partial"]
	\ar[ll] 
	\ar[u, "\simeq"']
	\\ 
	&
	{\Sigma} K(\bC(\cY\subseteq M))  & \ 
\end{tikzcd}
\]
where the horizontal arrows are induced by the obvious projections and inclusions. 
The lower triangle commutes by the naturality of the boundary operator for $C^{*}$-algebra $K$-theory  for maps of exact sequences of $C^{*}$-categories. The upper right vertical morphism is the left vertical equivalence in \eqref{ogjermf432r}. The 
 composition along the left part reflects the definition of $p\cap^{\cX}-$, and the other composition  is the definition \eqref{fqwfqweewdqed}  of $\partial^{MV}$.
% Since  the representative  $\frac{1+f}{2}$ of $p$ belongs to $\frac{\ell^{\infty}_{\cY}(M) }{ \ell^{\infty}(\cY)}$ we get the dotted factorization of the marked arrow.
  \end{proof}

We now assume that $M$ is a uniform bornological coarse space and that $\tilde f$ is in addition uniformly continuous.
 We define the closed subspaces
$M_{\pm}:=f^{-1}( \R^{\pm})$ and set $N:=M_{+}\cap M_{+}$.

 Then we have  $p\in K_{0}(C(\partial^{\cY}_{u}M))$ and 
 $\partial p\in K^{1}(\cY)$, where $\partial$ is the boundary map in $K$-theory associated to the exact sequence of $C^{*}$-algebras \eqref{gerwrgfrq} and we use  the grading convention   $\pi_{-*}K(C_{u}(\cY))\simeq K_{}^{*}(\cY)$.

Let $\cZ$ be a second big family.
 We have the coarse symbol pairing
$$\partial p\cap^{\cX\sigma}-:K^{\cX}_{\cZ}(M)\to \Sigma 
K^{\cX}_{\cZ}(\cY)$$
introduced in \cref{grjweopgergweffefwef}.
By \cref{keopgegwerferfw}.\ref{werkjgokwegergwerg9} the  decomposition $(M_{+},M_{-})$ (which is coarsely and uniformly excisive) of $M$ induces a Mayer-Vietoris fibre sequence
$$K^{\cX}_{\cZ}(M_{+})\oplus K^{\cX}_{\cZ}(M_{-})\to K^{\cX}_{\cZ}(M)\xrightarrow{\delta^{MV}} \Sigma K_{\cZ}^{\cX}(N)\ .$$
We let \begin{equation}\label{gwegerf3ferwfgreg}i^{\cX}:  K_{\cZ}^{\cX}(N)\to   K_{\cZ}^{\cX}(\cY)\quad \mbox{and} \quad i:K\cX(N\cap \cZ) \to K\cX(\cY\cap \cZ) 
\end{equation} denote the maps induced by the obvious inclusions. %Note that $i$ is an equivalence since the inclusion of $N$ into  every member of $\cY$ is a coarse equivalence.
 The square  \begin{equation}\label{vfdsvervfdvsdfvsdvdf}\xymatrix{ K_{\cZ}^{\cX}(N)\ar[r]^{i^{\cX}}\ar[d]^{a_{N,N\cap \cZ}} &  K_{\cZ}^{\cX}(\cY) \ar[d]^{a_{\cY,\cY\cap \cZ}} \\ K\cX(N\cap \cZ)\ar[r]_{i} 	  &K\cX(\cY\cap \cZ) }
\end{equation}
 commutes by the naturality of the  index map \cref{rtkohprethergetetg}.
 
\begin{lem}\label{kogpqgregergwrgrgw}
We have an equivalence
$$\partial p\cap^{\cX\sigma}- \simeq i^{\cX}\circ \delta^{MV}: K^{\cX}_{\cZ}(M)\to \Sigma K^{\cX}_{\cZ}(\cY)\ .$$
\end{lem}
\begin{proof}
We consider the following morphism of exact sequences of $C^{*}$-categories
\[
\begin{tikzcd}
0
\ar[d]
	&
		0
		\ar[d]
		&
			0
			\ar[d]
			\\
C_{u}(\cY)\otimes \bC(\cO_{\cZ}^{\infty}(M))
\ar[r]
\ar[d]
\ar[r, "\mu"]
	&
		\frac{\bC(\cO_{\cZ}^{\infty}(\cY)\subseteq \cO^{\infty}(M))}{\bC(\cO^{-}(\cY\cap \cZ)\subseteq \cO^{\infty}(M))}
		\ar[d]
		&
			\bC(\cO_{\cZ}^{\infty}(\cY)\subseteq \cO^{\infty}(M))
				\ar[l]
				\ar[d]
				\\
C_{u,\cY}(X)\otimes \bC(\cO_{\cZ}^{\infty}(M)) 
\ar[d]
\ar[r]
&
		\frac{\bC(\cO_{\cZ}^{\infty}(M)\subseteq \cO^{\infty}(M))}{\bC(\cO^{-}(\cY\cap \cZ)\subseteq \cO^{\infty}(M))}
		\ar[d]
		&
			\bC(\cO_{\cZ}^{\infty}(M)\subseteq \cO^{\infty}(M))
			\ar[l]
			\ar[d]
			\\
C_{u}(\partial_{u}^{\cY}M)\otimes \bC(\cO_{\cZ}^{\infty}(M))
\ar[r, "!"]
\ar[d]
&
	\frac{\bC(\cO_{\cZ}^{\infty}(M)\subseteq \cO^{\infty}(M))}{\bC(\cO_{\cZ}^{\infty}(\cY)\subseteq \cO^{\infty}(M))}
	\ar[d]
	&
		\frac{\bC(\cO_{\cZ}^{\infty}(M)\subseteq \cO^{\infty}(M))}{\bC(\cO_{\cZ}^{\infty}(\cY)\subseteq \cO^{\infty}(M))}
		\ar[l, equal]
		\ar[d]
		\\
0
	&
	0
		&
		0	 
\end{tikzcd}
\]
where the left vertical sequence is obtained by tensoring \eqref{gerwrgfrq} with $\bC(\cO_{\cZ}^{\infty}(M))$. % and the other two come from exact sequences of $C^{*}$-categories which  can easily be read off from the diagram.
{
To explain the arrow marked ``!'', notice that we have a map
\[
  \partial^{\cO^\infty(\cY)}\cO^\infty(M) \stackrel{\pr}{\longrightarrow} \partial^{\cY}M \stackrel{c}{\longrightarrow} \partial^{\cY}_u M   \ ,
\]
where the first map is the extension of the controlled map $\pr:\cO^{\infty}(M)\to M$ to a continuous map between the coronas (see \cref{RemFunctorialityCorona}) and the second map is the comparison map \eqref{ComparisonMap}.
The arrow marked ``!'' is then the composition $\nu \circ ((c \circ \pr)^* \otimes \id)$, where $\nu$ is the appropriate version of \eqref{iogjweoigwefewrfwerfwf11}. 
}

Applying $K$-theory and using the naturality of the boundary map in $K$-theory for morphisms of exact sequences of $C^{*}$-categories we get the two   lower right commutative squares in 
\[
\small
\begin{tikzcd}[column sep=-2cm]
K({\partial^{\cO^\infty(\cY)}\cO^\infty(M)})\times K_{\cZ}^{\cX}(M) 
\ar[dr]
\\
	&
	K\left({C(\partial^{\cO^\infty(\cY)}\cO^\infty(M))}  \otimes \bC(\cO_{\cZ}^{\infty}(M))\right)
	\ar[dr, "K(\nu)"] 
	\\
K(\partial_{u}^{\cY}M)\times K^{\cX}_{\cZ}(M)
\ar[uu, "{K((c \circ \pr)^{*}) \times \id}"]
\ar[dd, "\partial\times \id"']
\ar[dr]
	&
	&
	K\left(\frac{\bC(\cO_{\cZ}^{\infty}(M))}{\bC(\cO^{\infty}_{\cZ}(\cY)\subseteq \cO^{\infty}(M))}\right)
	\ar[dd, equal] 
		&\\
	&
	K(C(\partial_{u}^{\cY}M)\otimes \bC(\cO_{\cZ}^{\infty}(M)))
	\ar[dr, "K(!)", near end]
	\ar[uu, "K({(c \circ \pr)^*} \times \id)", near end]
	\ar[dd, "\partial"']
		&
		\\
\Sigma K(\cY)\times  K_{\cZ}^{\cX}(M)
\ar[dr]
&
	&
		K\left(\frac{\bC(\cO_{\cZ}^{\infty}(M)\subseteq \cO^{\infty}(M))}{\bC(\cO^{\infty}_{\cZ}(\cY)\subseteq \cO^{\infty}(M))}\right)
		\ar[dd, "\partial"]
		\ar[dr, equal]
			&
			\\
&
	\Sigma K(C_{u}(\cY)\otimes \bC(\cO_{\cZ}^{\infty}(M)))
	\ar[dr]	
	&
		&
			K\left(\frac{\bC(\cO_{\cZ}^{\infty}(M)\subseteq \cO^{\infty}(M))}{\cO^{\infty}_{\cZ}(\cY)\subseteq \cO^{\infty}(M)}\right)
			\ar[dd, "\partial"]
			\\
&
	&
		\Sigma K\left(\frac{\bC(\cO_{\cZ}^{\infty}(\cY)\subseteq \cO^{\infty}(M))}{\bC(\cO^{-}(\cY\cap \cZ)\subseteq \cO^{\infty}(M))}\right)
			&
			\\
&
	&
		&
			\Sigma K(\bC(\cO^{\infty}_{\cZ}(\cY)\subseteq \cO^{\infty}(M)))
			\ar[ul, "\simeq"'] 	
\end{tikzcd}
\]
%\[
%\footnotesize
%\begin{tikzcd}
%K(\frac{l_{ \cO^{\infty}(\cY)}^{\infty}(\cO^{\infty}(M))}{\ell^{\infty} (\cO^{\infty}(\cY)) } )\times K_{\cZ}^{\cX}(M) 
%\ar[r]
%	&
%	K(\frac{l_{ \cO^{\infty}(\cY)}^{\infty}(\cO^{\infty}(M))}{\ell^{\infty} (\cO^{\infty}(\cY)) } \otimes \bC(\cO_{\cZ}^{\infty}(M)))
%	\ar[r, "!"] 
%	&
%	K(\frac{\bC(\cO_{\cZ}^{\infty}(M))}{\bC(\cO^{\infty}_{\cZ}(\cY)\subseteq \cO^{\infty}(M))})
%	\ar[d, equal] 
%		&\\
%K(\partial_{u}^{\cY}M)\times K^{\cX}_{\cZ}(M)
%\ar[u, "\pr^{*}"]
%\ar[d, "\partial\times \id"]
%\ar[r]
%	&
%	K(C(\partial_{u}^{\cY}M)\otimes \bC(\cO_{\cZ}^{\infty}(M)))
%	\ar[r]
%	\ar[u, "K(\pr^{*}\otimes \id"]
%	\ar[d, "\partial"]
%		&
%		K(\frac{\bC(\cO_{\cZ}^{\infty}(M)\subseteq \cO^{\infty}(M))}{\bC(\cO^{\infty}_{\cZ}(\cY)\subseteq \cO^{\infty}(M))})
%		\ar[d, "\partial"]
%		\ar[d, "\partial"]
%			&
%			K(\frac{\bC(\cO_{\cZ}^{\infty}(M)\subseteq \cO^{\infty}(M))}{\cO^{\infty}_{\cZ}(\cY)\subseteq \cO^{\infty}(M)})
%			\ar[d, "\partial"]
%			\ar[l, equal]
%			\\
%\Sigma K(\cY)\times  K_{\cZ}^{\cX}(M)
%\ar[r]
%	& 
%	\Sigma K(C_{u}(\cY)\otimes \bC(\cO_{\cZ}^{\infty}(M)))
%	\ar[r]
%		&
%		\Sigma K(\frac{\bC(\cO_{\cZ}^{\infty}(\cY)\subseteq \cO^{\infty}(M))}{\bC(\cO^{-}(\cY\cap \cZ)\subseteq \cO^{\infty}(M))})
%			&
%			\Sigma K(\bC(\cO^{\infty}_{\cZ}(\cY)\subseteq \cO^{\infty}(M)))
%			\ar[l, "\simeq"] 	
%\end{tikzcd}
%\]
  The left horizontal maps are given by the symmetric monoidal structure of the $K$-theory functor for $C^{*}$-categories and the {left squares commute by bi-exactness of this structure}. 
{
Finally, the top right square commutes by functoriality of $K$, by definition of the arrow marked ``!''.
}
 
    If we fix $p$ in $K_{0}(\partial_{u}^{\cY}M) $, then the  {up}-right-down composition is 
    \[
    (c \circ \pr)^{*}p\cap^{\cX} -:K^{\cX}_{\cZ}(M)
    \to \Sigma K^{\cX}_{\cZ}(\cY)\stackrel{\mathrm{def}}{=}\Sigma K(\bC(\cO^{\infty}_{\cZ}(\cY)\subseteq \cO^{\infty}(M))) 
    \ .
    \] 
    By \cref{gojsekorpgergesrfrefs}  we   get
    \[
    {(c \circ \pr)^{*}}p\cap^{\cX}-\simeq \tilde  \partial^{MV}:K_{\cZ}^{\cX}(M)\to \Sigma K_{\cZ}^{\cX}(\cY)\ ,
    \]
     where $\tilde \partial_{MV}$ is the Mayer-Vietoris boundary for
 the decomposition of $\cO_{\cZ}^{\infty}(M)$ into $\cO_{\cZ}^{\infty}(\cM_{+})$ and $\cO_{\cZ}^{\infty}(\cM_{-})$.
 The down-right composition is 
 \[
 \partial p\cap^{\cX\sigma}-:K^{\cX}_{\cZ}(M)\to \Sigma K\cX(\cO_{\cZ}^{\infty}(M))\to \Sigma K_{\cZ}^{\cX}(\cY)\ .
 \]
   We have a map of squares
\[
 \begin{tikzcd}
\cO^{\infty}_{\cZ}(N)\ar[r]\ar[d]&\cO^{\infty}_{\cZ}(M_{+})\ar[d]\\\cO^{\infty}_{\cZ}(M_{-})\ar[r]&\cO^{\infty}_{\cZ}(M)	
 \end{tikzcd}
 \quad \to \quad
 \begin{tikzcd}
\cO^{\infty}_{\cZ}(\cY)\ar[r]\ar[d]&\cO^{\infty}_{\cZ}(\cM_{+})\ar[d]\\\cO^{\infty}_{\cZ}(\cM_{-})\ar[r]&\cO^{\infty}_{\cZ}(M)
\end{tikzcd}
 \]%\fuli{entsprechend unserer Konvention weiter oben können wir die Durchschnitte bei den Indizes auch weglassen. Sollen wir?}\color{black}
 which implies that
 $i^{\cX}\circ \delta^{MV}\simeq \tilde \partial^{MV}$.
 Combining these equivalences we get the assertion of the lemma.
\end{proof}

  We now assume that $M$ is represented by a complete Riemannian manifold, that $\tilde f:M\to \R$ is smooth, and that $0$ is a regular value. Then
  $N$  is an embedded codimension-one submanifold of $M$  with trivial normal bundle which separates $M$ into the two components $M_{-}$ and $M_{+}$.  
  We equip $N$ with the induced Riemannian metric. Then  the inclusion $i:N\to M$
  is a morphism  of bornological coarse spaces.
  We assume that $N$ has a product tubular neighbourhod of uniform width.

  If $\Dirac_{0}$ is a generalized Dirac operator on $M$  which is {positive} away from a big family $\cZ$ with symbol $\sigma_{\cZ}(\Dirac_{0})$ in $K^{\cX}_{\cZ}(M)$, then by \cref{tkopherthtere9} we know that%\fml{Das Theorem braucht aber Product Structure of Uniform Width...}
 \[
 \delta^{MV}(\sigma_{\cZ}(\Dirac_{0}))=\beta^{-1}\sigma_{N\cap\cZ}({\Dirac_{0|N}})
 \]
   in $K^{\cX}_{\cZ}(N)$ provided that $\Dirac_{0}$ has a product structure on the  product tubular neighbourhod of uniform width and
   $\Dirac_{0|N}$ is also positive away from $N\cap \cZ$.    We can conclude  a coarse version of Roe's partitioned manifold index theorem \cite{zbMATH04109704}.

\begin{kor}
\label{gwergregwerferfw} 
Under the assumption made above, we have an  equality
\[
i(\ind\cX(\Dirac_{0|N},\mathrm{on}\,N\cap \cZ))\simeq \beta \cdot \partial^{MV} \ind\cX(\Dirac_{0},\mathrm{on}\,\cZ)
\]
 in $K\cX(\cY \cap \cZ)$.
\end{kor}
 \begin{proof}
 Since the index map \cref{rtkohprethergetetg} is a natural transformation between $\Mod(KU)$-valued excisive functors it 
 commutes with Mayer-Vietoris maps. This is used in the marked equality below.
 We calculate
 \begin{eqnarray*}
 i(\ind\cX(\Dirac_{0|N},\mathrm{on}\,N\cap \cZ))
 &
 \stackrel{\text{\cref{jogergreffwfrf}}}{=}&
 i (a_{N,\cZ}(\sigma_{N\cap\cZ}({\Dirac_{0|N}})))\\
 &
 \stackrel{\text{\cref{tkopherthtere9}}}{=}&
 \beta \cdot i( a_{N,\cZ} (   \delta^{MV}\sigma_{\cZ}(\Dirac_{0}) ))\\
 &\stackrel{!}{=}&
 \beta   \cdot \partial^{MV} a_{M,\cZ}(\sigma_{\cZ}(\Dirac_{0}))
 \\
 &\stackrel{\text{\cref{jogergreffwfrf}}}{=}&
\beta  \cdot \partial^{MV} \ind\cX(\Dirac_{0},\mathrm{on}\,\cZ)
% 
% &\stackrel{\eqref{vfdsvervfdvsdfvsdvdf}}{=}&  a_{\cY,\cY\cap \cZ} (i^{\cX}(  \delta^{MV}\sigma_{\cZ}(\Dirac_{0}))) \\&\stackrel{\cref{kogpqgregergwrgrgw}}{=}& a_{\cY,\cZ}(  \partial p\cap^{\cX\sigma} \sigma_{\cZ}(\Dirac_{0})))\\&\stackrel{\cref{erokgpwergrwefwerfwerfwer}}{=}&  p\cap^{\cX} a_{M,\cZ}(\sigma_{\cZ}(\Dirac_{0}))\\&\stackrel{\cref{jogergreffwfrf}}{=}& p\cap^{\cX} \ind\cX(\Dirac_{0},\mathrm{on}\,\cZ)\\&\stackrel{\cref{gojsekorpgergesrfrefs}}{=}& \partial^{MV} \ind\cX(\Dirac_{0},\mathrm{on}\,\cZ)
\ .
\end{eqnarray*}
 %We now use that $i\circ \hat \partial^{MV}\simeq\partial^{MV}$ and that $i$ is an equivalence.
 \end{proof}

 \begin{rem}
 \cref{gwergregwerferfw} leads to the usual application to
 obstructions against positive scalar curvature in the coarse equivalence class of $M$. 
 %We fix a Riemannian metric $g'$  on $M$.\fml{Warum heisst die $g'$ und nicht $g$?}
  We assume that $i:N\to M$ is a coarse embedding with a uniform tubular neighbourhood along $N$.  
  Then $i:K\cX(N)\to K\cX(\cY)$ is an equivalence.
  Assume that {$M$ is spin and let} $\Dirac_{0}$ {be} the spin Dirac operator. 
Then $\Dirac_{0|N}$ is again the spin Dirac operator on $N$.   
If $\ind\cX(\Dirac_{0|N})\not=0$, then by \cref{gwergregwerferfw} and since $i$ is an equivalence also $\ind\cX(\Dirac_{0})\not=0$.

We call two complete Riemannian  metrics $g,g'$  on $M$ coarsely equivalent if $(M,g)$ and $(M,g')$ represent the same bornological coarse space.
If $g'$ is coarsely equivalent to $g$ and $\Dirac'_{0}$ is the associated Dirac operator, then $\ind\cX(\Dirac_{0})=\ind\cX(\Dirac'_{0})$.
We use here that two coarsely equivalent metrics induce the same index classes  for the spin Dirac operators.
This can again be shown similarly as \cref{bjgbpdgdfbdg} using the suspension theorem and the relative index theorem. 

{If $\ind\cX(\Dirac_{0|N})\not=0$,} we can conclude that $M$ can not have a metric $g'$  with uniformly positive scalar curvature in the coarse equivalence class of $M$.
It is important to note here that it is not required that $g'$ has a uniform tube at $N$, or even a product structure at $N$ at all. 
Otherwise the conclusion would be trivial since then the metric on $N$ induced by $g'$ has uniformly positive scalar curvature and therefore $\ind\cX(\Dirac_{0|N})=0$.

Using the supports we obtain a new type of obstruction. We say that a  class in $K\cX(N)$  comes from the big family $\cY$ in $N$ if it belongs to the image of
$K\cX(\cY)\to K\cX(N)$.

Let $\cZ$ be a big family in $M$.

\begin{kor}
If $\ind\cX(\Dirac_{0|N})$ does not come from $N\cap \cZ$, then there is no complete metric $g'$ in the coarse equivalence class  of $g$ of uniform positive scalar curvature on $M\setminus Z$  for any member $Z$ of $\cZ$. \hB
\end{kor}
\end{rem}

 Assume that $\Dirac_{0}$ is a generalized Dirac operator of  degree $k$ on a bundle $E_{0}\to M$ which is {positive} away from a big family $\cZ$ and that $\Dirac_{0|N}$ is  positive away from  $N\cap \cZ$. We further assume that  $\smash{\tilde f}$ is asymptotically constant away from $\cZ$. 
 We define $\Psi_{0}:M\to \End_{\Cl^{1}}(\Cl^{1})$ by
 $\Psi_{0}(m):=\tilde f(m)i\sigma$, where $\sigma$ in $\Cl^{1}$ is the odd, anti-selfadjoint generator with $\sigma^{2}=-1$ %\nml{acting from the left}\fml{$\Cl^1$ is commutative.}. 
 Then $\Psi_{0}$ is very positive away from $\cY$ and asymptotically constant away from $\cZ$. 
 We define the bundle $E:=E_{0}\stimes \Cl^{1}$ of $\Cl^{k+1}$-modules, $\Dirac:=\Dirac_{0}\stimes \id_{\Cl^{1}}$ and $\Psi:=\id_{E_{0}}\stimes \Psi_{0}$.
 
The function $\Psi_{0}$ defines a class $[\Psi_{0}]$ in $K^{1}(\cY)$. By \cref{wegokwpergerfwefrwefwf} we have
$$\sigma_{\cZ}(\Dirac+\Psi,\mathrm{on}\,\cY)=[\Psi_{0}] \cap^{\cX\sigma} \sigma_{\cZ}(\Dirac_{0})\ .$$ 
Combining  \eqref{gwegrwerfwerfwef}, \eqref{jgiowjeroijfoweferfwef} and  \eqref{bjbkjsldkbfjlsdfbvsdfvsvs} we have the well-known formula for the $K$-theoretic boundary map 
 \begin{equation}
 \label{rgwergeggwgre}
 [\Psi_{0}]=\partial p\quad \mbox{in}\quad K^{1}(\cY) \ . 
 \end{equation}

\begin{prop} In $K\cX(\cY\cap \cZ)$ 
we have  $$\ind\cX(\Dirac+\Psi,\mathrm{on}\,\cY\cap \cZ)=\beta^{-1} i( \ind\cX(\Dirac_{0|N},\mathrm{on}\,N\cap \cZ))$$
\end{prop}

\begin{proof} 
We have the following chain of equalities:
\begin{eqnarray*}
\ind\cX(\Dirac+\Psi,\mathrm{on}\,\cY\cap \cZ)
&\stackrel{\eqref{qefewdad}}{=}&
a_{\cY,\cZ }(\sigma_{\cZ}(\Dirac+\Psi,\mathrm{on}\,\cY))
\\
&\stackrel{\text{\cref{wegokwpergerfwefrwefwf}}, \eqref{rgwergeggwgre}}{=}& 
a_{\cY,\cZ }(\partial p\cap^{\cX\sigma}  \sigma_{\cZ}(\Dirac_{0}))
\\
&\stackrel{\text{\cref{erokgpwergrwefwerfwerfwer}}}{=}& 
p\cap^{{\cX}}a_{M,\cZ}(  \sigma_{\cZ}(\Dirac_{0})))
\\
&\stackrel{\text{\cref{jogergreffwfrf}}}{=}& 
p\cap^{\cX} \ind\cX(\Dirac_{0},\mathrm{on}\,\cZ)
\\
&\stackrel{\text{\cref{gojsekorpgergesrfrefs}}  }{=}&  
\partial^{MV}\ind\cX(\Dirac_{0},\mathrm{on}\,\cZ)
\\
&\stackrel{\text{\cref{gwergregwerferfw}}}{=}&
\beta^{-1} i( \ind\cX(\Dirac_{0|N},\mathrm{on}\,N\cap \cZ))\ .
\end{eqnarray*}
\end{proof}

\begin{rem}
If $N$ is compact, then $K(\cY)\simeq K\cX(*)\simeq KU$ and $K\cX(N)\simeq KU$ and under these indentifications 
$\ind\cX(\Dirac+\Psi,\mathrm{on}\,\cY)$ and 
$\ind\cX(D_{0|N})$ are the usual Fredholm indices.  If $k+1$ is even, then the equality of indices
$\ind(\Dirac+\Psi )=  \ind(\Dirac_{0|N})$ in $KU_{-k-1}$ is a special case of  Roe's partitioned manifold index theorem \cite{zbMATH04109704}, and also of the index theorem for Callias type operators shown in \cite{MR1348799}. 
  \hB
 \end{rem}
 
 \subsection{The multi-partitioned index theorem}\label{ewgkorpegrefwrf}

Let $M$ be a uniform bornological coarse space with a big family $\cZ$.
Assume that we have a finite family $\tilde f:=(\tilde f_{i})_{i=1,\dots,n}$ of uniformly continuous and controlled functions
$\tilde f_{i}:M\to \R$. Then we can iterate the constructions from \cref{koregperegfrefrfrefw}.  To $\tilde f_{i}$ we associate the big families $\cM_{i,\pm}$, as in \eqref{pclass},
$\cY_{i}$, and the class and $p_{i}$ in $K^{0}(\partial^{\cY_{i}}_{u}X)$.  For a subset $I$ of the poset  $ \langle n\rangle := \{1, \dots, n\}$, we further define 
\[
{
\cY_I :=\bigcap_{i\in I} \cY_{i} \ .
}
\]
If  $i \notin I$, then the decomposition {$(\cY_{I}\cap \cM_{i,-},\cY_{I}\cap \cM_{i,+})$ of $\cY_{I}$}
induces a Mayer-Vietoris boundary
\[
{
\partial^{MV}_{i}:\Sigma^{|I|}K_{\cZ}^{\cX}(\cY_{I})\to \Sigma^{|I|+1}K_{\cZ}^{\cX}(\cY_{I \cup \{i\}})\ .
}
\]
We avoid to use the embeddings \eqref{gwegerf3ferwfgreg} by considering excision with respect to the decomposition into two big families instead to the closed subspaces themselves, see e.g. the right triangle in \eqref{gweprkopwefklwpoekplml}. 
We have $\partial p_{i}\in K^{1}(\cY_{i})$ and 
 can form the product%\fml{Vorher hiess das $u$, aber das macht keinen Sinn, oder?}
\[
{
\partial p_{i_k}\cup \dots \cup \partial p_{i_1}\quad \mbox{in}\quad  K^{|I|}(\cY_{I} )\ , \qquad \text{for} \qquad I = \{i_1 < \dots < i_k\} \ .
}
  \]

By an iterated application of \cref{koppwegfrgwgregwe}
and the version of \cref{kogpqgregergwrgrgw} for $M$ replaced by a big family
we get:
\begin{kor}
For {$I = \{i_1 < \dots < i_k\} \subseteq \langle n\rangle$},
we have an   equivalence of functors
\[
(\partial p_{i_k}\cup \dots \cup \partial p_{i_1}) \cap^{\cX\sigma}-\simeq {\partial}_{i_k}^{MV}\circ \dots \circ  {\partial}^{MV}_{i_1}:K_{\cZ}^{\cX}(X)\to \Sigma^{|I|}K^{\cX}_{\cZ}(\cY_{I}) \ .
\]
 \end{kor}
 
Let ${{\cW}}$ be  a  big family on $M$ with the property
 $\tilde f({\cW})\subseteq \cB_{\R^{n}}$ so that $\tilde f$ becomes a morphism
 $\tilde f : (M,{\cW})\to (\R^{n},\cB_{\R^{n}})$ in $\BC^{(2)}$.
     We then  get an induced map 
     \[
     	\partial \tilde f : \partial^{{\cW}}M \to  \partial_h \R^n \to S^{n-1}\ .
     \]
If we perform the constructions above with the coordinate functions $(x_{i})_{i=1,\dots,n}$ of $\R^{n}$ instead of $(f_{i})_{i=1,\dots,n}$, then we get big families $\tilde \cY_{i}$ on $\R^{n}$ and classes $\partial \tilde p_{i}\in  K^{1}(\tilde \cY_{i} )$
such that $  ( \partial \tilde p_{n}\cup \dots \cup \partial \tilde p_{1})\in K^{n}(\tilde \cY_{{\langle n \rangle}})\cong  K_{c}^{n}(\R^{n})\cong \Z$ is a generator.
In view of the well-understood  long exact sequence in $K$-theory 
\[
K^{*}_{c}(\R^{n})\to K^{*}(\R^{n})\to K^{*}(S^{n-1})\stackrel{\partial}{\to} K^{*+1}_{c}(\R^{n})
\]
 there exists a unique class $\tilde u$ in $K^{n-1}(S^{n-1})$ such that $\tilde u_{|*}=0$ (i.e., $\tilde u$ is in the reduced summand) and $\partial \tilde  u=( \partial \tilde p_{{n}}\cup \dots \cup \partial \tilde p_{1})$. We set $u:= (\partial \tilde f)^{*}
\tilde u$  in $K^{n}(\cY_{{\langle n\rangle }})$ and conclude that $\partial u=( \partial p_{n}\cup \dots \cup \partial p_{1})$.

\begin{kor}\label{wrekopgpergrefwrefw}
The following square commutes:
\[
\xymatrix{K_{\cZ}^{\cX}({\cW}) \ar[d]^{a_{{\cW},\cZ}} \ar[rr]^-{{\partial}_{n}^{MV}\circ \dots \circ {\partial}^{MV}_{1}}&&\Sigma^{n}K_{\cZ}^{\cX}({\cW}\cap \cY_{\langle n\rangle}) \ar[d]^{a_{{\cW}\cap\cY_{{\langle n\rangle}},\cZ}} \\ K\cX({\cW}\cap \cZ)\ar[rr]^{u\cap^{\cX}} &&\Sigma^{n}K\cX({\cW}\cap \cY_{{\langle n\rangle}}\cap \cZ) }\ . 
\]
 \end{kor}

 We assume that $M$ comes from a complete Riemannian manifold.
 We consider a Callias type generalized Dirac {operator}  $\Dirac+\Psi$ such that $\Psi$  is very positive away from the big  family ${\cW}$ and asymptotically constant  away from the big family $\cZ$, and $\Dirac$ is {positive} away from $\cZ$.
We assume that $\tilde f$ is smooth and has $0$ in $\R^{n}$ as a regular value, and we set $N:=\tilde f^{-1}(\{0\})$.
 
 \begin{kor}
 \label{kopegwegergefw} 
 If $N$ has a tubular neighbourhood of uniform width with product structures, then in $K\cX({\cW}\cap \cY_{{\langle n\rangle}}\cap \cZ)$ we have the equality
 \[
 u\cap^{\cX}\ind\cX(\Dirac+\Psi,\mathrm{on}\,{\cW}\cap \cZ)\simeq \beta^{-n} i(\ind\cX((\Dirac+\Psi)_{|N},\mbox{on $N\cap {\cW}\cap \cZ$ }))\ .
 \]
 \end{kor}

\begin{proof}
We have
\begin{eqnarray*}
u\cap^{\cX}\ind\cX(\Dirac+\Psi,\mathrm{on}\,{\cW}\cap \cZ)
&\stackrel{\eqref{qefewdad}}{=}&
u\cap^{\cX} a_{{\cW},\cZ}(\sigma_{\cZ}(\Dirac+\Psi,\mathrm{on}\,{\cW}))
\\
&\stackrel{\text{\cref{wrekopgpergrefwrefw}}}{=}&
a_{{\cW}\cap {\cY_{{\langle n\rangle}}}, \cZ}( {\partial}_{n}^{MV}\circ \dots \circ {\partial}^{MV}_{1}   \sigma_{\cZ} (\Dirac+\Psi,\mathrm{on}\,{\cW}))
\\&=&a_{{\cW}\cap {\cY_{{\langle n\rangle}}}, \cZ}( {i^{\cX}\delta}_{n}^{MV}\circ \dots \circ {\delta}^{MV}_{1}   \sigma_{\cZ} (\Dirac+\Psi,\mathrm{on}\,{\cW}))
\\
&\stackrel{\text{\cref{tkopherthtere9}}}{=}&
 \beta^{-n} a_{{\cW}\cap \cY_{1,\dots,n}, \cZ}( i^{\cX}\sigma_{N\cap \cZ} ((\Dirac+\Psi)_{|N},\mathrm{on}\,N\cap {\cW}\cap \cZ))
 \\
 &\stackrel{\eqref{vfdsvervfdvsdfvsdvdf}}{=}& 
 \beta^{-n} i a_{N\cap {\cW}, \cZ}  (\sigma_{N\cap \cZ} ((\Dirac+\Psi)_{|N},\mathrm{on}\,N\cap {\cW}\cap \cZ))
 \\
 &\stackrel{\eqref{qefewdad}}{=}&
 \beta^{-n} i(\ind\cX((\Dirac+\Psi)_{|N},\mathrm{on}\,N\cap {\cW}\cap \cZ))\ .
\end{eqnarray*}
\end{proof}
   
 \begin{ex} 
We assume that ${\cW}\cap \cY_{{\langle n\rangle}}\cap \cZ\subseteq \cB$. Then
for every  member $ 
  W$ in ${\cW}\cap \cY_{{\langle n\rangle}}\cap \cZ$ the restriction $\tilde f_{|W}:W\to \R^{n}$ is proper. 
   Consequently we get a map 
   \begin{equation}\label{dfvdsfvwervsvsdfvsf} t_{{\tilde f}}:K\cX({\cW}\cap \cY_{{\langle n\rangle}}\cap \cZ)\stackrel{{\tilde f}_{*}}{\longrightarrow} K\cX(\R^{n}) \stackrel{ \tilde u\cap^{\cX}- }{\simeq } \Sigma^{n}K\cX(\cB_{\R^{n}})\simeq \Sigma^{n}KU
\end{equation} 
    
 If $\Dirac+\Psi$ is a  Callias type  operator of degree $k$ as above, then 
 following \cite[Def. 1.3]{MR3825192} we define the multipartitioned index by 
  $$
  t_{{\tilde f}}(u\cap^{\cX}\ind\cX(\Dirac+\Psi,\mathrm{on}\,{\cW}\cap  \cZ))\in \pi_{-n-k}KU\ .$$
 We furthermore have 
 $N\cap {\cW}\cap \cZ\subseteq N\cap \cB$ and therefore
 $$ {\tilde f}_{|N,*} \ind \cX((\Dirac+\Psi)_{|N}, \mathrm{on}\,N\cap {\cW}\cap \cZ )\in  K\cX_{n-k}(\cB _{\R^{n}})\simeq \pi_{n-k}KU\ .$$
 The following consequence of \cref{kopegwegergefw} is a generalization of the multi-partition index theorem \cite[Thm. 2.7]{MR3704253}, \cite[Thm. 1.4]{MR3825192}.

 \begin{kor}
 Under the assumptions of \cref{kopegwegergefw} we have 
 $$
 t_{{\tilde f}}(u\cap^{\cX}\ind\cX(\Dirac+\Psi,\mathrm{on}\,{\cW} \cap \cZ))= \beta^{-n} {\tilde f}_{|N,*} \ind \cX((\Dirac+\Psi)_{|N}, \mathrm{on}\,N\cap {\cW}\cap \cZ ) \ .$$
 \end{kor}
\end{ex}

\subsection{Slant products}\label{tlkhgwteggfregwgwergwregwregw}

In this section we interpret the main construction from \cite{Engel_2022}.
Let $X, Y$ be two uniform bornological coarse spaces.
Let $\cB_{Y}$ %\fuli{habe weiter unten alles daran angepaÃÂt. Nicht rot. Ich hoffe, daÃ es jetzt mehr Sinn macht.} 
be the bornology of $Y$ and consider the big family $X\times \cB_{Y}$ on
$X\otimes Y$. Pull-back of functions along the projection $\pr:X\times Y\to Y$ induces a map of instances of the exact sequence of $C^{*}$-algebras \eqref{gerwrgfrq} %\fml{Was ist hier $Y$? Und warum oben $C_0$ statt $C_u$?}
$$
\xymatrix{
0\ar[r]&C_{u}(\cB_{Y})\ar[d]\ar[r]&C_{u,\cB_Y}(Y)\ar[r]\ar[d]&C(\partial^{\cB_{Y}}_{u}Y)\ar[r]\ar[d]&0\\
0\ar[r]&C_{u}(X\otimes \cB_{Y})\ar[r]&C_{u,X\otimes \cB_{Y}}(X\otimes Y)\ar[r]&C(\partial_{u}^{X\otimes \cB_{Y}}(X\otimes Y))\ar[r]&0}\ .
$$

This gives a commutative square $$ \xymatrix{K(\partial^{\cB_{Y}}_{u}Y)\ar[r]^{\partial}\ar[d]^{\partial_{h}\pr^{*}} &\Sigma K(\cB_{Y}) \ar[d]^{\pr^{*}} \\ K(\partial^{X\otimes \cB_{Y}}_{u}(X\otimes Y))\ar[r]^{\partial} & \Sigma K(X\otimes \cB_{Y})}\ . $$
If $\theta$ is a class in $K_{k}(\partial^{\cB_{Y}}_{h}Y)$, then by \cref{erokgpwergrwefwerfwerfwer}
we get a commutative  diagram
$$\xymatrix{K^{\cX}(X\otimes Y) \ar[rr]^-{\pr^{*}\partial \theta\cap^{\cX\sigma}-}\ar[d]^{a_{X\otimes Y}} && \Sigma^{k+1}K^{\cX}(X\otimes \cB_{Y})\ar[d]^{a_{X\otimes B_{\cY}}} \ar[r]^{\kappa^{\sigma}}&\Sigma^{k+1}K^{\cX}(X)\ar[d]^{a_{X}}\\ K\cX(X\otimes Y)\ar[rr]^-{(\partial_{h}\pr)^{*}\theta\cap^{\cX}-} && \Sigma^{k+1} K\cX(X\otimes\cB_{Y}) \ar[r]^{\kappa}&\Sigma^{k+1}K\cX(X)} \ .
$$
The maps $\kappa$ and $\kappa^{\sigma}$ are well-defined since the restriction of $\pr$ to any member of $X\otimes\cB_{Y}$ is proper and hence a morphism in $\UBC$ and $\BC$. The horizontal compositions are the slant products first defined in \cite{Engel_2022} and denoted there by $/\theta$.
The fact that the outer square commutes yields a map of fibre sequences
   $$\xymatrix{
\Fib(a_{X\otimes Y})\ar@{..>}[d]^{/\theta}\ar[r]&K^{\cX}(X\otimes Y)\ar[r]^{a_{X\otimes Y}}\ar[d]^{/\theta}&K\cX(X\otimes Y) \ar[d]^{/\theta} \\
\Sigma^{k+1}\Fib(a_{X}) \ar[r]&\Sigma^{k+1}K^{\cX}(X)\ar[r]^{a_{X}}&\Sigma^{k+1}K\cX(X)}\ ,
$$
where the dotted arrow is determined by the universal property of the fibre. The commutativity 
 expresses the compatibility \cite[Thm. 4.10]{Engel_2022} of the slant product with the Higson-Roe sequence, see \cref{kopgwergregfrfwf}.
Most of  the formal properties of the slant product stated in \cite{Engel_2022}    
 are  consequences of the formal properties of the pairings stated in \cref{regijowergefrfwrfrf} and \cref{regijowergefrfwrfrf2}.

 \bibliographystyle{alpha}
\bibliography{forschung2021}

\end{document}